\documentclass[reqno, 12pt]{amsart}     % Basic GT/GTM/AGT style
\usepackage{amsmath}
\usepackage{amsfonts}
\usepackage{amssymb}
\usepackage{amsthm}
\usepackage{mathrsfs}
\usepackage{bbm}
\usepackage{graphicx, color}
\usepackage{tikz}
\usetikzlibrary{calc,decorations.markings, arrows.meta, positioning, cd, intersections, calc, decorations.pathmorphing, arrows, decorations.pathreplacing}

\usepackage{adjustbox}

\usepackage{bmpsize}

\usepackage{mathtools}

\usepackage{enumerate}

\definecolor{refkey}{rgb}{0.9451,0.2706,0.4941}
\definecolor{labelkey}{rgb}{0.9451,0.2706,0.4941}

\definecolor{mygreen}{rgb}{0,0.7,0.3}
\definecolor{myblue}{rgb}{0,0.50,1.20}%cyan={0,1.74,2.39}
\definecolor{myorange}{rgb}{1,0.5,0.1}

\usepackage{subfiles}
\usepackage
[colorlinks, linkcolor=blue, citecolor=red, urlcolor=cyan,pagebackref]
{hyperref}

\allowdisplaybreaks

\numberwithin{equation}{section}

\usepackage{cleveref}
\crefname{thm}{Theorem}{Theorems}
\crefname{cor}{Corollary}{Corollaries}
\crefname{lem}{Lemma}{Lemmas}
\crefname{sublem}{Sublemma}{Sublemmas}
\crefname{prop}{Proposition}{Propositions}
\crefname{dfn}{Definition}{Definitions}
\crefname{defi}{Definition}{Definitions}
\crefname{ex}{Example}{Examples}
\crefname{claim}{Claim}{Claims}
\crefname{conj}{Conjecture}{Conjectures}
\crefname{conv}{Notation}{Notations}
\crefname{rem}{Remark}{Remarks}
\crefname{rmk}{Remark}{Remarks}
\crefname{prob}{Problem}{Problems}
\crefname{figure}{Figure}{Figures}
\crefname{table}{Table}{Tables}
\crefname{section}{Section}{Sections}
\crefname{subsection}{Section}{Sections}
\crefname{appendix}{Appendix}{Appendices}
\crefname{introthm}{Theorem}{Theorems}
\crefname{introcor}{Corollary}{Corollaries}
\crefname{introconj}{Conjecture}{Conjectures}

\newtheorem{thm}{Theorem}[section]
\newtheorem{prop}[thm]{Proposition}
\newtheorem{cor}[thm]{Corollary}
\newtheorem{lem}[thm]{Lemma}

\newtheorem{introthm}{Theorem}

\newtheorem{introconj}[introthm]{Conjecture}

\theoremstyle{definition}
\newtheorem{dfn}[thm]{Definition}

\newtheorem{ex}[thm]{Example}

\newtheorem{conj}[thm]{Conjecture}
\newtheorem{conv}[thm]{Notation}
\theoremstyle{remark}

\newtheorem{rem}[thm]{Remark}

\newcommand{\corrected}
{\textcolor{red}{(corrected)}}

\newcommand*{\chom}{\mathcal{H}\kern -.5pt om}

\newcommand{\bZ}{\mathbb{Z}}
\newcommand{\bQ}{\mathbb{Q}}
\newcommand{\bR}{\mathbb{R}}

\newcommand{\bT}{\mathbb{T}}
\newcommand{\bM}{\mathbb{M}}

\newcommand{\bD}{\mathbb{D}}
\newcommand{\bP}{\mathbb{M}_\circ}
\newcommand{\bG}{\mathbb{G}}
\newcommand{\bA}{\mathbb{A}}
\newcommand{\bB}{\mathbb{B}}
\newcommand{\bE}{\mathbb{E}}

\newcommand{\bs}{{\boldsymbol{s}}}

\newcommand{\A}{\mathcal{A}}%----------shortened
\newcommand{\cA}{\mathcal{A}}

\newcommand{\cC}{\mathcal{C}}

\newcommand{\cF}{\mathcal{F}}

\newcommand{\cL}{\mathcal{L}}

\newcommand{\cO}{\mathcal{O}}
\newcommand{\cP}{\mathcal{P}}
\def\P{{\mathcal{P}}}

\newcommand{\cW}{\mathcal{W}}
\newcommand{\X}{\mathcal{X}}%--------------shortened
\newcommand{\cX}{\mathcal{X}}

\newcommand{\cZ}{\mathcal{Z}}

\newcommand{\sfa}{\mathsf{a}}
\newcommand{\sfx}{\mathsf{x}}

\newcommand{\Hom}{\mathrm{Hom}}

\newcommand{\tri}{\triangle}
\DeclareMathOperator{\sgn}{\mathrm{sgn}}
\newcommand{\trop}{\mathrm{trop}}
\newcommand{\pos}{\mathbb{R}_{>0}}

\newcommand{\eML}{\widehat{\mathcal{ML}}}

\newcommand{\bExch}{\bE \mathrm{xch}}
\newcommand{\uf}{\mathrm{uf}}
\newcommand{\f}{\mathrm{f}}

\newcommand{\Teich}{Teichm\"uller}
\newcommand{\fsl}{\mathfrak{sl}}
\newcommand{\sfp}{\mathsf{p}}
\renewcommand{\cL}{\mathcal{L}_{\fsl_3}}
\newcommand{\picW}[1]{\langle \cW_{#1} \rangle}
\newcommand{\pictW}[1]{\langle W_{#1} \rangle}

\newcommand{\ve}{\varepsilon}
\newcommand{\sfs}{{\mathsf{s}}}

\newcommand{\hL}{\widehat{L}}

\newcommand{\bott}{E,1}
\newcommand{\topp}{E,2}

\DeclareMathOperator{\interior}{\mathrm{int}}

\DeclareMathOperator{\Spec}{\mathrm{Spec}}

\makeatletter
\newcommand{\oset}[3][0ex]{%
  \mathrel{\mathop{#3}\limits^{
    \vbox to#1{\kern-2\ex@
    \hbox{$\scriptstyle#2$}\vss}}}}
\makeatother
\newcommand{\overbar}[1]{\oset{#1}{-\!\!\!-\!\!\!-}}

\makeatletter
\newcommand{\osetnear}[3][0ex]{%
  \mathrel{\mathop{#3}\limits^{
    \vbox to#1{\kern-.3\ex@
    \hbox{$\scriptstyle#2$}\vss}}}}
\makeatother
\newcommand{\overbarnear}[1]{\osetnear{#1}{-\!\!\!-\!\!\!-}}

\newcommand\qarrow[2]{\draw[->,shorten >=2pt,shorten <=2pt] (#1) -- (#2) [thick];} %arrow
\newcommand\qdrarrow[2]{\draw[->,dashed,shorten >=2pt,shorten <=2pt,bend right=0.5cm] (#1) to (#2) [thick];} %dashed arrow, right bend
\newcommand\qdlarrow[2]{\draw[->,dashed,shorten >=2pt,shorten <=2pt,bend left=0.5cm] (#1) to (#2) [thick];} %dashed arrow, left bend
\newcommand\qarrowbr[2]{\draw[->,shorten >=2pt,shorten <=2pt,bend right=0.5cm] (#1) to (#2) [thick];} %arrow_right bend
\newcommand\qarrowbl[2]{\draw[->,shorten >=2pt,shorten <=2pt,bend left=0.5cm] (#1) to (#2) [thick];} %arrow_left bend

\def\centerarc(#1)(#2:#3:#4)% Syntax: [draw options] (center) (initial angle:final angle:radius)
    {($(#1)+({#4*cos(#2)},{#4*sin(#2)})$) arc(#2:#3:#4) }

\tikzset{
  mid arrow/.style={postaction={decorate,decoration={
        markings,
        mark=at position .5 with {\arrow[#1]{stealth}}
      }}},
}

\tikzset{
    partial ellipse/.style args={#1:#2:#3}{
        insert path={+ (#1:#3) arc (#1:#2:#3)}
    }
}

\newcommand{\quiverplus}[3]{
\begin{scope}[>=latex]
{\color{mygreen}
    \path(#1) coordinate(x1);
    \path(#2) coordinate(x2);
    \path(#3) coordinate(x3);
    \foreach \l in {1,2}
    {
        \draw($(x1)!0.333*\l!(x2)$) circle(2pt) coordinate(x12\l);
        \draw($(x2)!0.333*\l!(x3)$) circle(2pt) coordinate(x23\l);
        \draw($(x3)!0.333*\l!(x1)$) circle(2pt) coordinate(x31\l);
    }
    \path($(x1)!0.5!(x2)$) coordinate(H);
    \draw($(x3)!0.667!(H)$) circle(2pt) coordinate(G);
    \qarrow{x121}{G}
    \qarrow{x231}{G}
    \qarrow{x311}{G}
    \qarrow{G}{x122}
    \qarrow{G}{x232}
    \qarrow{G}{x312}
    \qarrow{x312}{x121}
    \qarrow{x122}{x231}
    \qarrow{x232}{x311}
}
\end{scope}
}

\newcommand{\quiverminus}[3]{
\begin{scope}[>=latex]
{\color{mygreen}
    \path(#1) coordinate(x1);
    \path(#2) coordinate(x2);
    \path(#3) coordinate(x3);
    \foreach \l in {1,2}
    {
        \draw($(x1)!0.333*\l!(x2)$) circle(2pt) coordinate(x12\l);
        \draw($(x2)!0.333*\l!(x3)$) circle(2pt) coordinate(x23\l);
        \draw($(x3)!0.333*\l!(x1)$) circle(2pt) coordinate(x31\l);
    }
    \path($(x1)!0.5!(x2)$) coordinate(H);
    \draw($(x3)!0.667!(H)$) circle(2pt) coordinate(G);
    \path($(x1)!0.5!(x2)$) coordinate(z12);
    \path($(x2)!0.5!(x3)$) coordinate(z23);
    \path($(x3)!0.5!(x1)$) coordinate(z31);
    \path($(G)!0.7!(z12)$) coordinate(y12);
    \path($(G)!0.7!(z23)$) coordinate(y23);
    \path($(G)!0.7!(z31)$) coordinate(y31);
    \qarrow{x122}{G}
    \qarrow{x232}{G}
    \qarrow{x312}{G}
    \qarrow{G}{x121}
    \qarrow{G}{x231}
    \qarrow{G}{x311}
    \draw[<-,shorten >=2pt,shorten <=2pt] (x232) ..controls (y23) and (y12).. (x121) [thick];
    \draw[<-,shorten >=2pt,shorten <=2pt] (x312) ..controls (y31) and (y23).. (x231) [thick];
    \draw[<-,shorten >=2pt,shorten <=2pt] (x122) ..controls (y12) and (y31).. (x311) [thick];
}
\end{scope}
}

\newcommand{\quiversquare}[4]{
{\color{mygreen}
    \path(#1) coordinate(x1);
    \path(#2) coordinate(x2);
    \path(#3) coordinate(x3);
	\path(#4) coordinate(x4);
    \foreach \l in {1,2}
    {
        \draw($(x1)!0.333*\l!(x2)$) circle(2pt) coordinate(x12\l);
        \draw($(x2)!0.333*\l!(x3)$) circle(2pt) coordinate(x23\l);
        \draw($(x3)!0.333*\l!(x4)$) circle(2pt) coordinate(x34\l);
		\draw($(x4)!0.333*\l!(x1)$) circle(2pt) coordinate(x41\l);
		\draw($(x1)!0.333*\l!(x3)$) circle(2pt) coordinate(x13\l);	
	    \draw($(x2)!0.333*\l!(x4)$) circle(2pt) coordinate(x24\l);
	}
}
}

\newcommand{\squarefrozen}[8]{
{\color{mygreen}
\node[below right,scale=0.8] at (x121) {$#1$}; 
\node[below right,scale=0.8] at (x122) {$#2$};
\node[above right,scale=0.8] at (x231) {$#3$};
\node[above right,scale=0.8] at (x232) {$#4$};
\node[above left,scale=0.8] at (x341) {$#5$};
\node[above left,scale=0.8] at (x342) {$#6$};
\node[below left,scale=0.8] at (x411) {$#7$};
\node[below left,scale=0.8] at (x412) {$#8$};
}
}

\newcommand{\squareunfrozen}[4]{
{\color{mygreen}
\node[below right,scale=0.8] at (x131) {$#1$};
\node[above right,scale=0.8] at (x132) {$#2$};
\node[right,scale=0.8] at (x241) {$#3$};
\node[left,scale=0.8] at (x242) {$#4$};
}
}

\newcommand{\CoG}[3]{
    \path(#1) coordinate(x1);
    \path(#2) coordinate(x2);
    \path(#3) coordinate(x3);
    \path($(x1)!0.5!(x2)$) coordinate(H);
    \path($(x3)!0.667!(H)$) circle(2pt) coordinate(G);}
\newcommand{\sink}[3]{
    \CoG{#1}{#2}{#3}
    \draw[red,very thick,->-] (#1) -- (G);
    \draw[red,very thick,->-] (#2) -- (G);
    \draw[red,very thick,->-] (#3) -- (G);
}   
\newcommand{\source}[3]{
    \CoG{#1}{#2}{#3}
    \draw[red,very thick,-<-] (#1) -- (G);
    \draw[red,very thick,-<-] (#2) -- (G);
    \draw[red,very thick,-<-] (#3) -- (G);
}   

\newcommand{\bline}[3]{
    \path (#1)++(0,-#3) coordinate(m1);
    \path (#2)++(0,-#3) coordinate(m2);
    \filldraw[gray!30] (m1) -- (#1) -- (#2) -- (m2) --cycle;
    \draw[thick] (#1) -- (#2);
}
\newcommand{\tline}[3]{
    \path (#1)++(0,#3) coordinate(m1);
    \path (#2)++(0,#3) coordinate(m2);
    \filldraw[gray!30] (m1) -- (#1) -- (#2) -- (m2) --cycle;
    \draw[thick] (#1) -- (#2);
}

\tikzset{->-/.style 2 args={
	postaction={decorate},
	decoration={markings, mark=at position #1 with {\arrow[thick, #2]{>}}} 
    },
    ->-/.default={0.5}{}
}%refer arrow tips
\tikzset{-<-/.style 2 args={
	postaction={decorate},
	decoration={markings, mark=at position #1 with {\arrow[thick, #2]{<}}} 
    },
    -<-/.default={0.5}{}
}

\tikzset{
	overarc/.style={
		white, double=red, double distance=1.2pt, line width=2.4pt
	}
}

\newcommand{\Hweb}[4]{
\path($(#1)!0.5!(#3)$) coordinate(P);
\CoG{#1}{#2}{P}
    \draw[red,thick,->-] (#1) -- (G);
    \draw[red,thick,->-] (#2) -- (G);
    \path(G) coordinate (G');
\CoG{#3}{#4}{P}
    \draw[red,thick,-<-] (#3) -- (G);
    \draw[red,thick,-<-] (#4) -- (G);
\draw[red,thick,->-] (G) -- (G');
}

\newcommand{\arr}[1]{
    \draw[red,very thick,-latex] (#1)++(-0.18,0) --++(0.36,0);
}
\newcommand{\arl}[1]{
    \draw[red,very thick,-latex] (#1)++(0.18,0) --++(-0.36,0);
}

\pgfdeclarelayer{bg}    % declare background layer
\pgfsetlayers{bg,main}  % set the order of the layers (main is the standard layer)

\title{Unbounded $\mathfrak{sl}_3$-laminations and their shear coordinates}

\author[Tsukasa Ishibashi]{Tsukasa Ishibashi}
\address{Tsukasa Ishibashi, Mathematical Institute, Tohoku University, 
6-3 Aoba, Aramaki, Aoba-ku, Sendai, Miyagi 980-8578, Japan.}
\email{tsukasa.ishibashi.a6@tohoku.ac.jp}

\author[Shunsuke Kano]{Shunsuke Kano}
\address{Shunsuke Kano, Research Alliance Center for Mathematical Sciences, Tohoku University,
6-3 Aoba, Aramaki, Aoba-ku, Sendai, Miyagi 980-8578, Japan.}
\email{s.kano@tohoku.ac.jp}

\begin{document}

\begin{abstract}
Generalizing the work of Fock--Goncharov on rational unbounded laminations, we give a geometric model of the tropical points of the 
cluster variety $\mathcal{X}_{\mathfrak{sl}_3,\Sigma}$, which we call \emph{unbounded $\mathfrak{sl}_3$-laminations}, 
%moduli space $\mathcal{X}_{PGL_3,\Sigma}$ of framed $PGL_3$-local systems on a marked surface $\Sigma$ 
based on the Kuperberg's $\mathfrak{sl}_3$-webs.
We introduce their tropical cluster coordinates as an $\mathfrak{sl}_3$-analogue of the Thurston's shear coordinates associated with any ideal triangulation. 
%We also describe tropical points of Goncharov--Shen's moduli space $\mathcal{P}_{PGL_3,\Sigma}$. 
As 
%Then we give 
a tropical analogue of gluing morphisms among the moduli spaces $\mathcal{P}_{PGL_3,\Sigma}$ of Goncharov--Shen, we describe a geometric gluing procedure of unbounded $\mathfrak{sl}_3$-laminations with pinnings via ``shearings''. 
We also investigate a relation to the graphical basis of the $\mathfrak{sl}_3$-skein algebra \cite{IY21}, which conjecturally leads to a quantum duality map. 
\end{abstract}

\maketitle

%%%%%%%%%%%%%%%%%%%%   Start of main body of article

\section{Introduction}
%%% Measured laminations in classical, their role from the cluster point of view.

\subsection{Background}
The notion of measured geodesic laminations (or its equivalents, measured foliations) on a surface has been first introduced by W.~Thurston \cite{Thu88}, as a powerful geometric tool to study the mapping class groups and the large-scale geometry of the \Teich\ space. 
After a couple of decades, Fock--Goncharov \cite{FG07} studied Thurston's \emph{shear coordinates} on the space $\eML(\Sigma)$ of (enhanced) measured geodesic laminations on a marked surface $\Sigma$, which gives a global coordinate system parametrized by the interior edges of an ideal triangulations $\tri$ of $\Sigma$: $\eML(\Sigma) \xrightarrow{\sim} \bR^{e_{\interior}(\tri)}$. Moreover, they observed that these coordinates can be viewed as a ``tropical analogue'' of the cross-ratio coordinates\footnote{The cross-ratio coordinate is an exponential version of the shear coordinate on the \Teich\ space. In this paper, we always use the term ``shear coordinates'' for those on the lamination spaces.}
%and what we called the shear coordinates are called lamination shear coordinates. In this paper, we deal with only lamination shear coordinates so we refer to it as simply shear coordinate.
%To avoid confusion, we refer to the usual shear coordinate as cross-ratio coordinate in this section.}
on the enhanced \Teich\ space $\widehat{\mathcal{T}}(\Sigma) \xrightarrow{\sim} \bR^{e_{\interior}(\tri)}_{>0}$ studied by Chekhov--Fock \cite{CF}, as their coordinate transformation rule is exactly the tropical analogue of that for the latter. These facts indicate that there would be a universal algebraic object 
% (scheme)
behind the \Teich\ and lamination spaces: this idea leads to 
%by identifying their coordinate transformation formulae for flips with certain tropical cluster $\X$-transformations. 
%In terms of 
the theory of \emph{cluster varieties} developed in \cite{FG09}. In their terms, there is a cluster $\X$-variety\footnote{Here, the superscript ``uf'' just indicates that it has only unfrozen coordinates. It corresponds to the situation where the shear/cross-ratio coordinates are defined only for internal edges $e_\mathrm{int}(\tri)$ of an ideal triangulation $\tri$.} %\footnote{Cluster $\cX$-/$\cA$-variety is a scheme equipped with an atlas whose coordinate changes are given by cluster $\cX$-/$\cA$-transformations.} 
$\X_\Sigma^\uf$ associated with $\Sigma$ such that 
the spaces $\widehat{\mathcal{T}}(\Sigma)$ and  $\eML(\Sigma)$ are naturally identified with the spaces $\X_\Sigma^\uf(\pos)$, $\X_\Sigma^\uf(\bR^T)$ of positive real points and the real tropical points, respectively. We call the latter space $\X_\Sigma^\uf(\bR^T)$ the \emph{tropical cluster $\X$-variety} for short. 
%tropical cluster $\X$-variety $\X_\Sigma^\uf(\bR^T)$ associated with $\Sigma$, where the space $\widehat{\mathcal{T}}(\Sigma)$ is identified with the positive space $\X_\Sigma^\uf(\pos)$. 

In general, cluster varieties are schemes constructed from combinatorial data $\sfs$ (such as quivers) equipped with a birational atlas whose coordinate changes are given by specific rational transformations, called cluster transformations (see \cref{sec:appendix} for a short review of this theory). They always come in a dual pair $(\A_\sfs,\X_\sfs)$, forming a \emph{cluster ensemble}. The \emph{duality conjecture} is a profound conjecture of Fock--Goncharov \cite{FG09} that asks for a construction of ``duality maps"
%From the cluster algebra point of view, this can be served as a nice geometric model for the tropical space $\X_\Sigma^\uf(\bR^T)$. 
%It leads to the first construction of Fock--Goncharov duality map \cite{FG03}
\begin{align*}
    \mathbb{I}_\X: \X_\sfs(\bZ^T) \to \cO(\A_{\sfs^\vee}) \quad 
    (\text{resp. } \mathbb{I}_\A: \A_\sfs(\bZ^T) \to \cO(\X_{\sfs^\vee}))
\end{align*}
%for a punctured surface $\Sigma$.
which parametrizes a linear basis of the function ring $\cO(\A_{\sfs^\vee})$ (resp. $\cO(\X_{\sfs^\vee})$) of the dual cluster variety
%$\A$-variety (resp. $\cX$-variety) 
by the space $\X_\sfs(\bZ^T) \subset \X_\sfs(\bR^T)$ (resp. $\A_\sfs(\bZ^T) \subset \A_\sfs(\bR^T)$) of integral tropical points, satisfying certain strong axioms such as the positivity of structure constants. 

In the surface case, the spaces $\A_\Sigma(\pos)$ and $\A_\Sigma(\bR^T)$ are identified with the decorated \Teich\ and lamination spaces \cite{Penner,PP93} via the $\lambda$-length and intersection coordinates, respectively \cite{FG07}. 
The geometric realization of the tropical spaces $\A_\Sigma(\bZ^T)$ $\X_\Sigma^\uf(\bZ^T)$ by integral laminations \cite{FG07} leads to a topological construction of the duality maps $\mathbb{I}_\cX$ and $\mathbb{I}_\A$, and their required properties are proved recently by Mandel--Qin \cite{MQ} based on a comparison with the \emph{theta basis} of Gross--Hacking--Keel--Kontsevich \cite{GHKK}.
% We can also consider the dual version $\mathbb{I}_\A: \A_\sfs(\bZ^T) \to \cO(\X_{\sfs^\vee})$.
% Also, 
% % In the surface case, 
% $\A_\Sigma(\bZ^T)$ is identified with the space of bounded integral laminations (possibly with peripheral components) via intersection coordinates \cite{FG07}. 
These duality maps are two kinds of generalizations of the trace function basis for the function ring of the $SL_2$-character variety of a closed surface, parametrized by loops. 
%We will discuss other aspects of the identification $\eML(\Sigma)=\X_\Sigma^\uf(\bR^T)$ from the viewpoint of \Teich-Thurston theory in \cref{subsec:future} below.

\bigskip

%%% Expected generalizations for higher rank.
Strongly expected are ``higher rank" generalizations of the above picture. 
The cluster varieties $\cX^\uf_\Sigma$ and $\cA_\Sigma$ are birationally isomorphic to certain generalizations of the $PGL_2$- and $SL_2$-character varieties, respectively \cite{FG03}.
As a generalization for higher rank algebraic groups, there are cluster varieties $\cX^\uf_{\mathfrak{g}, \Sigma}$ and $\cA_{\mathfrak{g}, \Sigma}$ which are birationally isomorphic to the same kind of generalizations $\X_{G',\Sigma}$ and $\A_{G,\Sigma}$ of character varieties \cite{FG03,Le19,GS19}, where the defining combinatorial data for these cluster varieties only depend on the surface $\Sigma$ and a semisimple Lie algebra $\mathfrak{g}$. 
%, so we write $\cX^\uf_{\mathfrak{g}, \Sigma}$ and $\cA_{\mathfrak{g}, \Sigma}$.
In particular, $\cX^\uf_{\mathfrak{sl}_2, \Sigma} = \cX^\uf_\Sigma$ and $\cA_{\mathfrak{sl}_2, \Sigma} = \cA_\Sigma$ correspond to the case mentioned above. Goncharov--Shen \cite{GS19} introduced a cluster variety $\cX_{\mathfrak{g}, \Sigma}$ with frozen coordinates, which is birational to some extension $\P_{G',\Sigma}$ of $\X_{G',\Sigma}$. 
% Strongly expected are higher rank generalizations, replacing the Lie algebra $\fsl_2$ with any semisimple Lie algebra $\mathfrak{g}$. Indeed, there are cluster varieties $\A_{\mathfrak{g},\Sigma}$ and $\X_{\mathfrak{g},\Sigma}$ 
% %associated with the pair $(\mathfrak{g},\Sigma)$ that encodes the cluster structures of 
% related to the moduli spaces of $G$-local systems on $\Sigma$ with certain decoration data \cite{FG03,Le19,GS19}. 
%Here $G$ is the simply-connected algebraic group, and $G'$ is its adjoint group such that $\mathfrak{g}=\mathrm{Lie}(G)=\mathrm{Lie}(G')$. 
Hereupon, 
% In particular,
we have combinatorially defined tropical spaces $\A_{\mathfrak{g},\Sigma}(\bR^T)$ and $\X_{\mathfrak{g},\Sigma}(\bR^T)$,
% , which should parametrize linear bases of the function rings of the dual varieties with good properties.  
which should parametrize linear bases of the function rings of the dual varieties with good properties by the duality conjecture.
%from the viewpoint of Fock--Goncharov duality conjecture \cite{FG09}. 
The spaces $\A_{\mathfrak{g},\Sigma}(\bR^T)$ and $\X_{\mathfrak{g},\Sigma}(\bR^T)$ are widely expected to be certain spaces of $\mathfrak{g}$-webs on $\Sigma$, so that the duality maps are built from the web functions on the character variety. However, such a web description is still missing in general. 
We remark here that Le \cite{Le16} gave a description of these spaces in terms of certain configurations in the affine buildings, which should be ultimately related to $\mathfrak{g}$-webs based on the geometric Satake correspondence (see, for instance, \cite{FKK}). 

%%% The sl_3-case, recent progress. 
For the first non-trivial case $\mathfrak{g}=\fsl_3$, a major progress on the space $\A_{\mathfrak{sl}_3,\Sigma}(\bZ^T)$ has been made by Douglas--Sun \cite{DS20I,DS20II} and Kim \cite{Kim21}. They describe this space as an appropriate space of Kuperberg's $\fsl_3$-webs \cite{Kuperberg} by introducing an $\fsl_3$-version of the intersection coordinates with an ideal triangulation. 
%In their work, the Fock--Goncharov duality point of view plays a role of ansatz to define the intersection coordinates correctly (see also the earlier work by Xie \cite{Xie}). 
Their coordinates can also be extended to the space $\A_{\mathfrak{sl}_3,\Sigma}(\bQ^T)$ by scaling equivariance.

\subsection{Geometric model for the tropical space \texorpdfstring{$\X^\uf_{\mathfrak{sl}_3,\Sigma}(\bQ^T)$}{X(sl3,Sigma)}}
Our aim in this paper is to describe the tropical cluster variety $\X_{\mathfrak{sl}_3,\Sigma}(\bQ^T)$ on the dual side as a space of $\fsl_3$-webs with a different type of boundary conditions and some additional structures at punctures. 
We introduce the space $\cL^x(\Sigma,\bQ)$ of rational unbounded $\fsl_3$-laminations on $\Sigma$, which are certain equivalence classes of non-elliptic \emph{signed $\fsl_3$-webs} with positive rational weights (see \cref{subsec:webs}).
Then we define an $\fsl_3$-version of the shear coordinates of these objects with respect to an ideal triangulation $\tri$. As in the $\fsl_2$-case, we need to perturb the ends incident at punctures (and thus make them spiralling) so that they intersect with $\tri$ transversely. The spiralling directions are controlled by the signs assigned to each end of the $\fsl_3$-web, and this procedure leads to the notion of \emph{spiralling diagrams} (\cref{def:spiralling}) associated with signed $\fsl_3$-webs. After a careful study on the ``good positions'' of a spiralling diagram, we obtain well-defined shear coordinates. 

\begin{introthm}[\cref{thm:cluster_transf_unfrozen}]\label{introthm:shear_X}
For any marked surface $\Sigma$ satisfying the conditions {\rm (S1)--(S4)} in \cref{subsec:notation_marked_surface} and its ideal triangulation $\tri$ without self-folded triangles, we have a bijection
\begin{align}\label{eq:shear_X}
    \sfx^\uf_\tri: \cL^x(\Sigma,\bQ) \xrightarrow{\sim} \bQ^{I_\uf(\tri)},
\end{align}
which we call the shear coordinate system associated with $\tri$. Moreover, for any another ideal triangulation $\tri'$ of $\Sigma$, the coordinate transformation $\sfx_{\tri'}\circ \sfx_\tri^{-1}$ is a composite of tropical cluster $\X$-transformations. 
\end{introthm}
As a consequence, the shear coordinates combine to give an $MC(\Sigma)$-equivariant bijection 
\begin{align}\label{eq:shear_combined_X}
    \sfx^\uf_\bullet: \cL^x(\Sigma,\bQ) \xrightarrow{\sim} \X_{\fsl_3,\Sigma}^\uf(\bQ^T).
\end{align}
Therefore, our space $\cL^x(\Sigma,\bQ)$ of unbounded $\fsl_3$-laminations gives a geometric model for the tropical cluster $\X$-variety $\X_{\fsl_3,\Sigma}^\uf(\bQ^T)$. In other words, the space $\cL^x(\Sigma,\bQ)$ can be viewed as a tropical analogue of the moduli space $\X_{PGL_3,\Sigma}$ of framed $PGL_3$-local systems \cite{FG03}. 
%Here the symbol ``uf'' means that the spaces do not have \emph{frozen coordinates}. 

In \cref{subsec:reconstruction}, we give an explicit inverse map of $\sfx^\uf_\tri$ by gluing local building blocks according to the shear coordinates, in the same spirit as Fock--Goncharov. The coordinate transformation formula could be obtained by case-by-case as in \cite{DS20II} for the $\A$-side. However, in order to reduce the length of computation, we choose to derive it from the computation on the $\A$-side performed by Douglas--Sun after investigating their relation in detail (see \cref{introthm:shear_P} below). So the second statement in \cref{introthm:shear_X} follows from \cref{introthm:shear_P}.

%It is an immediate corollary of \cref{introthm:shear_X} that the bijection \eqref{eq:shear_combined_X} is equivariant under the natural action of the mapping class group $MC(\Sigma)$. Moreover, we will see that it is also equivariant under the \emph{Dynkin involution}, and further that the image of the space $\mathcal{L}_{\fsl_2}^x(\Sigma,\bQ)$ of rational unbounded ($\fsl_2$-)laminations under a natural embedding coincides with the fixed point locus of the Dynkin involution ( \cref{subsec:group_action}).

\subsection{Unbounded \texorpdfstring{$\fsl_3$}{sl3}-laminations with pinnings and their gluing}
In order to supply the frozen coordinates, we further introduce a larger space $\cL^p(\Sigma,\bQ)$ of unbounded $\fsl_3$-laminations \emph{with pinnings} by attaching additional data on boundary intervals, in the same spirit as Goncharov--Shen's construction of the moduli space $\P_{G',\Sigma}$ \cite{GS19}. As in their work, these additional data allow us to glue the $\fsl_3$-laminations along boundary intervals, which leads to the \emph{gluing map}
\begin{align}\label{eq:intro_amalgamation}
    q_{E_L,E_R}: \cL^p(\Sigma,\bQ) \to \cL^p(\Sigma',\bQ)
\end{align}
where $\Sigma'$ is the marked surface obtained from $\Sigma$ by gluing two boundary intervals $E_L,E_R$. 

The space $\cL^p(\Sigma,\bQ)$ is also suited for the comparison with the works of Douglas--Sun \cite{DS20I,DS20II} and Kim \cite{Kim21}.
Let $\cL^a(\Sigma,\bQ)$ denote the space of rational bounded $\fsl_3$-laminations, which essentially appears in these works. See \cref{rem:bounded_lamination}. 
Then we define a \emph{geometric ensemble map} 
\begin{align}\label{eq:intro_ensemble}
    \widetilde{p}: \cL^a(\Sigma,\bQ) \to \cL^p(\Sigma,\bQ)
\end{align}
by forgetting the peripheral components, and assigning pinnings in a certain way. When $\Sigma$ has no punctures, $\widetilde{p}$ gives a bijection. For these structures, we obtain the following:
%Let $p: \cL^a(\Sigma,\bQ) \to \cL^x(\Sigma,\bQ)$ denote the composite with the projection forgetting the pinnings. 

\begin{introthm}[\cref{thm:amalgamation}, \cref{prop:ensemble_map}, \cref{thm:cluster_transf}]\label{introthm:shear_P}
Under the same assumption as in \cref{introthm:shear_X}, we have a bijection 
\begin{align}\label{eq:shear_P}
    \sfx_\tri: \cL^p(\Sigma,\bQ) \xrightarrow{\sim} \bQ^{I(\tri)},
\end{align}
whose coordinate transformations are given by tropical cluster $\X$-transformations (including frozen coordinates). 
Via these coordinate systems:
\begin{enumerate}
    \item The gluing map $q_{E_L,E_R}$ coincides with the tropicalization of the amalgamation map \cite{FG06a}. 
    \item The geometric ensemble map $\widetilde{p}$ coincides with the tropicalization of the Goncharov--Shen extension of the ensemble map \eqref{eq:GS_extension}. 
\end{enumerate}
\end{introthm}
We will also see in \cref{subsec:Dynkin} that the shear coordinates are equivariant under the Dynkin involution $\ast$, which generates $\mathrm{Out}(SL_3)$. 
In particular, we have an $MC(\Sigma)\times \mathrm{Out}(SL_3)$-equivariant bijection 
\begin{align}\label{eq:shear_combined_P}
    \sfx_\tri: \cL^p(\Sigma,\bQ) \xrightarrow{\sim} \X_{\fsl_3,\Sigma}(\bQ^T).
\end{align}
In other words, the space $\cL^p(\Sigma,\bQ)$ can be viewed as a tropical analogue of the Goncharov--Shen's moduli space $\P_{PGL_3,\Sigma}$ \cite{GS19}. 

The property (1) allows one to reduce the computation of coordinate transformations to those for smaller surfaces. For a surface without punctures, the map $\widetilde{p}$ is a bijection and the property (2) shows that this map intertwines the two types of cluster transformations. This is our strategy to obtain the coordinate transformation formula for \eqref{eq:shear_P}. 

In our sequel paper \cite{IKp}, we will investigate the unbounded $\fsl_3$-laminations around punctures in detail, and study the tropicalizations of the cluster exact sequence of Fock--Goncharov \cite{FG09} and the Weyl group actions at punctures introduced by Goncharov--Shen \cite{GS18} in terms of $\fsl_3$-laminations. In the end, the bijections \eqref{eq:shear_combined_X} and \eqref{eq:shear_combined_P} turn out to be equivariant under the natural action of the group $(MC(\Sigma)\times \mathrm{Out}(SL_3)) \ltimes W(\fsl_3)^{\bP}$.

\subsection{Relation to the graphical basis of the skein algebra \texorpdfstring{$\mathscr{S}_{\fsl_3,\Sigma}^q$}{Sq(sl3,Sigma)}}
As mentioned in the beginning, our space $\cL^p(\Sigma,\bZ)\cong \X_{\fsl_3,\Sigma}(\bZ^T)$ is expected to parametrize a linear basis of the function ring $\cO(\A_{\fsl_3,\Sigma})$. 
When the marked surface has no punctures (hence the exchange matrix has full-rank), it is also expected to parametrize a linear basis of the \emph{quantum upper cluster algebra} $\cO_q(\A_{\fsl_3,\Sigma})$ of Berenstein--Zelevinsky \cite{BZ}. On the other hand, a skein model for $\cO_q(\A_{\fsl_3,\Sigma})$ is investigated in \cite{IY21} by the first named author and W.~Yuasa. They study a skein algebra $\mathscr{S}_{\fsl_3,\Sigma}^q$ with appropriate ``clasped'' skein relations at marked points, and constructed an inclusion of its boundary-localization $\mathscr{S}_{\fsl_3,\Sigma}^q[\partial^{-1}]$ into the quantum cluster algebra (and hence into $\cO_q(\A_{\fsl_3,\Sigma})$). 
Conjecturally these algebras coincide with each other. They give a $\bZ_q$-basis $\mathsf{BWeb}_{\fsl_3,\Sigma}$ of the skein algebra $\mathscr{S}_{\fsl_3,\Sigma}^q$ consisting of flat trivalent graphs. In this paper, we relate our integral $\fsl_3$-laminations with pinnings to the basis webs:

\begin{introthm}[\cref{thm:basis_correspondence}]\label{introthm:basis}
Assume that $\Sigma$ has no punctures. Then we have an $MC(\Sigma)\times \mathrm{Out}(SL_3)$-equivariant bijection
\begin{align*}
    \mathbb{I}_\X^q: \cL^p(\Sigma,\bZ)_+ \xrightarrow{\sim} \mathsf{BWeb}_{\fsl_3,\Sigma} \subset \mathscr{S}_{\fsl_3,\Sigma}^q, 
\end{align*}
where $\cL^p(\Sigma,\bZ)_+ \subset \cL^p(\Sigma,\bZ)$ denotes the subspace of dominant integral $\fsl_3$-laminations. Moreover, it is extended to a map $\mathbb{I}_\X^q: \cL^p(\Sigma,\bZ) \hookrightarrow 
% \subset
\mathscr{S}_{\fsl_3,\Sigma}^q[\partial^{-1}]$, whose image gives a $\bZ_q$-basis of $\mathscr{S}_{\fsl_3,\Sigma}^q[\partial^{-1}]$. 
\end{introthm}
The latter correspondence should be a basic ingredient for a construction of the \emph{quantum duality map} \cite{FG09} (see \cite[Conjecture 4.14]{Qin21} for a finer formulation as well as \cite{DM}). See \cref{subsec:FG duality} for a detailed discussion. 
Our general expectation is the following: 

\begin{introconj}
The basis $\mathbb{I}_\X^q(\cL^p(\Sigma,\bZ))$ is \emph{parametrized by tropical points} in the sense of \cite[Definition 4.13]{Qin21}. Namely, for any integral $\fsl_3$-lamination $\hL \in \cL^p(\Sigma,\bZ)$, the quantum Laurent expression of $\mathbb{I}_\X^q(\hL) \in \mathscr{A}^q_{\fsl_3,\Sigma}$ in the quantum cluster $\{A_i\}_{i \in I}$ associated with a vertex $\omega \in \bExch_{\fsl_3,\Sigma}$ has the leading term $\big[\prod_{i \in I}A_i^{\sfx_i(\hL)}\big]$ with respect to the dominance order (\cite[Definition 4.6]{Qin21}), where $\sfx^{(\omega)}=(\sfx_i)_{i \in I}$ is the shear coordinate system associated with $\omega$. 
\end{introconj}
Currently we do not know if it gives a basis with positivity (of Laurent expressions and/or structure constants),
or it requires a modification by using an $\fsl_3$-version of \emph{bracelets} \cite{Thu14}. See also \cite{AK,CKKO20,All} for the progress on the positivity problem for the $\fsl_2$-case. 

% When $\Sigma$ has punctures, our elementary moves (E3), (E4) at punctures that are used to define the equivalence of unbounded $\fsl_3$-laminations still seem to have skein-theoretic interpretations. It would be interesting to find an appropriate $\fsl_3$-skein algebra that has a basis parametrized by $\cL^p(\Sigma,\bZ)$.

\subsection{Future directions: real unbounded \texorpdfstring{$\fsl_3$}{sl3}-laminations}\label{subsec:real_lamination}
%Unbounded real laminations as degenerations of convex RP^2-structures?
%Dynamics of mapping class group action. Sign stability.
Let $\cL^x(\Sigma,\bR)$ be the completion of the space $\cL^x(\Sigma,\bQ)$ such that each shear coordinate system \eqref{eq:shear_X} extends to a homeomorphism $\sfx_\tri^\uf: \cL^x(\Sigma,\bR) \xrightarrow{\sim} \bR^{I_\uf(\tri)}$. It is well-defined since the cluster $\X$-transformations are Lipschitz continuous with respect to the Euclidean metrics on $\bQ^{I_\uf(\tri)}$. 
We call an element of $\cL^x(\Sigma,\bR)$ a \emph{real unbounded $\fsl_3$-lamination}, which is represented by a Cauchy sequence in $\cL^x(\Sigma,\bQ)$ with respect to shear coordinates. The space $\cL^x(\Sigma,\bR)$ has a natural PL structure, and considered to be an $\fsl_3$-analogue of the space $\eML(\Sigma)$ of measured geodesic laminations. Recall that in the \Teich--Thurston theory, the latter PL manifold plays the following roles (among others):

\begin{description}
    \item[Boundary at infinity of the \Teich\ space] The \emph{Thurston compactification} is a compactification of the \Teich\ space into a topological disk obtained by attaching the projectivization of $\eML(\Sigma)$, so that the mapping class group action is continuously extended. The measured geodesic laminations encode the ``rate'' of degenerations of geodesics in a divergent sequence in the \Teich\ space. 
    The Thurston compactification is identified with the \emph{Fock--Goncharov compactification} \cite{FG16,Le16,Ish19} $\overline{\X_\Sigma(\pos)}=\X_\Sigma(\pos) \cup \mathbb{S}\X_\Sigma(\bR^T)$, which is defined for any cluster $\X$-variety. 
    \item[Place for analyzing the pseudo-Anosov dynamics] The PL action of the mapping class group on $\eML(\Sigma)$ provides us rich information on the dynamics of pseudo-Anosov mapping classes. In particular, each pseudo-Anosov mapping class has the North-South dynamics on the projectivized space, and its unique attracting/repelling points are represented by a transverse pair of measured geodesic laminations. 
    A generalization of these specific properties for elements of a general cluster modular group is proposed in \cite{IK19,IK20a,IK20b}, which we call the \emph{sign stability}. The equivalence between the ``uniform'' sign stability and the pseudo-Anosov property is discussed in \cite{IK20a}, based on the identification $\eML(\Sigma) \cong \X_\Sigma^\uf(\bR^T)$.
\end{description}
It is natural to expect that the space $\cL^x(\Sigma,\bR)$ plays the same role in the $\fsl_3$-case. Since the positive real part $\X_{\fsl_3,\Sigma}^\uf(\pos)$ has been identified with the moduli space of convex $\mathbb{RP}^2$-structures on $\Sigma$, the real unbounded $\fsl_3$-laminations are expected to encode their degenerations. The PL action of a pseudo-Anosov mapping class on the space $\cL^x(\Sigma,\bR)$ is expected to provide more rich information, which may possibly lead to a finer classification of pseudo-Anosov mapping classes. Although a concrete description of a real unbounded $\fsl_3$-lamination as a certain geometric object (rather than a sequence) is still missing, the cluster algebraic interpretation of Thurston's train tracks studied in \cite{Kano_track} will be a useful tool. 

Generalizations of the Thurston's earthquake maps and the Hubburd--Masur theorem that relates measured foliations with quadratic differentials will be also interesting topics. A study on a cluster algebraic analogue of these theories is in progress by the authors with Takeru Asaka. 

\subsection*{Organization of the paper}
\begin{description}
\item[Main part (Sections \ref{sec:webs}--\ref{sec:amalgamation})]
In \cref{sec:webs}, we introduce rational unbounded $\fsl_3$-laminations and briefly discuss the relation to the works of Douglas--Sun \cite{DS20I,DS20II} and Kim \cite{Kim21}. We study the associated spiralling diagrams and define the shear coordinates in \cref{sec:shear_coord}. The bijectivity of the shear coordinate systems \eqref{eq:shear_X} is proved. 
In \cref{sec:amalgamation}, we introduce pinnings for rational unbounded $\fsl_3$-laminations and discuss their gluing and the extended ensemble map.  \cref{introthm:shear_P} is proved, and hence the proof of \cref{introthm:shear_X} is completed.

\item[Relation to the skein theory (\cref{subsec:FG duality})]
We investigate the relation to the skein algebra and quantum duality map in \cref{subsec:FG duality}. \cref{introthm:basis} is proved here. 
%In \cref{sec:ensemble}, we focus on the rational unbounded $\fsl_3$-laminations around a puncture. 
%After the classification of the spiralling diagrams on a punctured disk, we investigate the cluster exact sequence in terms of the tropical coordinates in \cref{subsec:cluster_exact_seq}. \cref{introthm:ensemble} is proved. In \cref{subsec:Weyl_action}, the Weyl group actions at punctures are introduced. \cref{introthm:Weyl_action} is proved.

\item[Proofs for the technical statements (\cref{sec:technicalities})]
The proofs for \cref{prop:spiralling_good_position} and \cref{prop:injectivity} are placed in \cref{sec:technicalities}. Logically they do not depend on the contents after the places where the statements are written. 
\end{description}
Basic terminology on the cluster varieties and the known results we need for the $\fsl_3$-case are collected in \cref{sec:appendix}.

\subsection*{Acknowledgements}
The authors thank to Hyun Kyu Kim for illuminating discussion on quantum duality maps. They also appreciate the anonymous referee's careful reading and suggestions. 
T. I. is grateful to Wataru Yuasa for a valuable discussion on $\fsl_3$-skein algebras in the early stage of this work. He is also grateful to Zhe Sun for explaining his works on the $\A$-side, and giving valuable comments on this work at several stages.
S. K. appreciates the support by Wataru Yuasa for his visit to Kyoto University in the spring of 2021.
The main part of this work was done during this visit.
T. I. is partially supported by JSPS KAKENHI (20K22304).
S. K. is partially supported by scientific research support of Research Alliance Center for Mathematical Sciences, Tohoku University.

\section{Unbounded \texorpdfstring{$\mathfrak{sl}_3$}{sl(3)}-laminations and their shear coordinates}\label{sec:webs}

\subsection{Marked surfaces and their triangulations}\label{subsec:notation_marked_surface}
A marked surface $(\Sigma,\bM)$ is a compact oriented surface $\Sigma$ together with a fixed non-empty finite set $\bM \subset \Sigma$ of \emph{marked points}. 
When the choice of $\bM$ is clear from the context, we simply denote a marked surface by $\Sigma$. 
A marked point is called a \emph{puncture} if it lies in the interior of $\Sigma$, and a \emph{special point} otherwise. 
Let $\bP=\bP(\Sigma)$ (resp. $\bM_\partial=\bM_\partial(\Sigma)$) denote the set of punctures (resp. special points), so that $\bM=\bP \sqcup \bM_\partial$. 
Let $\Sigma^*:=\Sigma \setminus \bP$. 
%We denote by $g$ the genus of $\Sigma$, $h$ the number of punctures, and $b$ the number of boundary components in the sequel. 
We always assume the following conditions:
\begin{enumerate}
    \item[(S1)] Each boundary component (if exists) has at least one marked point.
    \item[(S2)] $-2\chi(\Sigma^*)+|\bM_\partial| >0$.
    % , where $b$ denotes the number of boundary components.
    %\item[(S3)] If $g=0$ and $b=0$, then $h \geq 4$.
    %\item[(S3)] $(\Sigma,\bM)$ is not a disk with 2 special points on the boundary (biangle).
    \item[(S3)] $(\Sigma,\bM)$ is not a once-punctured disk with a single special point on the boundary.
    %$g(\Sigma) +|\bM| \geq 3$, where $g(\Sigma)$ is the genus of $\Sigma$. 
\end{enumerate}
We call a connected component of the punctured boundary $\partial^\ast \Sigma:=\partial\Sigma\setminus \bM_\partial$ a \emph{boundary interval}. The set of boundary intervals is denote by $\bB=\bB(\Sigma)$. We always endow each boundary interval with the orientation induced from $\partial\Sigma$. Then we have $|\bM_\partial|=|\bB|$. 
% \begin{align}\label{eq:correspondence S=B} 
%     \bM_\partial \xrightarrow{\sim} \bB,\quad m \mapsto E(m)
% \end{align}
% where $E(m)$ is the unique boundary interval having $m$ as its initial endpoint.

Unless otherwise stated, an \emph{isotopy} in a marked surface $(\Sigma,\bM)$ means an ambient isotopy in $\Sigma$ relative to $\bM$, which preserves each boundary interval setwisely. 
An \emph{ideal arc} in $(\Sigma,\bM)$ is an immersed arc in $\Sigma$ with endpoints in $\bM$ which has no self-intersection except possibly at its endpoints, and not isotopic to one point.
%An \emph{ideal arc} in $\Sigma$ is the isotopy class of an arc with endpoints in $\bM$ which is not contractible in $\Sigma^\ast$. 

An \emph{ideal triangulation} is a triangulation $\tri$ of $\Sigma$ whose set of $0$-cells (vertices) coincides with $\bM$. 
% An {\it ideal triangulation} $\tri$ of $\Sigma$ is a collection of ideal arcs such that
% \begin{itemize}
% \item each pair of ideal arcs in $\tri$ can intersect only at their endpoints;
% \item each region complementary to $\tri$ is a triangle whose edges (resp. vertices) are ideal arcs (resp. marked points).
% \end{itemize}
The conditions (S1), (S2) ensure the existence of such an ideal triangulation, and the positive integer in (S2) gives the number of 2-cells (triangles). The $1$-cells (edges) are necessarily ideal arcs. In this paper, we always consider an ideal triangulation without \emph{self-folded triangles} of the form
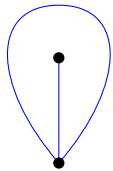
\begin{align*}
\begin{tikzpicture}[scale=0.7]
	%\draw [blue] (0,0) coordinate (v1) ellipse (2 and 2);
    \draw [blue] (0,0) -- (0,-2);
	\draw [blue] (0,-2) .. controls (-1.3,-0.5) and (-1.3,1) .. (0,1) .. controls (1.3,1) and (1.3,-0.5) .. (0,-2);
    %\node [fill, circle, inner sep=1.5pt] at (0,2) {};
    \node [fill, circle, inner sep=1.5pt] at (0,0) {};
    \node [fill, circle, inner sep=1.5pt] at (0,-2) {};
    %\node [blue,above] at (0,1) {$2$};
    %\node [blue,left] at (0,-1) {$1$};
	%\node [right] at (0,0) {$-$};
	%\node [red,left] at (0,0) {$p$};
\end{tikzpicture}
\end{align*}
Such an ideal triangulation exists by the condition (S3). See, for instance, \cite[Lemma 2.13]{FST}.
For an ideal triangulation $\tri$, denote the set of edges (resp. interior edges, triangles) of $\tri$ by $e(\tri)$ (resp. $e_{\interior}(\tri)$, $t(\tri)$). Since the boundary intervals belong to any ideal triangulation, $e(\tri)=e_{\interior}(\tri) \sqcup \bB$. By a computation on the Euler characteristics, we get
\begin{align*}
    &|e(\tri)|=-3\chi(\Sigma^*)+2|\bM_\partial|, \quad |e_{\interior}(\tri)|=-3\chi(\Sigma^*)+|\bM_\partial|, \\
    &|t(\tri)|=-2\chi(\Sigma^*)+|\bM_\partial|.
\end{align*}

\begin{figure}[h]
\centering
\begin{tikzpicture}
\draw[blue] (2.5,0) -- (0,2.5) -- (-2.5,0) -- (0,-2.5) --cycle;
\draw[blue, postaction={decorate}, decoration={markings,mark=at position 0.5 with {\arrow[scale=1.5,>=stealth]{>}}}] (0,-2.5) -- node[midway,right,scale=0.8] {$E$} (0,2.5);
\quiversquare{0,-2.5}{2.5,0}{0,2.5}{-2.5,0};
\draw(x131) node[mygreen, right=0.2em,scale=0.9]{$i^1(E)$};
\draw(x132) node[mygreen, left=0.2em,scale=0.9]{$i^2(E)$};
\end{tikzpicture}
    \caption{The set $I(\tri)$ of distinguished points.}
    \label{fig:sl3_triangulation}
\end{figure}
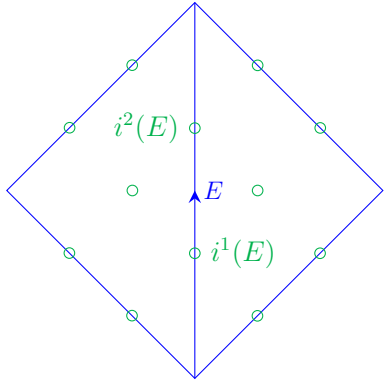

It is useful to equip $\tri$ with two distinguished points on the interior of each edge and one point in the interior of each triangle, as shown in \cref{fig:sl3_triangulation}. The set of such points is denoted by $I(\tri)=I_{\mathfrak{sl}_3}(\tri)$. This set will give the vertex set of the quiver $Q^\tri$ associated with $\tri$: see \cref{subsec:cluster_sl3}. 
Let $I^{\mathrm{edge}}(\tri)$ (resp. $I^{\mathrm{tri}}(\tri)$) denote the set of points on edges (resp. faces of triangles) so that $I(\tri)=I^{\mathrm{edge}}(\tri) \sqcup I^{\mathrm{tri}}(\tri)$, where we have a canonical bijection 
\begin{align*}
    t(\tri) \xrightarrow{\sim} I^{\mathrm{tri}}(\tri), \quad T \mapsto i(T).
\end{align*}
%For $k \in I^\mathrm{edge}(\tri)$, let $k^\mathrm{op}$ denote the other point on the same edge. 
When we need to label the two vertices on an edge $E \in e(\tri)$, we endow $E$ with an orientation. Then let $i^1(E) \in I(\tri)$ (resp. $i^2(E) \in I(\tri)$) denote the vertex closer to the initial (resp. terminal) endpoint of $E$. 
Let $I(\tri)_\f \subset I^\mathrm{edge}(\tri)$ (\lq\lq frozen'') be the subset consisting of the points on the boundary, and let $I(\tri)_\uf:=I(\tri) \setminus I(\tri)_\f$ (\lq\lq unfrozen''). The numbers
\begin{align*}
    |I(\tri)|= 2|e(\tri)| + |t(\tri)| =-8\chi(\Sigma^*)+5|\bM_\partial|
\end{align*}
and 
\begin{align*}
    |I(\tri)_\uf|= 2|e_{\interior}(\tri)| + |t(\tri)| =-8\chi(\Sigma^*)+3|\bM_\partial|
\end{align*}
will give the dimensions of the PL manifolds $\cL^p(\Sigma,\bR)$ and $\cL^x(\Sigma,\bR)$ respectively.
%, which we will study in this paper. 

\subsection{Unbounded \texorpdfstring{$\fsl_3$}{sl(3)}-laminations}\label{subsec:webs}
Recall that a \emph{uni-trivalent} graph is a (possibly disconnected and/or infinite) graph whose vertices have valency either one or three. It is allowed to have a loop component (\emph{i.e.}, a connected component without vertices). An \emph{orientation} of a uni-trivalent graph is an assignment of an orientation on each edge and loop such that any trivalent vertex is either a \emph{sink} or a \emph{source}, respectively:
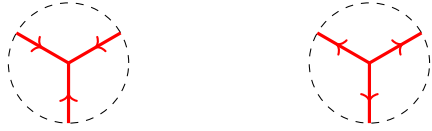
\begin{align*}
    \begin{tikzpicture}[scale=0.8]
        \draw[dashed] (0,0) circle(1cm);
        {\color{red} \sink{-90:1}{30:1}{150:1};}
        \begin{scope}[xshift=5cm]
        \draw[dashed] (0,0) circle(1cm);
        {\color{red} \source{-90:1}{30:1}{150:1};}
        \end{scope}
    \end{tikzpicture}
\end{align*}

An $\mathfrak{sl}_3$-web (or simply a \emph{web}) on a marked surface $\Sigma$ is an immersed oriented uni-trivalent graph $W$ on $\Sigma$ such that each univalent vertex lie in $\bP \cup \partial^\ast \Sigma$, and the other part is embedded into $\interior \Sigma^\ast$. 
It is said to be \emph{non-elliptic} if it has none of the following \emph{elliptic faces}:
%\footnote{Note that, on the other hand, a boundary H-face as shown in the left of \cref{fig:bdy H-move} is allowed. This should be compared with the \emph{weakly reduced webs} in \cite{FS20}.}:
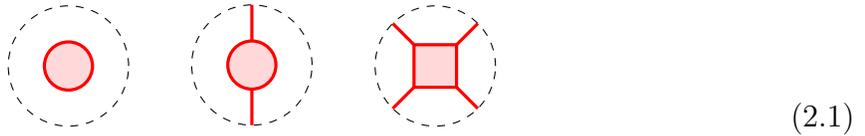
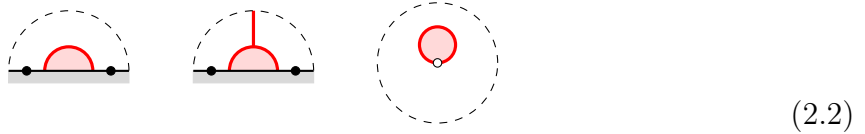
\begin{align}
    &\begin{tikzpicture}[scale=.8]
		\draw[dashed, fill=white] (0,0) circle [radius=1];
		\draw[very thick, red, fill=pink!60] (0,0) circle [radius=0.4];
	\end{tikzpicture}
	\hspace{2em}
	\begin{tikzpicture}[scale=.8]
		\draw[dashed, fill=white] (0,0) circle [radius=1];
		\draw[very thick, red] (0,-1) -- (0,-0.4);
		\draw[very thick, red] (0,0.4) -- (0,1);
		\draw[very thick, red, fill=pink!60] (0,0) circle [radius=0.4];
	\end{tikzpicture}
	\hspace{2em}
	\begin{tikzpicture}[scale=.8]
		\draw[dashed, fill=white] (0,0) circle [radius=1];
		\draw[very thick, red] (-45:1) -- (-45:0.5);
		\draw[very thick, red] (-135:1) -- (-135:0.5);
		\draw[very thick, red] (45:1) -- (45:0.5);
		\draw[very thick, red] (135:1) -- (135:0.5);
		\draw[very thick, red, fill=pink!60] (45:0.5) -- (135:0.5) -- (225:0.5) -- (315:0.5) -- cycle;
	\end{tikzpicture} \label{eq:elliptic face}\\
	%\hspace{2em}
	&\begin{tikzpicture}[scale=0.8,baseline=-0.7cm]
		\coordinate (P) at (-0.4,0) {};
		\coordinate (P') at (0.4,0) {};
		\coordinate (C) at (90:0.4) {};
		\draw[very thick, red, fill=pink!60] (P) to[out=north, in=west] (C) to[out=east, in=north] (P');
		\draw[dashed] (1,0) arc (0:180:1cm);
		\bline{-1,0}{1,0}{0.2}
		\draw[fill=black] (-0.7,0) circle(2pt);
		\draw[fill=black] (0.7,0) circle(2pt);
	\end{tikzpicture}
	\hspace{2em}
	\begin{tikzpicture}[scale=0.8,baseline=-0.7cm]
		\coordinate (P) at (-0.4,0) {};
		\coordinate (P') at (0.4,0) {};
		\coordinate (C) at (90:0.4) {};
		\draw[very thick, red, fill=pink!60] (P) to[out=north, in=west] (C) to[out=east, in=north] (P');
		\draw[very thick, red] (C) -- (90:1);
		\draw[dashed] (1,0) arc (0:180:1cm);
		\bline{-1,0}{1,0}{0.2}
		\draw[fill=black] (-0.7,0) circle(2pt);
		\draw[fill=black] (0.7,0) circle(2pt);
	\end{tikzpicture}
	\hspace{2em}
	\begin{tikzpicture}[scale=.8]
		\draw[dashed, fill=white] (0,0) circle [radius=1];
		\draw[red,very thick,fill=pink!60] (0,0.3) circle[radius=0.3]; 
		\draw[fill=white] (0,0) circle(2pt);
	\end{tikzpicture}
	%\hspace{2em}
% 	\begin{tikzpicture}[scale=.8]
% 		\draw[dashed, fill=white] (0,0) circle [radius=1];
% 		\draw[red,very thick,fill=pink!60] (0,0.3) circle[radius=0.3]; 
% 		\draw[red,very thick] (0,0.6)--(0,1);
% 		\draw[fill=white] (0,0) circle(2pt);
% 	\end{tikzpicture}
	\label{eq:elliptic_face_peripheral}
\end{align}
A web is said to be \emph{bounded} if none of its univalent vertices lie in $\bP$. 

We will mostly deal with finite webs, while infinite ones appear when (and only when) we discuss spiralling diagrams (\cref{def:spiralling}), which are still locally finite except possibly around punctures. When we simply say an ($\fsl_3$-)web below, it will mean a finite web. When the web in consideration can be infinite, we will say a ``possibly infinite web''. 

\begin{rem}
The exclusion of the internal faces in \eqref{eq:elliptic face} is usual in literature. Indeed, a web containing these faces can be written as a linear combination of non-elliptic webs in the skein algebra (see \cref{subsec:FG duality}) and hence not needed as a basis element. 
The first two faces in \eqref{eq:elliptic_face_peripheral} are excluded as variants of boundary skein relations \cite{IY21}. It is also related to the \emph{weakly reduced} condition in \cite{FS20}. The third one can be regarded as a variant for a boundary component without marked points.
%The exclusion of the last in \eqref{eq:elliptic_face_peripheral} will be also natural if one regard the punctures as boundary components without marked points. 
%then it comes from the \emph{reduced} condition in \cite{FS20}. 
%This condition is also crucial to ensure the finiteness of the intersection reduction procedure of the associated \emph{spiralling diagram} (\cref{def:spiralling}). See \cref{lem:relative_intersection_maximal}.
\end{rem}

\begin{ex}[Honeycomb webs]
Let $T \subset \interior\Sigma^\ast$ be an embedded triangle.
%, where the three vertices are allowed to land on a puncture or the boundary. 
For each positive integer $n$, 
%he triangle $T$ can be identified with the subset $\{(x,y,z) \mid x,y,z \geq 0,\ x+y+z =n\}$ of $\bR^3$. Here the three vertices correspond to $(n,0,0),(0,n,0),(0,0,n)$. Such a parametrization induces an $n$-triangulation of $T$, which is the triangulation by the intersection with the hyperplanes $x=a$, $y=b$, $z=c$ for integers $0 \leq a,b,c \leq n$. 
the incoming (resp. outgoing) \emph{honeycomb-web} (or \emph{pyramid web}) in $T$ of height $n$ is the $\mathfrak{sl}_3$-web dual to the \emph{$n$-triangulation} of $T$, oriented so that the outer-most edges are incoming to (resp. outgoing from) $T$. See the left picture in \cref{fig:honeycomb-web} for an example. We will also use a short-hand presentation as shown in the right \emph{loc.~cit}. 
The embedded image of a honeycomb web in $\Sigma$ is simply called a \emph{honeycomb}. The ends of a honeycomb can be connected with other oriented arcs or honeycombs on $\Sigma$. 

\begin{figure}[ht]
% \iffalse
\centering
\begin{tikzpicture}[scale=0.7]
\begin{scope}[xscale=0.5773]
\draw(-4,0) -- (4,0) -- (0,4) --cycle;
\foreach \y in {1,2,3}
{
\draw[gray] (-4+\y,\y) -- (4-\y,\y);
\draw[gray] (-4+\y,\y) -- (-4+2*\y,0);
\draw[gray] (4-\y,\y) -- (4-2*\y,0);
}
\foreach \i in {-3,-1,1,3}
{
\draw[red,very thick] (\i,-0.25) -- (\i,0.25) -- (\i+1,0.75);
\draw[red,very thick] (\i,0.25) -- (\i-1,0.75);
}
\foreach \i in {-2,0,2}
{
\draw[red,very thick] (\i,0.75) -- (\i,1.25) -- (\i+1,1.75);
\draw[red,very thick] (\i,1.25) -- (\i-1,1.75);
}
\foreach \i in {-1,1}
{
\draw[red,very thick] (\i,1.75) -- (\i,2.25) -- (\i+1,2.75);
\draw[red,very thick] (\i,2.25) -- (\i-1,2.75);
}
\draw[red,very thick] (0,2.75) -- (0,3.25) -- (1,3.75);
\draw[red,very thick] (0,3.25) -- (-1,3.75);
\end{scope}
\node[scale=1.5] at (3.5,2) {$=\mathrel{\mathop:}$};

\begin{scope}[xshift=7cm,yshift=1.333cm]
\draw (-30:8/3) -- (90:8/3) -- (210:8/3) --cycle;
\draw[red, very thick] (-30:1) -- (90:1) -- (210:1) --cycle;
\foreach \i in {1,2,3,4}
{
\draw[red,very thick] ($(-30:1)!\i/5!(90:1)$) --++(30:1.2);
\draw[red,very thick] ($(90:1)!\i/5!(210:1)$) --++(150:1.2);
\draw[red,very thick] ($(210:1)!\i/5!(-30:1)$) --++(-90:1.2);
}
\end{scope}
\end{tikzpicture}
% \fi%%
% \includegraphics[width=10cm]{honeycomb-web.png}
    \caption{A honeycomb-web on a triangle $T$ of height $n=4$ (left) and its short-hand presentation (right).}
    \label{fig:honeycomb-web}
\end{figure}
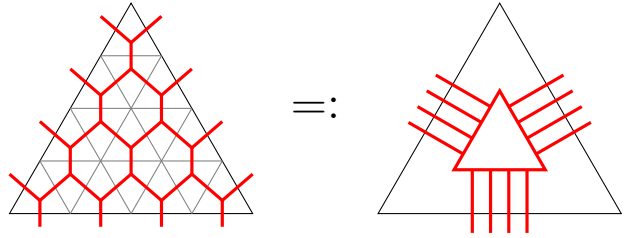
\end{ex}
A \emph{signed web} is a web on $\Sigma$ together with a sign ($+$ or $-$) assigned to each end incident to a puncture. The following patterns (and their orientation-reversals) of signed ends are called \emph{bad ends}:
\begin{align}\label{eq:puncture-admissible}
\begin{tikzpicture}
\draw[dashed] (0,0) circle(1cm);
\draw[red,very thick,->-] (-1,0) -- (0,0) node[below left]{$-\epsilon$};
\draw[red,very thick,-<-] (1,0) -- (0,0) node[below right]{$\epsilon$};
\filldraw[fill=white](0,0) circle(2pt);
% \begin{scope}[xshift=4cm]
% \draw[dashed] (0,0) circle(1cm);
% \draw[red,very thick,-<-] (-1,0) -- (0,0) node[below left,scale=0.8]{$+$};
% \draw[red,very thick,->-] (1,0) -- (0,0) node[below right,scale=0.8]{$-$};
% \filldraw[fill=white](0,0) circle(2pt);
% \end{scope}
\begin{scope}[xshift=4cm]
\draw[dashed] (0,0) circle(1cm);
\draw[red,very thick,-<-](0,0) arc(-90:90:0.3cm);
\draw[red,very thick,-<-](0,0) arc(-90:-270:0.3cm);
\draw[red,very thick,->-] (0,0.6) -- (0,1);
\node[red] at (-0.4,0) {$\epsilon$};
\node[red] at (0.4,0) {$\epsilon$};
\filldraw[fill=white](0,0) circle(2pt);
\end{scope}
\begin{scope}[xshift=8cm]
\draw[dashed] (0,0) circle(1cm);
\draw[red,very thick,->-](0,0) arc(-90:90:0.3cm);
\draw[red,very thick,->-](0,0) arc(-90:-270:0.3cm);
\draw[red,very thick,-<-] (0,0.6) -- (0,1);
\node[red] at (-0.4,0) {$\epsilon$};
\node[red] at (0.4,0) {$\epsilon$};
\filldraw[fill=white](0,0) circle(2pt);
\end{scope}
\end{tikzpicture}
\end{align}
Here $\epsilon \in \{+,-\}$. 
%The pair of ends of the first type is called the \emph{resolvable pair}. 
A signed web is said to be \emph{admissible} if it has no bad ends. In this paper, we always assume that the signed webs are admissible unless otherwise stated. 
A bounded web is naturally regarded as a signed web since we do not need to specify any signs. 

\begin{rem}
The latter two types of bad ends will be excluded simply because they will not contribute to the shear coordinates.
On the other hand, a pair of the first type will have non-trivial coordinates, while there is always another web that attains the same coordinates. So we only need admissible signed webs to realize the tropical space. 
%However, i
It turns out that we need to include the bad ends of first type to define the Weyl group actions at punctures \cite{IKp}.
\end{rem}

\bigskip
\paragraph{\textbf{Elementary moves of signed webs.}}
We are going to introduce several elementary moves for signed webs. The first two are defined for a web without signs.

\begin{enumerate}
\item[(E1)] Loop parallel-move (a.~k.~a.~\emph{flip move} \cite{FS20} or \emph{global parallel move} \cite{DS20I}):  
\begin{align}\label{eq:loop_parallel-move}
\begin{tikzpicture}[scale=1]
\draw(0,0) -- (2,0);
\draw(0,1.5) -- (2,1.5);
\draw[red,very thick,->-] (0.7,0.75) [partial ellipse=90:-90:0.2cm and 0.75cm];
\draw[red,very thick,dashed] (0.7,0.75) [partial ellipse=-90:-270:0.2cm and 0.75cm];
\draw[red,very thick,-<-] (1.3,0.75) [partial ellipse=90:-90:0.2cm and 0.75cm];
\draw[red,very thick,dashed] (1.3,0.75) [partial ellipse=-90:-270:0.2cm and 0.75cm];
% \draw[very thick,<->] (2.5,0.75) -- (3.5,0.75);
\node[scale=1.5] at (3,0.75) {$\sim$};
\begin{scope}[xshift=4cm]
\draw(0,0) -- (2,0);
\draw(0,1.5) -- (2,1.5);
\draw[red,very thick,-<-] (0.7,0.75) [partial ellipse=90:-90:0.2cm and 0.75cm];
\draw[red,very thick,dashed] (0.7,0.75) [partial ellipse=-90:-270:0.2cm and 0.75cm];
\draw[red,very thick,->-] (1.3,0.75) [partial ellipse=90:-90:0.2cm and 0.75cm];
\draw[red,very thick,dashed] (1.3,0.75) [partial ellipse=-90:-270:0.2cm and 0.75cm];
\end{scope}
\end{tikzpicture}
\end{align}

\item[(E2)] Boundary H-move:
\begin{align}\label{eq:boundary_H-move}
\begin{tikzpicture}
\filldraw[pink!60] (0.4,0) -- (0.4,0.4) -- (-0.4,0.4) -- (-0.4,0) --cycle;
\draw[very thick,red,-<-] (0.4,0) -- (0.4,0.4);
\draw[very thick,red,->-] (-0.4,0) -- (-0.4,0.4);
\draw[very thick,red,->-] (0.4,0.4) -- (0.4,0.917);
\draw[very thick,red,-<-] (-0.4,0.4) -- (-0.4,0.917);
\draw[very thick,red,->-] (0.4,0.4) -- (-0.4,0.4);
\draw[dashed] (1,0) arc (0:180:1cm);
\bline{-1,0}{1,0}{0.2}
\draw[fill=black] (-0.8,0) circle(2pt);
\draw[fill=black] (0.8,0) circle(2pt);
% \draw[very thick,<->] (1.5,0.5) -- (2.5,0.5);
\node[scale=1.5] at (2,0.5) {$\sim$};
\begin{scope}[xshift=4cm]
\draw[very thick,red,->-] (0.4,0) -- (0.4,0.917);
\draw[very thick,red,-<-] (-0.4,0) -- (-0.4,0.917);
\draw[dashed] (1,0) arc (0:180:1cm);
\bline{-1,0}{1,0}{0.2}
\draw[fill=black] (-0.8,0) circle(2pt);
\draw[fill=black] (0.8,0) circle(2pt);
\end{scope}
\end{tikzpicture}
\end{align}
Similarly for the opposite orientation. We call the face in the left-hand side a \emph{boundary H-face}.

\item[(E3)] Puncture H-moves:
\begin{align}\label{eq:puncture_H-move_1}
\begin{tikzpicture}
\draw[dashed, fill=white] (0,0) circle [radius=1];
\draw(60:0.6) coordinate(A);
\draw(120:0.6) coordinate(B);
\fill[pink!60] (A) -- (B) -- (0,0) --cycle; 
\draw[red,very thick,-<-={0.6}{}] (0,0)--(A);
\draw[red,very thick,->-={0.6}{}] (0,0)--(B);
\draw[red,very thick,->-={0.6}{}] (A)--(60:1);
\draw[red,very thick,-<-={0.6}{}] (B)--(120:1);
\draw[red,very thick,->-] (A) -- (B);
\node[red,scale=0.8,anchor=west] at (0.1,0) {$\epsilon$};
\node[red,scale=0.8,anchor=east] at (-0.1,0) {$\epsilon$};
\draw[fill=white] (0,0) circle(2pt);
% \draw[very thick,<->] (1.5,0) -- (2.5,0);
\node[scale=1.5] at (2,0) {$\sim$};
\begin{scope}[xshift=4cm]
\draw[dashed, fill=white] (0,0) circle [radius=1];
\draw(60:0.6) coordinate(A);
\draw(120:0.6) coordinate(B);
%\fill[pink!60] (A) -- (B) -- (0,0) --cycle; 
\draw[red,very thick,->-={0.6}{}] (0,0)--(60:1);
\draw[red,very thick,-<-={0.6}{}] (0,0)--(120:1);
%\draw[red,very thick] (A) -- (B);
\node[red,scale=0.8,anchor=west] at (0.1,0) {$\epsilon$};
\node[red,scale=0.8,anchor=east] at (-0.1,0) {$\epsilon$};
\draw[fill=white] (0,0) circle(2pt);
\end{scope}
\end{tikzpicture}
\end{align}
for $\epsilon \in \{+,-\}$, and
\begin{align}\label{eq:puncture_H-move_2}
\begin{tikzpicture}
\draw[dashed] (0,0) circle(1cm);
\filldraw[draw=red,very thick,->-={0.7}{},-<-={0.42}{},fill=pink!60] (0,0) ..controls (0.5,0.1) and (0.2,0.4).. (0,0.4) ..controls (-0.2,0.4) and (-0.5,0.1).. (0,0);
\draw[red,very thick,->-={0.7}{}] (0,0.4) -- (0,0.7);
\draw[red,very thick,-<-] (0,0.7) -- (60:1);
\draw[red,very thick,-<-] (0,0.7) -- (120:1);
\node[red,scale=0.8,anchor=east] at (-0.1,-0.1) {$+$};
\node[red,scale=0.8,anchor=west] at (0.1,-0.1) {$-$};
\draw[fill=white] (0,0) circle(2pt);
% \draw[very thick,<->] (1.5,0) -- (2.5,0);
\node[scale=1.5] at (2,0) {$\sim$};
\begin{scope}[xshift=4cm]
\draw[dashed] (0,0) circle(1cm);
\filldraw[draw=red,very thick,->-={0.7}{},-<-={0.42}{},fill=pink!60] (0,0) ..controls (0.5,0.1) and (0.2,0.4).. (0,0.4) ..controls (-0.2,0.4) and (-0.5,0.1).. (0,0);
\draw[red,very thick,->-={0.7}{}] (0,0.4) -- (0,0.7);
\draw[red,very thick,-<-] (0,0.7) -- (60:1);
\draw[red,very thick,-<-] (0,0.7) -- (120:1);
\node[red,scale=0.8,anchor=east] at (-0.1,-0.1) {$-$};
\node[red,scale=0.8,anchor=west] at (0.1,-0.1) {$+$};
\draw[fill=white] (0,0) circle(2pt);
% \draw[very thick,<->] (1.5,0) -- (2.5,0);
\node[scale=1.5] at (2,0) {$\sim$};
\end{scope}
\begin{scope}[xshift=8cm]
\draw[dashed] (0,0) circle(1cm);
\draw[red,very thick,-<-] (0,0) -- (60:1);
\draw[red,very thick,-<-] (0,0) -- (120:1);
\node[red,scale=0.8,anchor=east] at (-0.1,0) {$+$};
\node[red,scale=0.8,anchor=west] at (0.1,0) {$-$};
\draw[fill=white] (0,0) circle(2pt);
\end{scope}
\end{tikzpicture}
\end{align}
%for $\epsilon_1,\epsilon_2 \in \{+,-\}$. Here the case $\epsilon_1 \neq \epsilon_2$ involves bad ends. 
Similarly for the opposite orientation. We call the face in the left-hand side of \eqref{eq:puncture_H-move_1} a \emph{puncture H-face}. 
\end{enumerate}
%The loop parallel-move, boundary H-move and puncture H-move are collectively called the \emph{parallel-moves} of signed webs. 
%Two signed webs are \emph{parallel-equivalent} if they are transformed to each other by a finite sequence of parallel-moves. 
The following lemma is verified by using (E2) and the first one in (E3):

\begin{lem}\label{lem:arc-parallel-moves}
From the boundary and puncture H-moves, we get the following \lq\lq arc parallel-moves'' swapping parallel arcs with opposite orientations:
\begin{align*}
&\begin{tikzpicture}
\fill[gray!30](0,0) -- (0,1.5) -- (-0.2,1.5) -- (-0.2,0) --cycle;
\fill[gray!30](2,0) -- (2,1.5) -- (2.2,1.5) -- (2.2,0) --cycle;
\draw[thick](0,0) -- (0,1.5);
\draw[thick](2,0) -- (2,1.5);
\draw[red,very thick,->-] (0,1) to (2,1);
\draw[red,very thick,-<-] (0,0.5) to (2,0.5);
\fill(0,1.2) circle(2pt);
\fill(0,0.3) circle(2pt);
\fill(2,1.2) circle(2pt);
\fill(2,0.3) circle(2pt);
% \draw[very thick,<->] (2.5,0.75) -- (3.5,0.75);
\node[scale=1.5] at (3,0.75) {$\sim$};
\begin{scope}[xshift=4cm]
\fill[gray!30](0,0) -- (0,1.5) -- (-0.2,1.5) -- (-0.2,0) --cycle;
\fill[gray!30](2,0) -- (2,1.5) -- (2.2,1.5) -- (2.2,0) --cycle;
\draw[thick](0,0) -- (0,1.5);
\draw[thick](2,0) -- (2,1.5);
\draw[red,very thick,-<-] (0,1) to (2,1);
\draw[red,very thick,->-] (0,0.5) to (2,0.5);
\fill(0,1.2) circle(2pt);
\fill(0,0.3) circle(2pt);
\fill(2,1.2) circle(2pt);
\fill(2,0.3) circle(2pt);
\end{scope}
\end{tikzpicture}\\
&\begin{tikzpicture}
\fill[gray!30](0,0) -- (0,1.5) -- (-0.2,1.5) -- (-0.2,0) --cycle;
\draw[thick](0,0) -- (0,1.5);
\draw[red,very thick,->-] (0,1) to[out=0,in=150] (2,0.75);
\draw[red,very thick,-<-] (0,0.5) to[out=0,in=210] (2,0.75);
\fill(0,1.2) circle(2pt);
\fill(0,0.3) circle(2pt);
\draw[fill=white](2,0.75) circle(2pt);
\node[red,scale=0.8] at (1.9,0.75+0.25) {$\epsilon$};
\node[red,scale=0.8] at (1.9,0.75-0.25) {$\epsilon$};
% \draw[very thick,<->] (2.5,0.75) -- (3.5,0.75);
\node[scale=1.5] at (3,0.75) {$\sim$};
\begin{scope}[xshift=4cm]
\fill[gray!30](0,0) -- (0,1.5) -- (-0.2,1.5) -- (-0.2,0) --cycle;
\draw[thick](0,0) -- (0,1.5);
\draw[red,very thick,-<-] (0,1) to[out=0,in=150] (2,0.75);
\draw[red,very thick,->-] (0,0.5) to[out=0,in=210] (2,0.75);
\fill(0,1.2) circle(2pt);
\fill(0,0.3) circle(2pt);
\draw[fill=white](2,0.75) circle(2pt);
\node[red,scale=0.8] at (1.9,0.75+0.25) {$\epsilon$};
\node[red,scale=0.8] at (1.9,0.75-0.25) {$\epsilon$};
\end{scope}
\end{tikzpicture}\\
&\ 
\begin{tikzpicture}
\draw[red,very thick,->-] (0,0.75) to[bend left=20] (2,0.75);
\draw[red,very thick,-<-] (0,0.75) to[bend right=20] (2,0.75);
\draw[fill=white](0,0.75) circle(2pt);
\node[red,scale=0.8] at (0.1,0.75+0.25) {$\epsilon$};
\node[red,scale=0.8] at (0.1,0.75-0.25) {$\epsilon$};
\draw[fill=white](2,0.75) circle(2pt);
\node[red,scale=0.8] at (1.9,0.75+0.25) {$\epsilon'$};
\node[red,scale=0.8] at (1.9,0.75-0.25) {$\epsilon'$};
% \draw[very thick,<->] (2.5,0.75) -- (3.5,0.75);
\node[scale=1.5] at (3,0.75) {$\sim$};
\begin{scope}[xshift=4cm]
\draw[red,very thick,-<-] (0,0.75) to[bend left=20] (2,0.75);
\draw[red,very thick,->-] (0,0.75) to[bend right=20] (2,0.75);
\draw[fill=white](0,0.75) circle(2pt);
\node[red,scale=0.8] at (0.1,0.75+0.25) {$\epsilon$};
\node[red,scale=0.8] at (0.1,0.75-0.25) {$\epsilon$};
\draw[fill=white](2,0.75) circle(2pt);
\node[red,scale=0.8] at (1.9,0.75+0.25) {$\epsilon'$};
\node[red,scale=0.8] at (1.9,0.75-0.25) {$\epsilon'$};
\end{scope}
\end{tikzpicture}
\end{align*}
Here white (resp. black) circles stand for punctures (resp. special points), and $\epsilon,\epsilon' \in \{+,-\}$.
\end{lem}
Also note that we can always transform any signed web to a signed web without boundary H-faces (resp. puncture H-faces) by applying (E2) and (E3), respectively. Slightly generalizing the terminology in \cite{FS20}, such a signed web is said to be \emph{boundary-reduced} (resp. \emph{puncture-reduced}). It is said to be \emph{reduced} if it is both boundary- and puncture-reduced.

\begin{enumerate}
    \item [(E4)] Peripheral move: removing or creating a peripheral component:
    \begin{align}\label{eq:peripheral}
    \begin{tikzpicture}[scale=.8]
    \draw[dashed, fill=white] (0,0) circle [radius=1];
    \draw[very thick, red, fill=pink!60] (0,0) circle [radius=0.5];
    \filldraw[draw=black,fill=white] (0,0) circle(2.5pt); 
    \begin{scope}[xshift=5cm]
    \coordinate (P) at (-0.5,0) {};
    \coordinate (P') at (0.5,0) {};
    \coordinate (C) at (0,0.5) {};
    \draw[very thick, red, fill=pink!60] (P) to[out=north, in=west] (C) to[out=east, in=north] (P');
    \draw[dashed] (1,0) arc (0:180:1cm);
    \bline{-1,0}{1,0}{0.2}
    \draw[fill=black] (0,0) circle(2pt);
    \end{scope}
    \end{tikzpicture}
    \end{align}
    Moreover, we have the moves
    
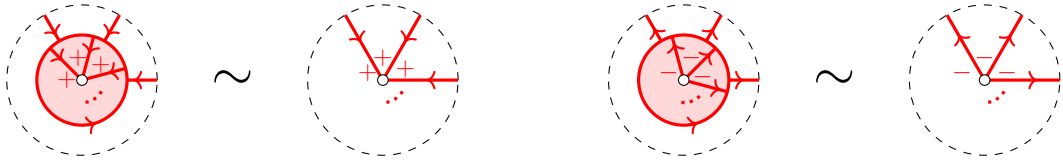
\begin{align*}
    \hspace{-5mm}
    \begin{tikzpicture}
    \draw[dashed, fill=white] (0,0) circle [radius=1];
    \fill[pink!60] (0,0) circle(0.6cm);
    \draw[red,very thick,->-={0.8}{}] (0,0) circle(0.6cm);
    \foreach \i in {120,60,0}
    { 
        \draw[red,very thick,-<-={0.8}{}] (\i+15:0) -- (\i+15:0.6);
        \draw[red,very thick,-<-] (\i:0.6) -- (\i:1);
    }
    \node[red,scale=0.8] at (-0.2,0) {$+$};
	\node[red,scale=0.8] at (-0.05,0.3) {$+$};
	\node[red,scale=0.8] at (0.25,0.2) {$+$};
    \draw[red,very thick,dotted] (-30:0.3) arc(-30:-90:0.3);
    \draw[fill=white] (0,0) circle(2pt);
    \node[scale=1.7] at (2,0) {$\sim$};
    \begin{scope}[xshift=4cm]
    \draw[dashed, fill=white] (0,0) circle [radius=1];
    \foreach \i in {120,60,0} \draw[red,very thick,-<-={0.7}{}] (\i:0) -- (\i:1);
    \node[red,scale=0.8] at (-0.2,0.1) {$+$};
    \node[red,scale=0.8] at (0,0.3) {$+$};
    \node[red,scale=0.8] at (0.3,0.15) {$+$};
    \draw[red,very thick,dotted] (-30:0.3) arc(-30:-90:0.3);
    \draw[fill=white] (0,0) circle(2pt);
    \end{scope}
    {\begin{scope}[xshift=7cm]
    \draw[dashed, fill=white] (0,0) circle [radius=1];
    \fill[pink!60] (0,0) circle(0.6cm);
    \draw[red,very thick,->-={0.8}{}] (0,0) circle(0.6cm);
    \foreach \i in {120,60,0}
    { 
        \draw[red,very thick,->-={0.85}{}] (\i-15:0) -- (\i-15:0.6);
        \draw[red,very thick,->-] (\i:0.6) -- (\i:1);
    }
    \node[red,scale=0.8] at (-0.2,0.1) {$-$};
    \node[red,scale=0.8] at (0.1,0.3) {$-$};
    \node[red,scale=0.8] at (0.25,0.05) {$-$};
    \draw[red,very thick,dotted] (-45:0.3) arc(-45:-105:0.3);
    \draw[fill=white] (0,0) circle(2pt);
    % \draw[thick,<->] (1.5,0) -- (2.5,0);
    \node[scale=1.7] at (2,0) {$\sim$};
    \end{scope}}
    {\begin{scope}[xshift=11cm]
    \draw[dashed, fill=white] (0,0) circle [radius=1];
    \foreach \i in {120,60,0} \draw[red,very thick,->-={0.7}{}] (\i:0) -- (\i:1);
    \node[red,scale=0.8] at (-0.3,0.1) {$-$};
    \node[red,scale=0.8] at (0,0.3) {$-$};
    \node[red,scale=0.8] at (0.3,0.1) {$-$};
    \draw[fill=white] (0,0) circle(2pt);
    \draw[red,very thick,dotted] (-30:0.3) arc(-30:-90:0.3);
    \end{scope}}
    \end{tikzpicture}
    \end{align*}
    Similarly for the opposite orientation.
\end{enumerate}
We will consider the equivalence relation on signed webs generated by isotopies of marked surfaces and the elementary moves (E1)--(E4). Observe that the moves (E1)--(E4) preserves the admissibility. On the other hand, a non-elliptic signed web may be equivalent to an elliptic web as the following example shows.
%Unless otherwise stated, we will only consider signed webs whose equivalence class does not contain any signed webs with bad ends in \eqref{eq:puncture-admissible}. 

\begin{ex}
We have
\begin{align*}
\begin{tikzpicture}
\draw[dashed, fill=white] (0,0) circle [radius=1];
\draw(60:0.6) coordinate(A);
\draw(120:0.6) coordinate(B);
%\fill[pink!60] (A) -- (B) -- (0,0) --cycle; 
\draw[red,very thick,-<-={0.6}{}] (0,0)--(A);
\draw[red,very thick,->-={0.6}{}] (0,0)--(B);
\draw[red,very thick,->-={0.6}{}] (A)--(60:1);
\draw[red,very thick,-<-={0.6}{}] (B)--(120:1);
\draw[red,very thick,->-] (A) -- (B);
\node[red,scale=0.8,anchor=west] at (0.1,0) {$+$};
\node[red,scale=0.8,anchor=east] at (-0.1,0) {$+$};
\draw[fill=white] (0,0) circle(2pt);
% \draw[very thick,<->] (1.5,0) -- (2.5,0);
\node[scale=1.5] at (1.5,0) {$\sim$};
\begin{scope}[xshift=-3cm]
\draw[dashed, fill=white] (0,0) circle [radius=1];
\draw(60:0.6) coordinate(A);
\draw(120:0.6) coordinate(B);
%\fill[pink!60] (A) -- (B) -- (0,0) --cycle; 
\draw[red,very thick,->-={0.6}{}] (0,0)--(60:1);
\draw[red,very thick,-<-={0.6}{}] (0,0)--(120:1);
%\draw[red,very thick] (A) -- (B);
\node[red,scale=0.8,anchor=west] at (0.1,0) {$+$};
\node[red,scale=0.8,anchor=east] at (-0.1,0) {$+$};
\draw[fill=white] (0,0) circle(2pt);
\node[scale=1.5] at (1.5,0) {$\sim$};
\end{scope}
\begin{scope}[xshift=3cm]
\draw[dashed, fill=white] (0,0) circle [radius=1];
\draw(60:0.5) coordinate(A);
\draw(120:0.5) coordinate(B);
\draw(60:0.8) coordinate(A');
\draw(120:0.8) coordinate(B');
\fill[pink!60] (A) -- (B) -- (B') -- (A') --cycle;
\draw[red,very thick,->-={0.6}{}] (0,0)--(A);
\draw[red,very thick,-<-={0.6}{}] (0,0)--(B);
\draw[red,very thick,-<-={0.4}{},->-={0.9}{}] (A)--(60:1);
\draw[red,very thick,->-={0.4}{},-<-={0.9}{}] (B)--(120:1);
\draw[red,very thick,-<-] (A) -- (B);
\draw[red,very thick,->-] (A') -- (B');
\node[red,scale=0.8,anchor=west] at (0.1,0) {$+$};
\node[red,scale=0.8,anchor=east] at (-0.1,0) {$+$};
\draw[fill=white] (0,0) circle(2pt);
\end{scope}
\end{tikzpicture}\ ,\quad 
\begin{tikzpicture}
\draw[dashed, fill=white] (0,0) circle [radius=1];
\draw[red,very thick,-<-,->-={0.9}{}] (0,0) -- (60:1);
\draw[red,very thick,-<-,->-={0.9}{}] (0,0) -- (120:1);
\draw[red,very thick,-<-] (0,0.8) -- (0,1);
\draw[red,very thick,-<-={0.5}{}] (0,0.8) -- (60:0.7);
\draw[red,very thick,-<-={0.5}{}] (0,0.8) -- (120:0.7);
\node[red,scale=0.8] at (-0.3,0.1) {$+$};
\node[red,scale=0.8] at (0.3,0.1) {$-$};
\draw[fill=white] (0,0) circle(2pt);
 \node[scale=1.7] at (1.5,0) {$\sim$};
\begin{scope}[xshift=3cm]
\draw[dashed, fill=white] (0,0) circle [radius=1];
\fill[pink!60] (60:0.7) ..controls (60:0.6) and (0.2,0.5).. (0,0.5) ..controls (-0.2,0.5) and (120:0.6).. (120:0.7) -- (0,0.8) -- (60:0.7);
\draw[red,very thick,-<-={0.15}{},->-={0.95}{}] (60:1) -- (60:0.7) ..controls (60:0.6) and (0.2,0.5).. (0,0.5) ..controls (-0.2,0.5) and (120:0.6).. (120:0.7) -- (120:1);
\draw[red,very thick,-<-] (0,0.8) -- (0,1);
\draw[red,very thick,-<-] (0,0.8) -- (60:0.7);
\draw[red,very thick,-<-] (0,0.8) -- (120:0.7);
\draw[red,very thick] (0,0.5) -- (0,0.3);
\draw[red,very thick,-<-={0.38}{},->-={0.72}{}] (0,0) ..controls (-0.3,0.1) and (-0.1,0.3).. (0,0.3) ..controls (0.1,0.3) and (0.3,0.1).. (0,0);
\node[red,scale=0.8] at (-0.3,0) {$-$};
\node[red,scale=0.8] at (0.3,0) {$+$};
    \draw[fill=white] (0,0) circle(2pt);
\end{scope}
\end{tikzpicture}
\end{align*}
by the puncture H-moves \eqref{eq:puncture_H-move_1} and \eqref{eq:puncture_H-move_2}, where the resulting signed webs are elliptic (having interior 4-gon faces).
\end{ex}

\begin{dfn}[rational unbounded $\mathfrak{sl}_3$-laminations]\label{def:unbounded laminations}
A \emph{rational unbounded $\mathfrak{sl}_3$-lamination} (or a rational $\mathfrak{sl}_3$-$\X$-lamination) on $\Sigma$ is an admissible, non-elliptic signed $\mathfrak{sl}_3$-web $W$ on $\Sigma$ equipped with a positive rational number (called the \emph{weight}) on each component, which is considered modulo the equivalence relation generated by isotopies and the following operations:
\begin{enumerate}
    \item Elementary moves (E1)--(E4) for the underlying signed webs. Here the corresponding components are assumed to have the same weights. 
    \item Combine a pair of isotopic loops with the same orientation with weights $u$ and $v$ into a single loop with the weight $u+v$. Similarly combine a pair of isotopic oriented arcs with the same orientation (and with the same signs if some of their ends are incident to punctures) into a single one by adding their weights. 
    \item For an integer $n \in \bZ_{>0}$ and a rational number $u \in \bQ_{>0}$, replace a component with weight $nu$ with its \emph{$n$-cabling} with weight $u$, which locally looks like
    
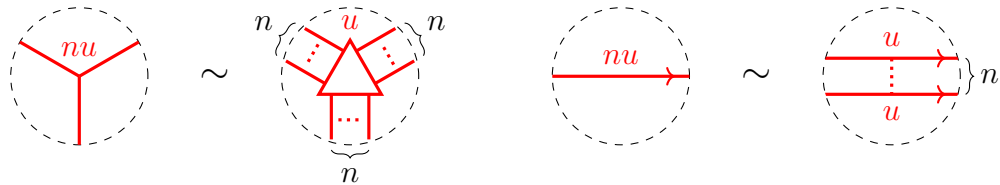
\begin{align*}
    \hspace{-5mm}
        \begin{tikzpicture}[scale=1.2]
        \draw[dashed] (0,0) circle(0.76cm);
        \foreach \i in {30,150,270}
        \draw[red,very thick] (0,0) -- (\i:0.76);
        \node[red] at (0,0.3) {$nu$};
        \node at (1.5,0) {\scalebox{1.2}{$\sim$}}; 
        \begin{scope}[xshift=3cm]
        \draw[red, very thick] (-30:0.4) -- (90:0.4) -- (210:0.4) --cycle;
        \foreach \i in {1,4}
        {
        \draw[red,very thick] ($(-30:0.4)!\i/5!(90:0.4)$) --++(30:0.5) coordinate(A\i);
        \draw[red,very thick] ($(90:0.4)!\i/5!(210:0.4)$) --++(150:0.5) coordinate(B\i);
        \draw[red,very thick] ($(210:0.4)!\i/5!(-30:0.4)$) --++(-90:0.5) coordinate(C\i);
        }
        \draw[red,very thick,dotted] (15:0.5) -- (45:0.5);
        \draw[red,very thick,dotted] (135:0.5) -- (165:0.5);
        \draw[red,very thick,dotted] (255:0.5) -- (285:0.5);
        \draw[decorate,decoration={brace,amplitude=3pt,raise=4pt}] (A4) --node[midway,xshift=12pt,yshift=8pt]{$n$} (A1);
        \draw[decorate,decoration={brace,amplitude=3pt,raise=4pt}] (B4) --node[midway,xshift=-12pt,yshift=8pt]{$n$} (B1);
        \draw[decorate,decoration={brace,amplitude=3pt,raise=4pt}] (C4) --node[midway,below=7pt]{$n$} (C1);
        \node[red] at (0,0.6) {$u$};
        \draw[dashed] (0,0) circle(0.76cm);
        \end{scope}
        \begin{scope}[xshift=6cm]
        \draw[dashed] (0,0) circle(0.76cm);
        \node at (1.5,0) {\scalebox{1.2}{$\sim$}};
        \clip (0,0) circle(0.76cm);
        \draw[red,very thick,->-={0.8}{}] (-1,0) --node[midway,above]{$nu$} (1,0);
        \end{scope}
        \begin{scope}[xshift=9cm]
        \draw[dashed] (0,0) circle(0.76cm);
        \draw[decorate,decoration={brace,amplitude=3pt,raise=2pt}] (0.76,0.2) --node[midway,right=0.3em]{$n$} (0.76,-0.2);
        \clip (0,0) circle(0.76cm);
        \draw[red,very thick,->-={0.8}{}] (-1,0.2) --node[above]{$u$} (1,0.2);
        \draw[red,very thick,dotted] (0,0.2) -- (0,-0.2);
        \draw[red,very thick,->-={0.8}{}] (-1,-0.2) --node[below]{$u$} (1,-0.2);
        \end{scope}
        \end{tikzpicture}
    \end{align*}
    For a loop or arc component, it is just a successive applications of the operation (2). One can also verify that the cabling operation is associative in the sense that the $n$-cabling followed by the $m$-cabling agrees with the $nm$-cabling, since $nm$-cabling is dual to the $m$-th subdivision of an $n$-triangulation (recall \cref{fig:honeycomb-web}). 
    %\item Resolve a pair of oriented ends incident to a common puncture with the same weight, with the opposite orientations and the opposite signs, into an oriented curve as in \cref{fig:puncture-equivalence}.
\end{enumerate}
\end{dfn}

See \cref{fig:global example} for a global example. 
Let $\cL^x(\Sigma,\bQ)$ denote the set of equivalence classes of the rational unbounded $\mathfrak{sl}_3$-laminations on $\Sigma$. We have a natural $\bQ_{>0}$-action on $\cL^x(\Sigma,\bQ)$ that simultaneously rescales the weights. 
A rational unbounded $\fsl_3$-lamination is said to be \emph{integral} if all the weights are integers. The subset of integral unbounded $\mathfrak{sl}_3$-laminations is denoted by $\cL^x(\Sigma,\bZ)$.

\begin{figure}[htbp]
%\iffalse
\centering
\begin{tikzpicture}[scale=.9]
\draw(0,-1.5) ellipse(0.2cm and 1cm);
\draw(0,1.5) ellipse(0.2cm and 1cm);
\node[fill,circle,inner sep=1.5pt] at (0.2,-1.5) {};
\node[fill,circle,inner sep=1.5pt] at (0,2.5) {};
\node[fill,circle,inner sep=1.5pt] at (0,0.5) {};
\draw(0,-0.5) ..controls (0.5,-0.5) and (0.5,0.5).. (0,0.5);
%surface
\draw(0,2.5) ..controls ++(0.5,-0.5) and (5,3).. node[inner sep=0,pos=0.9](T){} (6,3)
..controls (7,3) and (9,2).. (9,0)
..controls (9,-2) and (7,-3).. (6,-3)
..controls (5,-3) and (0.5,-2).. (0,-2.5) 
node[pos=0.4](A){} node[pos=0.2](B){};
%loops
\draw[red,very thick,-<-] (A) arc(-90:90:0.2cm and 1cm) 
coordinate(A');
\draw[red,very thick,dashed] (A) arc(-90:-270:0.2cm and 1cm);
\draw[red,very thick,->-={0.7}{}] (B) arc(-90:90:0.2cm and 1cm) 
coordinate(B');
\draw[red,very thick,dashed] (B) arc(-90:-270:0.2cm and 1cm);
%handle
\draw[shorten >=-15pt,shorten <=-15pt] (A') to[bend right=20] 
node[inner sep=0,pos=-0.15](A''){} node[inner sep=0,pos=1.15](B''){} 
(B');
\draw(A'') to[bend left=20] node[inner sep=0,pos=0.7](C){} (B'');
%puncture
\path (4,1) coordinate(P1);
\path(6,0.75) coordinate(P2);
%web vertices
\path (1,0) coordinate(X);
\path (2,0) coordinate(Y);
\path (2,1) coordinate(Z);
\path ($(0,-1.5)+(0.1*1.732,1*0.5)$) coordinate(BP1);
\path ($(0,1.5)+(0.1*1.732,-1*0.5)$) coordinate(BP2);
\path ($(0,-1.5)+(0.1*1.732,-1*0.5)$) coordinate(BP3);
\path ($(0,1.5)+(0.1*1.732,1*0.5)$) coordinate(BP4);
%web
\draw[red,very thick,->-] (X) -- (Y);
\draw[red,very thick,-<-] (Y) -- (Z);
\draw[red,very thick,->-] (X) to[out=-135,in=45] (BP1);
\draw[red,very thick,->-] (X) to[out=135,in=-45] (BP2);
\draw[red,very thick,->-] ($(A)+(0.1*1.732,1+1*0.5)$) to[out=180,in=-45] (Y);
\draw[red,very thick,-<-] ($(A)+(0.1*1.732,1-1*0.5)$) to[out=180,in=0] (BP3);
\draw[red,very thick,->-] (Z) to[out=180,in=0] (BP4);
\draw[red,very thick,->-] (Z) to [out=0,in=180] (P1);
\draw[red,very thick,->-={0.4}{}] (P2) ..controls ++(-120:0.5) and ($(C)+(0.3,0)$).. (C) node[inner sep=0,pos=0.5](S){};
\draw[red,very thick,->-] (P2) ..controls ++(120:0.5) and ($(T)+(0.3,0)$).. (T);
\draw[red,very thick,dashed] (T) ..controls ++(-0.3,0) and ($(C)+(-0.3,0)$).. (C);
\draw[red,very thick,-<-={0.7}{}] (S) to[out=150,in=-45] (P1);
\node[red,scale=0.9] at ($(P1)+(150:0.3)$) {$+$};
\node[red,scale=0.9] at ($(P1)+(-10:0.3)$) {$+$};
\node[red,scale=0.9] at ($(P2)+(80:0.3)$) {$+$};
\node[red,scale=0.9] at ($(P2)+(-90:0.3)$) {$-$};
\draw[fill=white](P1) circle(2pt);
\draw[fill=white](P2) circle(2pt);
\node[red] at ($(Z)+(0,0.3)$) {$u_1$};
\node[red] at ($(S)+(0.3,-0.1)$) {$u_2$};
\node[red] at ($(B)+(0.5,1)$) {$u_3$};
\end{tikzpicture}
%\fi%%
    \caption{An example of a rational unbounded $\fsl_3$-lamination. Here $u_1,u_2,u_3$ are arbitrary positive rational weights.}
    \label{fig:global example}
\end{figure}
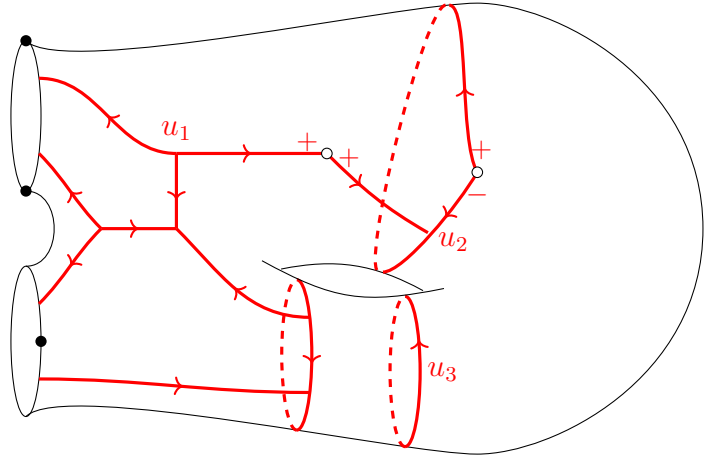
The set $\cL^x(\Sigma,\bQ)$ (resp. $\cL^x(\Sigma,\bZ)$) will be identified with the unfrozen part $\X_{\mathfrak{sl}_3,\Sigma}^\uf(\bQ^T)$ (resp. $\X_{\mathfrak{sl}_3,\Sigma}^\uf(\bZ^T)$) of the tropical cluster $\X$-variety associated with the pair $(\mathfrak{sl}_3,\Sigma)$ (see \cref{subsec:cluster_sl3}). 
%Before proceeding to the construction of coordinate systems on this space, let us briefly discuss several maps related to them.

\begin{conv}\label{notation:division of honeycombs}
In view of the equivalence relation (4), we will occasionally use the following equivalent notations for honeycombs:
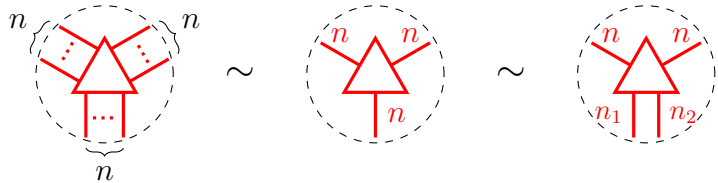
\begin{align*}
\begin{tikzpicture}[scale=1.2]
\draw[red,very thick] (-30:0.4) -- (90:0.4) -- (210:0.4) --cycle;
\foreach \i in {1,4}
{
\draw[red,very thick] ($(-30:0.4)!\i/5!(90:0.4)$) --++(30:0.5) coordinate(A\i);
\draw[red,very thick] ($(90:0.4)!\i/5!(210:0.4)$) --++(150:0.5) coordinate(B\i);
\draw[red,very thick] ($(210:0.4)!\i/5!(-30:0.4)$) --++(-90:0.5) coordinate(C\i);
}
\draw[red,very thick,dotted] (15:0.5) -- (45:0.5);
\draw[red,very thick,dotted] (135:0.5) -- (165:0.5);
\draw[red,very thick,dotted] (255:0.5) -- (285:0.5);
\draw[decorate,decoration={brace,amplitude=3pt,raise=4pt}] (A4) --node[midway,xshift=12pt,yshift=8pt]{$n$} (A1);
\draw[decorate,decoration={brace,amplitude=3pt,raise=4pt}] (B4) --node[midway,xshift=-12pt,yshift=8pt]{$n$} (B1);
\draw[decorate,decoration={brace,amplitude=3pt,raise=4pt}] (C4) --node[midway,below=7pt]{$n$} (C1);
\draw[dashed] (0,0) circle(0.76cm);
\node at (1.5,0) {\scalebox{1.2}{$\sim$}}; 
\begin{scope}[xshift=3cm]
\draw[red,very thick] (-30:0.4) -- (90:0.4) -- (210:0.4) --cycle;
\draw[red,very thick] ($(-30:0.4)!0.5!(90:0.4)$) --node[midway,above]{$n$} ++(30:0.5);
\draw[red,very thick] ($(90:0.4)!0.5!(210:0.4)$) --node[midway,above]{$n$} ++(150:0.5);
\draw[red,very thick] ($(210:0.4)!0.5!(-30:0.4)$) --node[midway,right]{$n$} ++(-90:0.5);
\draw[dashed] (0,0) circle(0.76cm);
\node at (1.5,0) {\scalebox{1.2}{$\sim$}}; 
\end{scope}
\begin{scope}[xshift=6cm]
\draw[red,very thick] (-30:0.4) -- (90:0.4) -- (210:0.4) --cycle;
\draw[red,very thick] ($(-30:0.4)!0.5!(90:0.4)$) --node[midway,above]{$n$} ++(30:0.5);
\draw[red,very thick] ($(90:0.4)!0.5!(210:0.4)$) --node[midway,above]{$n$} ++(150:0.5);
\draw[red,very thick] ($(210:0.4)!0.3!(-30:0.4)$) --node[midway,left,scale=0.9]{$n_1$} ++(-90:0.5);
\draw[red,very thick] ($(210:0.4)!0.7!(-30:0.4)$) --node[midway,right,scale=0.9]{$n_2$} ++(-90:0.5);
\draw[dashed] (0,0) circle(0.76cm);
\end{scope}
\end{tikzpicture}
\end{align*}
with $n_1+n_2=n$. We may also split an edge of weight $n$ with $k$ edges of weight $n_1,\dots,n_k$ with $n_1+\dots+n_k=n$.
\end{conv}

\begin{dfn}[Dynkin involution]\label{def:Dynkin_geometric}
The \emph{Dynkin involution} is the involutive automorphism
\begin{align*}
    \ast: \cL^x(\Sigma,\bQ) \to \cL^x(\Sigma,\bQ), \quad \hL \mapsto \hL^\ast,
\end{align*}
where $\hL^\ast$ is obtained from $\hL$ by reversing the orientation of every components of the underlying web, and keeping the signs at punctures intact. 
%, and $\delta^\ast=(\delta_E^\ast)_{E \in \mathbb{B}}$ is obtained from $\delta$ by the Dynkin involution on the coweight lattice: $\varpi^\ast_s:=\varpi_{3-s}$ for $s=1,2$. 
Since all the elementary moves (E1)--(E4) are equivariant under the orientation-reversion, this indeed defines an automorphism on $\cL^x(\Sigma,\bQ)$. 
\end{dfn}

\paragraph{\textbf{Bounded laminations and the ensemble map}}
%Douglas--Sun \cite{DS20I} investigated a closely related space $\mathcal{W}_{\Sigma}$ of webs, and defined certain coordinate systems associated with ideal triangulations of $\Sigma$ which behaves as the tropical $\A$-coordinates under the flips \cite{DS20II}. As treated by Kim \cite{Kim21}, their space of webs can be further enhanced by allowing the peripheral components surrounding punctures to have negative weights. Here we briefly explain a relation to their work. 

\begin{dfn}[rational bounded $\fsl_3$-laminations]\label{def:bounded laminations}
A \emph{rational bounded $\fsl_3$-lamination} (or a \emph{rational $\fsl_3$-$\A$-lamination}) on $\Sigma$ is a bounded non-elliptic $\fsl_3$-web $W$ on $\Sigma$ equipped with a rational number (called the \emph{weight}) on each component such that the weight on a non-peripheral component is positive. It is considered modulo the equivalence relation generated by isotopies and the operations (2)--(4) in \cref{def:unbounded laminations}.
\end{dfn}
Let $\cL^a(\Sigma,\bQ)$ denote the space of rational bounded $\fsl_3$-laminations. We have a natural $\bQ_{>0}$-action on $\cL^a(\Sigma,\bQ)$ that simultaneously rescales the weights. A rational bounded $\fsl_3$-lamination is said to be \emph{integral} if all the weights are integers. The subset of integral bounded $\fsl_3$-laminations is denoted by $\cL^a(\Sigma,\bZ)$. 

\begin{rem}\label{rem:bounded_lamination}
The space $\cL^a(\Sigma,\bZ)$ is the same one as the space $\A_L(\Sigma;\bZ)$ that appears in Kim's work (\cite[Definition 3.9]{Kim21})\footnote{Indeed, an element of our space $\cL^a(\Sigma,\bZ)$ can be represented by a reduced web (\cite[Definition 3.3]{Kim21}) by applying the boundary H-moves, and we can rescale the weights on honeycombs to be $1$ by the operation (4) in \cref{def:unbounded laminations}.}. The space $\mathcal{W}_\Sigma$ in Douglas--Sun's work (\cite[Definition 6]{DS20I}) is the subset of $\cL^a(\Sigma,\bZ)$ consisting of elements with positive peripheral weights. 
It is straightforward to extend their coordinate systems by $\bQ_{>0}$-equivariance to the rational case, and the space $\cL^a(\Sigma,\bQ)$ is identified with the tropical cluster $\A$-variety $\A_{\fsl_3,\Sigma}(\bQ^T)$ (\cite[Theorem 3.39]{Kim21})\footnote{Here note that there is a subset of $\cL^a(\Sigma,\bZ)$ formed by \emph{congruent} laminations (\cite[Definition 3.38]{Kim21}) which is identified with the tropical cluster $\A$-variety $\A_{\fsl_3,\Sigma}(\bZ^T)$.}.
\end{rem}

By forgetting the peripheral components, we get the \emph{geometric ensemble map}
\begin{align}\label{eq:ensemble_unfrozen}
    p: \cL^a(\Sigma,\bQ) \to \cL^x(\Sigma,\bQ).
\end{align}
We will see in \cref{sec:amalgamation} that the geometric ensemble map coincides with the cluster ensemble map \eqref{eq:ensemble_map} via the Douglas--Sun coordinates and our shear coordinates.

\section{Shear coordinates}\label{sec:shear_coord}

\subsection{Essential webs on polygons}
Let $\bD_k$ denote a disk with $k \geq 2$ special points. In what follows, we simply refer to $\bD_k$ as a \emph{$k$-gon}.
We say that an $\mathfrak{sl}_3$-web $W$ on $\bD_k$ is \emph{taut} if for any compact embedded arc $\alpha$ whose endpoints lie in a common boundary interval $E$, the number of intersection points of $W$ with $E$ does not exceed that of $W$ with $\alpha$. See \cref{fig:non-taut}. 
Following \cite{DS20I}, we call a non-elliptic, taut $\mathfrak{sl}_3$-web an \emph{essential} web. These essential webs on polygons are basic building blocks for the bounded $\mathfrak{sl}_3$-laminations studied in \cite{DS20I}. We recall the concrete description of the essential webs for $k=2,3$ following \cite[Sections 2.7 and 2.8]{DS20I} and \cite[Sections 8 and 9]{FS20}, including additional infinite webs needed for our purpose. 

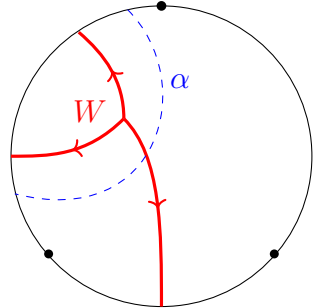
\begin{figure}[htbp]
    \centering
    \begin{tikzpicture}
    \draw (0,0) ellipse (2 and 2);
    \node [fill, circle, inner sep=1.3pt] at (0,2) {};
    \node [fill, circle, inner sep=1.2pt] at (-1.5,-1.3) {};
    \node [fill, circle, inner sep=1.2pt] at (1.5,-1.3) {};
    \draw [red, very thick, ->-](-0.5,0.5) .. controls (-0.5,1) and (-0.75,1.35) .. (-1.1,1.65);
    \draw [red, very thick, ->-](-0.5,0.5) .. controls (-1,0) and (-1.5,0) .. (-2,0);
    \draw [red, very thick, ->-](-0.5,0.5) .. controls (0,0) and (0,-1) .. (0,-2);
    \draw [dashed, blue](-0.45,1.95) .. controls (0.5,1) and (0,-1) .. (-1.95,-0.5);
    \node [blue] at (0.25,1) {$\alpha$};
    \node [red] at (-0.95,0.6) {$W$};
    \end{tikzpicture}
    \caption{Example of a non-taut web in $\bD_3$.}
    \label{fig:non-taut}
\end{figure}

\bigskip
\paragraph{\textbf{The biangle (2-gon) case.}}
Let $E_L,E_R$ denote the boundary intervals of a biangle $\bD_2$. A \emph{(finite) symmetric strand set} on $\bD_2$ is a pair $S=(S_L,S_R)$ of finite collections of disjoint oriented strands (\emph{i.e.}, germs of oriented arcs), where the oriented strands in $S_Z$ are located on $E_Z$ for $Z \in \{L,R\}$ such that the number of incoming (resp. outgoing) strands on $E_L$ is equal to the outgoing (resp. incoming) strands on $E_R$. See the left-most picture in \cref{fig:ladder-web} for an example. 

Given a symmetric strand set $S=(S_L,S_R)$, the associated \emph{ladder-web} $W(S)$ on $\bD_2$ is constructed as follows. First, let $W_\mathrm{br}(S)$ be the unique (up to ambient isotopy of $\bD_2$) collection of oriented curves connecting strands in $S_L$ with those in $S_R$ in the order-preserving and minimally-intersecting way. See the middle picture in \cref{fig:ladder-web}. It is characterized by the \emph{pairing map} $f:S_L \to S_R$, which is an order-preserving bijection that maps each incoming (resp. outgoing) strand of $S_L$ to an outgoing (resp. incoming) strand of $S_R$. 
The associated ladder-web $W(S)$ is obtained from $W_\mathrm{br}(S)$ by replacing each intersection with an H-web, as follows:
\begin{align}\label{eq:H-replacement}
    \begin{tikzpicture}[scale=0.8]
    \begin{scope}[rotate=90]
    \draw[dashed](0,0) circle(1.118cm);
    {\color{red}
    \draw[very thick,->-={0.7}{}] (-0.5,-1) ..controls (-0.5,-0.5) and (0.5,0.5).. (0.5,1);
    \draw[very thick,->-={0.3}{}] (-0.5,1) ..controls (-0.5,0.5) and (0.5,-0.5).. (0.5,-1);
    }
    \end{scope}
    \node[scale=1.5] at (2.5,0) {$\rightsquigarrow$};
    % \draw [thick,-{Classical TikZ Rightarrow[length=4pt]},decorate,decoration={snake,amplitude=2pt,pre length=2pt,post length=3pt}] (2,0) -- (3,0);
    %
    \begin{scope}[xshift=5cm,rotate=90]
    \draw[dashed](0,0) circle(1.118cm);
    {\color{red}
    \draw[very thick,->-] (-0.5,-1) -- (-0.5,0);
    \draw[very thick,->-] (-0.5,1) -- (-0.5,0);
    \draw[very thick,-<-] (-0.5,0) -- (0.5,0);
    \draw[very thick,->-] (0.5,0) -- (0.5,-1);
    \draw[very thick,->-] (0.5,0) -- (0.5,1);
    }
    \end{scope}
    \end{tikzpicture}
\end{align}
Conversely, the collection $W_\mathrm{br}(S)$ is called the \emph{braid representation} of the ladder-web $W(S)$. 
It is known that all the essential webs on $\bD_2$ arise in this way:

\begin{figure}[htbp]
\centering
\begin{tikzpicture}[scale=0.9]
\begin{scope}[rotate=90]
\fill(2,0) circle(2pt);
\fill(-2,0) circle(2pt);
\draw[blue](-2,0) to[out=70,in=180] (-1,1) -- (1,1) to[out=0,in=110] (2,0);
\draw[blue](-2,0) to[out=-70,in=180] (-1,-1) -- (1,-1) to[out=0,in=-110] (2,0);
\foreach \i in {-1,-0.5,0}
\draw[red,very thick,-latex] (\i,-1.18)--++(90:0.36);
\foreach \i in {0.5,1}
\draw[red,very thick,-latex] (\i,-1+0.18)--++(-90:0.36);
\foreach \i in {-1,1}
\draw[red,very thick,-latex] (\i,1.18)--++(-90:0.36);
\foreach \i in {-0.5,0,0.5}
\draw[red,very thick,-latex] (\i,1-0.18)--++(90:0.36);
\end{scope}

\begin{scope}[xshift=5cm,rotate=90]
\fill(2,0) circle(2pt);
\fill(-2,0) circle(2pt);
\draw[blue](-2,0) to[out=70,in=180] (-1,1) -- (1,1) to[out=0,in=110] (2,0);
\draw[blue](-2,0) to[out=-70,in=180] (-1,-1) -- (1,-1) to[out=0,in=-110] (2,0);
\foreach \i in {-1,-0.5,0}
\draw[red,thick,->-={0.2}{}] (\i,-1) ..controls (\i,0) and (\i+0.5,0).. (\i+0.5,1);
\draw[red,thick,->-={0.9}{}] (-1,1) ..controls (-1,0) and (0.5,0).. (0.5,-1);
\draw[red,thick,->-] (1,1) -- (1,-1);
\end{scope}

\begin{scope}[xshift=10cm,rotate=90]
\fill(2,0) circle(2pt);
\fill(-2,0) circle(2pt);
\draw[blue](-2,0) to[out=70,in=180] (-1,1) -- (1,1) to[out=0,in=110] (2,0);
\draw[blue](-2,0) to[out=-70,in=180] (-1,-1) -- (1,-1) to[out=0,in=-110] (2,0);
\foreach \i in {-1,-0.5,0}
{
\draw[red,thick,->-={0.3}{}] (\i,-1) -- (\i,0);
\draw[red,thick,-<-={0.3}{}] (\i+0.5,1) -- (\i+0.5,0);
\draw[red,thick] (\i,-\i-0.5) --++(0.5,0);
}
\draw[red,thick,->-={0.3}{}] (-1,1) -- (-1,0);
\draw[red,thick,-<-={0.3}{}] (0.5,-1) -- (0.5,0);
\draw[red,thick,-<-={0.3}{}] (0.5,1) -- (0.5,0);
\draw[red,thick,->-] (1,1) -- (1,-1);
\end{scope}
 \end{tikzpicture}
    \caption{Construction of the ladder-webs. Left: a symmetric set $S$. Middle: the corresponding collection of oriented curves $W_\mathrm{br}(S)$. Right: the associated ladder-web $W(S)$.}
    \label{fig:ladder-web}
\end{figure}
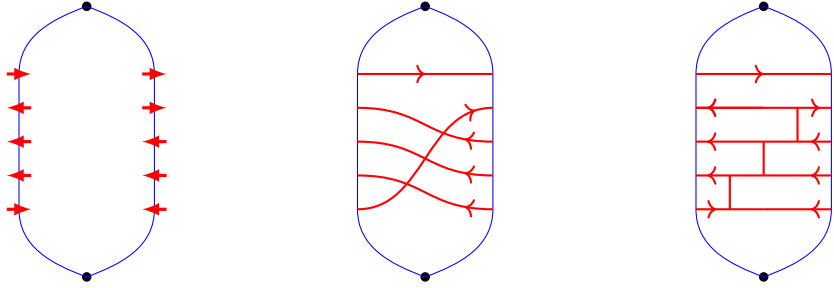

\begin{prop}[{\cite[Section 8]{FS20}, \cite[Proposition 19]{DS20I}}]
The ladder-web $W(S)$ is an essential web on $\bD_2$ for any symmetric strand set $S$. Conversely, given an essential web $W$ on $\bD_2$, there exists a unique symmetric strand set $S$ such that $W=W(S)$. 
\end{prop}
%We remark that a ladder-web may contain a collection of disjoint oriented arcs located on the corners of $\bD_2$. 
For the study of unbounded $\mathfrak{sl}_3$-webs, we need the following infinite extension of the symmetric strand sets.

\begin{dfn}[asymptotically periodic symmetric strand sets]\label{def:infinite strand set}
An \emph{asymptotically periodic symmetric strand set} $S=(S_L,S_R)$ on $\bD_2$ consists of countable collections $S_L$, $S_R$ of disjoint oriented strands, where the oriented strands in $S_Z$ are located on $E_Z$ without accumulation points for $Z \in \{L,R\}$. 
%, and one distinguished strand $s_i \in S_Z$ for each $Z \in \{L,R\}$. 
The oriented strands are required to be symmetric, and periodic away from a compact set (see \cref{fig:ladder_infinite}). Namely, we require that there exists a compact strip $K \subset \bD_2 \setminus \bM$ such that
\begin{itemize}
    %\item there exists a bijection $f:S_L \to S_R$ that sends incoming (resp. outgoing) strands in $S_L$ to outgoing (resp. incoming) strands in $S_R$, and preserves the ordering among the incoming strands and outgoing strands, respectively;
    %and $f_S(S_L)=s_2$;
    %\item there exists a compact set $K \subset \interior E_L$ such that the orientation pattern of the strands in $S_L$ that belongs to each connected component of $\interior E_L \setminus K$ is periodic.
    \item $K$ is bounded by two parallel arcs $\alpha_1,\alpha_2$ transverse to the boundary intervals of $\bD_2$, and $\alpha_1 \cup \alpha_2$ %is away from 
    avoiding the strand sets $S_L$, $S_R$;
    \item the pair $(S_L \cap K, S_R \cap K)$ is a finite symmetric strand set;
    \item the orientation patterns of the strands in the sets $S_L$ and $S_R$ that belong to $\bD_2 \setminus K$ are periodic, and the pairing map $f_K: S_L \cap K \to S_R \cap K$ of finite symmetric strand set can be extended to an order-preserving bijection $f:S_L \to S_R$ that maps each incoming (resp. outgoing) strand of $S_L$ to an outgoing (resp. incoming) strand of $S_R$.
    %the orientation patterns of the strands in the sets $S_L$ and $S_R$ that belong to $\bD_2 \setminus K$ are periodic, and their periods agree with each other in each connected component of $\bD_2 \setminus K$.
\end{itemize}
\end{dfn}
%The third condition implies that there exists an order-preserving bijection $f:S_L \to S_R$ that maps each incoming (resp. outgoing) strand of $S_L$ to an outgoing (resp. incoming) strand of $S_R$, which restricts to a pairing map $S_L \cap K \to S_R \cap K$ of finite symmetric strand set. 
Unlike the finite case, the pairing map $f$ may not be unique, as it depends on the choice of the compact strip $K$. 
%See \cref{fig:infinite_strand_set}. 
Given such a pair $(S,f)$, we get a collection $W_\mathrm{br}(S,f)$ of oriented curves mutually in a minimal position, and the associated ladder-web $W(S,f)$ just in the same manner as in the finite case. We call $W(S,f)$ the ladder-web associated with the pair $(S,f)$. It is possibly an infinite web.

\begin{figure}[ht]
    \centering
\begin{tikzpicture}[scale=0.9]
\fill[gray!20] (-1,0.75) -- (1,0.75) -- (1,-0.75) -- (-1,-0.75) --cycle;
\node at (0,0) {$K$};
\draw(-1,0.75)--node[midway,above]{$\alpha_1$}(1,0.75);
\draw(-1,-0.75)--node[midway,below]{$\alpha_2$}(1,-0.75);
\fill(0,2) circle(2pt);
\fill(0,-2.5) circle(2pt);
\draw[blue](0,-2.5) to[out=20,in=-90] (1,-1.5) -- (1,1) to[out=90,in=-20] (0,2);
\draw[blue](0,-2.5) to[out=160,in=-90] (-1,-1.5) -- (-1,1) to[out=90,in=-160] (0,2);
\foreach \i in {-0.3,0,0.6} \arr{-1,\i};
\foreach \i in {-0.6,0.3} \arl{-1,\i};
\foreach \i in {-0.3,0.3,0.6} \arr{1,\i};
\foreach \i in {-0.6,0} \arl{1,\i};
%periodic part
\foreach \i in {0.9,1.2} \arr{-1,\i};
\arl{-1,-0.9};
\arr{-1,-1.2};
\arl{-1,-1.5};
\arr{-1,-1.8};
\foreach \i in {0.9,1.2} \arr{1,\i};
\arr{1,-0.9};
\arl{1,-1.2};
\arr{1,-1.5};
\arl{1,-1.8};
\begin{scope}[xshift=4cm]
\fill[gray!20] (-1,0.75) -- (1,0.75) -- (1,-0.75) -- (-1,-0.75) --cycle;
\draw(-1,0.75)--(1,0.75);
\draw(-1,-0.75)--(1,-0.75);
\fill(0,2) circle(2pt);
\fill(0,-2.5) circle(2pt);
\draw[blue](0,-2.5) to[out=20,in=-90] (1,-1.5) -- (1,1) to[out=90,in=-20] (0,2);
\draw[blue](0,-2.5) to[out=160,in=-90] (-1,-1.5) -- (-1,1) to[out=90,in=-160] (0,2);
\foreach \i in {-0.3,0.6,0.9,1.2} \draw[red,thick,->-] (-1,\i) -- (1,\i);
\draw[red,thick,->-={0.6}{}] (1,-0.6) -- (-1,-0.6);
\draw[red,thick,->-={0.3}{},-<-={0.7}{}] (-1,0) -- (1,0);
\draw[red,thick,-<-={0.3}{},->-={0.7}{}] (-1,0.3) -- (1,0.3);
\draw[red,thick] (0,0) -- (0,0.3);
%periodic part
\draw[red,thick,->-={0.3}{},-<-={0.7}{}] (-1,-1.2) -- (1,-1.2);
\draw[red,thick,-<-={0.3}{},->-={0.7}{}] (-1,-0.9) -- (1,-0.9);
\draw[red,thick] (0,-1.2) -- (0,-0.9);
\draw[red,thick,->-={0.3}{},-<-={0.7}{}] (-1,-1.8) -- (1,-1.8);
\draw[red,thick,-<-={0.3}{},->-={0.7}{}] (-1,-1.5) -- (1,-1.5);
\draw[red,thick] (0,-1.5) -- (0,-1.8);
\end{scope}
\begin{scope}[xshift=8cm]
\fill[gray!20] (-1,0.75) -- (1,0.75-0.3) -- (1,-0.75-0.3) -- (-1,-0.75) --cycle;
\node at (0,0) {$K'$};
\draw(-1,0.75)--node[midway,above]{$\alpha'_1$}(1,0.75-0.3);
\draw(-1,-0.75)--node[midway,below]{$\alpha'_2$}(1,-0.75-0.3);
\fill(0,2) circle(2pt);
\fill(0,-2.5) circle(2pt);
\draw[blue](0,-2.5) to[out=20,in=-90] (1,-1.5) -- (1,1) to[out=90,in=-20] (0,2);
\draw[blue](0,-2.5) to[out=160,in=-90] (-1,-1.5) -- (-1,1) to[out=90,in=-160] (0,2);
\foreach \i in {-0.3,0,0.6} \arr{-1,\i};
\foreach \i in {-0.6,0.3} \arl{-1,\i};
\foreach \i in {-0.3,0.3,0.6} \arr{1,\i};
\foreach \i in {-0.6,0} \arl{1,\i};
%periodic part
\foreach \i in {0.9,1.2} \arr{-1,\i};
\arl{-1,-0.9};
\arr{-1,-1.2};
\arl{-1,-1.5};
\arr{-1,-1.8};
\foreach \i in {0.9,1.2} \arr{1,\i};
\arr{1,-0.9};
\arl{1,-1.2};
\arr{1,-1.5};
\arl{1,-1.8};
\end{scope}
\begin{scope}[xshift=12cm]
\fill[gray!20] (-1,0.75) -- (1,0.75-0.3) -- (1,-0.75-0.3) -- (-1,-0.75) --cycle;
\draw(-1,0.75)--(1,0.75-0.3);
\draw(-1,-0.75)--(1,-0.75-0.3);
\fill(0,2) circle(2pt);
\fill(0,-2.5) circle(2pt);
\draw[blue](0,-2.5) to[out=20,in=-90] (1,-1.5) -- (1,1) to[out=90,in=-20] (0,2);
\draw[blue](0,-2.5) to[out=160,in=-90] (-1,-1.5) -- (-1,1) to[out=90,in=-160] (0,2);
\foreach \i in {1.2,0.9,0.6,0} \draw[red,thick,->-] (-1,\i) -- (1,\i-0.3);
\draw[red,thick,-<-] (-1,0.3) -- (1,0);
\draw[red,thick,->-={0.3}{},-<-={0.7}{}] (-1,-0.3) -- (1,-0.6);
\draw[red,thick,-<-={0.3}{},->-={0.7}{}] (-1,-0.6) -- (1,-0.9);
\draw[red,thick] (0,-0.45) -- (0,-0.75);
%periodic part
\foreach \i in {-1.2} \draw[red,thick,->-] (-1,\i) -- (1,\i-0.3);
\foreach \i in {-0.9,-1.5} \draw[red,thick,-<-] (-1,\i) -- (1,\i-0.3);
\draw[red,thick,->-={0.9}{}] (-0.95,-1.8) --++(0.5*2,-0.5*0.3);
\draw[red,thick,-<-={0.9}{}] (0.95,1.2) --++(-0.5*2,0.5*0.3);
\end{scope}
\end{tikzpicture}
    \caption{An asymptotically periodic symmetric strand set and the associated ladder webs corresponding to the two choices of compact strips $K$ and $K'$.} 
    \label{fig:ladder_infinite}
\end{figure}

\begin{dfn}
An \emph{unbounded essential web} on $\bD_2$ is the isotopy class of the ladder-web associated with a pair $(S,f)$ as above.
\end{dfn}

Among the others, the following way of fixing a pairing map turns out to be useful in this paper.
%\cref{subsec:reconstruction}. 

\begin{dfn}\label{def:pinning}
A \emph{pinning} of an asymptotically periodic symmetric strand set $S=(S_L,S_R)$ is a pair $\sfp_Z=(p_Z^+,p_Z^-)$ of points in $E_Z$ away from the set $S_Z$ for $Z \in \{L,R\}$. 
The resulting tuple 
$\widehat{S}:=(S;\sfp_L,\sfp_R)$ is called a \emph{pinned symmetric strand set}. 
\end{dfn}
Then we define the pairing map
%associated collection $W(\widehat{S})$ of oriented curves mutually in minimal position
as follows. 
For $Z \in \{L,R\}$, let us decompose $S_Z=S_Z^+ \sqcup S_Z^-$, where $S_Z^+$ (resp. $S_Z^-$) denotes the subset of incoming (resp. outgoing) strands. 
Then there exist orientation-reversing homeomorphisms $f_\pm: E_L \to E_R$ such that $f_\pm(S_L^\pm) = S_R^\mp$ and $f_\pm(p_L^\pm) = p_R^\mp$. 
Then we get the unique pairing map 
\begin{align*}
    f_{\widehat{S}}:=f_+\sqcup f_-: S_L^+ \sqcup S_L^- \to S_R^- \sqcup S_R^+,
\end{align*}
which determines the collection $W_{\mathrm{br}}(\widehat{S}):=W_{\mathrm{br}}(S,f_{\widehat{S}})$ of oriented curves and the associated ladder-web $W(\widehat{S}):=W(S,f_{\widehat{S}})$.

\bigskip
\paragraph{\textbf{The triangle (3-gon) case.}}
Let $\bD_3$ be a triangle. Recall that we have honeycomb-webs on $\bD_3$, which are dual to $n$-triangulations of $\bD_3$. 

\begin{prop}[{\cite[Theorem 19]{FS20}, \cite[Proposition 22]{DS20I}}]\label{lem:triangle_essential}
A honeycomb-web is 
%rung-less 
reduced (\emph{rung-less} in terms of \cite{DS20I})
and essential. Conversely, any connected reduced essential web on $\bD_3$ having at least one trivalent vertex is a honeycomb-web. 
\end{prop}
Consequently, any reduced essential web on $\bD_3$ consists of a unique (possibly empty) honeycomb component together with a collection of disjoint oriented arcs located on the corners of $\bD_3$. These oriented arcs are called \emph{corner arcs}. Similarly to the biangle case, we may allow the collection of corner arcs to be semi-infinite and asymptotically periodic. 

\begin{dfn}
An \emph{unbounded reduced essential web} on $\bD_3$ is the isotopy class of a disjoint union of a (possibly empty) reduced essential web on $\bD_3$ and at most one semi-infinite periodic collection of corner arcs around each corner. 
\end{dfn}

% \begin{figure}
% \begin{tikzpicture}
% \draw(-2,0) -- (2,0) -- (0,2*1.732) --cycle;
% \fill (-2,0) circle(2pt);
% \fill (2,0) circle(2pt);
% \fill (0,2*1.732) circle(2pt);
% \foreach \i in {0,0.2,0.4,0.6,0.8}
% \draw[red,very thick] (0+\i,0) arc(180:120:2-\i);

% \begin{scope}[xshift=5cm]
% \draw(-2,0) -- (2,0) -- (0,2*1.732) --cycle;
% \fill (-2,0) circle(2pt);
% \fill (2,0) circle(2pt);
% \fill (0,2*1.732) circle(2pt);
% \foreach \i in {0,0.2,0.4,0.6,0.8}
% \draw[red,very thick] (0+\i,0) arc(180:120:2-\i);
% {\begin{scope}[xshift=2cm]
% \foreach \i in {0,0.2,0.4,0.6,0.8}
% \draw[red,very thick] (170-40*\i:2-\i) -- (170-40*\i:2-\i-0.2);
% \end{scope}}
% \end{scope}

% \begin{scope}[xshift=10cm]
% \draw(-2,0) -- (2,0) -- (0,2*1.732) --cycle;
% \fill (-2,0) circle(2pt);
% \fill (2,0) circle(2pt);
% \fill (0,2*1.732) circle(2pt);
% \foreach \i in {0,0.2,0.4,0.6,0.8}
% \draw[red,very thick] (0+\i,0) arc(180:120:2-\i);
% {\begin{scope}[xshift=2cm]
% \foreach \i in {0,0.2,0.4,0.6,0.8}
% \draw[red,very thick] (130+40*\i:2-\i) -- (130+40*\i:2-\i-0.2);
% \end{scope}}
% \end{scope}
% \end{tikzpicture}
%     \caption{Admissible infinite webs on $\bD_3$. Here an arbitrary orientation of webs is allowed.}
%     \label{fig:admissible_triangle}
% \end{figure}

%An infinite local parallel-move (TO BE WRITTEN)

\subsection{Good position of an unbounded \texorpdfstring{$\fsl_3$}{sl(3)}-lamination}
Let $\tri$ be an ideal triangulation of $\Sigma$ without self-folded triangles. Recall from \cite[Section 3]{DS20I} that a bounded $\mathfrak{sl}_3$-web $W$ on $\Sigma$ is \emph{generic} with respect to $\tri$ if none of its trivalent vertices intersect with the edges of $\tri$, and $W$ intersects with $\tri$ transversely. A \emph{generic isotopy} is an isotopy of webs through generic webs. Recall the \emph{parallel-equivalence} of bounded webs, which is the equivalence relation generated by isotopies of marked surface and the loop parallel-move (E1). 
A generic bounded web $W$ is said to be in \emph{minimal position} with respect to $\tri$ if it minimizes the sum of the intersection numbers with the edges of $\tri$ among those parallel-equivalent to $W$. Then we have:

\begin{prop}[{\cite[Section 6]{FS20}, \cite[Proposition 27]{DS20I}}]\label{prop:good_position_bounded}
Any parallel-equivalence class of non-elliptic bounded webs on $\Sigma$ has a representative in minimal position with respect to $\tri$. Moreover, such a representative is unique up to a sequence of $H$-moves across edges of $\tri$ (\cref{fig:H-move}), loop parallel-moves, and generic isotopies. 
\end{prop}
Indeed, the minimal position is realized by appropriately applying the \emph{intersection reduction moves} (a.k.a. \emph{tightening moves}) across edges of $\tri$ shown in \cref{fig:tightening}.

\begin{figure}[htbp]
    \centering
\begin{tikzpicture}
\fill[pink!70] (0.3,0.5) -- (0.3,-0.5) -- (0,-0.5) -- (0,0.5) --cycle;
%\draw[blue] (2,0) -- (0,2) -- (-2,0) -- (0,-2) --cycle;
\draw[blue] (0,1) -- (0,-1);
\draw[blue,thick,dashed] (0,0.8) ..controls ++(-20:1) and ($(0,-0.8)+(20:1)$).. (0,-0.8);
\draw[red,very thick] (-0.3,0.5) -- (0.9,0.5);
\draw[red,very thick] (-0.3,-0.5) -- (0.9,-0.5);
\draw[red,very thick] (0.3,0.5) -- (0.3,-0.5);
\draw[thick,<->] (1.5,0) -- (2.5,0);
\begin{scope}[xshift=4cm,xscale=-1]
\fill[pink!70] (0.3,0.5) -- (0.3,-0.5) -- (0,-0.5) -- (0,0.5) --cycle;
%\draw[blue] (2,0) -- (0,2) -- (-2,0) -- (0,-2) --cycle;
\draw[blue] (0,1) -- (0,-1);
\draw[blue,thick,dashed] (0,0.8) ..controls ++(-20:1) and ($(0,-0.8)+(20:1)$).. (0,-0.8);
\draw[red,very thick] (-0.3,0.5) -- (0.9,0.5);
\draw[red,very thick] (-0.3,-0.5) -- (0.9,-0.5);
\draw[red,very thick] (0.3,0.5) -- (0.3,-0.5);
\end{scope}
\end{tikzpicture}
    \caption{The H-move across an arc.}
    \label{fig:H-move}
\end{figure}
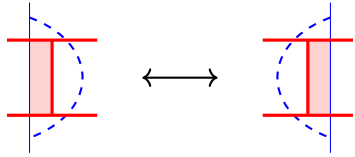

\begin{figure}[htbp]
    \centering
\begin{tikzpicture}
\fill[pink!70] (0,0.5) ..controls (0.8,0.5) and (0.8,-0.5).. (0,-0.5) -- (0,0.5);
%\draw[blue] (2,0) -- (0,2) -- (-2,0) -- (0,-2) --cycle;
\draw[blue] (0,1) -- (0,-1);
\draw[red,very thick] (-0.3,0.5) -- (0,0.5) ..controls (0.8,0.5) and (0.8,-0.5).. (0,-0.5) -- (-0.3,-0.5);
\draw[blue,thick,dashed] (0,0.8) ..controls ++(-10:1.5) and ($(0,-0.8)+(10:1.5)$).. (0,-0.8);
\draw[thick,->] (1.5,0) -- (2.5,0);
\begin{scope}[xshift=4cm]
%\draw[blue] (2,0) -- (0,2) -- (-2,0) -- (0,-2) --cycle;
\draw[blue] (0,1) -- (0,-1);
\draw[red,very thick]  (-0.8,0.5) ..controls (0,0.5) and (0,-0.5).. (-0.8,-0.5);
\end{scope}
\begin{scope}[xshift=7cm]
\fill[pink!70] (0,0.5) -- (0.5,0) -- (0,-0.5) --cycle;
%\draw[blue] (2,0) -- (0,2) -- (-2,0) -- (0,-2) --cycle;
\draw[blue] (0,1) -- (0,-1);
\draw[blue,thick,dashed] (0,0.8) ..controls ++(-20:1) and ($(0,-0.8)+(20:1)$).. (0,-0.8);
\draw[red,very thick] (-0.25,0.75) -- (0.5,0) -- (-0.25,-0.75);
\draw[red,very thick] (0.5,0) -- (1,0);
\draw[thick,->] (1.5,0) -- (2.5,0);
\end{scope}
\begin{scope}[xshift=11cm]
%\draw[blue] (2,0) -- (0,2) -- (-2,0) -- (0,-2) --cycle;
\draw[blue] (0,1) -- (0,-1);
\draw[red,very thick] (-0.85,0.75) -- (-0.1,0) -- (-0.85,-0.75);
\draw[red,very thick] (-0.1,0) -- (0.4,0);
\end{scope}
\end{tikzpicture}
    \caption{The intersection reduction moves across an arc.}
    \label{fig:tightening}
\end{figure}
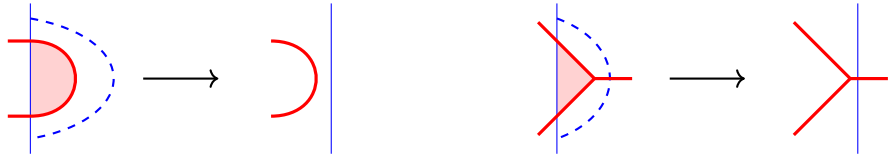

The \emph{split ideal triangulation} $\widehat{\tri}$ is obtained from $\tri$ by replacing each edge $E$ into a biangle $B_E$.
We say that a bounded web $W$ on $\Sigma$ is in \emph{good position} with respect to $\widehat{\tri}$ if the restrictions $W \cap B_E$ for $E \in e(\tri)$ (resp. $W \cap T$ for $T \in t(\tri)$) are an essential (resp. reduced essential) webs. Then it is known that any parallel-equivalence class of non-elliptic bounded webs on $\Sigma$ has a representative in good position with respect to $\widehat{\tri}$; such a representative is unique up to a sequence of modified H-moves (\cref{fig:modified H-move}), loop parallel-moves, and generic isotopies for $\widehat{\tri}$ (\cite[Corollary 18]{FS20} and \cite[Proposition 30]{DS20I}). Using such a representative, the Douglas--Sun coordinates are defined (\cite[Section 4]{DS20I}). 

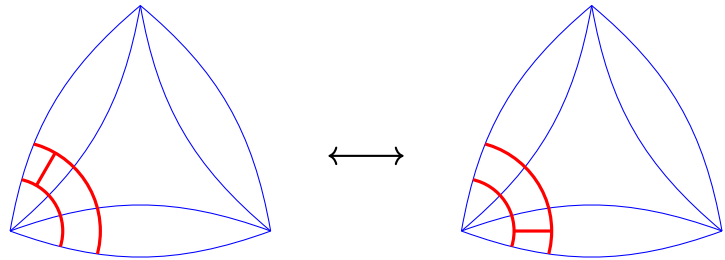
\begin{figure}[htbp]
    \centering
\begin{tikzpicture}
\foreach \i in {-30,90,210}
{
\draw[blue](\i:2) to[bend left=20pt] (\i+120:2);
\draw[blue](\i:2) to[bend right=20pt] (\i+120:2);
}
\begin{scope}
\clip(-30:2) to[bend right=20pt] (90:2) to[bend right=20pt] (210:2) to[bend right=20pt] (-30:2);
\draw[red,very thick] (210:2) circle(0.7cm);
\draw[red,very thick] (210:2) circle(1.2cm);
\path(210:2) --++(60:0.7) coordinate(A);
\draw[red,very thick] (A) --++(60:0.5);
\end{scope}

\draw[thick,<->] (2.5,0) -- (3.5,0);

\begin{scope}[xshift=6cm]
\foreach \i in {-30,90,210}
{
\draw[blue](\i:2) to[bend left=20pt] (\i+120:2);
\draw[blue](\i:2) to[bend right=20pt] (\i+120:2);
}
{\begin{scope}
\clip(-30:2) to[bend right=20pt] (90:2) to[bend right=20pt] (210:2) to[bend right=20pt] (-30:2);
\draw[red,very thick] (210:2) circle(0.7cm);
\draw[red,very thick] (210:2) circle(1.2cm);
\path(210:2) --++(0:0.7) coordinate(A);
\draw[red,very thick] (A) --++(0:0.5);
\end{scope}}
\end{scope}
\end{tikzpicture}
    \caption{The modified H-move \cite{DS20I} (a.k.a.~crossbar pass \cite{FS20}) across a corner.}
    \label{fig:modified H-move}
\end{figure}

Now let us consider a signed web $W$ on $\Sigma$. In this case, $W$ is no more parallel-equivalent to a web in good position in the above sense. To resolve this, we introduce the following notion:

\begin{dfn}[spiralling diagram]\label{def:spiralling}
Let $W$ be a non-elliptic signed web on $\Sigma$. 
Then the associated \emph{spiralling diagram} $\cW$ is a (possibly infinite and non-compact) $\mathfrak{sl}_3$-web obtained by the following two steps. 
\begin{enumerate}
    \item In a small disk neighborhood $D_p$ of each puncture $p \in \bP$, deform each end of $W$ incident to $p$ into an infinitely spiralling curve, according to their signs as shown in \cref{fig:spiral}. Let $\cW'$ be the resulting diagram.
    %Then it may occur that the resulting diagram $\cW'$ has infinitely many self-intersections accumulating at the punctures. If not, let $\cW:=\cW'$. 
    \item A pair of ends incident to a common puncture $p$ with the opposite sign produce infinitely many intersections in $\cW'$. We then modify these intersections into H-webs in a periodic manner, as follows. 
    By applying an isotopy in $D_p$, we can make these intersections only occurring in a single half-biangle $B_p$ in $D_p$ with special point $p$, without producing additional intersections\footnote{Concretely, this can be done as follows. If we fix a polar coordinates $(r,\theta)$, $r < r_0$ for some $r_0 > 0$ on the punctured disk $D_p\setminus \{p\}$, each spiralling curve can be modeled by the logarithmic spiral $\ell_\pm(a):\ \theta=\pm\log(ar)$ for some parameter $a >0$. Then an elementary calculation shows that the intersection points of $\ell_+(a_1)$ and $\ell_-(a_2)$ lie on a single line, which is viewed as the union of two rays. Then we can collectively push these rays into a chosen half-biangle $B_p$ only by smoothly varying the coordinate function $\theta$. By the standard argument involving a smooth cut-off function, we can also modify this \lq\lq angular'' isotopy to be identity near $\partial D_p$.}. 
    Then $\cW' \cap B_p=W_{\mathrm{br}}(S_p)$ for an asymptotically periodic symmetric strand set $S_p$ on $B_p$. By replacing the biangle part $W_{\mathrm{br}}(S_p)$ with the associated ladder-web $W(S_p)$, we get the spiralling diagram $\cW$. Since $\cW\cap (D_p \setminus B_p)$ consists of oriented corner arcs, the result does not depend on the choice of $B_p$.
\end{enumerate}
\end{dfn}
See \cref{fig:spiral_example} for a local example. A global example arising from \cref{fig:global example} is shown in \cref{fig:global example spiralling}. 
%The parallel-equivalence for the spiralling diagrams is similarly defined so that two parallel-equivalent signed webs produce parallel-equivalent spiralling diagrams. Here isotopies are required to be identity near punctures.

\begin{figure}[htbp]
    \centering
\begin{tikzpicture}
\draw[dashed] (-2.5,-1.5) circle(2cm);
\draw [red,thick](-3,0.45) .. controls (-2.5,0) and (-2.9,-0.8) .. (-2.5,-1.5);
\filldraw[fill=white] (-2.5,-1.5) circle(2pt);
\node[red] at (-2.5,0.2) {$W$};
\node[red] at (-2.3,-1.7) {$+$};
\node at (-2.7,-1.7) {$p$};
\draw (5,-1.5) circle(2pt);
\draw[dashed] (5,-1.5) circle(2cm);

\draw [red,thick](4.5,0.45) .. controls (5,0) and (5.55,-1.05) .. (5.55,-1.5) .. controls (5.55,-1.85) and (5.25,-2) .. (5,-2) .. controls (4.75,-2) and (4.55,-1.8) .. (4.55,-1.5) .. controls (4.55,-1.25) and (4.75,-1.1) .. (5,-1.1) .. controls (5.25,-1.1) and (5.4,-1.25) .. (5.4,-1.5) .. controls (5.4,-1.75) and (5.2,-1.85) .. (5,-1.85) .. controls (4.85,-1.85) and (4.7,-1.7) .. (4.7,-1.5) .. controls (4.7,-1.35) and (4.85,-1.25) .. (5,-1.25) .. controls (5.15,-1.25) and (5.25,-1.35) .. (5.25,-1.5) .. controls (5.25,-1.6) and (5.15,-1.7) .. (5,-1.7) .. controls (4.9,-1.7) and (4.85,-1.6) .. (4.85,-1.5);
\draw [red, thick, dotted](4.85,-1.5) .. controls (4.85,-1.3) and (5.15,-1.3) .. (5.15,-1.5);

\draw [thick,-{Classical TikZ Rightarrow[length=4pt]},decorate,decoration={snake,amplitude=2pt,pre length=2pt,post length=3pt}](0.65,-1.5) -- (2,-1.5);
\end{tikzpicture}
    \caption{Construction of a spiralling diagram. The negative sign similarly produce an end spiralling counter-clockwisely.}
    \label{fig:spiral}
\end{figure}
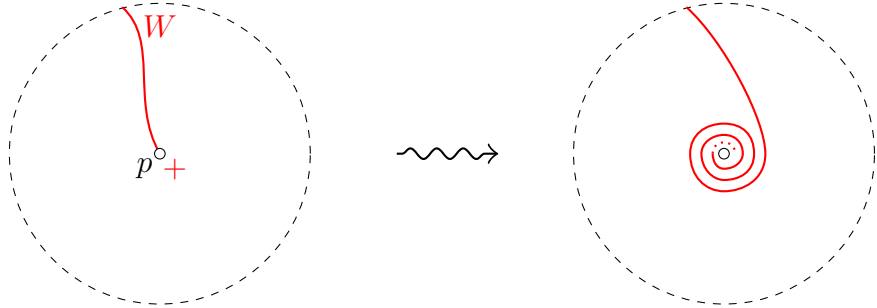

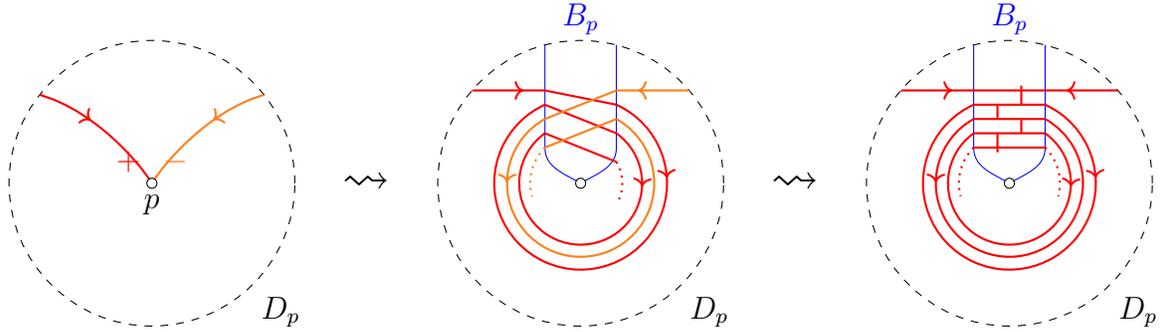
\begin{figure}[htbp]
\centering
\begin{tikzpicture}[scale=.8]
\begin{scope}
\draw[dashed] (0,0) circle(2cm);
\node at (1.8,-1.8) {$D_p$};
\clip(0,0) circle(2cm);
\draw[red,thick,->-] (-2,1.3) ..controls (-1,1.3) and (120:0.3).. (0,0);
%node[above left]{$+$};
\node[red,scale=0.8] at (-0.3,0.1) {$+$}; 
\draw[myorange,thick,->-] (2,1.3) ..controls (1,1.3) and (60:0.3).. (0,0);
%node[above right]{$-$};
\node[myorange,scale=0.8] at (0.3,0.1) {$+$}; 
\draw[fill=white](0,0) circle(2pt) node[below]{$p$};
\end{scope}
% \draw [very thick,-{Classical TikZ Rightarrow[length=4pt]},decorate,decoration={snake,amplitude=2pt,pre length=2pt,post length=3pt}](2.5,0) -- (3.5,0);
\node[scale=1.5] at (3,0) {$\rightsquigarrow$};

\begin{scope}[xshift=6cm]
\draw[dashed] (0,0) circle(2cm);
\node at (1.8,-1.8) {$D_p$};
\node[blue] at (0,2.3) {$B_p$};
\clip(0,0) circle(2cm);
\draw[blue] (0,0) to[out=150,in=-90] (-0.5,0.5) -- (-0.5,2);
\draw[blue] (0,0) to[out=30,in=-90] (0.5,0.5) -- (0.5,2);
\draw[fill=white](0,0) circle(2pt);
\draw[red,thick,->-={0.8}{}] (-2,1.3) -- (-0.5,1.3);
\draw[red,thick] (-0.5,1.3) -- (0.5,1.1);
\draw[red,thick] (-0.5,1.1) -- (0.5,0.7);
\draw[red,thick] (-0.5,0.7) -- (0.5,0.3);
\draw[red,thick,->-={0.2}{}] (0.5,1.1) arc(65.56:-180-65.56:1.209);
\draw[red,thick,->-={0.2}{}] (0.5,0.7) arc(54.46:-180-54.46:0.86);
\draw[red,thick,dotted] (0.5,0.3) arc(30.96:-30:0.583);

\draw[myorange,thick,->-={0.8}{}] (2,1.3) -- (0.5,1.3);
\draw[myorange,thick] (0.5,1.3) -- (-0.5,0.9);
\draw[myorange,thick] (0.5,0.9) -- (-0.5,0.5);
\draw[myorange,thick,->-={0.2}{}] (-0.5,0.9) arc(-180-60.95:60.95:1.0296);
\draw[myorange,thick,dotted] (-0.5,0.5) arc(-180-45:-180+15:0.5*1.414);
\end{scope}
% \draw [very thick,-{Classical TikZ Rightarrow[length=4pt]},decorate,decoration={snake,amplitude=2pt,pre length=2pt,post length=3pt}](8.5,0) -- (9.5,0);
\node[scale=1.5] at (9,0) {$\rightsquigarrow$};

\begin{scope}[xshift=12cm]
\draw[dashed] (0,0) circle(2cm);
\node at (1.8,-1.8) {$D_p$};
\node[blue] at (0,2.3) {$B_p$};
\clip(0,0) circle(2cm);
\draw[blue] (0,0) to[out=150,in=-90] (-0.5,0.5) -- (-0.5,2);
\draw[blue] (0,0) to[out=30,in=-90] (0.5,0.5) -- (0.5,2);
\draw[fill=white](0,0) circle(2pt);
\draw[red,thick,->-={0.8}{}] (-2,1.3) -- (-0.5,1.3);
\draw[red,thick] (-0.5,1.3) -- (0.5,1.3);
\draw[red,thick] (-0.5,1.1) -- (0.5,1.1);
\draw[red,thick] (-0.5,0.9) -- (0.5,0.9);
\draw[red,thick] (-0.5,0.7) -- (0.5,0.7);
\draw[red,thick] (-0.5,0.5) -- (0.5,0.5);

\draw[red,thick] ($(-0.5,1.3)!2/3!(0.5,1.3)$) -- ($(-0.5,1.1)!2/3!(0.5,1.1)$);
\draw[red,thick] ($(-0.5,1.1)!1/3!(0.5,1.1)$) -- ($(-0.5,0.9)!1/3!(0.5,0.9)$);
\draw[red,thick] ($(-0.5,0.9)!2/3!(0.5,0.9)$) -- ($(-0.5,0.7)!2/3!(0.5,0.7)$);
\draw[red,thick] ($(-0.5,0.7)!1/3!(0.5,0.7)$) -- ($(-0.5,0.5)!1/3!(0.5,0.5)$);

\draw[red,thick,->-={0.2}{}] (0.5,1.1) arc(65.56:-180-65.56:1.209);
\draw[red,thick,->-={0.2}{}] (0.5,0.7) arc(54.46:-180-54.46:0.86);
\draw[red,thick,dotted] (0.5,0.5) arc(45:-15:0.5*1.414);

\draw[red,thick,->-={0.8}{}] (2,1.3) -- (0.5,1.3);
\draw[red,thick,->-={0.2}{}] (-0.5,0.9) arc(-180-60.95:60.95:1.0296);
\draw[red,thick,dotted] (-0.5,0.5) arc(-180-45:-180+15:0.5*1.414);
\end{scope}
\end{tikzpicture}
    \caption{Construction of a spiralling diagram. Replace intersections with H-webs in a periodic manner.}
    \label{fig:spiral_example}
\end{figure}

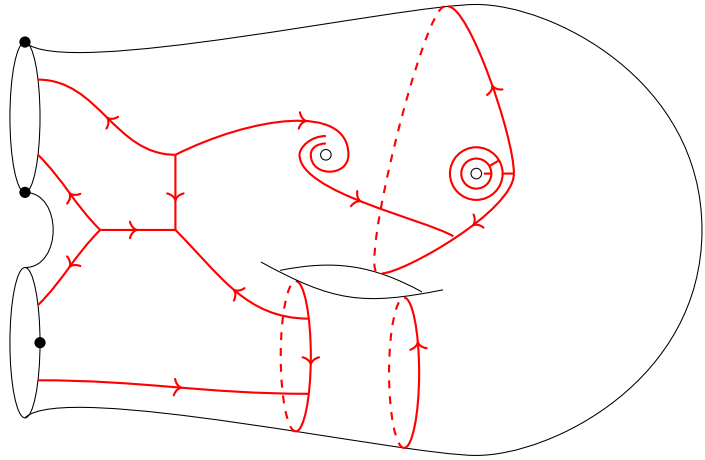
\begin{figure}[htbp]
%\iffalse
\centering
\begin{tikzpicture}[scale=.9]
\draw(0,-1.5) ellipse(0.2cm and 1cm);
\draw(0,1.5) ellipse(0.2cm and 1cm);
\node[fill,circle,inner sep=1.5pt] at (0.2,-1.5) {};
\node[fill,circle,inner sep=1.5pt] at (0,2.5) {};
\node[fill,circle,inner sep=1.5pt] at (0,0.5) {};
\draw(0,-0.5) ..controls (0.5,-0.5) and (0.5,0.5).. (0,0.5);
%surface
\draw(0,2.5) ..controls ++(0.5,-0.5) and (5,3).. node[inner sep=0,pos=0.9](T){} (6,3)
..controls (7,3) and (9,2).. (9,0)
..controls (9,-2) and (7,-3).. (6,-3)
..controls (5,-3) and (0.5,-2).. (0,-2.5) 
node[pos=0.4](A){} node[pos=0.2](B){};
%loops
\draw[red,thick,-<-] (A) arc(-90:90:0.2cm and 1cm) 
coordinate(A');
\draw[red,thick,dashed] (A) arc(-90:-270:0.2cm and 1cm);
\draw[red,thick,->-={0.7}{}] (B) arc(-90:90:0.2cm and 1cm) 
coordinate(B');
\draw[red,thick,dashed] (B) arc(-90:-270:0.2cm and 1cm);
%handle
\draw[shorten >=-15pt,shorten <=-15pt] (A') to[bend right=20] 
node[inner sep=0,pos=-0.15](A''){} node[inner sep=0,pos=1.15](B''){} 
(B');
\draw(A'') to[bend left=20] node[inner sep=0,pos=0.7](C){} (B'');
%puncture
\path (4,1) coordinate(P1);
\path(6,0.75) coordinate(P2);
%web vertices
\path (1,0) coordinate(X);
\path (2,0) coordinate(Y);
\path (2,1) coordinate(Z);
\path ($(0,-1.5)+(0.1*1.732,1*0.5)$) coordinate(BP1);
\path ($(0,1.5)+(0.1*1.732,-1*0.5)$) coordinate(BP2);
\path ($(0,-1.5)+(0.1*1.732,-1*0.5)$) coordinate(BP3);
\path ($(0,1.5)+(0.1*1.732,1*0.5)$) coordinate(BP4);
%web
\draw[red,thick,->-] (X) -- (Y);
\draw[red,thick,-<-] (Y) -- (Z);
\draw[red,thick,->-] (X) to[out=-135,in=45] (BP1);
\draw[red,thick,->-] (X) to[out=135,in=-45] (BP2);
\draw[red,thick,->-] ($(A)+(0.1*1.732,1+1*0.5)$) to[out=180,in=-45] (Y);
\draw[red,thick,-<-] ($(A)+(0.1*1.732,1-1*0.5)$) to[out=180,in=0] (BP3);
\draw[red,thick,->-] (Z) to[out=180,in=0] (BP4);
\draw[red,thick,->-] (Z)
..controls ++(1,0.5) and ($(P1)+(0.3,0.7)$).. ($(P1)+(0.3,0)$)
..controls ++(0,-0.3) and ($(P1)+(-0.2,-0.3)$).. ($(P1)+(-0.2,0)$)
..controls ++(0,0.1) and ($(P1)+(-0.1,0.15)$).. ($(P1)+(0,0.15)$);
\draw[red,thick,->-={0.4}{}] (P2)++(0.5,0) ..controls ++(-90:0.5) and ($(C)+(0.3,0)$).. (C) node[inner sep=0,pos=0.5](S){};
\draw[red,thick,->-] (P2)++(0.5,0) ..controls ++(90:0.5) and ($(T)+(0.3,0)$).. (T);
\draw[red,thick] (P2) circle (0.35cm);
\draw[red,thick] (P2) circle (0.2cm);
\draw[red,thick] (P2)++(0.5,0) -- ($(P2)+(0.35,0)$);
\draw[red,thick] ($(P2)+(30:0.35)$) -- ($(P2)+(30:0.2)$);
\draw[red,thick] ($(P2)+(0:0.2)$) -- ($(P2)+(0:0.1)$);
\draw[red,thick,dashed] (T) ..controls ++(-0.3,0) and ($(C)+(-0.3,0)$).. (C);
\draw[red,thick,-<-={0.5}{}] (S) 
..controls ++(150:0.5) and ($(P1)+(-0.35,-0.5)$).. ($(P1)+(-0.35,0)$)
..controls ++(0,0.1) and ($(P1)+(-0.2,0.25)$).. ($(P1)+(0,0.25)$);
\draw[fill=white](P1) circle(2pt);
\draw[fill=white](P2) circle(2pt);
\end{tikzpicture}
%\fi%%
    \caption{A global example of spiralling diagram arising from the underlying signed non-elliptic web in \cref{fig:global example}.}
    \label{fig:global example spiralling}
\end{figure}
%Given an integral $\mathfrak{sl}_3$-lamination $\hL \in \cL^x(\Sigma,\bZ)$, take its representative in the following manner. 
% \begin{enumerate}
%     \item First take a representative without peripheral leaves, each of whose components have a weight $1$ by using the operations (1) and (3) in \cref{def:unbounded laminations}.
%     \item Resolve each pair of ends at a puncture with opposite orientations and signs according to \cref{fig:puncture-equivalence}. 
%     \item By applying an isotopy, make it in minimal position with respect to the split ideal triangulation $\widehat{\tri}$. Here in a slightly generalized sense, non-transverse intersections of $W$ with $\tri$ at the punctures are allowed.
%     \item Finally, deform it to a \emph{spiralling diagram} acoording to the signatures assigned to its end at punctures, as shown in \cref{fig:spiral}. 
% \end{enumerate}
% Then the resulting diagram may have certain infinite self-intersections around punctures, arising from pairs of ends with the opposite signs, both incoming or both outgoing. 

\begin{dfn}
The spiralling diagram $\cW$ is in a \emph{good position} with respect to a split triangulation $\widehat{\tri}$ if the intersection $\cW\cap B_E$ (resp. $\cW \cap T$) is an unbounded essential (resp. reduced essential) local web for each $E \in e(\tri)$ and $T \in t(\tri)$.  
\end{dfn}
The loop parallel-move and the boundary H-move of a spiralling diagram are similarly defined as before, so that the construction of spiralling diagram from a signed web is equivariant under these moves. 
We define the \emph{modified periodic H-move} of a spiralling diagram in a good position across a corner to be the periodic application of the modified H-move to be the periodic parts of the unbounded essential local webs on biangles. By a \emph{strict isotopy} relative to a split triangulation $\widehat{\tri}$, we mean an isotopy on a marked surface $\Sigma$ which is the identity on each edge of $\widehat{\tri}$ and a neighborhood of each puncture. 

\begin{thm}[Proof in \cref{sec:good_position}]\label{prop:spiralling_good_position}
Any spiralling diagram arising from a non-elliptic signed web on $\Sigma$ can be isotoped into a good position with respect to $\widehat{\tri}$ by a finite sequence of intersection reduction moves, H-moves, and strict isotopies relative to $\widehat{\tri}$.
%and off a neighborhood of each puncture. 
Moreover, such a good position is unique up to a sequence of modified H-moves, modified periodic H-moves, loop parallel-moves, boundary H-moves, and strict isotopies relative to $\widehat{\tri}$. 
\end{thm}
Indeed, we can obtain a representative in a good position by successively applying the intersection reduction moves (\cref{fig:tightening}) and then pushing the H-faces into biangles by the H-move (\cref{fig:H-move}). An example of this procedure is illustrated in \cref{fig:canonical_position_example}. 
The main issue here is to ensure that this procedure always terminates in finite steps, which is discussed in \cref{sec:good_position} in detail.

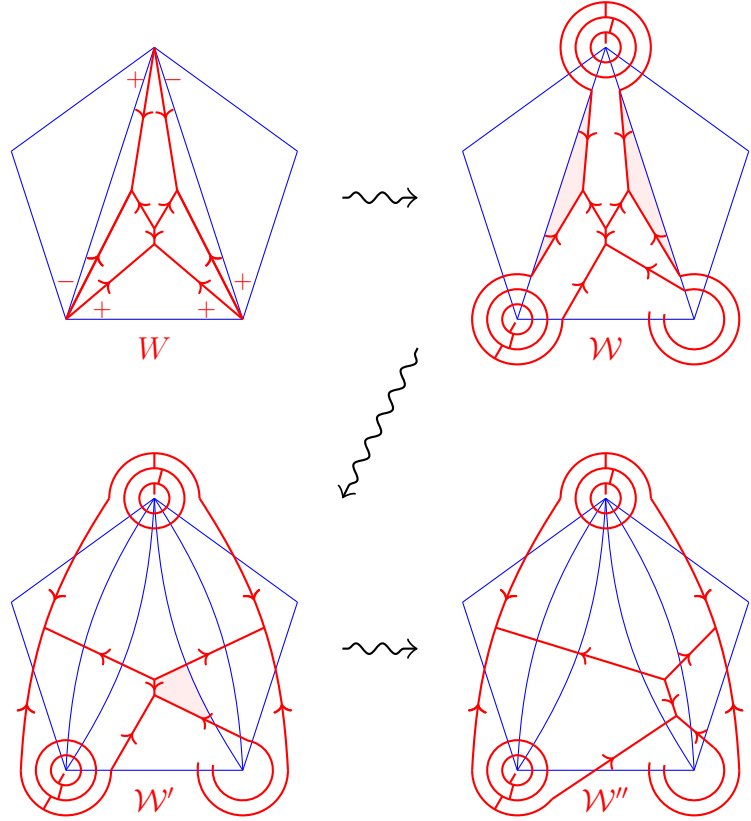
\begin{figure}
%\iffalse
\centering
\begin{tikzpicture}
\begin{scope}[rotate=18]
\foreach \x in {0,1,2,3,4}
\draw[blue] (72*\x:2) coordinate(A\x) -- (72*\x+72:2);
\foreach \x in {0,1,2,3,4}
{
\foreach \y in {0,1,2,3,4}
\draw($(A\x)!0.5!(A\y)$) coordinate(A\x\y);
}
\end{scope}
\draw($(A1)!2/3!(A34)$) coordinate(G);
\draw[blue] (A1) -- (A4);
\draw[blue] (A1) -- (A3);
\draw(A13) ++(-18:0.3) coordinate(B13);
\draw(A14) ++(-180+18:0.3) coordinate(B14);
\draw(A34) ++(90:1) coordinate(B34);

{\color{red}
\draw[thick,->-] (A1) -- (B13);
\draw[thick,->-] (A1) -- (B14);
\draw[thick,->-] (A3) -- (B13);
\draw[thick,->-] (A4) -- (B14);
\draw[thick,->-] (A3) -- (B34);
\draw[thick,->-] (A4) -- (B34);
\draw[thick,->-] (A3) -- (B13);
\draw[thick,->-] (A4) -- (B14);
\draw[thick,->-={0.7}{}] (G) -- (B13);
\draw[thick,->-={0.7}{}] (G) -- (B14);
\draw[thick,->-={0.7}{}] (G) -- (B34);

\node[scale=0.8] at ($(A1)+(-120:0.5)$) {$+$};
\node[scale=0.8] at ($(A1)+(-60:0.5)$) {$-$};
\node[scale=0.8] at ($(A3)+(15:0.5)$) {$+$};
\node[scale=0.8] at ($(A3)+(90:0.5)$) {$-$};
\node[scale=0.8] at ($(A4)+(180-15:0.5)$) {$+$};
\node[scale=0.8] at ($(A4)+(90:0.5)$) {$+$};
\node at (0,-2) {$W$};
}

{\begin{scope}[xshift=6cm]
\begin{scope}[rotate=18]
\foreach \x in {0,1,2,3,4}
\draw[blue] (72*\x:2) coordinate(A\x) -- (72*\x+72:2);
\foreach \x in {0,1,2,3,4}
{
\foreach \y in {0,1,2,3,4}
\draw($(A\x)!0.5!(A\y)$) coordinate(A\x\y);
}
\end{scope}
\draw($(A1)!2/3!(A34)$) coordinate(G);
\draw[blue] (A1) -- (A4);
\draw[blue] (A1) -- (A3);
\draw(A13) ++(-18:0.3) coordinate(B13);
\draw(A14) ++(-180+18:0.3) coordinate(B14);
\draw(A34) ++(90:1) coordinate(B34);

\draw[thick] (A1)++(-90+18:0.6) coordinate(D14);
\draw[thick] (A1)++(270-18:0.6) coordinate(D13);
\draw[thick] (A3)++(72:0.6) coordinate(D31);
\draw[thick] (A3)++(0:0.6) coordinate(D34);
\draw[thick] (A4)++(90+18:0.6) coordinate(D41);
\draw[thick] (A4)++(90+18:0.4) coordinate(D43);
\begin{pgfonlayer}{bg} 
\filldraw[pink!30] (D31) -- (B13) -- (D13) --cycle;
\filldraw[pink!30] (D41) -- (B14) -- (D14) --cycle;
\end{pgfonlayer} 

{\color{red}
\draw[thick,->-={0.7}{}] (G) -- (B13);
\draw[thick,->-={0.7}{}] (G) -- (B14);
\draw[thick,->-={0.7}{}] (G) -- (B34);
\draw[thick] (A1) circle(0.2cm);
\draw[thick] (A1) circle(0.4cm);
\draw[thick] (A1)++(270-18:0.6) arc(270-18:-90+18:0.6cm);
\draw[thick,->-] (A1)++(-90+18:0.6) --(B14);
\draw[thick,->-] (A1)++(270-18:0.6) --(B13);
\draw[thick] (A1)++(90:0.6) --++(-90:0.2);
\draw[thick] (A1)++(75:0.4) --++(-105:0.2);
\draw[thick] (A1)++(90:0.2) --++(-90:0.15);
\draw[thick] (A3) circle(0.2cm);
\draw[thick] (A3) circle(0.4cm);
\draw[thick] (A3)++(0:0.6) arc(0:-360+72:0.6cm);
\draw[thick,->-] (A3)++(72:0.6)--(B13);
\draw[thick,->-] (A3)++(0:0.6) --(B34);
\draw[thick] (A3)++(-120:0.6) --++(60:0.2);
\draw[thick] (A3)++(-105:0.4) --++(75:0.2);
\draw[thick] (A3)++(-120:0.2) --++(60:0.15);
\draw[thick] (A4)++(90+18:0.6) arc(90+18:-190:0.6cm);
\draw[thick] (A4)++(90+18:0.4) arc(90+18:-190:0.4cm);
\draw[thick,->-] (A4)++(90+18:0.6) --(B14);
\draw[thick,->-] (A4)++(90+18:0.4) --(B34);
\node at (0,-2) {$\cW$};
}
\end{scope}}

{\begin{scope}[yshift=-6cm]
\begin{scope}[rotate=18]
\foreach \x in {0,1,2,3,4}
\draw[blue] (72*\x:2) coordinate(A\x) -- (72*\x+72:2);
\foreach \x in {0,1,2,3,4}
{
\foreach \y in {0,1,2,3,4}
\draw($(A\x)!0.5!(A\y)$) coordinate(A\x\y);
}
\end{scope}
\draw($(A1)!2/3!(A34)$) coordinate(G);
\draw[blue] (A1) to[bend left=15] (A4);
\draw[blue,name path=P] (A1) to[bend right=15] (A4);
\draw[blue] (A1) to[bend left=15] (A3);
\draw[blue] (A1) to[bend right=15] (A3);
\draw(A13) ++(-18:0.3) coordinate(B13);
\draw(A14) ++(-180+18:0.3) coordinate(B14);
\draw(A34) ++(90:1) coordinate(B34);

\draw[thick] (A1)++(0:0.6) coordinate(D14);
\draw[thick] (A1)++(180:0.6) coordinate(D13);
\draw[thick] (A3)++(180:0.6) coordinate(D31);
\draw[thick] (A3)++(0:0.6) coordinate(D34);
\draw[thick] (A4)++(0:0.6) coordinate(D41);
\draw[thick] (A4)++(90+18:0.4) coordinate(D43);

{\color{red}
\draw[thick,->-={0.4}{},-<-={0.8}{}] (D13) to[bend right=15] node[midway]{} coordinate(E13) (D31);
\draw[thick,->-={0.4}{},-<-={0.8}{}] (D14) to[bend left=15] node[midway]{} coordinate(E14) (D41);
\draw[thick,->-={0.7}{}] (G) -- (B34);
\draw[thick] (A1) circle(0.2cm);
\draw[thick] (A1) circle(0.4cm);
\draw[thick] (A1)++(180:0.6) arc(180:0:0.6cm);
\draw[thick] (A1)++(270-18:0.6);
\draw[thick] (A1)++(90:0.6) --++(-90:0.2);
\draw[thick] (A1)++(75:0.4) --++(-105:0.2);
\draw[thick] (A1)++(90:0.2) --++(-90:0.15);
\draw[thick] (A3) circle(0.2cm);
\draw[thick] (A3) circle(0.4cm);
\draw[thick] (A3)++(0:0.6) arc(0:-180:0.6cm);
\draw[thick,->-] (A3)++(0:0.6) --(B34);
\draw[thick] (A3)++(-120:0.6) --++(60:0.2);
\draw[thick] (A3)++(-105:0.4) --++(75:0.2);
\draw[thick] (A3)++(-120:0.2) --++(60:0.15);
\draw[thick] (A4)++(0:0.6) arc(0:-190:0.6cm);
\draw[thick] (A4)++(80:0.4) arc(80:-190:0.4cm);
\draw[thick] (A4)++(90+18:0.6);
\draw[thick,->-,name path=P1] (A4)++(80:0.4) --(B34);
\draw[thick,->-] (G) -- (E13);
\draw[thick,->-,name path=P2] (G) -- (E14);

\draw[name intersections={of=P and P1,by=F1}];
\path[name intersections={of=P and P2,by=F2}];
\begin{pgfonlayer}{bg}  % select the background layer
    \filldraw[pink!30] (F1) -- (B34) -- (G) -- (F2);
\end{pgfonlayer}
\node at (0,-2) {$\cW'$};
}
\end{scope}}

{\begin{scope}[xshift=6cm,yshift=-6cm]
\begin{scope}[rotate=18]
\foreach \x in {0,1,2,3,4}
\draw[blue] (72*\x:2) coordinate(A\x) -- (72*\x+72:2);
\foreach \x in {0,1,2,3,4}
{
\foreach \y in {0,1,2,3,4}
\draw($(A\x)!0.5!(A\y)$) coordinate(A\x\y);
}
\end{scope}
\draw($(A1)!2/3!(A4)$) coordinate(G);
\draw[blue] (A1) to[bend left=15] (A4);
\draw[blue,name path=P] (A1) to[bend right=15] (A4);
\draw[blue] (A1) to[bend left=15] (A3);
\draw[blue] (A1) to[bend right=15] (A3);
\draw(A13) ++(-18:0.3) coordinate(B13);
\draw(A14) ++(-180+18:0.3) coordinate(B14);
\draw($(A1)!0.8!(A4)$) coordinate(B34);

\draw[thick] (A1)++(0:0.6) coordinate(D14);
\draw[thick] (A1)++(180:0.6) coordinate(D13);
\draw[thick] (A3)++(180:0.6) coordinate(D31);
\draw[thick] (A3)++(0:0.6) coordinate(D34);
\draw[thick] (A4)++(0:0.6) coordinate(D41);
\draw[thick] (A4)++(90+18:0.4) coordinate(D43);

{\color{red}
\draw[thick,->-={0.4}{},-<-={0.8}{}] (D13) to[bend right=15] node[midway]{} coordinate(E13) (D31);
\draw[thick,->-={0.4}{},-<-={0.8}{}] (D14) to[bend left=15] node[midway]{} coordinate(E14) (D41);
\draw[thick,->-={0.7}{}] (G) -- (B34);
\draw[thick] (A1) circle(0.2cm);
\draw[thick] (A1) circle(0.4cm);
\draw[thick] (A1)++(180:0.6) arc(180:0:0.6cm);
\draw[thick] (A1)++(270-18:0.6);
\draw[thick] (A1)++(90:0.6) --++(-90:0.2);
\draw[thick] (A1)++(75:0.4) --++(-105:0.2);
\draw[thick] (A1)++(90:0.2) --++(-90:0.15);
\draw[thick] (A3) circle(0.2cm);
\draw[thick] (A3) circle(0.4cm);
\draw[thick] (A3)++(-40:0.6) arc(-40:-180:0.6cm);
\draw[thick,->-] (A3)++(-40:0.6) --(B34);
\draw[thick] (A3)++(-120:0.6) --++(60:0.2);
\draw[thick] (A3)++(-105:0.4) --++(75:0.2);
\draw[thick] (A3)++(-120:0.2) --++(60:0.15);
\draw[thick] (A4)++(0:0.6) arc(0:-190:0.6cm);
\draw[thick] (A4)++(50:0.4) arc(50:-190:0.4cm);
\draw[thick] (A4)++(90+18:0.6);
\draw[thick,->-] (A4)++(50:0.4) --(B34);
\draw[thick,->-] (G) -- (E13);
\draw[thick,->-] (G) -- (E14);

\node at (0,-2) {$\cW''$};
}
\end{scope}}
\draw [thick,-{Classical TikZ Rightarrow[length=4pt]},decorate,decoration={snake,amplitude=2pt,pre length=2pt,post length=3pt}](2.5,0) --(3.5,0);
\draw [thick,-{Classical TikZ Rightarrow[length=4pt]},decorate,decoration={snake,amplitude=2pt,pre length=2pt,post length=3pt}](2.5,-6) --(3.5,-6);
\draw [thick,-{Classical TikZ Rightarrow[length=4pt]},decorate,decoration={snake,amplitude=2pt,pre length=2pt,post length=3pt}](3.5,-2) --(2.5,-4);
\end{tikzpicture}
%\fi%%
    \caption{An example of the procedure to place a spiralling diagram in good position.}
    \label{fig:canonical_position_example}
\end{figure}

While the spiralling diagram itself is suited to discuss its good position, the following \emph{braid representation} will be useful to define the shear coordinates:

\begin{dfn}[Braid representation of a spiralling diagram]
Let $\cW$ be a spiralling diagram in a good position with respect to $\widehat{\tri}$. Then its \emph{braid representation} $\cW_{\mathrm{br}}^\tri$ is obtained from $\cW$ by replacing the unbounded essential web $\cW\cap B_E$ on each biangle $B_E$ with its braid representation. 
\end{dfn}
The braid representation is closely related to (an unbounded version of) \emph{global picture} \cite[Definition 55]{DS20I}. See also \cref{sec:traveler}.

\subsection{Definition of the shear coordinates}\label{subsec:shear}
Now we define the shear coordinates associated with an ideal triangulation $\tri$ of $\Sigma$ without self-folded triangles. Let $\widehat{\tri}$ be the associated split triangulation. 

Given a rational $\fsl_3$-lamination $\hL \in \cL^x(\Sigma,\bQ)$, represent it by an $\fsl_3$-web $W$ together with rational weights on its components and signs at the ends incident to punctures. Let $\cW$ be the associated spiralling diagram together with rational weights on the components, placed in good position with respect to $\widehat{\tri}$. Let $\cW_{\mathrm{br}}^\tri$ be its braid representation, together with well-assigned rational weights on its components. The shear coordinates of $\hL$ are going to be defined out of $\cW_{\mathrm{br}}^\tri$. 

For each $E \in e_{\interior}(\tri)$, let $Q_E$ be the unique quadrilateral containing $E$ as its diagonal, regarded as the union of two triangles $T_L,T_R$ and the biangle $B_E$. By \cref{lem:triangle_essential}, the restriction of $\cW_{\mathrm{br}}^\tri$ to each of $T_L$ and $T_R$ has at most one honeycomb web, which is represented by a triangular symbol as in \cref{notation:division of honeycombs}. We call any strand in the braid representative $\cW_{\mathrm{br}}^\tri \cap Q_E$ that is incident to the triangular symbol in $T_L$ (if exists) a \emph{$T_L$-strand}. Similarly, we define \emph{$T_R$-strands}. It is possible that an arc is both $T_L$- and $T_R$-strand, in which case it connects the two honeycombs. 
%We call these components \emph{honeycomb components}. 
By removing the $T_L$- and $T_R$-strands, remaining is a collection of (possibly intersecting) oriented curves, which we call the \emph{curve components}. See \cref{fig:shear_example} below.

% \begin{figure}
%     \centering
%     \includegraphics[width=14cm]{braid_rep2.png}
%     \caption{Another example. Here we have an edge connecting two honeycomb webs, which does not contribute to the $\fsl_3$-shear coordinates.}
%     \label{fig:braid_rep2}
% \end{figure}

\begin{dfn}[$\fsl_3$-shear coordinates]
The \emph{($\fsl_3$-)shear coordinate system} 
\begin{align*}
    \sfx^\uf_\tri(\hL) =(\sfx_i^\tri(\hL))_{i \in I_\uf(\tri)} \in \bQ^{I_\uf(\tri)}
\end{align*}
is defined as follows. First, for each $E \in e_{\interior}(\tri)$, the coordinates assigned to the four vertices in the interior of $Q_E$ only depends on the restriction $\cW_{\mathrm{br}}^\tri \cap Q_E$. 
\begin{enumerate}
    \item Each curve component contributes to the edge coordinates according to the rule shown in \cref{fig:shear_curve}.
    \item The honeycomb on the triangle $T_L$ contributes to $\sfx^\uf_\tri(\hL)$ as in \cref{fig:shear_honeycomb}. Namely, the face coordinate counts the height of the honeycomb web, where a sink (resp. source) is counted positively (resp. negatively). 
    The edge coordinates counts the contributions from $T_L$-strands, where we have $n_1$ left-turning ones, $n_2$ straight-going ones (which are also $T_R$-strands), and $n_3$ right-turning ones. 
    \item The honeycomb on the triangle $T_R$ and the $T_R$-strands contribute in the symmetric way with respect to the $\pi$ rotation of the figure. 
    %oriented edges emanating from a sink- (resp. source-) honeycomb component which turn left (resp. right). 
\end{enumerate}
Then the shear coordinates are defined to be the weighted sums of these contributions. 
\end{dfn}

\begin{figure}[tbp]
    \centering
\begin{tikzpicture}
\draw[blue] (2,0) -- (0,2) -- (-2,0) -- (0,-2) --cycle;
\draw[blue] (0,-2) to[bend left=30pt] (0,2);
\draw[blue] (0,-2) to[bend right=30pt] (0,2);
\node[blue] at (-1.5,-1.5) {$T_L$};
\node[blue] at (1.5,-1.5) {$T_R$};
\draw[very thick,red,->-] (-1.2,1) --node[midway,below=0.3em]{$m$} (1,-1.2);
\draw[very thick,red,-<-] (-1,1.2) --node[midway,above=0.3em]{$n$} (1.2,-1);
\draw[thick,|->] (2.5,0) --node[midway,above]{$\sfx^\uf_\tri$} (3.5,0); 
\begin{scope}[xshift=6cm]
\draw[blue] (2,0) -- (0,2) -- (-2,0) -- (0,-2) --cycle;
\draw[blue] (0,-2) to (0,2);
{\color{mygreen}
\draw (0.8,0) circle(2pt) node[right]{$0$};
\draw (-0.8,0) circle(2pt) node[left]{$0$};
\draw (0,0.8) circle(2pt) node[above right]{$n$};
\draw(0,-0.8) circle(2pt)node[below right]{$m$};
}
\end{scope}

\begin{scope}[yshift=-5cm]
\draw[blue] (2,0) -- (0,2) -- (-2,0) -- (0,-2) --cycle;
\draw[blue] (0,-2) to[bend left=30pt] (0,2);
\draw[blue] (0,-2) to[bend right=30pt] (0,2);
\node[blue] at (-1.5,-1.5) {$T_L$};
\node[blue] at (1.5,-1.5) {$T_R$};
\draw[very thick,red,-<-] (-1.2,-1) --node[midway,above=0.3em]{$n$} (1,1.2);
\draw[very thick,red,->-] (-1,-1.2) --node[midway,below=0.3em]{$m$} (1.2,1);
\draw[thick,|->] (2.5,0) --node[midway,above]{$\sfx^\uf_\tri$} (3.5,0); 
{\begin{scope}[xshift=6cm]
\draw[blue] (2,0) -- (0,2) -- (-2,0) -- (0,-2) --cycle;
\draw[blue] (0,-2) to (0,2);
{\color{mygreen}
\draw (0.8,0) circle(2pt) node[right]{$0$};
\draw (-0.8,0) circle(2pt) node[left]{$0$};
\draw (0,0.8) circle(2pt) node[above right]{$-n$};
\draw(0,-0.8) circle(2pt)node[below right]{$-m$};
}
\end{scope}}
\end{scope}
\end{tikzpicture}
    \caption{Contributions from curve components.}
    \label{fig:shear_curve}
\end{figure}
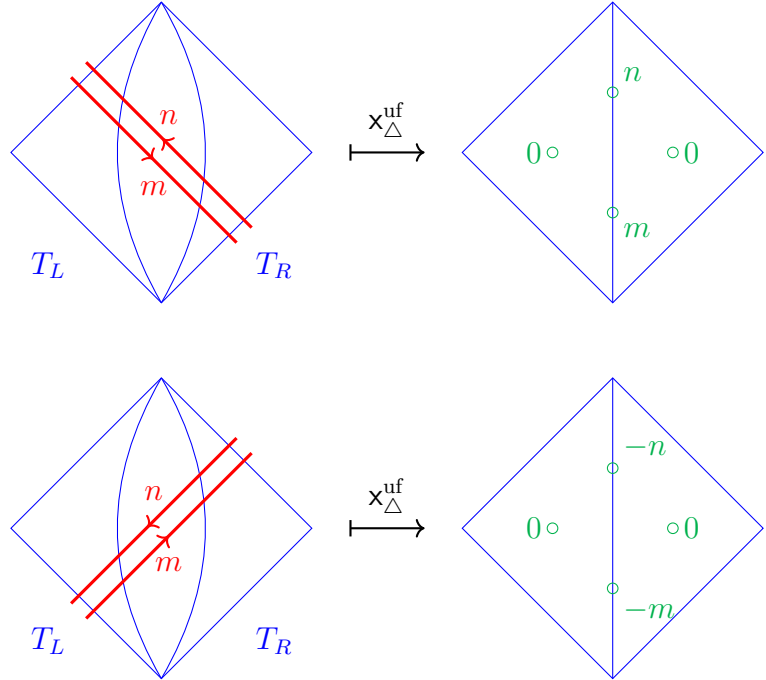

\begin{figure}[htbp]
    \centering
\begin{tikzpicture}
\draw[blue] (2,0) -- (0,2) -- (-2,0) -- (0,-2) --cycle;
\draw[blue] (0,-2) to[bend left=30pt] (0,2);
\draw[blue] (0,-2) to[bend right=30pt] (0,2);
\node[blue] at (-1.5,-1.5) {$T_L$};
\node[blue] at (1.5,-1.5) {$T_R$};
\draw[red,very thick] (-1.4,0) --++(0.6,0.6) --++(0,-1.2) --cycle;
\draw[red,very thick,->-] (-1.6,0.6) node[above]{$n$} -- (-1.2,0.2);
\draw[red,very thick,->-] (-1.6,-0.6) node[below]{$n$} -- (-1.2,-0.2);
\draw[red,dashed] (1.4,0) --++(-0.6,0.6) --++(0,-1.2) --cycle;
\draw[red,very thick,-<-] (-0.8,0) -- (0.8,0);
%node[below left=-0.1em]{$n_2$};
\node[red] at (0.4,-0.2) {$n_2$};
\draw[red,very thick,-<-] (-0.8,0.2) to[out=0, in=-120] (0.8,1.4) node[above]{$n_1$};
\draw[red,very thick,-<-] (-0.8,-0.2) to[out=0, in=120] (0.8,-1.4) node[below]{$n_3$};
\draw[thick,|->] (2.5,0) --node[midway,above]{$\sfx^\uf_\tri$} (3.5,0); 
\begin{scope}[xshift=6cm]
\draw[blue] (2,0) -- (0,2) -- (-2,0) -- (0,-2) --cycle;
\draw[blue] (0,-2) to (0,2);
{\color{mygreen}
\draw (0.8,0) circle(2pt) node[right]{$0$};
\draw (-0.8,0) circle(2pt) node[left]{$n$};
\draw (0,0.8) circle(2pt) node[above right]{$-n_1$};
\draw(0,-0.8) circle(2pt)node[below right]{$0$};
}
\end{scope}

\begin{scope}[yshift=-5cm]
\draw[blue] (2,0) -- (0,2) -- (-2,0) -- (0,-2) --cycle;
\draw[blue] (0,-2) to[bend left=30pt] (0,2);
\draw[blue] (0,-2) to[bend right=30pt] (0,2);
\node[blue] at (-1.5,-1.5) {$T_L$};
\node[blue] at (1.5,-1.5) {$T_R$};
\draw[red,very thick] (-1.4,0) --++(0.6,0.6) --++(0,-1.2) --cycle;
\draw[red,very thick,-<-] (-1.6,0.6) node[above]{$n$} -- (-1.2,0.2);
\draw[red,very thick,-<-] (-1.6,-0.6) node[below]{$n$} -- (-1.2,-0.2);
\draw[red,dashed] (1.4,0) --++(-0.6,0.6) --++(0,-1.2) --cycle;
\draw[red,very thick,->-] (-0.8,0) -- (0.8,0);
%node[below left=-0.1em]{$n_2$};
\node[red] at (0.4,-0.2) {$n_2$};
\draw[red,very thick,->-] (-0.8,0.2) to[out=0, in=-120] (0.8,1.4) node[above]{$n_1$};
\draw[red,very thick,->-] (-0.8,-0.2) to[out=0, in=120] (0.8,-1.4) node[below]{$n_3$};
\draw[thick,|->] (2.5,0) --node[midway,above]{$\sfx^\uf_\tri$} (3.5,0); 
{\begin{scope}[xshift=6cm]
\draw[blue] (2,0) -- (0,2) -- (-2,0) -- (0,-2) --cycle;
\draw[blue] (0,-2) to (0,2);
{\color{mygreen}
\draw (0.8,0) circle(2pt) node[right]{$0$};
\draw (-0.8,0) circle(2pt) node[left]{$-n$};
\draw (0,0.8) circle(2pt) node[above right]{$0$};
\draw(0,-0.8) circle(2pt)node[below right]{$n_3$};
}
\end{scope}}
\end{scope}
\end{tikzpicture}
    \caption{Contributions from the honeycomb of height $n=n_1+n_2+n_3$ on the triangle $T_L$. Observe that the $n_2$ straight-going $T_L$-strands do not contribute.}
    \label{fig:shear_honeycomb}
\end{figure}
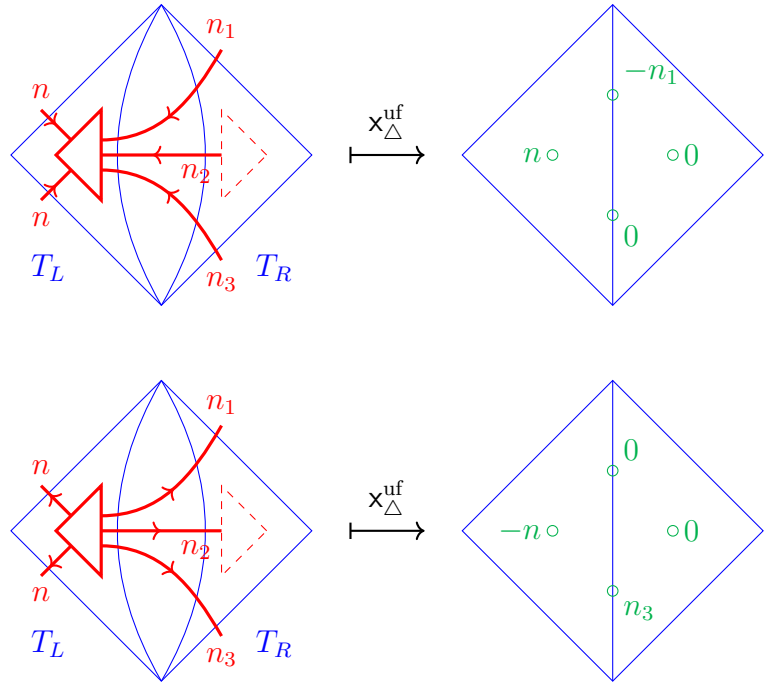

\begin{rem}\label{rem:decomposition_honeycomb}
\begin{enumerate}
    \item Notice that the rule shown in \cref{fig:shear_curve} is an \lq\lq oriented version'' of the Thurston's shear coordinates (see \cref{subsec:principal}). Indeed, the sign of contribution is determined by the crossing pattern as in the $\fsl_2$-case, and it contributes to the coordinates \emph{on the right side} of the oriented curve. 
    \item The shear coordinates of the first honeycomb component shown in \cref{fig:shear_honeycomb} is the same as the sum of shear coordinates of the three honeycomb components shown in \cref{fig:shear_honeycomb_decomposition}. 
    %This viewpoint will be generalized to the \emph{coordinate sum} of spiralling diagrams: see \cref{sec:ensemble}.
\end{enumerate}
\end{rem}

\begin{figure}[htbp]
\centering
\begin{tikzpicture}[scale=.9]
\draw[blue] (2,0) -- (0,2) -- (-2,0) -- (0,-2) --cycle;
\draw[blue] (0,-2) to[bend left=30pt] (0,2);
\draw[blue] (0,-2) to[bend right=30pt] (0,2);
\draw[red,very thick] (-1.4,0) --++(0.6,0.6) --++(0,-1.2) --cycle;
\draw[red,very thick,->-] (-1.6,0.6) node[above]{$n_1$} -- (-1.2,0.2);
\draw[red,very thick,->-] (-1.6,-0.6) node[below]{$n_1$} -- (-1.2,-0.2);
%\draw[red,dashed] (1.4,0) --++(-0.6,0.6) --++(0,-1.2) --cycle;
%\draw[red,very thick,-<-] (-0.8,0) -- (0.8,0) node[below left=0.2em]{$n_2$};
\draw[red,very thick,-<-] (-0.8,0) to[out=0, in=-120] (0.8,1.4) node[above]{$n_1$};
%\draw[red,very thick,-<-] (-0.8,-0.2) to[out=0, in=120] (0.8,-1.4) node[below]{$n_3$};

{\begin{scope}[xshift=5cm]
\draw[blue] (2,0) -- (0,2) -- (-2,0) -- (0,-2) --cycle;
\draw[blue] (0,-2) to[bend left=30pt] (0,2);
\draw[blue] (0,-2) to[bend right=30pt] (0,2);
\draw[red,very thick] (-1.4,0) --++(0.6,0.6) --++(0,-1.2) --cycle;
\draw[red,very thick,->-] (-1.6,0.6) node[above]{$n_2$} -- (-1.2,0.2);
\draw[red,very thick,->-] (-1.6,-0.6) node[below]{$n_2$} -- (-1.2,-0.2);
\draw[red,dashed] (1.4,0) --++(-0.6,0.6) --++(0,-1.2) --cycle;
\draw[red,very thick,-<-] (-0.8,0) -- (0.8,0) node[below left=0.2em]{$n_2$};
%\draw[red,very thick,-<-] (-0.8,0.2) to[out=0, in=-120] (0.8,1.4) node[above]{$n_1$};
%\draw[red,very thick,-<-] (-0.8,-0.2) to[out=0, in=120] (0.8,-1.4) node[below]{$n_3$};
\end{scope}}

{\begin{scope}[xshift=10cm]
\draw[blue] (2,0) -- (0,2) -- (-2,0) -- (0,-2) --cycle;
\draw[blue] (0,-2) to[bend left=30pt] (0,2);
\draw[blue] (0,-2) to[bend right=30pt] (0,2);
\draw[red,very thick] (-1.4,0) --++(0.6,0.6) --++(0,-1.2) --cycle;
\draw[red,very thick,->-] (-1.6,0.6) node[above]{$n_3$} -- (-1.2,0.2);
\draw[red,very thick,->-] (-1.6,-0.6) node[below]{$n_3$} -- (-1.2,-0.2);
%\draw[red,dashed] (1.4,0) --++(-0.6,0.6) --++(0,-1.2) --cycle;
%\draw[red,very thick,-<-] (-0.8,0) -- (0.8,0) node[below left=0.2em]{$n_2$};
%\draw[red,very thick,-<-] (-0.8,0.2) to[out=0, in=-120] (0.8,1.4) node[above]{$n_1$};
\draw[red,very thick,-<-] (-0.8,0) to[out=0, in=120] (0.8,-1.4) node[below]{$n_3$};
\end{scope}}
\end{tikzpicture}
    \caption{Basic honeycomb components}
    \label{fig:shear_honeycomb_decomposition}
\end{figure}
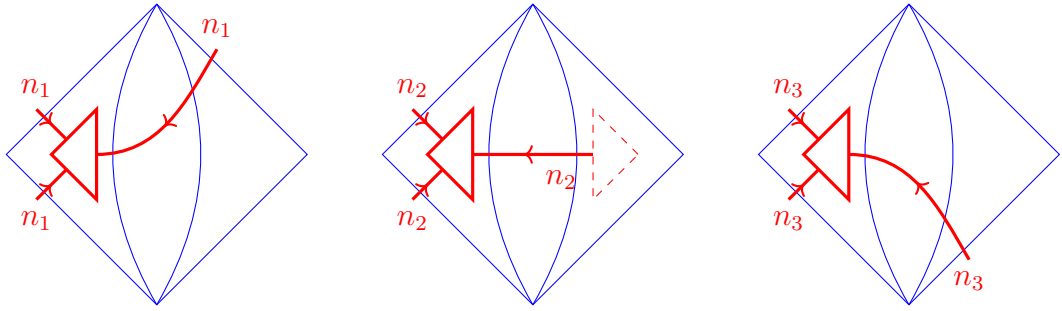

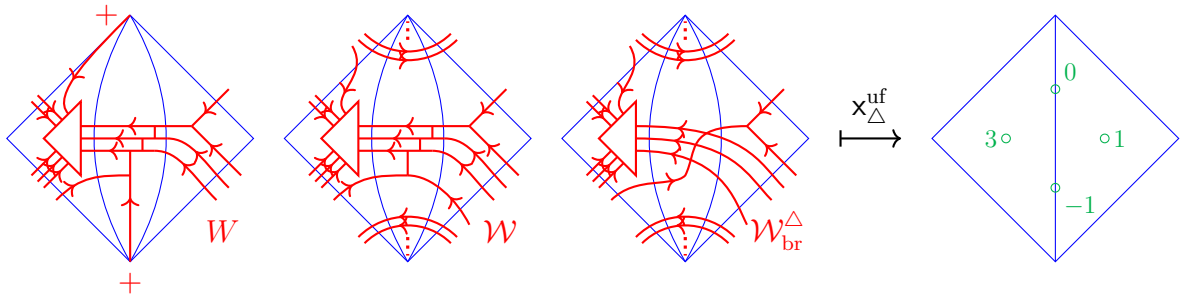
\begin{figure}[htp]
    \centering
    \begin{tikzpicture}[scale=.7]
\draw[blue] (2,0) -- (0,2) -- (-2,0) -- (0,-2) --cycle;
\draw[blue] (0,-2) to[bend left=30pt] (0,2);
\draw[blue] (0,-2) to[bend right=30pt] (0,2);
\draw[red,thick] (-1.4,0) --++(0.6,0.6) --++(0,-1.2) --cycle;
\draw[red,thick,->-={0.7}{}] (-1.6,0.6) -- (-1.2,0.2);
\draw[red,thick,->-={0.7}{}] (-1.5,0.7) -- (-1.1,0.3);
%\draw[red,thick,->-] (-1.4,0.8) -- (-1.0,0.4);
\draw[red,thick,->-={0.7}{}] (0,2) node[left,scale=0.8]{$+$} ..controls (-0.6,1.4) and (-1.3,0.7).. (-1.0,0.4);
\draw[red,thick,->-={0.7}{}] (-1.6,-0.6) -- (-1.2,-0.2);
\draw[red,thick,->-={0.7}{}] (-1.5,-0.7) -- (-1.1,-0.3);
\draw[red,thick,->-={0.7}{}] (-1.4,-0.8) -- (-1.0,-0.4);
\draw[red,thick,-<-] (-0.8,0.2) -- (1,0.2);
\draw[red,thick,-<-] (1,0.2) -- (1.6,0.8);
\draw[red,thick,-<-] (1,0.2) -- (1.8,-0.6);
\draw[red,thick,-<-] (-0.8,0) -- (0.6,0);
\draw[red,thick,-<-] (0.6,0) ..controls (1,0) and (1.2,-0.4).. (1.6,-0.8);
\draw[red,thick,-<-] (-0.8,-0.2) -- (0.4,-0.2);
\draw[red,thick,-<-] (0.4,-0.2) ..controls (0.8,-0.2) and (1,-0.6).. (1.4,-1.0);
\draw[red,thick] (0.4,0.2) -- (0.4,0);
\draw[red,thick] (0.2,0) -- (0.2,-0.2);
\draw[red,thick,-<-] (0,-0.2) -- (0,-2) node[below,scale=0.8]{$+$};
\draw[red,thick,->-={0.3}{}] (-1.2,-1.0) ..controls (-0.8,-0.6) and (-0.5,-0.6).. (0,-0.6);
\node[red] at (1.5,-1.5) {$W$};
\begin{scope}[xshift=4.5cm]
\draw[blue] (2,0) -- (0,2) -- (-2,0) -- (0,-2) --cycle;
\draw[blue] (0,-2) to[bend left=30pt] (0,2);
\draw[blue] (0,-2) to[bend right=30pt] (0,2);
\draw[red,thick] (-1.4,0) --++(0.6,0.6) --++(0,-1.2) --cycle;
\draw[red,thick,->-={0.7}{}] (-1.6,0.6) -- (-1.2,0.2);
\draw[red,thick,->-={0.7}{}] (-1.5,0.7) -- (-1.1,0.3);
%\draw[red,thick,->-] (-1.4,0.8) -- (-1.0,0.4);
\draw[red,thick,->-={0.7}{}] (-0.9,1.5) ..controls (-0.6,1) and (-1.3,0.7).. (-1.0,0.4);
\draw[red,thick,->-] (-0.8,1.6) to[out=-45,in=-135] (0.8,1.6);
\draw[red,thick,->-] (-0.7,1.7) to[out=-45,in=-135] (0.7,1.7);
\draw[red,very thick,dotted] (0,1.6) -- (0,1.9);
\draw[red,thick,->-={0.7}{}] (-1.6,-0.6) -- (-1.2,-0.2);
\draw[red,thick,->-={0.7}{}] (-1.5,-0.7) -- (-1.1,-0.3);
\draw[red,thick,->-={0.7}{}] (-1.4,-0.8) -- (-1.0,-0.4);
\draw[red,thick,-<-] (-0.8,0.2) -- (1,0.2);
\draw[red,thick,-<-] (1,0.2) -- (1.6,0.8);
\draw[red,thick,-<-] (1,0.2) -- (1.8,-0.6);
\draw[red,thick,-<-] (-0.8,0) -- (0.6,0);
\draw[red,thick,-<-] (0.6,0) ..controls (1,0) and (1.2,-0.4).. (1.6,-0.8);
\draw[red,thick,-<-] (-0.8,-0.2) -- (0.4,-0.2);
\draw[red,thick,-<-] (0.4,-0.2) ..controls (0.8,-0.2) and (1,-0.6).. (1.4,-1.0);
\draw[red,thick] (0.4,0.2) -- (0.4,0);
\draw[red,thick] (0.2,0) -- (0.2,-0.2);
\draw[red,thick] (0,-0.2) -- (0,-0.6);
\draw[red,thick,->-={0.3}{}] (-1.2,-1.0) ..controls (-0.8,-0.6) and (-0.5,-0.6).. (0,-0.6);
\draw[red,thick,-<-={0.7}{}] (0,-0.6) ..controls (0.4,-0.6) and (0.8,-1).. (1.0,-1.4);
\draw[red,thick,-<-] (-0.8,-1.6) to[out=45,in=135] (0.8,-1.6);
\draw[red,thick,-<-] (-0.7,-1.7) to[out=45,in=135] (0.7,-1.7);
\draw[red,very thick,dotted] (0,-1.6) -- (0,-1.9);
\node[red] at (1.5,-1.5) {$\cW$};
\end{scope}
\begin{scope}[xshift=9cm]
\draw[blue] (2,0) -- (0,2) -- (-2,0) -- (0,-2) --cycle;
\draw[blue] (0,-2) to[bend left=30pt] (0,2);
\draw[blue] (0,-2) to[bend right=30pt] (0,2);
\draw[red,thick] (-1.4,0) --++(0.6,0.6) --++(0,-1.2) --cycle;
\draw[red,thick,->-={0.7}{}] (-1.6,0.6) -- (-1.2,0.2);
\draw[red,thick,->-={0.7}{}] (-1.5,0.7) -- (-1.1,0.3);
\draw[red,thick,->-={0.7}{}] (-0.9,1.5) ..controls (-0.6,1) and (-1.3,0.7).. (-1.0,0.4);
\draw[red,thick,->-] (-0.8,1.6) to[out=-45,in=-135] (0.8,1.6);
\draw[red,thick,->-] (-0.7,1.7) to[out=-45,in=-135] (0.7,1.7);
\draw[red,very thick,dotted] (0,1.6) -- (0,1.9);
\draw[red,thick,->-={0.7}{}] (-1.6,-0.6) -- (-1.2,-0.2);
\draw[red,thick,->-={0.7}{}] (-1.5,-0.7) -- (-1.1,-0.3);
\draw[red,thick,->-={0.7}{}] (-1.4,-0.8) -- (-1.0,-0.4);
\draw[red,thick,-<-={0.65}{}] (1,0.2) ..controls (0,0.2) and (0.2,-0.2).. (0,-0.6) ..controls (-0.2,-0.8) and (-0.8,-0.6).. (-1.2,-1.0);
\draw[red,thick,-<-] (1,0.2) -- (1.6,0.8);
\draw[red,thick,-<-] (1,0.2) -- (1.8,-0.6);
\draw[red,thick,-<-={0.3}{}] (-0.8,0.2) ..controls(1,0.2) and  (1.2,-0.4).. (1.6,-0.8);
\draw[red,thick,-<-={0.3}{}] (-0.8,0) ..controls (0.8,0) and (1.0,-0.6).. (1.4,-1.0);
\draw[red,thick,-<-={0.3}{}] (-0.8,-0.2) ..controls (0.6,-0.2) and (0.8,-1.0).. (1.0,-1.4);
\draw[red,thick,-<-] (-0.8,-1.6) to[out=45,in=135] (0.8,-1.6);
\draw[red,thick,-<-] (-0.7,-1.7) to[out=45,in=135] (0.7,-1.7);
\draw[red,very thick,dotted] (0,-1.6) -- (0,-1.9);
\node[red] at (1.6,-1.5) {$\cW_\mathrm{br}^\tri$};
\end{scope}

\draw[thick,|->] (11.5,0) --node[midway,above]{$\sfx^\uf_\tri$} (12.5,0); 

\begin{scope}[xshift=15cm]
\draw[blue] (2,0) -- (0,2) -- (-2,0) -- (0,-2) --cycle;
\draw[blue] (0,-2) to (0,2);
{\color{mygreen}
\draw (0.8,0) circle(2pt) node[right, scale=.8]{$1$};
\draw (-0.8,0) circle(2pt) node[left, scale=.8]{$3$};
\draw (0,0.8) circle(2pt) node[above right, scale=.8]{$0$};
\draw(0,-0.8) circle(2pt)node[below right, scale=.8]{$-1$};
}
\end{scope}
\end{tikzpicture}
    \caption{An example of a signed web $W$ restricted to $Q_E$, the associated spiralling diagram $\cW$, its braid representation $\cW_\mathrm{br}^\tri$, its shear coordinates are shown order. In $\cW_\mathrm{br}^\tri$, there are two honeycomb components and infinitely many curve components.}
    \label{fig:shear_example}
\end{figure}

\begin{prop}\label{prop:shear_well-defined}
The shear coordinate system $\sfx^\uf_\tri(\hL) \in \bQ^{I_\uf(\tri)}$ is well-defined, and we get a map
\begin{align*}
    \sfx^\uf_\tri: \cL^x(\Sigma,\bQ) \to \bQ^{I_\uf(\tri)}.
\end{align*}
\end{prop}

\begin{proof}
It is not hard to see that the operations appearing in \cref{prop:spiralling_good_position} that move a spiralling diagram in a good position to another good position do not change the shear coordinates. For example, the modified H-move always involves a pair of oriented curves in the opposite directions in the braid representation and hence preserves the contribution from the pair. It follows that the shear coordinates are well-defined for a given spiralling diagram, not depending on the choice of a good position with respect to $\widehat{\tri}$. 

We need to check that the elementary moves (E1)--(E4) of signed webs do not change the shear coordinates. It is easy to see the invariance for the loop parallel-move (E1). The braid representatives of spiralling diagrams associated with the local signed webs in  \eqref{eq:boundary_H-move}--\eqref{eq:puncture_H-move_2} are obtained as follows:
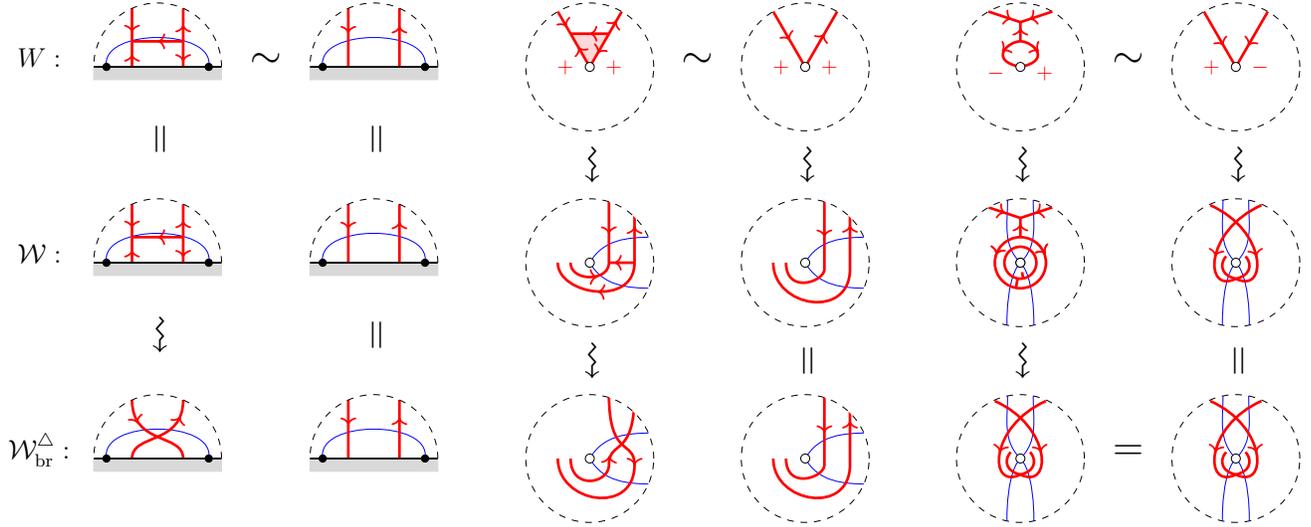
\begin{equation*}
\adjustbox{scale=0.7, center}{
\begin{tikzcd}
W: \hspace{-1cm}&
\begin{tikzpicture}
\draw[blue] (-0.8,0) to[bend left=100] (0.8,0);
\draw[very thick,red,-<-] (0.4,0) -- (0.4,0.4);
\draw[very thick,red,->-] (-0.4,0) -- (-0.4,0.4);
\draw[very thick,red,->-] (0.4,0.4) -- (0.4,0.917);
\draw[very thick,red,-<-] (-0.4,0.4) -- (-0.4,0.917);
\draw[very thick,red,->-] (0.4,0.4) -- (-0.4,0.4);
\draw[dashed] (1,0) arc (0:180:1cm);
\bline{-1,0}{1,0}{0.2}
\draw[fill=black] (-0.8,0) circle(1.6pt);
\draw[fill=black] (0.8,0) circle(1.6pt);
% \node at (0,-0.5) {$W=\cW$};
\end{tikzpicture}
\ar[d, phantom, "="{scale=1.5,rotate=-90}]
\ar[r, phantom, "\sim"{scale=1.5}]
&
\begin{tikzpicture}
\draw[blue] (-0.8,0) to[bend left=100] (0.8,0);
\draw[very thick,red,->-={0.7}{}] (0.4,0) -- (0.4,0.917);
\draw[very thick,red,-<-={0.7}{}] (-0.4,0) -- (-0.4,0.917);
\draw[dashed] (1,0) arc (0:180:1cm);
\bline{-1,0}{1,0}{0.2}
\draw[fill=black] (-0.8,0) circle(1.6pt);
\draw[fill=black] (0.8,0) circle(1.6pt);
\end{tikzpicture}
\ar[d, phantom, "="{scale=1.5,rotate=-90}]
&
\begin{tikzpicture}
\draw[dashed, fill=white] (0,0) circle [radius=1];
\draw(60:0.6) coordinate(A);
\draw(120:0.6) coordinate(B);
\fill[pink!60] (A) -- (B) -- (0,0) --cycle; 
\draw[red,very thick,-<-={0.6}{}] (0,0)--(A);
\draw[red,very thick,->-={0.6}{}] (0,0)--(B);
\draw[red,very thick,->-={0.6}{}] (A)--(60:1);
\draw[red,very thick,-<-={0.6}{}] (B)--(120:1);
\draw[red,very thick,->-] (A) -- (B);
\node[red,scale=0.8,anchor=west] at (0.1,0) {$+$};
\node[red,scale=0.8,anchor=east] at (-0.1,0) {$+$};
\draw[fill=white] (0,0) circle(2pt);
\end{tikzpicture}
\ar[d, phantom, "\rightsquigarrow"{scale=1.5, rotate=-90}]
\ar[r, phantom, "\sim"{scale=1.5}]
&
\begin{tikzpicture}
\draw[dashed, fill=white] (0,0) circle [radius=1];
\draw(60:0.6) coordinate(A);
\draw(120:0.6) coordinate(B);
%\fill[pink!60] (A) -- (B) -- (0,0) --cycle; 
\draw[red,very thick,->-={0.6}{}] (0,0)--(60:1);
\draw[red,very thick,-<-={0.6}{}] (0,0)--(120:1);
%\draw[red,very thick] (A) -- (B);
\node[red,scale=0.8,anchor=west] at (0.1,0) {$+$};
\node[red,scale=0.8,anchor=east] at (-0.1,0) {$+$};
\draw[fill=white] (0,0) circle(2pt);
\end{tikzpicture}
\ar[d, phantom, "\rightsquigarrow"{scale=1.5, rotate=-90}]
&
\begin{tikzpicture}
\draw[dashed, fill=white] (0,0) circle [radius=1];
\draw[red,very thick,-<-={0.7}{}] (0,0) ..controls (0.5,0.1) and (0.2,0.4).. (0,0.4);
\draw[red,very thick,-<-={0.7}{}] (0,0) ..controls (-0.5,0.1) and (-0.2,0.4).. (0,0.4);
\draw[red,very thick,->-={0.7}{}] (0,0.4) -- (0,0.7);
\draw[red,very thick,-<-] (0,0.7) -- (60:1);
\draw[red,very thick,-<-] (0,0.7) -- (120:1);
\node[red,scale=0.8,anchor=east] at (-0.1,-0.1) {$-$};
\node[red,scale=0.8,anchor=west] at (0.1,-0.1) {$+$};
\draw[fill=white] (0,0) circle(2pt);
\end{tikzpicture}
\ar[d, phantom, "\rightsquigarrow"{scale=1.5, rotate=-90}]
\ar[r, phantom, "\sim"{scale=1.5}]
&
\begin{tikzpicture}
\draw[dashed] (0,0) circle(1cm);
\draw[red,very thick,-<-] (0,0) -- (60:1);
\draw[red,very thick,-<-] (0,0) -- (120:1);
\node[red,scale=0.8,anchor=east] at (-0.1,0) {$+$};
\node[red,scale=0.8,anchor=west] at (0.1,0) {$-$};
\draw[fill=white] (0,0) circle(2pt);
\end{tikzpicture}
\ar[d, phantom, "\rightsquigarrow"{scale=1.5, rotate=-90}]
\\
\cW: \hspace{-1cm}& 
\begin{tikzpicture}
\draw[blue] (-0.8,0) to[bend left=100] (0.8,0);
\draw[very thick,red,-<-] (0.4,0) -- (0.4,0.4);
\draw[very thick,red,->-] (-0.4,0) -- (-0.4,0.4);
\draw[very thick,red,->-] (0.4,0.4) -- (0.4,0.917);
\draw[very thick,red,-<-] (-0.4,0.4) -- (-0.4,0.917);
\draw[very thick,red,->-] (0.4,0.4) -- (-0.4,0.4);
\draw[dashed] (1,0) arc (0:180:1cm);
\bline{-1,0}{1,0}{0.2}
\draw[fill=black] (-0.8,0) circle(1.6pt);
\draw[fill=black] (0.8,0) circle(1.6pt);
% \node at (0,-0.5) {$W=\cW$};
\end{tikzpicture}
\ar[d, phantom, "\rightsquigarrow"{scale=1.5, rotate=-90}]
&
\begin{tikzpicture}
\draw[blue] (-0.8,0) to[bend left=100] (0.8,0);
\draw[very thick,red,->-={0.7}{}] (0.4,0) -- (0.4,0.917);
\draw[very thick,red,-<-={0.7}{}] (-0.4,0) -- (-0.4,0.917);
\draw[dashed] (1,0) arc (0:180:1cm);
\bline{-1,0}{1,0}{0.2}
\draw[fill=black] (-0.8,0) circle(1.6pt);
\draw[fill=black] (0.8,0) circle(1.6pt);
\end{tikzpicture}
\ar[d, phantom, "="{scale=1.5,rotate=-90}]
&
\begin{tikzpicture}
\draw[dashed, fill=white] (0,0) circle [radius=1];
\begin{scope}
    \clip (0,0) circle(1cm);
    \draw[blue] (0,0) to[out=60,in=180] (1,0.4);
    \draw[blue] (0,0) to[out=-60,in=180] (1,-0.4);
    \draw[red,very thick,-<-] (0.3,0) -- (0.7,0);
    \draw[red,very thick,-<-={0.55}{}] (0.7,1) -- (0.7,0);
    \draw[red,very thick,->-] (0.7,0) ..controls (0.7,-0.7) and (-0.5,-0.5).. (-0.5,0);
    \draw[red,very thick,->-={0.55}{}] (0.3,1) -- (0.3,0);
    \draw[red,very thick,-<-] (0.3,0) ..controls (0.3,-0.3) and (-0.3,-0.3).. (-0.3,0);
\end{scope}
\draw[fill=white] (0,0) circle(2pt);
\end{tikzpicture}
\ar[d, phantom, "\rightsquigarrow"{scale=1.5,rotate=-90}]
&
\begin{tikzpicture}
\draw[dashed, fill=white] (0,0) circle [radius=1];
\begin{scope}
    \clip (0,0) circle(1cm);
    \draw[blue] (0,0) to[out=60,in=180] (1,0.4);
    \draw[blue] (0,0) to[out=-60,in=180] (1,-0.4);
    \draw[red,very thick,->-={0.3}{}] (0.3,1) ..controls ++(0,-1) and (0.3,0.3).. (0.3,-0.1) ..controls (0.3,-0.3) and (-0.3,-0.5).. (-0.3,0);
    \draw[red,very thick,-<-={0.18}{}] (0.7,1) ..controls ++(0,-1) and (0.7,0.3).. (0.7,-0.1) ..controls (0.7,-0.9) and (-0.5,-0.7).. (-0.5,0);
\end{scope}
\draw[fill=white] (0,0) circle(2pt);
\end{tikzpicture}
\ar[d, phantom, "="{scale=1.5, rotate=-90}]
&
\begin{tikzpicture}
\draw[dashed, fill=white] (0,0) circle [radius=1];
\begin{scope}
    \clip (0,0) circle(1cm);
    \draw[blue] (0,0) to[out=45,in=-90] (0.2,1);
    \draw[blue] (0,0) to[out=135,in=-90] (-0.2,1);
	\draw[blue] (0,0) to[out=-45,in=-90] (0.2,-1);
    \draw[blue] (0,0) to[out=-135,in=-90] (-0.2,-1);
    \draw[red,very thick,-<-={0.1}{},->-={0.45}{}] (0,0) circle(0.4cm);
	\draw[red,very thick] (0,0) circle(0.25cm);
	\draw[red,very thick] (-100:0.4) -- (-100:0.25);
    \draw[red,very thick] (-80:0.25) -- (-80:0.15);
    \draw[red,very thick,->-={0.7}{}] (0,0.4) -- (0,0.7);
    \draw[red,very thick,-<-={0.7}{}] (0,0.7) -- (60:1);
    \draw[red,very thick,-<-={0.7}{}] (0,0.7) -- (120:1);
    \draw[fill=white] (0,0) circle(2pt);
\end{scope}
\draw[fill=white] (0,0) circle(2pt);
\end{tikzpicture}
\ar[d, phantom, "\rightsquigarrow"{scale=1.5, rotate=-90}]
&
\begin{tikzpicture}
\draw[dashed, fill=white] (0,0) circle [radius=1];
\begin{scope}
    \clip (0,0) circle(1cm);
    \draw[blue] (0,0) to[out=45,in=-90] (0.2,1);
    \draw[blue] (0,0) to[out=135,in=-90] (-0.2,1);
	\draw[blue] (0,0) to[out=-45,in=-90] (0.2,-1);
    \draw[blue] (0,0) to[out=-135,in=-90] (-0.2,-1);
    \draw [red,very thick,->-](0.4,0.9) .. controls (-0.5,0.5) and (-0.5,-0.5) .. (0,-0.25) .. controls (0.25,-0.2) and (0.25,0) .. (0.15,0.1);
	\begin{scope}[xscale=-1]
	\draw [red,very thick,->-](0.4,0.9) .. controls (-0.5,0.5) and (-0.5,-0.5) .. (0,-0.25) .. controls (0.25,-0.2) and (0.25,0) .. (0.15,0.1);
	\end{scope}
\end{scope}
\draw[fill=white] (0,0) circle(2pt);
\end{tikzpicture}
\ar[d, phantom, "="{scale=1.5, rotate=-90}]
\\
\cW_{\mathrm{br}}^\tri: \hspace{-1cm}&
\begin{tikzpicture}
\draw[blue] (-0.8,0) to[bend left=100] (0.8,0);
\draw[very thick,red,-<-={0.8}{}] (0.4,0) ..controls (0.4,0.4) and (-0.4,0.2).. (-0.4,0.917);
\draw[very thick,red,->-={0.8}{}] (-0.4,0) ..controls (-0.4,0.4) and (--0.4,0.2).. (0.4,0.917);
\draw[dashed] (1,0) arc (0:180:1cm);
\bline{-1,0}{1,0}{0.2}
\draw[fill=black] (-0.8,0) circle(1.6pt);
\draw[fill=black] (0.8,0) circle(1.6pt);
% \node at (0,-0.5) {$\cW_\mathrm{br}$};
\end{tikzpicture}
&
\begin{tikzpicture}
\draw[blue] (-0.8,0) to[bend left=100] (0.8,0);
\draw[very thick,red,->-={0.7}{}] (0.4,0) -- (0.4,0.917);
\draw[very thick,red,-<-={0.7}{}] (-0.4,0) -- (-0.4,0.917);
\draw[dashed] (1,0) arc (0:180:1cm);
\bline{-1,0}{1,0}{0.2}
\draw[fill=black] (-0.8,0) circle(1.6pt);
\draw[fill=black] (0.8,0) circle(1.6pt);
\end{tikzpicture}
&
\begin{tikzpicture}
\draw[dashed, fill=white] (0,0) circle [radius=1];
\begin{scope}
    \clip (0,0) circle(1cm);
    \draw[blue] (0,0) to[out=60,in=180] (1,0.4);
    \draw[blue] (0,0) to[out=-60,in=180] (1,-0.4);
    \draw[red,very thick,->-={0.4}{}] (0.3,1) ..controls ++(0,-1) and (0.7,0.3).. (0.7,-0.1) ..controls (0.7,-0.9) and (-0.5,-0.7).. (-0.5,0);
    \draw[red,very thick,-<-={0.55}{}] (0.7,1) ..controls ++(0,-1) and (0.3,0.3).. (0.3,-0.1) ..controls (0.3,-0.4) and (-0.3,-0.5).. (-0.3,0);
\end{scope}
\draw[fill=white] (0,0) circle(2pt);
\end{tikzpicture}
&
\begin{tikzpicture}
\draw[dashed, fill=white] (0,0) circle [radius=1];
\begin{scope}
    \clip (0,0) circle(1cm);
    \draw[blue] (0,0) to[out=60,in=180] (1,0.4);
    \draw[blue] (0,0) to[out=-60,in=180] (1,-0.4);
    \draw[red,very thick,->-={0.3}{}] (0.3,1) ..controls ++(0,-1) and (0.3,0.3).. (0.3,-0.1) ..controls (0.3,-0.3) and (-0.3,-0.5).. (-0.3,0);
    \draw[red,very thick,-<-={0.18}{}] (0.7,1) ..controls ++(0,-1) and (0.7,0.3).. (0.7,-0.1) ..controls (0.7,-0.9) and (-0.5,-0.7).. (-0.5,0);
\end{scope}
\draw[fill=white] (0,0) circle(2pt);
\end{tikzpicture}
&
\begin{tikzpicture}
\draw[dashed, fill=white] (0,0) circle [radius=1];
\begin{scope}
    \clip (0,0) circle(1cm);
    \draw[blue] (0,0) to[out=45,in=-90] (0.2,1);
    \draw[blue] (0,0) to[out=135,in=-90] (-0.2,1);
	\draw[blue] (0,0) to[out=-45,in=-90] (0.2,-1);
    \draw[blue] (0,0) to[out=-135,in=-90] (-0.2,-1);
    \draw [red,very thick,->-](0.4,0.9) .. controls (-0.5,0.5) and (-0.5,-0.5) .. (0,-0.25) .. controls (0.25,-0.2) and (0.25,0) .. (0.15,0.1);
	\begin{scope}[xscale=-1]
	\draw [red,very thick,->-](0.4,0.9) .. controls (-0.5,0.5) and (-0.5,-0.5) .. (0,-0.25) .. controls (0.25,-0.2) and (0.25,0) .. (0.15,0.1);
	\end{scope}
\end{scope}
\draw[fill=white] (0,0) circle(2pt);
\end{tikzpicture}
\ar[r, phantom, "="{scale=1.5}]
&
\begin{tikzpicture}
\draw[dashed, fill=white] (0,0) circle [radius=1];
\begin{scope}
    \clip (0,0) circle(1cm);
    \draw[blue] (0,0) to[out=45,in=-90] (0.2,1);
    \draw[blue] (0,0) to[out=135,in=-90] (-0.2,1);
	\draw[blue] (0,0) to[out=-45,in=-90] (0.2,-1);
    \draw[blue] (0,0) to[out=-135,in=-90] (-0.2,-1);
    \draw [red,very thick,->-](0.4,0.9) .. controls (-0.5,0.5) and (-0.5,-0.5) .. (0,-0.25) .. controls (0.25,-0.2) and (0.25,0) .. (0.15,0.1);
	\begin{scope}[xscale=-1]
	\draw [red,very thick,->-](0.4,0.9) .. controls (-0.5,0.5) and (-0.5,-0.5) .. (0,-0.25) .. controls (0.25,-0.2) and (0.25,0) .. (0.15,0.1);
	\end{scope}
\end{scope}
\draw[fill=white] (0,0) circle(2pt);
\end{tikzpicture}
\end{tikzcd}}
\end{equation*}
Here the braid representatives are not quite the same in the first two cases, but both have the same shear coordinates. 
%Observe that the right-most ones are the braid representatives of spiralling diagrams arising from the right-hand sides of \eqref{eq:boundary_H-move}--\eqref{eq:puncture_H-move_2}. 
Thus the shear coordinates are invariant under the moves (E2) and (E3). The invariance under the peripheral move (E4) is similarly verified, where the signed web in the left-hand side produces a peripheral component in its spiralling diagram.
%The braid representative of spiralling diagram arising from the signed web $W_1$ there has a peripheral component if and only if the product web $W'_1$ has circle components. 

The shear coordinates are clearly invariant under the operations (2) and (3) in \cref{def:unbounded laminations}, and hence do not depend on the choice of a signed $\bQ_{>0}$-weighted web representing an unbounded $\fsl_3$-lamination. 
\end{proof}

\begin{conv}\label{conv:shear_coordinates}
We will write $\sfx^\tri_T:=\sfx^\tri_{i(T)}$ for a triangle $T$ of $\tri$, and $\sfx^\tri_{E,s}:=\sfx^\tri_{i^s(E)}$ for an oriented edge $E$ of $\tri$ and $s=1,2$. Here recall the notations in \cref{subsec:notation_marked_surface}.
\end{conv}

\subsection{Reconstruction}\label{subsec:reconstruction}
We are going to give an inverse map $\xi_\tri:\bQ^{I_\uf(\tri)} \to \cL^x(\Sigma,\bQ)$ of the shear coordinate system associated with an ideal triangulation $\tri$. 

Given $(\widetilde{\sfx}_i)_i \in \bQ^{I_\uf(\tri)}$, choose a positive integer $u \in \bZ_{>0}$ such that $\sfx_i:=u\widetilde{\sfx}_i$ are integral for all $i \in I_\uf(\tri)$. We will use a notation similar to \cref{conv:shear_coordinates} for these tuples. 
On each triangle $T \in t(\tri)$, first draw a honeycomb web of height $|\sfx_T|$ of sink type (resp. source type) if $\sfx_T \geq 0$ (resp. $\sfx_T < 0$). Moreover, on each corner of $T$, draw an semi-infinite collection of disjoint corner arcs with alternating orientations such that 
\begin{itemize}
    \item they are disjoint from the honeycomb web (placed on the center of $T$), 
    \item they accumulate only at the marked points of the triangle, and 
    \item the farthest one from the marked point is oriented clockwise. 
\end{itemize}
See \cref{fig:gluing block}. Then we get an unbounded reduced essential web $W_T$ on each triangle $T$. We are going to glue these local blocks together to form an integral unbounded $\fsl_3$-lamination $\xi_\tri((\sfx_i)_i) \in \cL^x(\Sigma,\bZ)$. 

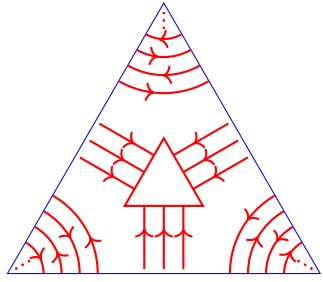
\begin{figure}[t]
    \centering
\begin{tikzpicture}[scale=1.2]
\draw[blue] (-30:2) coordinate(A) -- (90:2) coordinate(B) -- (210:2) coordinate(C) --cycle;
\begin{scope}
\clip (-30:2) -- (90:2) -- (210:2) --cycle;
\draw[red,thick] (-30:0.5) -- (90:0.5) -- (210:0.5) --cycle;
\foreach \i in {1,2,3}
{
\foreach \x in {0,120,240}
\draw[red,thick,rotate=\x,-<-] ($(-30:0.5)!\i/4!(90:0.5)$) --++(30:0.7);
}
\foreach \i in {0.6,1}
{
\foreach \x in {0,120,240}
    {
    \draw(-30+\x:2)++(120+\x:\i) coordinate(a);
    \draw(-30+\x:2)++(120+\x:\i-0.2) coordinate(b);
    \draw[red,thick,->-={0.4}{}] (a) arc(120+\x:180+\x:\i);
    \draw[red,thick,-<-={0.6}{}] (b) arc(120+\x:180+\x:\i-0.2);
    \draw[red,thick,dotted] (-30+\x:1.9) -- (-30+\x:1.65);
    }
}
\end{scope}
\end{tikzpicture}
    \caption{The building block for reconstruction from the shear coordinates when $\sfx_{T}=+3$.}
    \label{fig:gluing block}
\end{figure}

Now let us concentrate on a quadrilateral $Q_E$ in the ideal triangulation $\tri$ which contains two triangles $T_L$ and $T_R$ that share an interior edge $E$. We fix an orientation of $E$ so that $T_L$ lies on the left, hence have two edge coordinates $\sfx_{\bott}$ and $\sfx_{\topp}$ as well as two face coordinates $\sfx_{T_L}$ and $\sfx_{T_R}$:
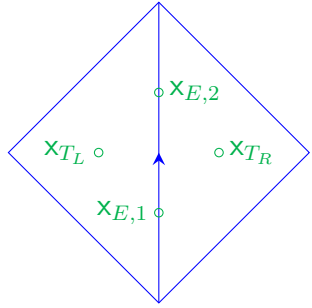
\begin{align*}
\begin{tikzpicture}[scale=0.8]
\draw[blue] (0,2.5) -- (-2.5,0) -- (0,-2.5) -- (2.5,0) --cycle;
\draw[blue, postaction={decorate}, decoration={markings,mark=at position 0.5 with {\arrow[scale=1.5,>=stealth]{>}}}] (0,-2.5) -- (0,2.5);
\draw[mygreen] (1,0) circle(2pt) node[right]{$\sfx_{T_R}$};
\draw[mygreen] (-1,0) circle(2pt) node[left]{$\sfx_{T_L}$};
\draw[mygreen] (0,-1) circle(2pt) node[left]{$\sfx_{E,1}$};
\draw[mygreen] (0,1) circle(2pt) node[right]{$\sfx_{E,2}$};
\end{tikzpicture}
\end{align*}
Consider a biangle $B_E$ in the split ideal triangulation $\widehat{\tri}$ obtained by fattening $E$, which is bounded by boundary intervals $E_L$ and $E_R$ of $T_L$ and $T_R$, respectively. 
For $Z \in \{L,R\}$, let $S_{Z}=S_{Z}^+\sqcup S_{Z}^-$ denote the set of ends of the web $W_{T_Z}$ on $E_Z$, where $S_{Z}^+$ (resp. $S_{Z}^-$) consists of the ends incoming to (resp. outgoing from) the biangle $B_E$. Then $S=(S_L,S_R)$ defines an asymptotically periodic symmetric strand set (\cref{def:infinite strand set}). 
Let us define its pinning by the following rule:
\begin{itemize}
    \item For $Z \in \{L,R\}$, choose orientation-preserving parametrizations $\phi_{Z}^\pm:\bR \to E_Z$ so that $\phi_{Z}^\pm(\frac{1}{2}+\bZ)=S_{Z}^\pm$, and $\phi_Z^\pm(\bR_{<0}) \cap S_Z^\pm$ consists of all the strands coming from the corner arcs around the initial marked point of $E_Z$. 
    \item Let $p_Z^\pm:=\phi_{Z}^\pm(n_{Z}^\pm) \in E_Z$ for $Z \in \{L,R\}$, where $n_{Z}^\pm \in \bZ$ are given by%\footnote{This strange rule will be clarified in \cref{sec:amalgamation} in relation with the \emph{pinnings} of $\fsl_3$-laminations.}
    \begin{align}\label{eq:gluing_rule}
    \begin{aligned}
    n_{L}^+&:=\sfx_{\bott},& n_{R}^-&:=[\sfx_{T_R}]_+,
    \\ 
    % &n_{L}^-:=\sfx_{\topp} + [\sfx_{T_L}]_+, & &n_{R}^+:=0,
    n_{L}^-&:= [\sfx_{T_L}]_+,& n_{R}^+&:= \sfx_{\topp},
    \end{aligned}
    \end{align}
\end{itemize}
where we use the notation $[x]_+:=\max\{0,x\}$.
Then we get a pinned symmetric strand set $\widehat{S}_{E}:=(S;\sfp_L,\sfp_R)$ with the pinnings $\sfp_Z:=(p_Z^+,p_Z^-)$ for $Z \in \{L,R\}$. 
Let $W_{\mathrm{br}}(\widehat{S}_E)$ denote the associated collection of oriented curves in $B_E$. 

\begin{rem}\label{rem:gluing_symmetry}
%Despite the asymmetry of the rule \eqref{eq:gluing_rule}, the resulting collection $W_{\mathrm{br}}(\widetilde{S}_E)$ is unchanged if we swap the role of the two triangles $T_L$ and $T_R$. 
The resulting collection $W_{\mathrm{br}}(\widehat{S}_E)$ is invariant under the transformation
\begin{align*}
\begin{aligned}
    n_{L}^+ &\mapsto n_{L}^+ - k,& n_{R}^- &\mapsto n_{R}^- + k,
    \\ 
    n_{L}^- &\mapsto n_{L}^- - l,& n_{R}^+ &\mapsto n_{R}^+ + l
\end{aligned}
\end{align*}
for $k,l \in \bZ$. 
%Then by choosing $k=\sfx_{\bott}$ and $l=\sfx_{\topp}$, we get the equivalent rule where the roles of $T_L$ and $T_R$ are swapped.
\end{rem}

Gluing together the local webs $W_T$ for $T \in t(\tri)$ and the curves in $W_{\mathrm{br}}(\widehat{S}_E)$ for $E \in e(\tri)$, we get a (possibly infinite) collection $\cW_{\mathrm{br}}^\tri((\sfx_i)_i)$ of webs on $\Sigma$. 
%It turns out that the collection $\cW_{\mathrm{br}}^\tri((\sfx'_i)_i)$ is the braid representative of a spiralling diagram. 
The following lemma shows that it has correct shear coordinates.

\begin{lem}\label{lem:braid-shear}
We have $\sfx_k^\tri(\cW_{\mathrm{br}}^\tri((\sfx_i)_i)) = \sfx_k$ for all $k \in I_\uf(\tri)$.
\end{lem}

\begin{proof}
Let us concentrate on a quadrilateral $Q=T_L \cup B_E \cup T_R$. It is easy to see $\sfx_{T_Z}^\tri(\hL)=\sfx_{T_Z}$ for $Z \in \{L,R\}$. The equalities $\sfx^\tri_{\bott}(\hL)=\sfx_{\bott}$ and $\sfx^\tri_{\topp}(\hL)=\sfx_{\topp}$ can be also verified case-by-case, divided according to the signs of $\sfx_{T_L}$ and $\sfx_{T_R}$. See Figures \ref{fig:reconstruction1}--\ref{fig:reconstruction3}. Here we draw the pictures by separating the gluing procedures $S_L^- \to S_R^+$ and $S_L^+ \to S_R^-$ into two sheets: the result is obtained by overlaying the two diagrams drawn on the right. 

For example, let us consider the example shown in \cref{fig:reconstruction1}. In the case $\sfx_{\topp}\geq 0$ (as in this example), there are $\sfx_{\topp}$ many lines from South-East to North-West that contribute positively. One can imagine the other cases by varying this example: if we decrease $\sfx_{\topp}$, then the point $p_R^+$ moves upward and the gluing pattern is shifted. When $-\sfx_{T_L} \leq \sfx_{\topp}< 0$, negative contributions come from the honeycomb in $T_L$. When $\sfx_{\topp} < -\sfx_{T_L}$, there are also lines from South-West to North-East that contribute negatively. Thus we get $\sfx_{\topp}^\tri(\hL)=\sfx_{\topp}$. The check for $\sfx_{\bott}$ is similar. One can check the other cases from \cref{fig:reconstruction2,fig:reconstruction3} by a similar manner.
\end{proof}

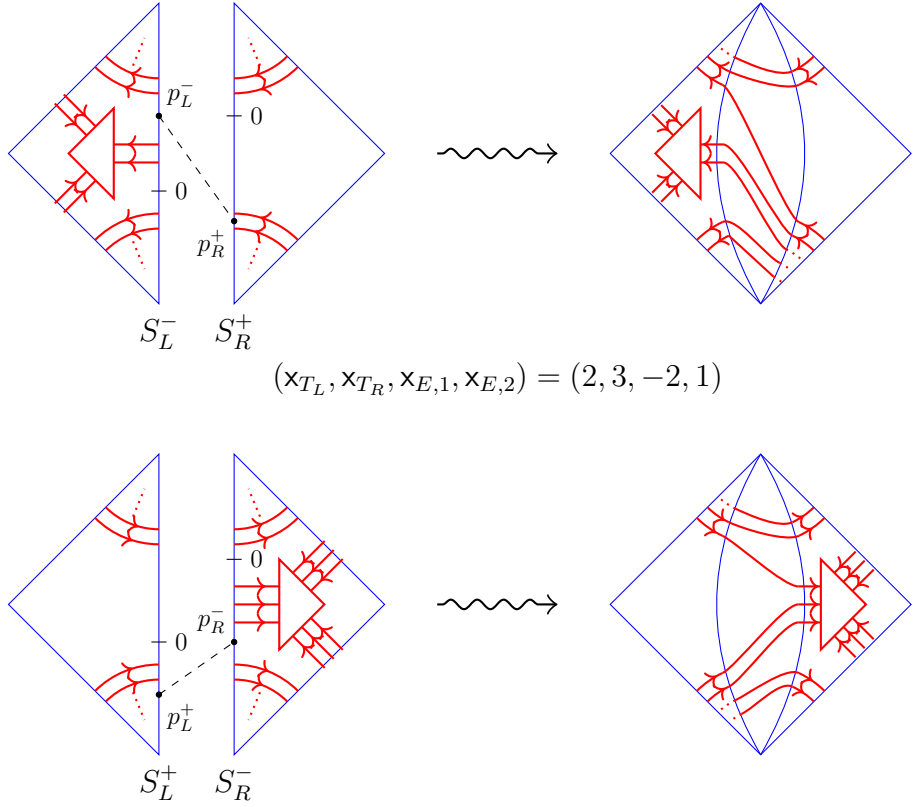
\begin{figure}[htbp]
% \iffalse
\begin{tikzpicture}
\draw[blue] (-2,0) -- (0,2) -- (0,-2) --cycle;
\draw[red,thick] (-1.2,0) coordinate(A) --++(0.6,0.6) coordinate(B) --++(0,-1.2) coordinate(C) --cycle;
\draw[red,thick,-<-] ($(A)!0.6!(B)$) --++(135:0.6);
\draw[red,thick,-<-] ($(A)!0.4!(B)$) --++(135:0.6);
\draw[red,thick,-<-] ($(A)!0.6!(C)$) --++(225:0.6);
\draw[red,thick,-<-] ($(A)!0.4!(C)$) --++(225:0.6);
\draw[red,thick,-<-] ($(B)!0.6!(C)$) --++(0:0.6);
\draw[red,thick,-<-] ($(B)!0.4!(C)$) --++(0:0.6);
\node[below] at (0,-2) {$S_L^-$};

\foreach \x in {1.2,1.0}
{
\draw[red,thick,->-] (0,-2+\x) arc(90:135:\x); 
\draw[red,thick,->-] (0,2-\x) arc(-90:-135:\x);
}
\draw[red,thick,dotted] ($(0,-2)+(112.5:0.9)$)--++(-67.5:0.4);
\draw[red,thick,dotted] ($(0,2)+(-112.5:0.9)$)--++(67.5:0.4);

\draw(-0.1,-0.5) -- (0.1,-0.5) node[right,scale=0.8]{$0$};
\filldraw(0,0.5) circle(1pt) node[above right,scale=0.8]{$p_L^-$} coordinate(p1);

\begin{scope}[xshift=1cm]
\draw[blue] (2,0) -- (0,2) -- (0,-2) --cycle;
\node[below] at (0,-2) {$S_R^+$};

\foreach \x in {1.2,1.0}
{
\draw[red,thick,-<-] (0,-2+\x) arc(90:45:\x); 
\draw[red,thick,-<-] (0,2-\x) arc(-90:-45:\x);
}
\draw[red,thick,dotted] ($(0,-2)+(67.5:0.9)$)--++(-112.5:0.4);
\draw[red,thick,dotted] ($(0,2)+(-67.5:0.9)$)--++(112.5:0.4);

\draw(-0.1,0.5) -- (0.1,0.5) node[right,scale=0.8]{$0$};
\filldraw(0,-0.9) circle(1pt) node[below left,scale=0.8]{$p_R^+$} coordinate(p2);
\end{scope}
\draw[dashed] (p1) -- (p2);
\draw [thick,-{Classical TikZ Rightarrow[length=4pt]},decorate,decoration={snake,amplitude=1.8pt,pre length=2pt,post length=3pt}](3.7,0) --(5.3,0);

\begin{scope}[xshift=8cm]
\draw[blue] (2,0) -- (0,2) -- (-2,0) -- (0,-2) --cycle;
\draw[blue] (0,-2) to[bend left=30pt] (0,2);
\draw[blue] (0,-2) to[bend right=30pt] (0,2);
\draw[red,thick] (-1.4,0) coordinate(A) --++(0.6,0.6) coordinate(B) --++(0,-1.2) coordinate(C) --cycle;
\draw[red,thick,-<-] ($(A)!0.6!(B)$) --++(135:0.4);
\draw[red,thick,-<-] ($(A)!0.4!(B)$) --++(135:0.4);
\draw[red,thick,-<-] ($(A)!0.6!(C)$) --++(225:0.4);
\draw[red,thick,-<-] ($(A)!0.4!(C)$) --++(225:0.4);
%left
\draw[red,thick,-<-] ($(B)!0.6!(C)$) --++(0:0.3) coordinate(L1);
\draw[red,thick,-<-] ($(B)!0.4!(C)$) --++(0:0.3) coordinate(L2);
\draw (L1) ++(0.2,0) coordinate(L1');
\draw (L2) ++(0.2,0) coordinate(L2');
\foreach \x in {5,6}
{
\draw[red,thick,->-] ($(0,-2)+(115:\x/5)$) coordinate(LB\x) arc(115:135:\x/5); 
\draw (LB\x) ++(25:0.2) coordinate (LB\x');
\draw[red,thick,->-] ($(0,2)+(-115:\x/5)$) coordinate(LT\x) arc(-115:-135:\x/5);
\draw (LT\x) ++(-25:0.2) coordinate (LT\x');
}
\foreach \x in {4}
{
\draw[red,thick,dotted] ($(0,2)+(-115:\x/5)$) coordinate(LT\x) arc(-115:-132:\x/5);
\draw (LT\x) ++(-25:0.2) coordinate (LT\x');
}
%\draw[red,thick,dotted] ($(0,-2)+(112.5:0.9)$)--++(-67.5:0.4);
%\draw[red,thick,dotted] ($(0,2)+(-112.5:0.9)$)--++(67.5:0.4);
%right
\foreach \x in {5,6}
{
\draw[red,thick,-<-] ($(0,-2)+(65:\x/5)$) coordinate(RB\x) arc(65:45:\x/5); 
\draw (RB\x) ++(180-25:0.2) coordinate (RB\x');
\draw[red,thick,-<-] ($(0,2)+(-65:\x/5)$) coordinate(RT\x) arc(-65:-45:\x/5);
\draw (RT\x) ++(-180+25:0.2) coordinate (RT\x');
}
\foreach \x in {2,3,4}
{
\draw[red,thick,dotted] ($(0,-2)+(65:\x/5)$) coordinate(RB\x) arc(65:47:\x/5); 
\draw (RB\x) ++(180-25:0.2) coordinate (RB\x');
}
%\draw[red,thick,dotted] ($(0,-2)+(67.5:0.9)$)--++(-112.5:0.4);
%\draw[red,thick,dotted] ($(0,2)+(-67.5:0.9)$)--++(112.5:0.4);
\draw[red,thick] (LT5) ..controls (LT5') and (RT6')..  (RT6);
\draw[red,thick] (LT6) ..controls (LT6') and (RB6').. (RB6);
\draw[red,thick] (L2) ..controls (L2') and (RB5').. (RB5);
\draw[red,thick] (L1) ..controls (L1') and (RB4').. (RB4);
\draw[red,thick] (LT4) ..controls (LT4') and (RT5').. (RT5);
\draw[red,thick] (LB6) ..controls (LB6') and (RB3').. (RB3);
\draw[red,thick] (LB5) -- (RB2);
\end{scope}
{\begin{scope}[yshift=-6cm]
\draw[blue] (-2,0) -- (0,2) -- (0,-2) --cycle;
\node[below] at (0,-2) {$S_L^+$};

\foreach \x in {1.2,1.0}
{
\draw[red,thick,-<-] (0,-2+\x) arc(90:135:\x); 
\draw[red,thick,-<-] (0,2-\x) arc(-90:-135:\x);
}
\draw[red,thick,dotted] ($(0,-2)+(112.5:0.9)$)--++(-67.5:0.4);
\draw[red,thick,dotted] ($(0,2)+(-112.5:0.9)$)--++(67.5:0.4);

\draw(-0.1,-0.5) -- (0.1,-0.5) node[right,scale=0.8]{$0$};
\filldraw(0,-1.2) circle(1pt) node[below right,scale=0.8]{$p_L^+$} coordinate(p1);

\begin{scope}[xshift=1cm]
\draw[blue] (2,0) -- (0,2) -- (0,-2) --cycle;
\draw[red,thick] (1.2,0) coordinate(A) --++(-0.6,-0.6) coordinate(B) --++(0,+1.2) coordinate(C) --cycle;
\draw[red,thick,-<-] ($(A)!0.7!(B)$) --++(-45:0.6);
\draw[red,thick,-<-] ($(A)!0.5!(B)$) --++(-45:0.6);
\draw[red,thick,-<-] ($(A)!0.3!(B)$) --++(-45:0.6);
\draw[red,thick,-<-] ($(A)!0.7!(C)$) --++(45:0.6);
\draw[red,thick,-<-] ($(A)!0.5!(C)$) --++(45:0.6);
\draw[red,thick,-<-] ($(A)!0.3!(C)$) --++(45:0.6);
\draw[red,thick,-<-] ($(B)!0.7!(C)$) --++(180:0.6);
\draw[red,thick,-<-] ($(B)!0.5!(C)$) --++(180:0.6);
\draw[red,thick,-<-] ($(B)!0.3!(C)$) --++(180:0.6);
\node[below] at (0,-2) {$S_R^-$};

\foreach \x in {1.2,1.0}
{
\draw[red,thick,->-] (0,-2+\x) arc(90:45:\x); 
\draw[red,thick,->-] (0,2-\x) arc(-90:-45:\x);
}
\draw[red,thick,dotted] ($(0,-2)+(67.5:0.9)$)--++(-112.5:0.4);
\draw[red,thick,dotted] ($(0,2)+(-67.5:0.9)$)--++(112.5:0.4);

\draw(-0.1,0.6) -- (0.1,0.6) node[right,scale=0.8]{$0$};
\filldraw(0,-0.5) circle(1pt) node[above left,scale=0.8]{$p_R^-$} coordinate(p2);
\end{scope}
\draw[dashed] (p1) -- (p2);
\draw [thick,-{Classical TikZ Rightarrow[length=4pt]},decorate,decoration={snake,amplitude=1.8pt,pre length=2pt,post length=3pt}](3.7,0) --(5.3,0);

\begin{scope}[xshift=8cm]
\draw[blue] (2,0) -- (0,2) -- (-2,0) -- (0,-2) --cycle;
\draw[blue] (0,-2) to[bend left=30pt] (0,2);
\draw[blue] (0,-2) to[bend right=30pt] (0,2);
\draw[red,thick] (1.4,0) coordinate(A) --++(-0.6,-0.6) coordinate(B) --++(0,+1.2) coordinate(C) --cycle;
\draw[red,thick,-<-] ($(A)!0.7!(B)$) --++(-45:0.4);
\draw[red,thick,-<-] ($(A)!0.5!(B)$) --++(-45:0.4);
\draw[red,thick,-<-] ($(A)!0.3!(B)$) --++(-45:0.4);
\draw[red,thick,-<-] ($(A)!0.7!(C)$) --++(45:0.4);
\draw[red,thick,-<-] ($(A)!0.5!(C)$) --++(45:0.4);
\draw[red,thick,-<-] ($(A)!0.3!(C)$) --++(45:0.4);
%left
\draw[red,thick,-<-] ($(B)!0.7!(C)$) --++(180:0.3) coordinate(R1);
\draw[red,thick,-<-] ($(B)!0.5!(C)$) --++(180:0.3) coordinate(R2);
\draw[red,thick,-<-] ($(B)!0.3!(C)$) --++(180:0.3) coordinate(R3);
\draw (R1) ++(-0.2,0) coordinate(R1');
\draw (R2) ++(-0.2,0) coordinate(R2');
\draw (R3) ++(-0.2,0) coordinate(R3');
\foreach \x in {5,6}
{
\draw[red,thick,-<-] ($(0,-2)+(115:\x/5)$) coordinate(LB\x) arc(115:135:\x/5); 
\draw (LB\x) ++(25:0.2) coordinate (LB\x');
\draw[red,thick,-<-] ($(0,2)+(-115:\x/5)$) coordinate(LT\x) arc(-115:-135:\x/5);
\draw (LT\x) ++(-25:0.2) coordinate (LT\x');
}
\foreach \x in {4}
{
\draw[red,thick,dotted] ($(0,2)+(-115:\x/5)$) coordinate(LT\x) arc(-115:-132:\x/5);
\draw (LT\x) ++(-25:0.2) coordinate (LT\x');
}
\foreach \x in {3,4}
{
\draw[red,thick,dotted] ($(0,-2)+(115:\x/5)$) coordinate(LB\x) arc(115:132:\x/5); 
\draw (LB\x) ++(25:0.2) coordinate (LB\x');
}
%\draw[red,thick,dotted] ($(0,-2)+(112.5:0.9)$)--++(-67.5:0.4);
%\draw[red,thick,dotted] ($(0,2)+(-112.5:0.9)$)--++(67.5:0.4);
%right
\foreach \x in {5,6}
{
\draw[red,thick,->-] ($(0,-2)+(65:\x/5)$) coordinate(RB\x) arc(65:45:\x/5); 
\draw (RB\x) ++(180-25:0.2) coordinate (RB\x');
\draw[red,thick,->-] ($(0,2)+(-65:\x/5)$) coordinate(RT\x) arc(-65:-45:\x/5);
\draw (RT\x) ++(-180+25:0.2) coordinate (RT\x');
}

%\draw[red,thick,dotted] ($(0,-2)+(67.5:0.9)$)--++(-112.5:0.4);
%\draw[red,thick,dotted] ($(0,2)+(-67.5:0.9)$)--++(112.5:0.4);

\draw[red,thick] (LB5) ..controls (LB5') and (R3')..  (R3);
\draw[red,thick] (LB6) ..controls (LB6') and (R2')..  (R2);
\draw[red,thick] (LT6) ..controls (LT6') and (R1')..  (R1);
\draw[red,thick] (LT5) ..controls (LT5') and (RT6')..  (RT6);
\draw[red,thick] (LT4) ..controls (LT4') and (RT5')..  (RT5);
\draw[red,thick] (LB4) ..controls (LB4') and (RB6')..  (RB6);
\draw[red,thick] (LB3) ..controls (LB3') and (RB5')..  (RB5);
\end{scope}
\end{scope}}
\node at (4.5,-3) {$(\sfx_{T_L},\sfx_{T_R},\sfx_{\bott},\sfx_{\topp})=(2,3,-2,1)$};
\end{tikzpicture}
% \fi%%
    \caption{An example for the case $\sfx_{T_L}\geq 0$ and $\sfx_{T_R}\geq 0$.}
    \label{fig:reconstruction1}
\end{figure}

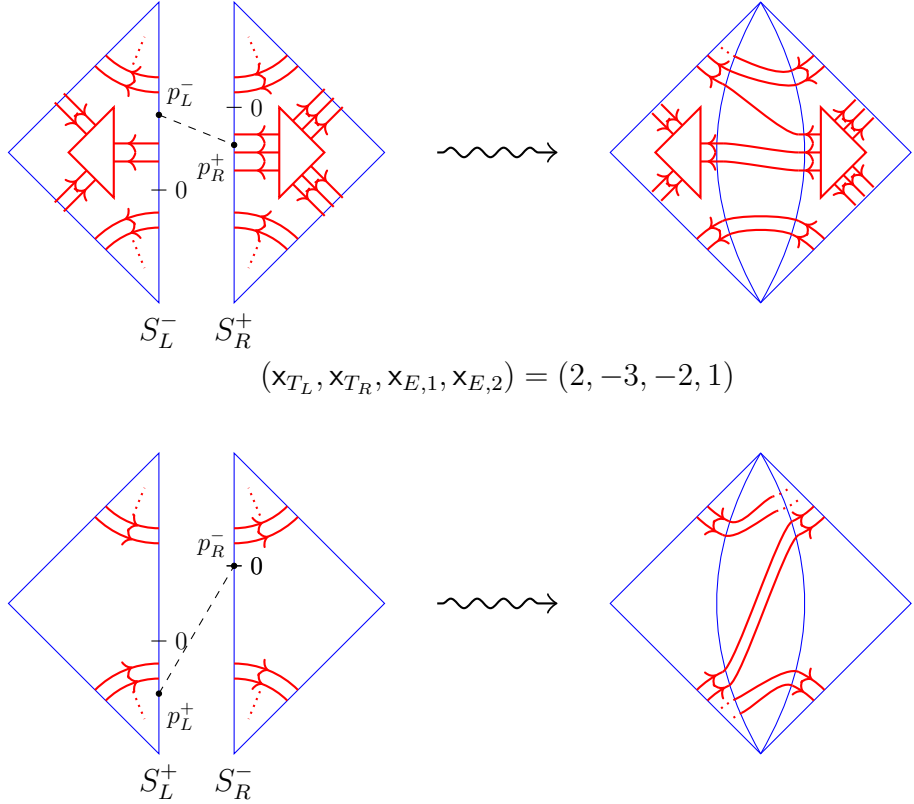
\begin{figure}[htbp]
% \iffalse
\begin{tikzpicture}
\draw[blue] (-2,0) -- (0,2) -- (0,-2) --cycle;
\draw[red,thick] (-1.2,0) coordinate(A) --++(0.6,0.6) coordinate(B) --++(0,-1.2) coordinate(C) --cycle;
\draw[red,thick,-<-] ($(A)!0.6!(B)$) --++(135:0.6);
\draw[red,thick,-<-] ($(A)!0.4!(B)$) --++(135:0.6);
\draw[red,thick,-<-] ($(A)!0.6!(C)$) --++(225:0.6);
\draw[red,thick,-<-] ($(A)!0.4!(C)$) --++(225:0.6);
\draw[red,thick,-<-] ($(B)!0.6!(C)$) --++(0:0.6);
\draw[red,thick,-<-] ($(B)!0.4!(C)$) --++(0:0.6);
\node[below] at (0,-2) {$S_L^-$};

\foreach \x in {1.2,1.0}
{
\draw[red,thick,->-] (0,-2+\x) arc(90:135:\x); 
\draw[red,thick,->-] (0,2-\x) arc(-90:-135:\x);
}
\draw[red,thick,dotted] ($(0,-2)+(112.5:0.9)$)--++(-67.5:0.4);
\draw[red,thick,dotted] ($(0,2)+(-112.5:0.9)$)--++(67.5:0.4);

\draw(-0.1,-0.5) -- (0.1,-0.5) node[right,scale=0.8]{$0$};
\filldraw(0,0.5) circle(1pt) node[above right,scale=0.8]{$p_L^-$} coordinate(p1);

\begin{scope}[xshift=1cm]
\draw[blue] (2,0) -- (0,2) -- (0,-2) --cycle;
\draw[red,thick] (1.2,0) coordinate(A) --++(-0.6,-0.6) coordinate(B) --++(0,+1.2) coordinate(C) --cycle;
\draw[red,thick,->-] ($(A)!0.7!(B)$) --++(-45:0.6);
\draw[red,thick,->-] ($(A)!0.5!(B)$) --++(-45:0.6);
\draw[red,thick,->-] ($(A)!0.3!(B)$) --++(-45:0.6);
\draw[red,thick,->-] ($(A)!0.7!(C)$) --++(45:0.6);
\draw[red,thick,->-] ($(A)!0.5!(C)$) --++(45:0.6);
\draw[red,thick,->-] ($(A)!0.3!(C)$) --++(45:0.6);
\draw[red,thick,->-] ($(B)!0.7!(C)$) --++(180:0.6);
\draw[red,thick,->-] ($(B)!0.5!(C)$) --++(180:0.6);
\draw[red,thick,->-] ($(B)!0.3!(C)$) --++(180:0.6);
\node[below] at (0,-2) {$S_R^+$};

\foreach \x in {1.2,1.0}
{
\draw[red,thick,-<-] (0,-2+\x) arc(90:45:\x); 
\draw[red,thick,-<-] (0,2-\x) arc(-90:-45:\x);
}
\draw[red,thick,dotted] ($(0,-2)+(67.5:0.9)$)--++(-112.5:0.4);
\draw[red,thick,dotted] ($(0,2)+(-67.5:0.9)$)--++(112.5:0.4);

\draw(-0.1,0.6) -- (0.1,0.6) node[right,scale=0.8]{$0$};
\filldraw(0,0.1) circle(1pt) node[below left,scale=0.8]{$p_R^+$} coordinate(p2);
\end{scope}
\draw[dashed] (p1) -- (p2);
\draw [thick,-{Classical TikZ Rightarrow[length=4pt]},decorate,decoration={snake,amplitude=1.8pt,pre length=2pt,post length=3pt}](3.7,0) --(5.3,0);

\begin{scope}[xshift=8cm]
\draw[blue] (2,0) -- (0,2) -- (-2,0) -- (0,-2) --cycle;
\draw[blue] (0,-2) to[bend left=30pt] (0,2);
\draw[blue] (0,-2) to[bend right=30pt] (0,2);

\draw[red,thick] (-1.4,0) coordinate(A) --++(0.6,0.6) coordinate(B) --++(0,-1.2) coordinate(C) --cycle;
\draw[red,thick,-<-] ($(A)!0.6!(B)$) --++(135:0.4);
\draw[red,thick,-<-] ($(A)!0.4!(B)$) --++(135:0.4);
\draw[red,thick,-<-] ($(A)!0.6!(C)$) --++(225:0.4);
\draw[red,thick,-<-] ($(A)!0.4!(C)$) --++(225:0.4);
\draw[red,thick,-<-={0.7}{}] ($(B)!0.6!(C)$) --++(0:0.3) coordinate(L1);
\draw[red,thick,-<-={0.7}{}] ($(B)!0.4!(C)$) --++(0:0.3) coordinate(L2);
\draw (L1) ++(0.2,0) coordinate(L1');
\draw (L2) ++(0.2,0) coordinate(L2');

\draw[red,thick] (1.4,0) coordinate(A) --++(-0.6,-0.6) coordinate(B) --++(0,+1.2) coordinate(C) --cycle;
\draw[red,thick,->-] ($(A)!0.7!(B)$) --++(-45:0.4);
\draw[red,thick,->-] ($(A)!0.5!(B)$) --++(-45:0.4);
\draw[red,thick,->-] ($(A)!0.3!(B)$) --++(-45:0.4);
\draw[red,thick,->-] ($(A)!0.7!(C)$) --++(45:0.4);
\draw[red,thick,->-] ($(A)!0.5!(C)$) --++(45:0.4);
\draw[red,thick,->-] ($(A)!0.3!(C)$) --++(45:0.4);
%left
\draw[red,thick,->-={0.7}{}] ($(B)!0.7!(C)$) --++(180:0.3) coordinate(R1);
\draw[red,thick,->-={0.7}{}] ($(B)!0.5!(C)$) --++(180:0.3) coordinate(R2);
\draw[red,thick,->-={0.7}{}] ($(B)!0.3!(C)$) --++(180:0.3) coordinate(R3);
\draw (R1) ++(-0.2,0) coordinate(R1');
\draw (R2) ++(-0.2,0) coordinate(R2');
\draw (R3) ++(-0.2,0) coordinate(R3');
\foreach \x in {5,6}
{
\draw[red,thick,->-] ($(0,-2)+(115:\x/5)$) coordinate(LB\x) arc(115:135:\x/5); 
\draw (LB\x) ++(25:0.2) coordinate (LB\x');
\draw[red,thick,->-] ($(0,2)+(-115:\x/5)$) coordinate(LT\x) arc(-115:-135:\x/5);
\draw (LT\x) ++(-25:0.2) coordinate (LT\x');
}
\foreach \x in {4}
{
\draw[red,thick,dotted] ($(0,2)+(-115:\x/5)$) coordinate(LT\x) arc(-115:-132:\x/5);
\draw (LT\x) ++(-25:0.2) coordinate (LT\x');
}
%\foreach \x in {3,4}
%{
%\draw[red,thick,dotted] ($(0,-2)+(115:\x/5)$) coordinate(LB\x) arc(115:132:\x/5); 
%\draw (LB\x) ++(25:0.2) coordinate (LB\x');
%}
%\draw[red,thick,dotted] ($(0,-2)+(112.5:0.9)$)--++(-67.5:0.4);
%\draw[red,thick,dotted] ($(0,2)+(-112.5:0.9)$)--++(67.5:0.4);
%right
\foreach \x in {5,6}
{
\draw[red,thick,-<-] ($(0,-2)+(65:\x/5)$) coordinate(RB\x) arc(65:45:\x/5); 
\draw (RB\x) ++(180-25:0.2) coordinate (RB\x');
\draw[red,thick,-<-] ($(0,2)+(-65:\x/5)$) coordinate(RT\x) arc(-65:-45:\x/5);
\draw (RT\x) ++(-180+25:0.2) coordinate (RT\x');
}

%\draw[red,thick,dotted] ($(0,-2)+(67.5:0.9)$)--++(-112.5:0.4);
%\draw[red,thick,dotted] ($(0,2)+(-67.5:0.9)$)--++(112.5:0.4);

\draw[red,thick] (LB5) ..controls (LB5') and (RB5')..  (RB5);
\draw[red,thick] (LB6) ..controls (LB6') and (RB6')..  (RB6);
\draw[red,thick] (LT6) ..controls (LT6') and (R1')..  (R1);
\draw[red,thick] (LT5) ..controls (LT5') and (RT6')..  (RT6);
\draw[red,thick] (LT4) ..controls (LT4') and (RT5')..  (RT5);
\draw[red,thick] (L2) ..controls (L2') and (R2')..  (R2);
\draw[red,thick] (L1) ..controls (L1') and (R3')..  (R3);
\end{scope}
{\begin{scope}[yshift=-6cm]

\draw[blue] (-2,0) -- (0,2) -- (0,-2) --cycle;
\node[below] at (0,-2) {$S_L^+$};

\foreach \x in {1.2,1.0}
{
\draw[red,thick,-<-] (0,-2+\x) arc(90:135:\x); 
\draw[red,thick,-<-] (0,2-\x) arc(-90:-135:\x);
}
\draw[red,thick,dotted] ($(0,-2)+(112.5:0.9)$)--++(-67.5:0.4);
\draw[red,thick,dotted] ($(0,2)+(-112.5:0.9)$)--++(67.5:0.4);

\draw(-0.1,-0.5) -- (0.1,-0.5) node[right,scale=0.8]{$0$};
\filldraw(0,-1.2) circle(1pt) node[below right,scale=0.8]{$p_L^+$} coordinate(p1);

\begin{scope}[xshift=1cm]
\draw[blue] (2,0) -- (0,2) -- (0,-2) --cycle;
\draw(-0.1,0.5) -- (0.1,0.5) node[right,scale=0.8]{$0$};
\node[below] at (0,-2) {$S_R^-$};

\foreach \x in {1.2,1.0}
{
\draw[red,thick,->-] (0,-2+\x) arc(90:45:\x); 
\draw[red,thick,->-] (0,2-\x) arc(-90:-45:\x);
}
\draw[red,thick,dotted] ($(0,-2)+(67.5:0.9)$)--++(-112.5:0.4);
\draw[red,thick,dotted] ($(0,2)+(-67.5:0.9)$)--++(112.5:0.4);

\draw(-0.1,0.5) -- (0.1,0.5) node[right,scale=0.8]{$0$};
\filldraw(0,0.5) circle(1pt) node[above left,scale=0.8]{$p_R^-$} coordinate(p2);
\end{scope}
\draw[dashed] (p1) -- (p2);
\draw [thick,-{Classical TikZ Rightarrow[length=4pt]},decorate,decoration={snake,amplitude=1.8pt,pre length=2pt,post length=3pt}](3.7,0) --(5.3,0);

\begin{scope}[xshift=8cm]
\draw[blue] (2,0) -- (0,2) -- (-2,0) -- (0,-2) --cycle;
\draw[blue] (0,-2) to[bend left=30pt] (0,2);
\draw[blue] (0,-2) to[bend right=30pt] (0,2);

\foreach \x in {5,6}
{
\draw[red,thick,-<-] ($(0,-2)+(115:\x/5)$) coordinate(LB\x) arc(115:135:\x/5); 
\draw (LB\x) ++(25:0.2) coordinate (LB\x');
\draw[red,thick,-<-] ($(0,2)+(-115:\x/5)$) coordinate(LT\x) arc(-115:-135:\x/5);
\draw (LT\x) ++(-25:0.2) coordinate (LT\x');
}
\foreach \x in {3,4}
{
\draw[red,thick,dotted] ($(0,-2)+(115:\x/5)$) coordinate(LB\x) arc(115:132:\x/5);
\draw (LB\x) ++(25:0.2) coordinate (LB\x');
}
%right
\foreach \x in {5,6}
{
\draw[red,thick,->-] ($(0,-2)+(65:\x/5)$) coordinate(RB\x) arc(65:45:\x/5); 
\draw (RB\x) ++(180-25:0.2) coordinate (RB\x');
\draw[red,thick,->-] ($(0,2)+(-65:\x/5)$) coordinate(RT\x) arc(-65:-45:\x/5);
\draw (RT\x) ++(-180+25:0.2) coordinate (RT\x');
}
\foreach \x in {3,4}
{
\draw[red,thick,dotted] ($(0,2)+(-75:\x/5)$) coordinate(RT\x) arc(-75:-48:\x/5);
\draw (RT\x) ++(-155:0.2) coordinate (RT\x');
}

\draw[red,thick] (LB5) ..controls (LB5') and (RT6')..  (RT6);
\draw[red,thick] (LB6) ..controls (LB6') and (RT5')..  (RT5);
\draw[red,thick] (LT6) ..controls (LT6') and (RT4')..  (RT4);
\draw[red,thick] (LT5) ..controls (LT5') and (RT3')..  (RT3);
\draw[red,thick] (LB4) ..controls (LB4') and (RB6')..  (RB6);
\draw[red,thick] (LB3) ..controls (LB3') and (RB5')..  (RB5);

\end{scope}
\end{scope}}
\node at (4.5,-3) {$(\sfx_{T_L},\sfx_{T_R},\sfx_{\bott},\sfx_{\topp})=(2,-3,-2,1)$};
\end{tikzpicture}
% \fi%%
    \caption{An example for the case $\sfx_{T_L}\geq 0$ and $\sfx_{T_R}\leq 0$. The case $\sfx_{T_L}\leq 0$ and $\sfx_{T_R}\geq 0$ follows by symmetry (\cref{rem:gluing_symmetry}).}
    \label{fig:reconstruction2}
\end{figure}

\begin{figure}[htbp]
\begin{tikzpicture}

\draw[blue] (-2,0) -- (0,2) -- (0,-2) --cycle;
%\draw[red,thick] (-1.2,0) coordinate(A) --++(0.6,0.6) coordinate(B) --++(0,-1.2) coordinate(C) --cycle;
%\draw[red,thick,-<-] ($(A)!0.6!(B)$) --++(135:0.6);
%\draw[red,thick,-<-] ($(A)!0.4!(B)$) --++(135:0.6);
%\draw[red,thick,-<-] ($(A)!0.6!(C)$) --++(225:0.6);
%\draw[red,thick,-<-] ($(A)!0.4!(C)$) --++(225:0.6);
%\draw[red,thick,-<-] ($(B)!0.6!(C)$) --++(0:0.6);
%\draw[red,thick,-<-] ($(B)!0.4!(C)$) --++(0:0.6);
\node[below] at (0,-2) {$S_L^-$};

\foreach \x in {1.2,1.0}
{
\draw[red,thick,->-] (0,-2+\x) arc(90:135:\x); 
\draw[red,thick,->-] (0,2-\x) arc(-90:-135:\x);
}
\draw[red,thick,dotted] ($(0,-2)+(112.5:0.9)$)--++(-67.5:0.4);
\draw[red,thick,dotted] ($(0,2)+(-112.5:0.9)$)--++(67.5:0.4);

\draw(-0.1,-0.5) node[left,scale=0.8]{$0$} -- (0.1,-0.5);
\filldraw(0,-0.5) circle(1pt) node[below right,scale=0.8]{$p_L^-$} coordinate(p1);

\begin{scope}[xshift=1cm]
\draw[blue] (2,0) -- (0,2) -- (0,-2) --cycle;
\draw[red,thick] (1.2,0) coordinate(A) --++(-0.6,-0.6) coordinate(B) --++(0,+1.2) coordinate(C) --cycle;
\draw[red,thick,->-] ($(A)!0.7!(B)$) --++(-45:0.6);
\draw[red,thick,->-] ($(A)!0.5!(B)$) --++(-45:0.6);
\draw[red,thick,->-] ($(A)!0.3!(B)$) --++(-45:0.6);
\draw[red,thick,->-] ($(A)!0.7!(C)$) --++(45:0.6);
\draw[red,thick,->-] ($(A)!0.5!(C)$) --++(45:0.6);
\draw[red,thick,->-] ($(A)!0.3!(C)$) --++(45:0.6);
\draw[red,thick,->-] ($(B)!0.7!(C)$) --++(180:0.6);
\draw[red,thick,->-] ($(B)!0.5!(C)$) --++(180:0.6);
\draw[red,thick,->-] ($(B)!0.3!(C)$) --++(180:0.6);
\node[below] at (0,-2) {$S_R^+$};

\foreach \x in {1.2,1.0}
{
\draw[red,thick,-<-] (0,-2+\x) arc(90:45:\x); 
\draw[red,thick,-<-] (0,2-\x) arc(-90:-45:\x);
}
\draw[red,thick,dotted] ($(0,-2)+(67.5:0.9)$)--++(-112.5:0.4);
\draw[red,thick,dotted] ($(0,2)+(-67.5:0.9)$)--++(112.5:0.4);

\draw(-0.1,0.5) -- (0.1,0.5) node[right,scale=0.8]{$0$};
\filldraw(0,0.1) circle(1pt) node[above left,scale=0.8]{$p_R^+$} coordinate(p2);
\end{scope}
\draw[dashed] (p1) -- (p2);
\draw [thick,-{Classical TikZ Rightarrow[length=4pt]},decorate,decoration={snake,amplitude=1.8pt,pre length=2pt,post length=3pt}](3.7,0) --(5.3,0);

\begin{scope}[xshift=8cm]
\draw[blue] (2,0) -- (0,2) -- (-2,0) -- (0,-2) --cycle;
\draw[blue] (0,-2) to[bend left=30pt] (0,2);
\draw[blue] (0,-2) to[bend right=30pt] (0,2);

%\draw[red,thick] (-1.4,0) coordinate(A) --++(0.6,0.6) coordinate(B) --++(0,-1.2) coordinate(C) --cycle;
%\draw[red,thick,-<-] ($(A)!0.6!(B)$) --++(135:0.4);
%\draw[red,thick,-<-] ($(A)!0.4!(B)$) --++(135:0.4);
%\draw[red,thick,-<-] ($(A)!0.6!(C)$) --++(225:0.4);
%\draw[red,thick,-<-] ($(A)!0.4!(C)$) --++(225:0.4);
%\draw[red,thick,-<-={0.7}{}] ($(B)!0.6!(C)$) --++(0:0.3) coordinate(L1);
%\draw[red,thick,-<-={0.7}{}] ($(B)!0.4!(C)$) --++(0:0.3) coordinate(L2);
%\draw (L1) ++(0.2,0) coordinate(L1');
%\draw (L2) ++(0.2,0) coordinate(L2');

\draw[red,thick] (1.4,0) coordinate(A) --++(-0.6,-0.6) coordinate(B) --++(0,+1.2) coordinate(C) --cycle;
\draw[red,thick,->-] ($(A)!0.7!(B)$) --++(-45:0.4);
\draw[red,thick,->-] ($(A)!0.5!(B)$) --++(-45:0.4);
\draw[red,thick,->-] ($(A)!0.3!(B)$) --++(-45:0.4);
\draw[red,thick,->-] ($(A)!0.7!(C)$) --++(45:0.4);
\draw[red,thick,->-] ($(A)!0.5!(C)$) --++(45:0.4);
\draw[red,thick,->-] ($(A)!0.3!(C)$) --++(45:0.4);
%left
\draw[red,thick,->-={0.7}{}] ($(B)!0.7!(C)$) --++(180:0.3) coordinate(R1);
\draw[red,thick,->-={0.7}{}] ($(B)!0.5!(C)$) --++(180:0.3) coordinate(R2);
\draw[red,thick,->-={0.7}{}] ($(B)!0.3!(C)$) --++(180:0.3) coordinate(R3);
\draw (R1) ++(-0.2,0) coordinate(R1');
\draw (R2) ++(-0.2,0) coordinate(R2');
\draw (R3) ++(-0.2,0) coordinate(R3');
\foreach \x in {5,6}
{
\draw[red,thick,->-] ($(0,-2)+(115:\x/5)$) coordinate(LB\x) arc(115:135:\x/5); 
\draw (LB\x) ++(25:0.2) coordinate (LB\x');
\draw[red,thick,->-] ($(0,2)+(-115:\x/5)$) coordinate(LT\x) arc(-115:-135:\x/5);
\draw (LT\x) ++(-25:0.2) coordinate (LT\x');
}
\foreach \x in {4}
{
\draw[red,thick,dotted] ($(0,2)+(-115:\x/5)$) coordinate(LT\x) arc(-115:-132:\x/5);
\draw (LT\x) ++(-25:0.2) coordinate (LT\x');
}
\foreach \x in {3,4}
{
\draw[red,thick,dotted] ($(0,-2)+(115:\x/5)$) coordinate(LB\x) arc(115:132:\x/5); 
\draw (LB\x) ++(25:0.2) coordinate (LB\x');
}

%right
\foreach \x in {5,6}
{
\draw[red,thick,-<-] ($(0,-2)+(65:\x/5)$) coordinate(RB\x) arc(65:45:\x/5); 
\draw (RB\x) ++(180-25:0.2) coordinate (RB\x');
\draw[red,thick,-<-] ($(0,2)+(-65:\x/5)$) coordinate(RT\x) arc(-65:-45:\x/5);
\draw (RT\x) ++(-180+25:0.2) coordinate (RT\x');
}

%\draw[red,thick,dotted] ($(0,-2)+(67.5:0.9)$)--++(-112.5:0.4);
%\draw[red,thick,dotted] ($(0,2)+(-67.5:0.9)$)--++(112.5:0.4);

\draw[red,thick] (LB5) ..controls (LB5') and (R3')..  (R3);
\draw[red,thick] (LB6) ..controls (LB6') and (R2')..  (R2);
\draw[red,thick] (LT6) ..controls (LT6') and (R1')..  (R1);
\draw[red,thick] (LT5) ..controls (LT5') and (RT6')..  (RT6);
\draw[red,thick] (LT4) ..controls (LT4') and (RT5')..  (RT5);
\draw[red,thick] (LB4) ..controls (LB4') and (RB6')..  (RB6);
\draw[red,thick] (LB3) ..controls (LB3') and (RB5')..  (RB5);
\end{scope}

{\begin{scope}[yshift=-6cm]
\draw[blue] (-2,0) -- (0,2) -- (0,-2) --cycle;
\draw[red,thick] (-1.2,0) coordinate(A) --++(0.6,0.6) coordinate(B) --++(0,-1.2) coordinate(C) --cycle;
\draw[red,thick,->-] ($(A)!0.6!(B)$) --++(135:0.6);
\draw[red,thick,->-] ($(A)!0.4!(B)$) --++(135:0.6);
\draw[red,thick,->-] ($(A)!0.6!(C)$) --++(225:0.6);
\draw[red,thick,->-] ($(A)!0.4!(C)$) --++(225:0.6);
\draw[red,thick,->-] ($(B)!0.6!(C)$) --++(0:0.6);
\draw[red,thick,->-] ($(B)!0.4!(C)$) --++(0:0.6);
\node[below] at (0,-2) {$S_L^+$};

\foreach \x in {1.2,1.0}
{
\draw[red,thick,-<-] (0,-2+\x) arc(90:135:\x); 
\draw[red,thick,-<-] (0,2-\x) arc(-90:-135:\x);
}
\draw[red,thick,dotted] ($(0,-2)+(112.5:0.9)$)--++(-67.5:0.4);
\draw[red,thick,dotted] ($(0,2)+(-112.5:0.9)$)--++(67.5:0.4);

\draw(-0.1,-0.5) -- (0.1,-0.5) node[right,scale=0.8]{$0$};
\filldraw(0,-1.2) circle(1pt) node[below right,scale=0.8]{$p_L^+$} coordinate(p1);

\begin{scope}[xshift=1cm]
\draw[blue] (2,0) -- (0,2) -- (0,-2) --cycle;
\node[below] at (0,-2) {$S_R^-$};

\foreach \x in {1.2,1.0}
{
\draw[red,thick,->-] (0,-2+\x) arc(90:45:\x); 
\draw[red,thick,->-] (0,2-\x) arc(-90:-45:\x);
}
\draw[red,thick,dotted] ($(0,-2)+(67.5:0.9)$)--++(-112.5:0.4);
\draw[red,thick,dotted] ($(0,2)+(-67.5:0.9)$)--++(112.5:0.4);

\draw(-0.1,0.5) -- (0.1,0.5) node[right,scale=0.8]{$0$};
\filldraw(0,0.5) circle(1pt) node[above left,scale=0.8]{$p_R^-$} coordinate(p2);
\end{scope}
\draw[dashed] (p1) -- (p2);
\draw [thick,-{Classical TikZ Rightarrow[length=4pt]},decorate,decoration={snake,amplitude=1.8pt,pre length=2pt,post length=3pt}](3.7,0) --(5.3,-0);

\begin{scope}[xshift=8cm]
\draw[blue] (2,0) -- (0,2) -- (-2,0) -- (0,-2) --cycle;
\draw[blue] (0,-2) to[bend left=30pt] (0,2);
\draw[blue] (0,-2) to[bend right=30pt] (0,2);

\draw[red,thick] (-1.4,0) coordinate(A) --++(0.6,0.6) coordinate(B) --++(0,-1.2) coordinate(C) --cycle;
\draw[red,thick,->-] ($(A)!0.6!(B)$) --++(135:0.4);
\draw[red,thick,->-] ($(A)!0.4!(B)$) --++(135:0.4);
\draw[red,thick,->-] ($(A)!0.6!(C)$) --++(225:0.4);
\draw[red,thick,->-] ($(A)!0.4!(C)$) --++(225:0.4);
\draw[red,thick,->-={0.7}{}] ($(B)!0.6!(C)$) --++(0:0.3) coordinate(L1);
\draw[red,thick,->-={0.7}{}] ($(B)!0.4!(C)$) --++(0:0.3) coordinate(L2);
\draw (L1) ++(0.2,0) coordinate(L1');
\draw (L2) ++(0.2,0) coordinate(L2');

\foreach \x in {5,6}
{
\draw[red,thick,-<-] ($(0,-2)+(115:\x/5)$) coordinate(LB\x) arc(115:135:\x/5); 
\draw (LB\x) ++(25:0.2) coordinate (LB\x');
\draw[red,thick,-<-] ($(0,2)+(-115:\x/5)$) coordinate(LT\x) arc(-115:-135:\x/5);
\draw (LT\x) ++(-25:0.2) coordinate (LT\x');
}
\foreach \x in {3,4}
{
\draw[red,thick,dotted] ($(0,-2)+(115:\x/5)$) coordinate(LB\x) arc(115:132:\x/5); 
\draw (LB\x) ++(25:0.2) coordinate (LB\x');
}

%right
\foreach \x in {5,6}
{
\draw[red,thick,->-] ($(0,-2)+(65:\x/5)$) coordinate(RB\x) arc(65:45:\x/5); 
\draw (RB\x) ++(180-25:0.2) coordinate (RB\x');
\draw[red,thick,->-] ($(0,2)+(-65:\x/5)$) coordinate(RT\x) arc(-65:-45:\x/5);
\draw (RT\x) ++(-180+25:0.2) coordinate (RT\x');
}
\foreach \x in {2,5/2,3,4}
{
\draw[red,thick,dotted] ($(0,2)+(-65:\x/5)$) coordinate(RT\x) arc(-65:-48:\x/5);
\draw (RT\x) ++(180+25:0.2) coordinate (RT\x');
}

%\draw[red,thick,dotted] ($(0,-2)+(67.5:0.9)$)--++(-112.5:0.4);
%\draw[red,thick,dotted] ($(0,2)+(-67.5:0.9)$)--++(112.5:0.4);

\draw[red,thick] (LB5) ..controls (LB5') and (RT6')..  (RT6);
\draw[red,thick] (LB6) ..controls (LB6') and (RT5')..  (RT5);
\draw[red,thick] (LT6) ..controls (LT6') and (RT5/2')..  (RT5/2);
\draw[red,thick] (LT5) ..controls (LT5') and (RT2')..  (RT2);
\draw[red,thick] (L1) ..controls (L1') and (RT4')..  (RT4);
\draw[red,thick] (L2) ..controls (L2') and (RT3')..  (RT3);
\draw[red,thick] (LB4) ..controls (LB4') and (RB6')..  (RB6);
\draw[red,thick] (LB3) ..controls (LB3') and (RB5')..  (RB5);
\end{scope}
\end{scope}}
\node at (4.5,-3) {$(\sfx_{T_L},\sfx_{T_R},\sfx_{\bott},\sfx_{\topp})=(-2,-3,-2,1)$};
\end{tikzpicture}
    \caption{An example for the case $\sfx_{T_L}\leq 0$ and $\sfx_{T_R}\leq 0$.}
    \label{fig:reconstruction3}
\end{figure}

The collection $\cW_{\mathrm{br}}^\tri((\sfx_i)_i)$ is the braid representative of the spiralling diagram associated to an unbounded integral $\fsl_3$-lamination $\xi_\tri((\sfx_i)_i)$, which is obtained as follows:
\begin{description}
    \item[Step 1] First remove the peripheral components around the marked points (both special points and punctures) from $\cW_{\mathrm{br}}^\tri((\sfx_i)_i)$. Then, remaining are finitely many components. 
    \item[Step 2] Replace each spiralling end around a puncture $p$ with an end incident to $p$, while encoding the spiralling directions in signs by reversing the rule in \cref{fig:spiral}. Then we get a collection $W_\mathrm{br}^\tri((\sfx_i)_i)$ of signed webs, which we call a \emph{braid representative} of a signed web. 
    It contains at most finitely many intersections of curves only in biangles. Here we can rearrange $W_\mathrm{br}^\tri((\sfx_i)_i)$ so that no pair of curves form a bigon by applying a Reidemeister II-type isotopy if necessary (cf.~\emph{square removing algorithm} in \cite{DS20I}). See \cref{fig:square_resolution}. Observe that this operation does not affect the shear coordinates. 
    \item[Step 3] Replace each intersection of curves in a biangle with an H-web by the rule \eqref{eq:H-replacement}. Then we get a signed $\fsl_3$-web $W$ on $\Sigma$, which has no elliptic faces. Indeed, we have no 0-gon nor 2-gon faces by construction, and possible emergence of 4-gon faces has been eliminated in Step 2.
    %\item[Step 4] Finally, resolve the 4-gon faces  following the \emph{square removing algorithm} in \cite{DS20I}\footnote{
    %From the skein point of view, we choose one term in the right-hand side of the 4-gon removal relation (see \eqref{eq:skein:4-gon_rel}) which does not decrease the intersection number with $\widehat{\tri}$.}. See \cref{fig:square_resolution}. 
    %Since there are finitely many such 4-gons, it follows directly from the proof of \cite[Lemma 52]{DS20I} that we eventually get a non-elliptic signed web $W$ whose equivalence class does not depend on the choices of 4-gons and its order which we resolve. 
    \end{description}
Then $\xi_\tri((\sfx_i)_i) \in \cL^x(\Sigma,\bZ)$ is
defined to be the unbounded integral $\fsl_3$-lamination represented by the non-elliptic signed web $W$ (with weight $1$ on each component). 
Set $\xi_\tri((\widetilde{\sfx}_i)_i):=u^{-1}\cdot\xi_\tri((\sfx_i)_i) \in \cL^x(\Sigma,\bQ)$. 

Thus we get the map $\xi_\tri: \bQ^{I_\uf(\tri)} \to \cL^x(\Sigma,\bQ)$, which is clearly $\bQ_{>0}$-equivariant. We are going to show that this map indeed gives the inverse map of $\sfx^\uf_\tri$. The following direction is easier:

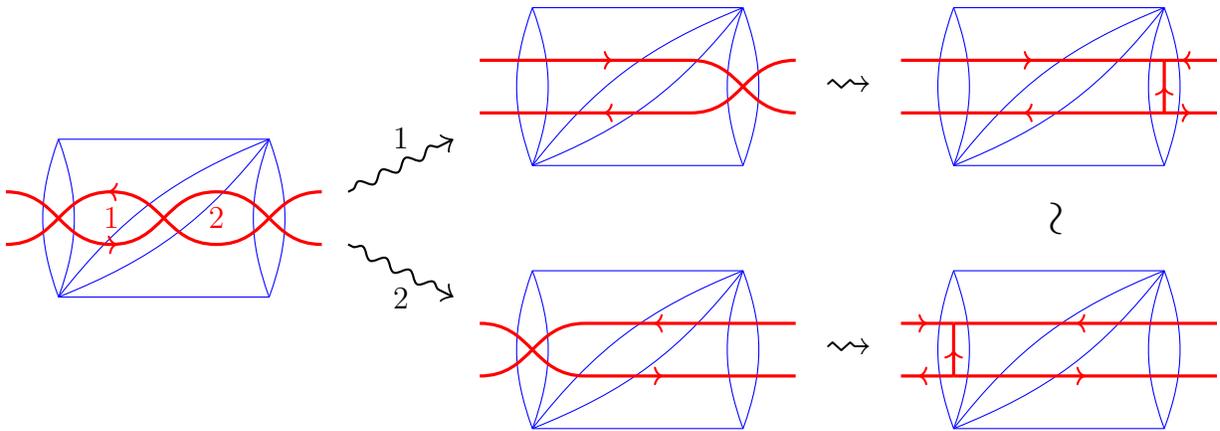
\begin{figure}[htbp]
\begin{tikzpicture}[scale=.58]
\draw[blue] (-2,-1.5) to[bend left=20pt] (-2,1.5) -- (2,1.5) to[bend left=20pt] (2,-1.5) --cycle;
\draw[blue] (-2,-1.5) to[bend right=20pt] (-2,1.5);
\draw[blue] (2,-1.5) to[bend left=20pt] (2,1.5);
\draw[blue] (-2,-1.5) to[bend left=15pt] (2,1.5);
\draw[blue] (-2,-1.5) to[bend right=15pt] (2,1.5);

\node[red] at (-1,0) {$1$};
\node[red] at (1,0) {$2$};
\draw [red, very thick, ->-={0.35}{}](-3,0.5) .. controls (-2,0.5) and (-2,-0.5) .. (-1,-0.5) .. controls (0,-0.5) and (0,0.5) .. (1,0.5) .. controls (2,0.5) and (2,-0.5) .. (3,-0.5);
\draw [red, very thick, -<-={0.35}{}](-3,-0.5) .. controls (-2,-0.5) and (-2,0.5) .. (-1,0.5) .. controls (0,0.5) and (0,-0.5) .. (1,-0.5) .. controls (2,-0.5) and (2,0.5) .. (3,0.5);

\draw [thick,-{Classical TikZ Rightarrow[length=4pt]},decorate,decoration={snake,amplitude=1.8pt,pre length=2pt,post length=3pt}](3.5,0.5) --node[midway,above=0.2em]{$1$} (5.5,1.5);
\draw [thick,-{Classical TikZ Rightarrow[length=4pt]},decorate,decoration={snake,amplitude=1.8pt,pre length=2pt,post length=3pt}](3.5,-0.5) --node[midway,below=0.2em]{$2$} (5.5,-1.5);

\begin{scope}[xshift=9cm,yshift=2.5cm]
\draw[blue] (-2,-1.5) to[bend left=20pt] (-2,1.5) -- (2,1.5) to[bend left=20pt] (2,-1.5) --cycle;
\draw[blue] (-2,-1.5) to[bend right=20pt] (-2,1.5);
\draw[blue] (2,-1.5) to[bend left=20pt] (2,1.5);
\draw[blue] (-2,-1.5) to[bend left=15pt] (2,1.5);
\draw[blue] (-2,-1.5) to[bend right=15pt] (2,1.5);
\draw [red, very thick, ->-={0.4}{}](-3,0.5) .. controls (-2,0.5)  and (0,0.5) .. (1,0.5) .. controls (2,0.5) and (2,-0.5) .. (3,-0.5);
\draw [red, very thick, -<-={0.4}{}](-3,-0.5) .. controls (-2,-0.5) and (0,-0.5) .. (1,-0.5) .. controls (2,-0.5) and (2,0.5) .. (3,0.5);
\end{scope}

\node[scale=1.5] at (13,2.5) {$\rightsquigarrow$};

\begin{scope}[xshift=17cm,yshift=2.5cm]
\draw[blue] (-2,-1.5) to[bend left=20pt] (-2,1.5) -- (2,1.5) to[bend left=20pt] (2,-1.5) --cycle;
\draw[blue] (-2,-1.5) to[bend right=20pt] (-2,1.5);
\draw[blue] (2,-1.5) to[bend left=20pt] (2,1.5);
\draw[blue] (-2,-1.5) to[bend left=15pt] (2,1.5);
\draw[blue] (-2,-1.5) to[bend right=15pt] (2,1.5);
\draw[red,very thick,->-] (-3,0.5) -- (2,0.5);
\draw[red,very thick,-<-] (2,0.5) -- (3,0.5);
\draw[red,very thick,-<-] (-3,-0.5) -- (2,-0.5);
\draw[red,very thick,->-] (2,-0.5) -- (3,-0.5);
\draw[red,very thick,->-] (2,-0.5) -- (2,0.5);
\end{scope}

\begin{scope}[xshift=9cm,yshift=-2.5cm]
\draw[blue] (-2,-1.5) to[bend left=20pt] (-2,1.5) -- (2,1.5) to[bend left=20pt] (2,-1.5) --cycle;
\draw[blue] (-2,-1.5) to[bend right=20pt] (-2,1.5);
\draw[blue] (2,-1.5) to[bend left=20pt] (2,1.5);
\draw[blue] (-2,-1.5) to[bend left=15pt] (2,1.5);
\draw[blue] (-2,-1.5) to[bend right=15pt] (2,1.5);
\draw [red, very thick, ->-={0.6}{}](-3,0.5) .. controls (-2,0.5) and (-2,-0.5) .. (-1,-0.5) .. controls (0,-0.5) and (2,-0.5) .. (3,-0.5);
\draw [red, very thick, -<-={0.6}{}](-3,-0.5) .. controls (-2,-0.5) and (-2,0.5) .. (-1,0.5) .. controls (0,0.5) and (2,0.5) .. (3,0.5);
\end{scope}

\node[scale=1.5, rotate=90] at (17, 0) {$\sim$};

\node[scale=1.5] at (13,-2.5) {$\rightsquigarrow$};

\begin{scope}[xshift=17cm,yshift=-2.5cm]
\draw[blue] (-2,-1.5) to[bend left=20pt] (-2,1.5) -- (2,1.5) to[bend left=20pt] (2,-1.5) --cycle;
\draw[blue] (-2,-1.5) to[bend right=20pt] (-2,1.5);
\draw[blue] (2,-1.5) to[bend left=20pt] (2,1.5);
\draw[blue] (-2,-1.5) to[bend left=15pt] (2,1.5);
\draw[blue] (-2,-1.5) to[bend right=15pt] (2,1.5);
\draw[red,very thick,->-] (-3,0.5) -- (-2,0.5);
\draw[red,very thick,-<-] (-2,0.5) -- (3,0.5);

\draw[red,very thick,-<-] (-3,-0.5) -- (-2,-0.5);
\draw[red,very thick,->-] (-2,-0.5) -- (3,-0.5);

\draw[red,very thick,->-] (-2,-0.5) -- (-2,0.5);
\end{scope}
\end{tikzpicture}
    % \caption{Resolution of 4-gons. We have two ways of resolutions, which produce parallel-equivalent webs.}
    \caption{Reidemeister II-type isotopy. We have two ways of applications of this isotopy, which produce equivalent webs.}
    \label{fig:square_resolution}
\end{figure}

\begin{prop}\label{prop:surjectivity}
We have $\sfx^\uf_\tri \circ \xi_\tri=\mathrm{id}_{\bQ^{I_\uf(\tri)}}$. 
\end{prop}

\begin{proof}
By the $\bQ_{>0}$-equivariance, it suffices to consider an integral tuple $(\sfx_i)_i \in \bZ^{I_\uf(\tri)}$. Notice that by construction, the collection $\cW_{\mathrm{br}}^\tri((\sfx_i)_i)$ arising from the gluing construction above is exactly the braid representative of the spiralling diagram associated with the underlying signed web of the $\fsl_3$-lamination $\hL:=\xi_\tri((\sfx_i)_i) \in \cL^x(\Sigma,\bZ)$. Therefore the shear coordinates $(\sfx^\tri_i(\hL))$ can be directly read off from the collection $\cW_{\mathrm{br}}^\tri((\sfx_i)_i)$. Hence the assertion follows from \cref{lem:braid-shear}.
\end{proof}

\begin{thm}[Proof in \cref{sec:traveler}]\label{prop:injectivity}
We have $\xi_\tri \circ \sfx^\uf_\tri=\mathrm{id}_{\cL^x(\Sigma,\bQ)}$. In particular, the shear coordinates gives a bijection $\xi_\tri: \bQ^{I_\uf(\tri)} \xrightarrow{\sim} \cL^x(\Sigma,\bQ)$.
\end{thm}
See \cref{sec:traveler} for a proof. 
The main ingredient of the proof is an unbounded version of the Fellow-Traveler Lemma (\cite[Lemma 57]{DS20I}) with respect to the shear coordinates. 

Recall from \cref{subsec:cluster_sl3} that the ideal triangulations $\tri$ correspond to certain seeds in the mutation class $\sfs(\Sigma, \fsl_3)$. The following theorem states that the associated shear coordinate systems $\sfx^\uf_\tri$ are related by tropical cluster Poisson transformations:
%The choice of the ideal triangulations $\tri$ for the map $\sfx^\uf_\tri$ is nothing but the choice of the specific seeds for the mutation class $\sfs(\Sigma, \fsl_3)$:

\begin{thm}\label{thm:cluster_transf_unfrozen}
For any two ideal triangulations $\tri$ and $\tri'$ of $\Sigma$, the coordinate transformation $\sfx^\uf_{\tri'}\circ (\sfx^\uf_\tri)^{-1}: \bQ^{I_\uf(\tri)} \to \bQ^{I_\uf(\tri)}$ is a composite of tropical cluster Poisson transformations. In particular, we get an $MC(\Sigma)$-equivariant identification $\sfx^\uf_\bullet:\cL^x(\Sigma,\bQ) \xrightarrow{\sim} \X_{\fsl_3,\Sigma}^\uf(\bQ^T)$. 
\end{thm}
Since it is classically known that any two ideal triangulations of the same marked surface can be connected by a finite sequence of flips, it suffices to show that a flip corresponds to a composite of tropical cluster Poisson transformations. Although it can be directly checked in a similar way to \cite[Section 4]{DS20II}, we are going to reduce it to the Douglas--Sun's result via the ensemble map and the gluing technique developed in \cref{sec:amalgamation}.

\subsection{Relation to the rational unbounded $\fsl_2$-laminations}\label{subsec:principal}
Recall the space $\mathcal{L}_{\fsl_2}^x(\Sigma,\bQ)$ of rational unbounded ($\fsl_2$-)laminations from \cite{FG07}. It consists of the following data:
\begin{itemize}
    \item A collection of immersed unoriented loops and arcs such that each endpoint lies in $\bP \cup \partial^\ast \Sigma$, and the other part is embedded into $\interior \Sigma$. It is required to have no elliptic faces (the first one in \eqref{eq:elliptic face} nor the first and last ones in \eqref{eq:elliptic_face_peripheral}).
    \item A positive rational weight on each component.
    \item A sign $\sigma_p \in \{+,0,-\}$ for each puncture $p \in \bP$ such that $\sigma_p=0$ if and only if there are no component incident to $p$.
\end{itemize}
They are considered modulo removal/creation of peripheral components as in \eqref{eq:peripheral}, and the weighted isotopy as in \cref{def:unbounded laminations} (2). Given an ideal triangulation $\tri$ of $\Sigma$, the ($\fsl_2$-)shear coordinate $\sfx_\tri=(\sfx_E^\tri)_{E\in e(\tri)}: \mathcal{L}_{\fsl_2}^x(\Sigma,\bQ) \xrightarrow{\sim} \bQ^{e(\tri)}$ (cf.~\cite{FG07}) is defined by first constructing a spiralling diagram according to the sign $\sigma_p$, and counting the following contributions with weights from the curves in that diagram: 

\begin{figure}[h]
    \centering
\begin{tikzpicture}
\draw[blue] (2,0) -- (0,2) -- (-2,0) -- (0,-2) --cycle;
\draw[blue] (0,-2) to (0,2);
\draw[red,thick] (-1.2,1) -- (1,-1.2);
{\color{mygreen}
\draw (0,0) circle(2pt) node[right]{$+1$};
}

\begin{scope}[xshift=6cm]
\draw[blue] (2,0) -- (0,2) -- (-2,0) -- (0,-2) --cycle;
\draw[blue] (0,-2) to (0,2);
\draw[red,thick] (-1.2,-1) -- (1,1.2);
{\color{mygreen}
\draw (0,0) circle(2pt) node[right]{$-1$};
}
\end{scope}
\end{tikzpicture}
    \caption{Contributions to the $\fsl_2$-shear coordinates.}
    \label{fig:shear_curve_sl2}
\end{figure}
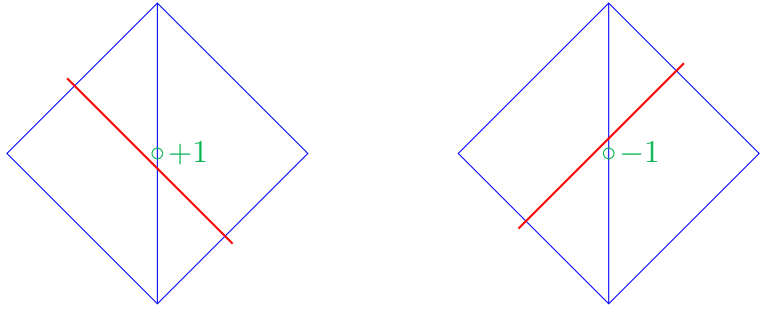

An embedding $\iota_\mathrm{prin}: \mathcal{L}_{\fsl_2}^x(\Sigma,\bQ) \to \cL^x(\Sigma,\bQ)$ is defined so that
\begin{itemize}
    \item each curve $\gamma$ with weight $u \in \bQ_{>0}$ is sent to its parallel copies $\gamma_1,\gamma_2$ with the same weight $u$ with the opposite orientations;
    \item if an arc $\gamma$ is incident to a puncture $p$, then the corresponding ends of the oriented curves $\gamma_1,\gamma_2$ are assigned the sign $\sigma_p \in \{+,-\}$.
\end{itemize}
One can easily verify that it is indeed well-defined. We call $\iota_{\mathrm{prin}}$ the \emph{principal embedding}, as it is a tropical analogue of the morphism $\X_{SL_2,\Sigma} \to \X_{SL_3,\Sigma}$ induced by the principal embedding $\fsl_2 \to \fsl_3$. The following is a tropical analogue of the statement given in \cite[Section 2.5.3]{FG07c}:

\begin{prop}\label{prop:principal_embedding}
The image $\iota_{\mathrm{prin}}(\mathcal{L}_{\fsl_2}^x(\Sigma,\bQ))$ coincides with the fixed point locus of the Dynkin involution $\ast$ (\cref{def:Dynkin_geometric}). In the shear coordinate system $\sfx_\tri$ associated with any ideal triangulation $\tri$, it is characterized by the equations
\begin{align*}
    &\sfx_{E,1}^\tri=\sfx_{E,2}^\tri \hspace{-2cm} &\hspace{-2cm} \mbox{for each $E \in e(\tri)$}, \\
    &\sfx_T^\tri=0 \hspace{-2cm} &\hspace{-2cm} \mbox{for each $T \in t(\tri)$}.
\end{align*}
\end{prop}

\begin{proof}
The first assertion follows from the second one, by \cref{prop:Dynkin-cluster} below. The second assertion is easily verified by comparing the definitions of $\fsl_2$- and $\fsl_3$-shear coordinates. Indeed, we have $\sfx_E^\tri(\hL)=\sfx_{E,1}^\tri(\iota_{\mathrm{prin}}(\hL))=\sfx_{E,2}^\tri(\iota_{\mathrm{prin}}(\hL))$ and $\sfx_T^\tri(\iota_\mathrm{prin}(\hL))=0$, where $(\sfx_E^\tri)_{E \in e(\tri)}$ denotes the $\fsl_2$-shear coordinate system. 
\end{proof}

\section{Rational \texorpdfstring{$\P$}{P}-laminations, their gluing and the mutation equivariance}\label{sec:amalgamation}
In this section, we introduce the space of \emph{rational $\P$-laminations} by considering some additional data on boundary intervals and define a coordinate system $\sfx_\tri$ extending $\sfx^\uf_\tri$. These additional data allow us to introduce the \emph{gluing map} between these spaces. 
Under this extended situation, we discuss the relation to the Douglas--Sun's tropical $\A$-coordinates \cite{DS20I}, and prove that the coordinates $\sfx_\tri$ transform correctly under flips. 

\subsection{Rational unbounded \texorpdfstring{$\fsl_3$}{sl(3)}-laminations with pinnings}
It has been stated that the space $\cL^x(\Sigma,\bQ)$ of rational unbounded $\fsl_3$-laminations is identified with the unfrozen part $\X_{\fsl_3,\Sigma}^\uf(\bQ^T)$ of the tropical cluster $\X$-variety. 
In order to obtain the entire tropical cluster $\X$-variety, we further equip the rational laminations with additional data on boundary intervals. Let $\mathsf{P}^\vee=\bZ \varpi_1^\vee \oplus \bZ \varpi_2^\vee$ be the coweight lattice of $\fsl_3$, and $\mathsf{P}^\vee_\bQ:= \mathsf{P}^\vee \otimes \bQ$. Let us consider the direct sum
\begin{align*}
    H_\partial(\bQ^T):=\bigoplus_{E \in \mathbb{B}} \mathsf{P}^\vee_\bQ
\end{align*}
of the coweight lattices over $\bQ$, one for each boundary interval.

\begin{dfn}[rational unbounded $\fsl_3$-laminations with pinnings]
We introduce the space
\begin{align*}
    \cL^p(\Sigma,\bQ):=\cL^x(\Sigma,\bQ) \times H_\partial(\bQ^T),
\end{align*}
and call its elements \emph{rational unbounded $\fsl_3$-laminations with pinnings} (or \emph{rational ($\fsl_3$-)$\cP$-laminations}). The datum in the second factor is written as $\nu=(\nu_E)_{E \in \mathbb{B}}$ with $\nu_E= \nu_E^+\varpi_1^\vee + \nu_E^-\varpi_2^\vee$, $\nu_E^\pm \in \bQ$.
\end{dfn}
The data $\nu=(\nu_E)_E$ will be related to the pinning in the sense of \cref{def:pinning} when we consider their gluings, thus the terminology. 
We have a natural $\bQ_{>0}$-action on $\cL^p(\Sigma,\bQ)$ given by $u.(\hL,\nu):=(u.\hL,(u\nu_E)_E)$ 
for $u \in \bQ_{>0}$ and $(\hL,\nu = (\nu_E)_E) \in \cL^p(\Sigma,\bQ)$. The Dynkin involution (\cref{def:Dynkin_geometric}) is extended as
\begin{align}\label{eq:Dynkin_pinning}
    \ast: \cL^p(\Sigma,\bQ) \to \cL^p(\Sigma,\bQ), \quad (\hL,(\nu_E)_{E \in \bB}) \mapsto (\hL^\ast,(\nu^\ast_E)_{E \in \bB}),
\end{align}
where $\nu^\ast=(\nu_E^\ast)_{E \in \mathbb{B}}$ is obtained from $\nu$ by the Dynkin involution on the coweight lattice: $\varpi^\ast_s:=\varpi_{3-s}$ for $s=1,2$. 
There is a projection 
\begin{align*}
    \pi_\uf: \cL^p(\Sigma,\bQ) \to \cL^x(\Sigma,\bQ)
\end{align*}
forgetting the second factor, which is equivariant under these structures. 
%It will be identified with the projection $\X_{\mathfrak{sl}_3,\Sigma}(\bQ^T) \to \X_{\mathfrak{sl}_3,\Sigma}^\uf(\bQ^T)$ forgetting the frozen coordinates. 
A rational $\P$-lamination $(\hL,\nu)$ is said to be \emph{integral} if $\hL \in \cL^x(\Sigma,\bZ)$ and $p_E \in \mathsf{P}^\vee$ for all $E \in \mathbb{B}$. 
%It is said to be \emph{dominant} if it moreover satisfies $\nu_E \in \mathsf{P}_+^\vee:= \bZ_+\varpi_1^\vee + \bZ_+ \varpi_2^\vee$ for all $E \in \mathbb{B}$. 

\begin{rem}
The space $\cL^p(\Sigma,\bQ)$ is introduced as a tropical analogue of the moduli space $\P_{PGL_3,\Sigma}$ of framed $PGL_3$-local systems with pinnings on $\Sigma$ \cite{GS19}. We have a dominant morphism $\P_{PGL_3,\Sigma} \to \X_{PGL_3,\Sigma}$, which is a principal $H_\partial:=H^{\mathbb{B}}$-bundle over its image. Here $H \subset PGL_3$ denote the Cartan subgroup. 
As a tropical analogue, we may naturally consider the bundle 
\begin{align}\label{eq:bundle_structure}
    0 \to H_\partial(\bQ^T) \to \P_{PGL_3,\Sigma}(\bQ^T) \to \X_{PGL_3,\Sigma}(\bQ^T) \to 0.
\end{align} 
The space $\cL^p(\Sigma,\bQ)$ is regarded as the total space $\P_{PGL_3,\Sigma}(\bQ^T)$ with a fixed trivialization. See also \cref{rem:natural_definition} below. 
\end{rem}

\paragraph{\textbf{Shear coordinates on $\cL^p(\Sigma,\bQ)$}}
Given an ideal triangulation $\tri$ of $\Sigma$, we are going to define a shear coordinate system
\begin{align*}
    \sfx_\tri=(\sfx^\tri_i)_{i \in I(\tri)}: \cL^p(\Sigma,\bQ) \to \bQ^{I(\tri)}
\end{align*}
which extends $\sfx_\tri^\uf$ on $\cL^x(\Sigma,\bQ)$. 
%we have studied in the previous section (denoted by the same symbol). 
For $(\hL,\nu) \in \cL^p(\Sigma,\bQ)$ and an unfrozen index $i \in I_\uf(\tri)$, let $\sfx^\tri_i(\hL,\nu):=\sfx^\tri_i(\hL)$ be the shear coordinate of the underlying rational unbounded lamination. 

We define the frozen coordinate $\sfx^\tri_{E,s}(\hL,\nu)$ for $s=1,2$ associated to a boundary interval $E \in \bB$, as follows. Let $W$ be a non-elliptic signed $\bQ_{>0}$-weighted web without peripheral components representing $\hL$, and $\cW$ its spiralling diagram in a good position with respect to the split triangulation $\widehat{\tri}$. 
%(in fact, the spiralling is irrelevant here since we focus on boundary intervals). 
By convention, $E$ is endowed with the orientation induced from $\partial\Sigma$. Then $\sfx_{E,1}^\tri$ (resp. $\sfx_{E,2}^\tri$) is assigned to the vertex of the $\fsl_3$-triangulation on $E$ closer to the initial (resp. terminal) endpoint.  
Let $m \in \bM_\partial$ be the initial endpoint of $E$, and $T \in t(\tri)$ the unique triangle having $E$ as an edge.  
Let
$\alpha_E^+(\hL)$ (resp. $\alpha_E^-(\hL)$) be the total weight of the oriented corner arcs in $\cW \cap T$ bounding the special point $m$ in the clockwise (resp. counter-clockwise) direction, hence incoming to (resp. outgoing from) the external biangle $B_E$ if we consider the split triangulation $\widehat{\tri}$. See \cref{fig:boundary_coordinate}. 
%Let $\sigma_T(\hL):=\max\{0,\sfx_T(\hL)\}$ the weight of the sink in $W \cap T$. 
Then we define
\begin{align}
\begin{aligned}\label{eq:frozen_coordinates}
    \sfx_{E,1}^\tri(\hL,\nu)&:=\nu_E^+ - \alpha_E^+(\hL), \\
    \sfx_{E,2}^\tri(\hL,\nu)&:=\nu_E^-  - \alpha_E^-(\hL) - [\sfx_T(\hL)]_+.
\end{aligned}
\end{align}

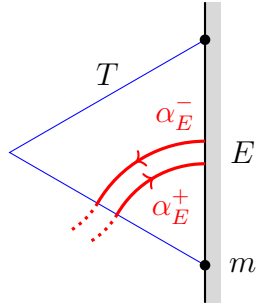
\begin{figure}[h]
\centering
\begin{tikzpicture}
\draw[blue] (0,0) coordinate(A) -- (0,3) coordinate(B) --++(-150:3) coordinate(C) --cycle;
\fill[gray!30] (0,-0.5) -- (0,3.5) -- (0.2,3.5) -- (0.2,-0.5) --cycle;
\draw[thick] (0,-0.5) -- (0,3.5);
\fill(0,0) circle(2pt);
\fill(0,3) circle(2pt);
\draw[red,very thick,-<-] ($(A)!0.55!(C)$) arc(150:90:3*0.55);
\draw[red,very thick,->-] ($(A)!0.45!(C)$) arc(150:90:3*0.45);
\draw[red,very thick,dotted] ($(A)!0.55!(C)$) arc(-30:-60:1);
\draw[red,very thick,dotted] ($(A)!0.45!(C)$) arc(-30:-60:1);
\node at (0.5,1.5) {$E$};
\node at (0.5,0) {$m$};
\node[above] at ($(B)!0.5!(C)$) {$T$};
\node[red,above left] at ($(A)!0.55!(B)$) {$\alpha_E^-$};
\node[red] at (120:3*0.3) {$\alpha_E^+$};
\end{tikzpicture}
    \caption{The corner arcs relevant to the boundary shear coordinate.}
    \label{fig:boundary_coordinate}
\end{figure}

\begin{prop}\label{prop:shear_P}
The shear coordinate system gives a bijection $\sfx_\tri: \cL^p(\Sigma,\bQ) \xrightarrow{\sim} \bQ^{I(\tri)}$.
\end{prop}

\begin{proof}
Given $(\sfx_i)_{i \in I(\tri)} \in \bQ^{I(\tri)}$,
we can reconstruct the underlying rational unbounded lamination $\hL$ from the unfrozen part $(\sfx_i)_{i \in I(\tri)_\uf}$ as in \cref{subsec:reconstruction}. Then the datum $\nu$ is uniquely determined by the relation \eqref{eq:frozen_coordinates}.
\end{proof}
The following is immediate from the definition:

\begin{lem}\label{lem:cluster_projection}
The map $\pi_\uf:\cL^p(\Sigma,\bQ) \to \cL^x(\Sigma,\bQ)$ is a \emph{cluster projection}. Namely, we have a commutative diagram
\begin{equation*}
    \begin{tikzcd}
    \cL^p(\Sigma,\bQ) \ar[r,"\sfx_\tri"] \ar[d,"\pi_\uf"'] & \bQ^{I(\tri)} \ar[d] \\
    \cL^x(\Sigma,\bQ) \ar[r,"\sfx_\tri^\uf"'] & \bQ^{I_\uf(\tri)}
    \end{tikzcd}
\end{equation*}
for any ideal triangulation $\tri$ of $\Sigma$, where the right vertical map is the projection forgetting the frozen coordinates.
\end{lem}

\subsection{Gluing of laminations}
Let $\Sigma$ be a (possibly disconnected) marked surface, and $E_L,E_R \in B(\Sigma)$ distinct boundary intervals. Then we can form a new marked surface $\Sigma'$ from $\Sigma$ by gluing $E_L$ with $E_R$. 
As a tropical analogue of the gluing morphism $\P_{PGL_3,\Sigma} \to \P_{PGL_3,\Sigma'}$ \cite[Lemma 2.14]{GS19}, we are going to introduce a map
\begin{align*}
    q_{E_L,E_R}: \cL^p(\Sigma,\bQ) \to \cL^p(\Sigma',\bQ)
\end{align*}
between the corresponding spaces of rational $\P$-laminations. The map $q_{E_L,E_R}$ will be defined so that equivariant with respect to the $\bQ_{>0}$-action, and invariant under the action $\alpha_{E_L,E_R}: \mathsf{P}^\vee_\bQ \curvearrowright \cL^p(\Sigma,\bQ)$ given by the shift
%only for $E_L$ and $E_R$ components for $\nu = (\nu_E)_E$ of each $(\hL, \nu) \in \cL^p(\Sigma, \bQ)$ as
\begin{align}\label{eq:shift_action}
    \mu.(\nu_{E_L},\nu_{E_R}) := (\nu_{E_L}+\mu, \nu_{E_R}-\mu^\ast)
\end{align}
for $\mu=a\varpi_1^\vee+ b\varpi_2^\vee \in \mathsf{P}^\vee_\bQ$, where $\mu^\ast:=b\varpi_1^\vee+ a\varpi_2^\vee$, and keeping other $\nu_E$, $E \neq E_L,E_R$ intact. 

Let $(\hL,\nu) \in \cL^p(\Sigma,\bZ)$ be an integral $\P$-lamination. Represent the integral unbounded $\fsl_3$-lamination $\hL$ by a non-elliptic signed web $W$ with weight $1$ on every component. 
Around each special point of $E_L$ and $E_R$, draw a semi-infinite collection of disjoint corner arcs with alternating orientations that accumulates only at the special point so that they are disjoint from $W$. Here we choose the orientation of the farthest corner arc from the special point to be clock-wise, as in \cref{subsec:reconstruction}.
Insert a biangle $B$ between $E_L$ and $E_R$, and identify $\Sigma'= \Sigma \cup B$. Notice that the ends of $W$ on $E_L$ and $E_R$, together with those of the additional corner arcs, defines an asymptotically periodic symmetric strand set $S = (S_L,S_R)$ on $B$. 
We equip $S$ with a pinning $\sfp^\pm_Z$ for $Z \in \{L,R\}$ by the following rule:
\begin{itemize}
    \item Choose continuous parametrizations $\psi_{Z}^\pm:\bR \to E_Z$ so that $\psi_{Z}^\pm(\frac{1}{2}+\bZ)=S_{Z}^\pm$, and $\psi_Z^\pm(\bR_{<0}) \cap S_Z^\pm$ consists of all the strands coming from the additional corner arcs around the initial marked point of $E_Z$. 
    \item Then set $p_Z^\pm:=\psi_{Z}^\pm(\nu_{E_Z}^\pm) \in E_Z$.
\end{itemize}
Then we get a pinned symmetric strand set $\widehat{S}:=(S;\sfp_L,\sfp_R)$ on the biangle $B$. Let $W_{\mathrm{br}}(\widehat{S})$ be the associated collection of oriented curves in $B$. Gluing the web $W$ with the collection $W_{\mathrm{br}}(\widehat{S})$, we get an infinite collection $\cW'_{\mathrm{br}}$ of webs on $\Sigma'=\Sigma\cup B$. The initial (resp. terminal) marked point of $E_L$ is identified with the terminal (resp. initial) marked point of $E_R$, and regarded as new marked points in $\Sigma'$. For each of these new marked points, do the followings:
\begin{itemize}
    \item If it is a special point, then remove the peripheral components around this point from $\cW'_{\mathrm{br}}$.
    \item If it is a puncture, then remove the peripheral components and replace each spiralling end around this point with a signed end, while encoding the spiralling directions in signs by reversing the rule in \cref{fig:spiral}. Then there remain at most finitely many intersections in $B$.
    \item Finally, replace each intersection of curves in $B$ with an H-web by the rule \eqref{eq:H-replacement}.
\end{itemize}
Thus we get a non-elliptic signed web $W'$ on $\Sigma'$, which represents an integral $\P$-lamination $\hL'=q_{E_L,E_R}(\hL) \in \cL^p(\Sigma,\bZ)$. The construction is clearly invariant for the action of $H_\partial(\bQ^T)$ by \cref{rem:gluing_symmetry}, and $\bZ_{>0}$-equivariant. Thus it can be extended $\bQ_{>0}$-equivariantly.  

\begin{dfn}
The thus obtained map $q_{E_L,E_R}: \cL^p(\Sigma,\bQ) \to \cL^p(\Sigma',\bQ)$ is called the \emph{gluing map} along $E_L$ and $E_R$.
\end{dfn}
In view of \cref{rem:gluing_symmetry}, we immediately have:

\begin{lem}\label{lem:shift-invariance}
The gluing map $q_{E_L,E_R}$ is invariant under the shift action \eqref{eq:shift_action} of $\mathsf{P}^\vee_\bQ$.
\end{lem}

Any ideal triangulation $\tri$ of $\Sigma$ naturally induces a triangulation $\tri'$ of $\Sigma'$, where the edges $E_L$ and $E_R$ are identified and give an interior edge $E$ of $\tri$. The 
%quiver vertices 
points in $I(\tri)$ on these edges are identified as $i^s(E_L)=i^{s^\ast}(E_R)$ for $s=1,2$ with $s^\ast:=3-s$. 
The points of $I(\tri)$ away from the edges $E_L$ and $E_R$ are naturally identified with the corresponding points of $I(\tri')$. 

\begin{thm}\label{thm:amalgamation}
The gluing map $q_{E_L,E_R}$ is the \emph{tropicalized amalgamation}. Namely, for any ideal triangulation $\tri$ of $\Sigma$ and the induced triangulation $\tri'$ of $\Sigma'$, it satisfies
\begin{align*}
    q_{E_L,E_R}^\ast\sfx^{\tri'}_{E,s} = \sfx^\tri_{E_L,s} + \sfx^\tri_{E_R,s^\ast}
\end{align*}
for $s=1,2$. 
Here $E$ inherits an orientation from $E_L$ (so that from the bottom to the top, when we draw $E_L$ on the left). The other coordinates are kept intact: $q_{E_L,E_R}^\ast\sfx^{\tri'}_{i} = \sfx^\tri_{i}$ for $i \in I(\tri') \setminus \{i^s(E)\}_{s=1,2}$.
\end{thm}

\begin{proof}
The last statement is clear from the definition. To see the relation between the coordinates on the edges $E_L$, $E_R$ and $E$, it suffices to consider an integral lamination $\hL \in \cL^p(\Sigma,\bZ)$ by $\bQ_{>0}$-equivariance. Write $L':=q_{E_L,E_R}(\hL)$ and $\sfx_i:=\sfx_i^\tri(\hL)$ for $i \in I(\tri)$. Recall the reconstruction procedure of the integral lamination $\hL'$ from its shear coordinates, and
compare the gluing parameters
\begin{align}\label{eq:amalgamation_rule_1}
    \begin{aligned}
    \nu_{E_L}^+&= \sfx_{E_L,1} + \alpha_{E_L}^+, & \nu_{E_R}^-&= \sfx_{E_R,2} + [\sfx_{T_R}]_+ + \alpha_{E_R}^-, \\
    \nu_{E_L}^-&=\sfx_{E_L,2} + [\sfx_{T_L}]_+ + \alpha_{E_L}^-, & \nu_{E_R}^+&= \sfx_{E_R,1} + \alpha_{E_R}^+
    \end{aligned}
\end{align} 
with the integers appearing in \eqref{eq:gluing_rule}. By \cref{lem:shift-invariance}, the result of gluing is unchanged under the modification 
\begin{align}\label{eq:amalgamation_rule_2}
    \begin{aligned}
    \widetilde{\nu}_{E_L}^+&:= (\sfx_{E_L,1}+\sfx_{E_R,2}) + \alpha_{E_L}^+, & \widetilde{\nu}_{E_R}^-&:= [\sfx_{T_R}]_+ + \alpha_{E_R}^-, \\
    \widetilde{\nu}_{E_L}^-&:= [\sfx_{T_L}]_+ + \alpha_{E_L}^-, & \widetilde{\nu}_{E_R}^+&:= (\sfx_{E_L,2}+\sfx_{E_R,1})+\alpha_{E_R}^+
    \end{aligned}
\end{align}
by the shift action \eqref{eq:shift_action}. 
On the other hand, since there are \lq\lq original'' corner arcs of $\hL$ in $T_L$ and $T_R$ before adding infinite collections of corner arcs in the gluing procedure, the parametrizations of edges are related by 
\begin{align*}
    \phi^\pm_Z(n) = \psi^\pm_Z(n+\alpha_{E_Z}^\pm)
\end{align*}
for $n \in \bZ$ and $Z \in \{L,R\}$. See \cref{fig:difference_gluing}. These comparisons on the two gluing constructions show that $\hL'=q_{E_L,E_R}(\hL)$ if and only if $\sfx_{E,s}(\hL')=\sfx_{E_L,s}(\hL) + \sfx_{E_R,s^\ast}(\hL)$ for $s=1,2$. 
\end{proof}

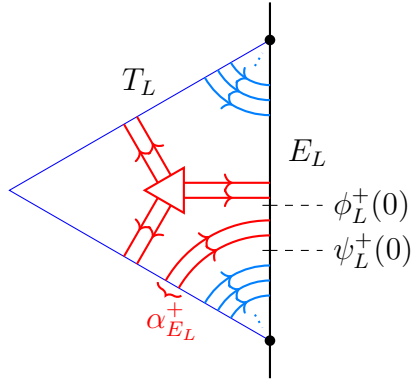
\begin{figure}[h]
\centering
\begin{tikzpicture}
\draw[blue] (0,0) coordinate(A) -- (0,4) coordinate(B) --++(-150:4) coordinate(C) --cycle;
%\fill[gray!30] (0,-0.5) -- (0,3.5) -- (0.2,3.5) -- (0.2,-0.5) --cycle;
\draw[thick] (0,-0.5) -- (0,4.5);
\fill(0,0) circle(2pt);
\fill(0,4) circle(2pt);
\draw[red,thick] (C)++(1.8,0) coordinate(C1) --++(-30:0.6) coordinate(A1) --($(C1)+(30:0.6)$) coordinate(B1) --cycle;
\begin{scope}
\clip (0,0) coordinate(A) -- (0,4) coordinate(B) --++(-150:4) coordinate(C) --cycle;
\foreach \i in {1,2}
\draw[red,thick,->-] ($(A1)!\i/3!(B1)$) --++(1.3,0);
\foreach \i in {1,2}
\draw[red,thick,->-] ($(C1)!\i/3!(B1)$) --++(120:1);
\foreach \i in {1,2}
\draw[red,thick,->-] ($(A1)!\i/3!(C1)$) --++(-120:1);
\end{scope}
\foreach \i in {0.4,0.35}
\draw[red,thick,->-] ($(A)!\i!(C)$) arc(150:90:4*\i);
\foreach \i in {0.25,0.20,0.15}
\draw[myblue,thick,->-] ($(A)!\i!(C)$) arc(150:90:4*\i);
\draw[myblue,thick,dotted] (120:4*0.10) -- (120:4*0.05); 
\foreach \i in {0.25,0.20,0.15}
\draw[myblue,thick,->-] ($(B)!\i!(C)$) arc(-150:-90:4*\i);
\draw[myblue,thick,dotted] ($(0,4)+(-120:4*0.10)$) -- ($(0,4)+(-120:4*0.05)$);  
\draw (-0.1,4*0.30) -- (0.1,4*0.30);
\draw[dashed] (0.2,4*0.30) -- (0.7,4*0.30)
node[anchor=west]{$\psi_L^+(0)$};
\draw (-0.1,4*0.45) -- (0.1,4*0.45);
\draw[dashed] (0.2,4*0.45) -- (0.7,4*0.45)
node[anchor=west]{$\phi_L^+(0)$};
\draw[red,thick,decorate,decoration={brace,amplitude=3pt,raise=2pt}] (150:4*0.33) --node[midway,below=0.2em]{$\alpha_{E_L}^+$} (150:4*0.42);
\node[above=0.3em] at ($(B)!0.5!(C)$) {$T_L$};
\node at (0.5,2.5) {$E_L$};
\end{tikzpicture}
    \caption{Comparison of two edge parametrizations. A part of the web representing $\hL$ which will be incoming to the bigon $B_E$ is shown in red, and the additional corner arcs are shown in blue.}
    \label{fig:difference_gluing}
\end{figure}

\begin{rem}\label{rem:natural_definition}
In view of the gluing construction presented above, the definition of the integral unbounded $\fsl_3$-laminations with pinnings can be modified slightly more geometrically as integral unbounded $\fsl_3$-laminations equipped with infinitely many corner arcs around special points and choices of points $p_E^\pm \in E$ for each $E \in \bB$, in place of the datum $\nu_E \in \mathsf{P}^\vee$. It gives a right description of the tropical analogue of $\P_{PGL_3,\Sigma}(\bZ^T)$ without fixing a trivialization of the bundle \eqref{eq:bundle_structure}.  
We do not pursue an extension of this description to the rational case. 
%It will be interesting to pursue a more natural definition of the space $\cL^p(\Sigma,\bQ)$ without splitting into $\cL^x(\Sigma,\bQ)$ and $H_\partial(\bQ^T)$. 
\end{rem}

\subsection{Extended ensemble map}
Recall the geometric ensemble map \eqref{eq:ensemble_unfrozen}. 
We extend it by 
\begin{align*}
    \widetilde{p}: \cL^a(\Sigma,\bQ) \to \cL^p(\Sigma,\bQ), \quad L \mapsto (p(L),(\nu_E)_E),
\end{align*}
where $\nu_E^+$ (resp. $\nu_E^-$) is minus the total weight of the peripheral components with the clockwise (resp. counter-clockwise) orientation around the initial marked point of $E$. We have a commutative diagram 
\begin{equation*}
    \begin{tikzcd}
    \cL^a(\Sigma,\bQ) \ar[r,"\widetilde{p}"] \ar[rd,"p"'] & \cL^p(\Sigma,\bQ) \ar[d,"\pi_\uf"] \\
     & \cL^x(\Sigma,\bQ).
    \end{tikzcd}
\end{equation*}

\begin{lem}\label{lem:ensemble_bijection}
If $\Sigma$ has no punctures, then $\widetilde{p}: \cL^a(\Sigma,\bQ) \to \cL^p(\Sigma,\bQ)$ gives a bijection.
\end{lem}

\begin{proof}
In this case, the only datum that the map $p$ loses is the weights of peripheral components around special points. This can be uniquely recovered from the tuple $(\nu_E)_E$. 
\end{proof}
On the integral points, we have $\widetilde{p}(\cL^a(\Sigma,\bZ)) \subset \cL^p(\Sigma,\bZ)$.
%is a sub-lattice of index $3$. 

\begin{prop}\label{prop:ensemble_map}
The extended geometric ensemble map $\widetilde{p}: \cL^a(\Sigma,\bQ) \to \cL^p(\Sigma,\bQ)$ coincides with the Goncharov--Shen extension of the ensemble map \eqref{eq:GS_extension}. Namely, it satisfies
\begin{align}\label{eq:ensemble_relation}
    \widetilde{p}^* \sfx_i^\tri = \sum_{j \in I(\tri)} (\varepsilon^\tri_{ij}+m_{ij}) \sfa_j^\tri
\end{align}
for any ideal triangulation $\tri$ of $\Sigma$ and $i \in I(\tri)$, where:
\begin{itemize}
    \item $(\sfa_j^\tri)_{j \in I(\tri)}$ denotes the tropical $\A$-coordinates on $\cL^a(\Sigma,\bQ)$ associated with $\tri$, which is the one-third of the Douglas--Sun's coordinates;
    \item  $\varepsilon^\tri=(\varepsilon^\tri_{ij})_{i,j \in I(\tri)}$ denotes the exchange matrix defined in \cref{subsec:cluster_sl3};
    \item $M=(m_{ij})_{i,j \in I_\f(\tri)}$ is the half-integral symmetric matrix given in \eqref{eq:m-matrix}. 
\end{itemize}

\end{prop}
In particular, by forgetting the pinnings and frozen coordinates, we see that the geometric ensemble map $p: \cL^a(\Sigma,\bQ) \to \cL^x(\Sigma,\bQ)$ coincides with the ensemble map \eqref{eq:ensemble_map}.

\begin{proof}
In view of the local nature of the definitions of coordinate systems and the exchange matrix, it suffices to consider the case where $\Sigma$ is a triangle or a quadrilateral.  
Indeed, for $i=i(T) \in I^\mathrm{tri}(\tri)$, it suffices to focus on the triangle $T$ containing it; for $i=i^s(E) \in I^\mathrm{edge}(\tri) \cap I_\uf(\tri)$ consider the quadrilateral containing the interior edge $E$ as a diagonal; for $i=i^s(E) \in I^\mathrm{edge}(\tri) \cap I_\f(\tri)$ consider the triangle $T$ having the boundary interval $E$ as one of its sides. 

\begin{description}
\item[Triangle case] For the $\fsl_3$-quiver associated with the unique ideal triangulation of a triangle $T$, label its vertices as:
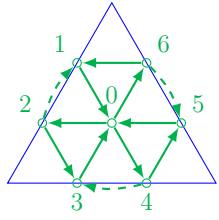
\begin{align*}
\begin{tikzpicture}[scale=0.8]
\draw[blue] (-30:2) -- (90:2) -- (210:2) --cycle;
\draw[mygreen] (0,0) circle(2pt);
\foreach \i in {0,1,2}
{
\begin{scope}[rotate=\i*120,>=latex]
\draw[mygreen] ($(-30:2)!1/3!(90:2)$) circle(2pt) coordinate(X\i);
\draw[mygreen] ($(-30:2)!2/3!(90:2)$) circle(2pt) coordinate(Y\i);
{\color{mygreen}
\qarrow{$(-30:2)!2/3!(90:2)$}{$(90:2)!1/3!(210:2)$}
\qarrow{$(-30:2)!1/3!(90:2)$}{0,0}
\qarrow{0,0}{$(90:2)!2/3!(210:2)$}
\qdlarrow{$(-30:2)!2/3!(90:2)$}{$(-30:2)!1/3!(90:2)$}
}
\end{scope}
}
{\color{mygreen}
\node[above left] at (X1) {\scalebox{0.8}{$1$}};
\node[above left] at (Y1) {\scalebox{0.8}{$2$}};
\node[below] at (X2) {\scalebox{0.8}{$3$}};
\node[below] at (Y2) {\scalebox{0.8}{$4$}};
\node[above right] at (X0) {\scalebox{0.8}{$5$}};
\node[above right] at (Y0) {\scalebox{0.8}{$6$}};
\node[above=0.3em] at (0,0) {\scalebox{0.8}{$0$}};
}
\end{tikzpicture}
\end{align*}
Then the expected relation \eqref{eq:ensemble_relation} reads as
\begin{align*}
    \widetilde{p}^*\sfx_0&=\sfa_{2}+\sfa_{4}+\sfa_{6}-(\sfa_{1}+\sfa_{3}+\sfa_{5}), \\
    \widetilde{p}^*\sfx_{1}&=\sfa_{0} - \sfa_1 - \sfa_6, \\
    \widetilde{p}^*\sfx_{2}&=\sfa_1 + \sfa_3 - \sfa_2 - \sfa_0, \\
    \widetilde{p}^*\sfx_{3}&=\sfa_{0} - \sfa_3 - \sfa_2, \\
    \widetilde{p}^*\sfx_{4}&=\sfa_3 + \sfa_5 - \sfa_4 - \sfa_0, \\
    \widetilde{p}^*\sfx_{5}&=\sfa_{0} - \sfa_5 - \sfa_4, \\
    \widetilde{p}^*\sfx_{6}&=\sfa_5 + \sfa_1 - \sfa_6 - \sfa_0.
\end{align*}
The tropical $\A$-coordinates of essential webs on $T$ are defined as the weighted sum of the coordinates of its components. See \cite[Section 4.3]{DS20I}. Therefore it suffices to check the relations for the corner arcs and the sink-/source-honeycombs of height $1$, whose coordinates are shown in \cref{fig:triangle_coordinates_comparison}. Then the relations between the two coordinates can be easily verified.

\item[Quadrilateral case] For the $\fsl_3$-quiver associated with an ideal triangulation $\tri$ of a quadrilateral $Q$, label its vertices as:
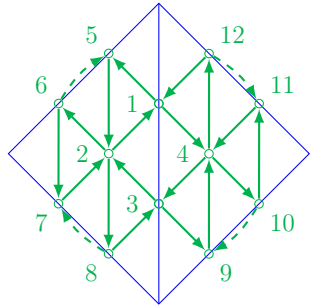
\begin{align*}
\begin{tikzpicture}[scale=0.8,>=latex]
\draw[blue] (0,2.5) coordinate(A) -- (-2.5,0) coordinate(B) -- (0,-2.5) coordinate(C) -- (2.5,0) coordinate(D) --cycle;
\draw[blue] (A) -- (C);
{\color{mygreen}
\quiverplus{0,-2.5}{0,2.5}{-2.5,0};
\qdlarrow{x232}{x231};
\qdlarrow{x312}{x311};
\node[left=0.3em] at (x122) {\scalebox{0.8}{$1$}};
\node[left=0.3em] at (x121) {\scalebox{0.8}{$3$}};
\node[left=0.3em] at (G) {\scalebox{0.8}{$2$}};
\node[above left] at (x231) {\scalebox{0.8}{$5$}};
\node[above left] at (x232) {\scalebox{0.8}{$6$}};
\node[below left] at (x311) {\scalebox{0.8}{$7$}};
\node[below left] at (x312) {\scalebox{0.8}{$8$}};
\quiverplus{0,-2.5}{2.5,0}{0,2.5};
\qdlarrow{x122}{x121};
\qdlarrow{x232}{x231};
\node[left=0.3em] at (G) {\scalebox{0.8}{$4$}};
\node[below right] at (x121) {\scalebox{0.8}{$9$}};
\node[below right] at (x122) {\scalebox{0.8}{$10$}};
\node[above right] at (x231) {\scalebox{0.8}{$11$}};
\node[above right] at (x232) {\scalebox{0.8}{$12$}};
}
\end{tikzpicture}
\end{align*}
The remaining relations to be checked are:
\begin{align}\label{eq:ensemble_edge_coordinate}
    \begin{aligned}
    \widetilde{p}^*\sfx_1 &= \sfa_5 + \sfa_4 - \sfa_2 - \sfa_{12},\\
    \widetilde{p}^*\sfx_3 &= \sfa_2 + \sfa_9 - \sfa_4 - \sfa_8.
    \end{aligned}
\end{align}
The tropical $\A$-coordinate assigned to a vertex $i \in I(\tri)$ only depends on the restriction of a given web to the triangle which contains $i$. In particular, we can choose the braid representative with respect to $\widehat{\tri}$ for the computation, since the biangle part does not matter. Then both $\A$- and $\X$-coordinates are weighted sums of contributions from the components of the braid representative. It is easy to verify that the both sides of the equations in \eqref{eq:ensemble_edge_coordinate} vanish for the corner arcs around the marked points $Q$. For the curve and honeycomb components that contribute to the shear coordinates, the expected relations are easily verified from \cref{fig:quadrilateral_coordinates_comparison,fig:quadrilateral_coordinates_comparison_2}. Here notice that, for instance, the coordinates of the honeycomb component $H_{n_1,n_2,n_3}$ shown in the top of \cref{fig:shear_honeycomb} can be computed as 
$\mathsf{z}_\tri(H_{n_1,n_2,n_3})= n_1\mathsf{z}_\tri(\tau_+^L) + n_2\mathsf{z}_\tri(h) + n_3\mathsf{z}_\tri(\tau_+^R)$ 
for $\mathsf{z} \in \{\sfa,\sfx\}$. Together with this observation, the eight patterns shown in \cref{fig:quadrilateral_coordinates_comparison,fig:quadrilateral_coordinates_comparison_2} exhausts all the patterns up to symmetry. 
\end{description}
Thus the proposition is proved.
\end{proof}

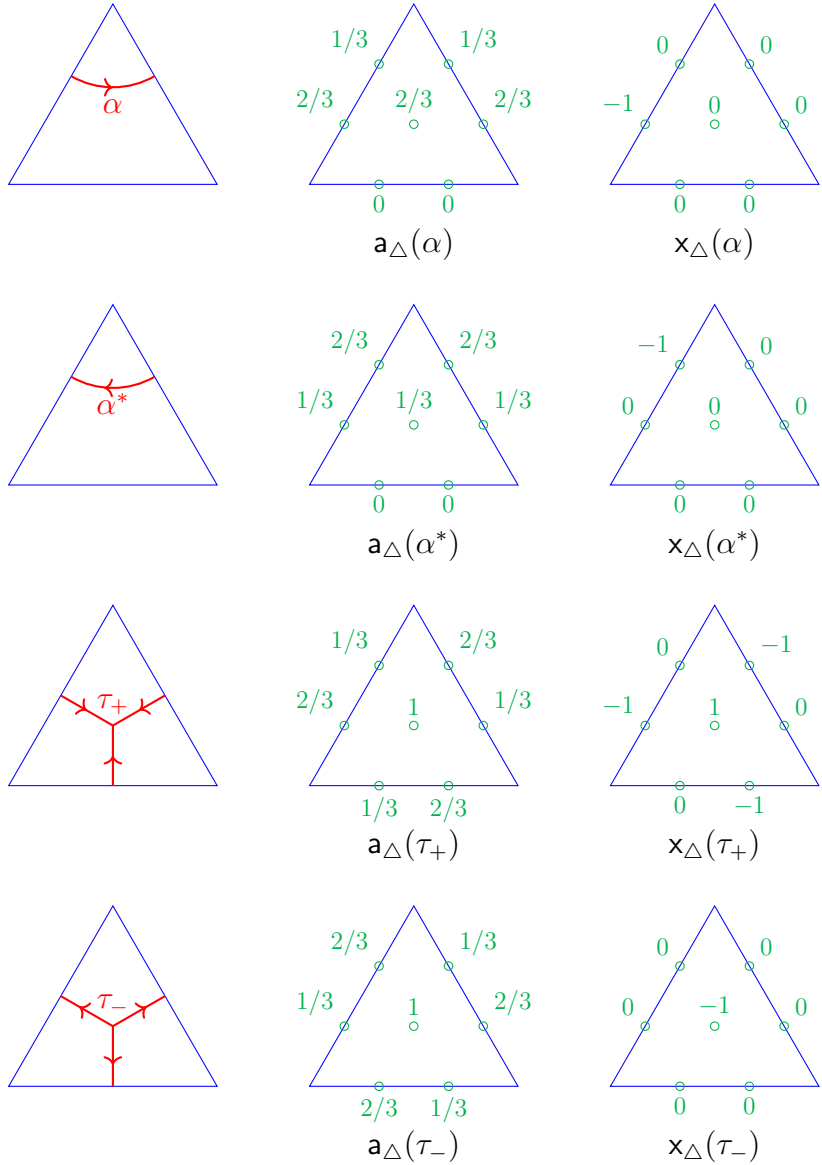
\begin{figure}[htbp]
\centering
\begin{tikzpicture}[scale=0.75]
\draw[blue] (-30:2) -- (90:2) -- (210:2) --cycle;
\draw[red,thick,->-] ($(90:2)!0.4!(210:2)$) arc(-120:-60:0.8*1.732);
\node[red,above] at (0,0) {$\alpha$};

\begin{scope}[xshift=5cm]
\draw[blue] (-30:2) -- (90:2) -- (210:2) --cycle;
\draw[mygreen] (0,0) circle(2pt);
\foreach \i in {0,1,2}
{
\begin{scope}[rotate=\i*120,>=latex]
\draw[mygreen] ($(-30:2)!1/3!(90:2)$) circle(2pt) coordinate(X\i);
\draw[mygreen] ($(-30:2)!2/3!(90:2)$) circle(2pt) coordinate(Y\i);
\end{scope}
}
{\color{mygreen}
\node[above left] at (X1) {\scalebox{0.8}{$1/3$}};
\node[above left] at (Y1) {\scalebox{0.8}{$2/3$}};
\node[below] at (X2) {\scalebox{0.8}{$0$}};
\node[below] at (Y2) {\scalebox{0.8}{$0$}};
\node[above right] at (X0) {\scalebox{0.8}{$2/3$}};
\node[above right] at (Y0) {\scalebox{0.8}{$1/3$}};
\node[above] at (0,0) {\scalebox{0.8}{$2/3$}};
}
\node at (0,-2) {$\sfa_\tri(\alpha)$};
\end{scope}

\begin{scope}[xshift=10cm]
\draw[blue] (-30:2) -- (90:2) -- (210:2) --cycle;
\draw[mygreen] (0,0) circle(2pt);
\foreach \i in {0,1,2}
{
\begin{scope}[rotate=\i*120,>=latex]
\draw[mygreen] ($(-30:2)!1/3!(90:2)$) circle(2pt) coordinate(X\i);
\draw[mygreen] ($(-30:2)!2/3!(90:2)$) circle(2pt) coordinate(Y\i);
\end{scope}
}
{\color{mygreen}
\node[above left] at (X1) {\scalebox{0.8}{$0$}};
\node[above left] at (Y1) {\scalebox{0.8}{$-1$}};
\node[below] at (X2) {\scalebox{0.8}{$0$}};
\node[below] at (Y2) {\scalebox{0.8}{$0$}};
\node[above right] at (X0) {\scalebox{0.8}{$0$}};
\node[above right] at (Y0) {\scalebox{0.8}{$0$}};
\node[above] at (0,0) {\scalebox{0.8}{$0$}};
}
\node at (0,-2) {$\sfx_\tri(\alpha)$};
\end{scope}

\begin{scope}[yshift=-5cm]
\draw[blue] (-30:2) -- (90:2) -- (210:2) --cycle;
\draw[red,thick,-<-] ($(90:2)!0.4!(210:2)$) arc(-120:-60:0.8*1.732);
\node[red,above] at (0,0) {$\alpha^\ast$};
\end{scope}

\begin{scope}[xshift=5cm,yshift=-5cm]
\draw[blue] (-30:2) -- (90:2) -- (210:2) --cycle;
\draw[mygreen] (0,0) circle(2pt);
\foreach \i in {0,1,2}
{
\begin{scope}[rotate=\i*120,>=latex]
\draw[mygreen] ($(-30:2)!1/3!(90:2)$) circle(2pt) coordinate(X\i);
\draw[mygreen] ($(-30:2)!2/3!(90:2)$) circle(2pt) coordinate(Y\i);
\end{scope}
}
{\color{mygreen}
\node[above left] at (X1) {\scalebox{0.8}{$2/3$}};
\node[above left] at (Y1) {\scalebox{0.8}{$1/3$}};
\node[below] at (X2) {\scalebox{0.8}{$0$}};
\node[below] at (Y2) {\scalebox{0.8}{$0$}};
\node[above right] at (X0) {\scalebox{0.8}{$1/3$}};
\node[above right] at (Y0) {\scalebox{0.8}{$2/3$}};
\node[above] at (0,0) {\scalebox{0.8}{$1/3$}};
}
\node at (0,-2) {$\sfa_\tri(\alpha^\ast)$};
\end{scope}

\begin{scope}[xshift=10cm,yshift=-5cm]
\draw[blue] (-30:2) -- (90:2) -- (210:2) --cycle;
\draw[mygreen] (0,0) circle(2pt);
\foreach \i in {0,1,2}
{
\begin{scope}[rotate=\i*120,>=latex]
\draw[mygreen] ($(-30:2)!1/3!(90:2)$) circle(2pt) coordinate(X\i);
\draw[mygreen] ($(-30:2)!2/3!(90:2)$) circle(2pt) coordinate(Y\i);
\end{scope}
}
{\color{mygreen}
\node[above left] at (X1) {\scalebox{0.8}{$-1$}};
\node[above left] at (Y1) {\scalebox{0.8}{$0$}};
\node[below] at (X2) {\scalebox{0.8}{$0$}};
\node[below] at (Y2) {\scalebox{0.8}{$0$}};
\node[above right] at (X0) {\scalebox{0.8}{$0$}};
\node[above right] at (Y0) {\scalebox{0.8}{$0$}};
\node[above] at (0,0) {\scalebox{0.8}{$0$}};
}
\node at (0,-2) {$\sfx_\tri(\alpha^\ast)$};
\end{scope}

\begin{scope}[yshift=-10cm]
\draw[blue] (-30:2) -- (90:2) -- (210:2) --cycle;
\foreach \i in {-30,90,210}
\draw[red,thick,->-] ($(\i:2)!0.5!(\i+120:2)$) -- (0,0);
\node[red,above] at (0,0) {$\tau_+$};
\end{scope}

\begin{scope}[xshift=5cm,yshift=-10cm]
\draw[blue] (-30:2) -- (90:2) -- (210:2) --cycle;
\draw[mygreen] (0,0) circle(2pt);
\foreach \i in {0,1,2}
{
\begin{scope}[rotate=\i*120,>=latex]
\draw[mygreen] ($(-30:2)!1/3!(90:2)$) circle(2pt) coordinate(X\i);
\draw[mygreen] ($(-30:2)!2/3!(90:2)$) circle(2pt) coordinate(Y\i);
\end{scope}
}
{\color{mygreen}
\node[above left] at (X1) {\scalebox{0.8}{$1/3$}};
\node[above left] at (Y1) {\scalebox{0.8}{$2/3$}};
\node[below] at (X2) {\scalebox{0.8}{$1/3$}};
\node[below] at (Y2) {\scalebox{0.8}{$2/3$}};
\node[above right] at (X0) {\scalebox{0.8}{$1/3$}};
\node[above right] at (Y0) {\scalebox{0.8}{$2/3$}};
\node[above] at (0,0) {\scalebox{0.8}{$1$}};
}
\node at (0,-2) {$\sfa_\tri(\tau_+)$};
\end{scope}

\begin{scope}[xshift=10cm,yshift=-10cm]
\draw[blue] (-30:2) -- (90:2) -- (210:2) --cycle;
\draw[mygreen] (0,0) circle(2pt);
\foreach \i in {0,1,2}
{
\begin{scope}[rotate=\i*120,>=latex]
\draw[mygreen] ($(-30:2)!1/3!(90:2)$) circle(2pt) coordinate(X\i);
\draw[mygreen] ($(-30:2)!2/3!(90:2)$) circle(2pt) coordinate(Y\i);
\end{scope}
}
{\color{mygreen}
\node[above left] at (X1) {\scalebox{0.8}{$0$}};
\node[above left] at (Y1) {\scalebox{0.8}{$-1$}};
\node[below] at (X2) {\scalebox{0.8}{$0$}};
\node[below] at (Y2) {\scalebox{0.8}{$-1$}};
\node[above right] at (X0) {\scalebox{0.8}{$0$}};
\node[above right] at (Y0) {\scalebox{0.8}{$-1$}};
\node[above] at (0,0) {\scalebox{0.8}{$1$}};
}
\node at (0,-2) {$\sfx_\tri(\tau_+)$};
\end{scope}

\begin{scope}[yshift=-15cm]
\draw[blue] (-30:2) -- (90:2) -- (210:2) --cycle;
\foreach \i in{-30,90,210}
\draw[red,thick,-<-] ($(\i:2)!0.5!(\i+120:2)$) -- (0,0);
\node[red,above] at (0,0) {$\tau_-$};
\end{scope}

\begin{scope}[xshift=5cm,yshift=-15cm]
\draw[blue] (-30:2) -- (90:2) -- (210:2) --cycle;
\draw[mygreen] (0,0) circle(2pt);
\foreach \i in {0,1,2}
{
\begin{scope}[rotate=\i*120,>=latex]
\draw[mygreen] ($(-30:2)!1/3!(90:2)$) circle(2pt) coordinate(X\i);
\draw[mygreen] ($(-30:2)!2/3!(90:2)$) circle(2pt) coordinate(Y\i);
\end{scope}
}
{\color{mygreen}
\node[above left] at (X1) {\scalebox{0.8}{$2/3$}};
\node[above left] at (Y1) {\scalebox{0.8}{$1/3$}};
\node[below] at (X2) {\scalebox{0.8}{$2/3$}};
\node[below] at (Y2) {\scalebox{0.8}{$1/3$}};
\node[above right] at (X0) {\scalebox{0.8}{$2/3$}};
\node[above right] at (Y0) {\scalebox{0.8}{$1/3$}};
\node[above] at (0,0) {\scalebox{0.8}{$1$}};
}
\node at (0,-2) {$\sfa_\tri(\tau_-)$};
\end{scope}

\begin{scope}[xshift=10cm,yshift=-15cm]
\draw[blue] (-30:2) -- (90:2) -- (210:2) --cycle;
\draw[mygreen] (0,0) circle(2pt);
\foreach \i in {0,1,2}
{
\begin{scope}[rotate=\i*120,>=latex]
\draw[mygreen] ($(-30:2)!1/3!(90:2)$) circle(2pt) coordinate(X\i);
\draw[mygreen] ($(-30:2)!2/3!(90:2)$) circle(2pt) coordinate(Y\i);
\end{scope}
}
{\color{mygreen}
\node[above left] at (X1) {\scalebox{0.8}{$0$}};
\node[above left] at (Y1) {\scalebox{0.8}{$0$}};
\node[below] at (X2) {\scalebox{0.8}{$0$}};
\node[below] at (Y2) {\scalebox{0.8}{$0$}};
\node[above right] at (X0) {\scalebox{0.8}{$0$}};
\node[above right] at (Y0) {\scalebox{0.8}{$0$}};
\node[above] at (0,0) {\scalebox{0.8}{$-1$}};
}
\node at (0,-2) {$\sfx_\tri(\tau_-)$};
\end{scope}

\end{tikzpicture}
    \caption{Two types of coordinates of component webs on a triangle $T$. All the webs shown here have weight $1$.}
    \label{fig:triangle_coordinates_comparison}
\end{figure}

\begin{figure}[htbp]
    \centering
\begin{tikzpicture}[scale=0.58, every node/.style={scale=0.85}]
%1st type
\draw[blue] (2.5,0) -- (0,2.5) -- (-2.5,0) -- (0,-2.5) --cycle;
\draw[blue] (0,-2.5) to[bend left=30pt] (0,2.5);
\draw[blue] (0,-2.5) to[bend right=30pt] (0,2.5);
\draw[thick,red,->-] (-1.25,1.25) --node[midway,below=0.3em]{$\alpha_+$} (1.25,-1.25);

\begin{scope}[xshift=6cm]
\draw[blue] (2.5,0) -- (0,2.5) -- (-2.5,0) -- (0,-2.5) --cycle;
\draw[blue] (0,-2.5) to (0,2.5);
\quiversquare{0,-2.5}{2.5,0}{0,2.5}{-2.5,0};
{\color{mygreen}
\node[below right,scale=0.8] at (x121) {$2/3$}; 
\node[below right,scale=0.8] at (x122) {$1/3$};
\node[above right,scale=0.8] at (x231) {$0$};
\node[above right,scale=0.8] at (x232) {$0$};
\node[above left,scale=0.8] at (x341) {$1/3$};
\node[above left,scale=0.8] at (x342) {$2/3$};
\node[below left,scale=0.8] at (x411) {$0$};
\node[below left,scale=0.8] at (x412) {$0$};
\node[below right,scale=0.8] at (x131) {$2/3$};
\node[above right,scale=0.8] at (x132) {$1/3$};
\node[right,scale=0.8] at (x241) {$1/3$};
\node[left,scale=0.8] at (x242) {$2/3$};
}
\node at (0,-3) {$\sfa_\tri(\alpha_+)$};
\end{scope}

\begin{scope}[xshift=12cm]
\draw[blue] (2.5,0) -- (0,2.5) -- (-2.5,0) -- (0,-2.5) --cycle;
\draw[blue] (0,-2.5) to (0,2.5);
\quiversquare{0,-2.5}{2.5,0}{0,2.5}{-2.5,0};
\squarefrozen{-1}{0}{0}{0}
{0}{-1}{0}{0};
\squareunfrozen{1}{0}{0}{0};
\node at (0,-3) {$\sfx_\tri(\alpha_+)$};
\end{scope}
%2nd type
\begin{scope}[yshift=-6.5cm]
\draw[blue] (2.5,0) -- (0,2.5) -- (-2.5,0) -- (0,-2.5) --cycle;
\draw[blue] (0,-2.5) to[bend left=30pt] (0,2.5);
\draw[blue] (0,-2.5) to[bend right=30pt] (0,2.5);
\draw[thick,red,->-] (-1.25,-1.25) --node[midway,above=0.3em]{$\alpha_-$} (1.25,1.25);
\end{scope}

\begin{scope}[xshift=6cm,yshift=-6.5cm]
\draw[blue] (2.5,0) -- (0,2.5) -- (-2.5,0) -- (0,-2.5) --cycle;
\draw[blue] (0,-2.5) to (0,2.5);
\quiversquare{0,-2.5}{2.5,0}{0,2.5}{-2.5,0};
\squarefrozen{0}{0}{2/3}{1/3}
{0}{0}{1/3}{2/3};
\squareunfrozen{2/3}{1/3}{2/3}{1/3};
\node at (0,-3) {$\sfa_\tri(\alpha_-)$};
\end{scope}

\begin{scope}[xshift=12cm,yshift=-6.5cm]
\draw[blue] (2.5,0) -- (0,2.5) -- (-2.5,0) -- (0,-2.5) --cycle;
\draw[blue] (0,-2.5) to (0,2.5);
\quiversquare{0,-2.5}{2.5,0}{0,2.5}{-2.5,0};
\squarefrozen{0}{0}{0}{0}
{0}{0}{0}{0};
\squareunfrozen{-1}{0}{0}{0};
\node at (0,-3) {$\sfx_\tri(\alpha_-)$};
\end{scope}
%3rd type
\begin{scope}[yshift=-13cm]
\draw[blue] (2.5,0) -- (0,2.5) -- (-2.5,0) -- (0,-2.5) --cycle;
\draw[blue] (0,-2.5) to[bend left=30pt] (0,2.5);
\draw[blue] (0,-2.5) to[bend right=30pt] (0,2.5);
\draw[thick,red,->-] (-2*5/6,-1*5/6) -- (-1*5/6,0);
\draw[thick,red,->-] (-2*5/6,1*5/6) -- (-1*5/6,0);
\draw[thick,red,-<-] (-1*5/6,0) to[out=0,in=-135] (1.25,1.25);
\node[red] at (0,-0.25) {$\tau_+^L$};
\end{scope}

\begin{scope}[xshift=6cm,yshift=-13cm]
\draw[blue] (2.5,0) -- (0,2.5) -- (-2.5,0) -- (0,-2.5) --cycle;
\draw[blue] (0,-2.5) to (0,2.5);
\quiversquare{0,-2.5}{2.5,0}{0,2.5}{-2.5,0};
\squarefrozen{0}{0}{1/3}{2/3}
{1/3}{2/3}{1/3}{2/3};
\squareunfrozen{1/3}{2/3}{1/3}{1};
\node at (0,-3) {$\sfa_\tri(\tau_+^L)$};
\end{scope}

\begin{scope}[xshift=12cm,yshift=-13cm]
\draw[blue] (2.5,0) -- (0,2.5) -- (-2.5,0) -- (0,-2.5) --cycle;
\draw[blue] (0,-2.5) to (0,2.5);
\quiversquare{0,-2.5}{2.5,0}{0,2.5}{-2.5,0};
\squarefrozen{0}{0}{0}{0}
{0}{-1}{0}{-1};
\squareunfrozen{0}{-1}{0}{1};
\node at (0,-3) {$\sfx_\tri(\tau_+^L)$};
\end{scope}
%4th type
\begin{scope}[yshift=-19.5cm]
\draw[blue] (2.5,0) -- (0,2.5) -- (-2.5,0) -- (0,-2.5) --cycle;
\draw[blue] (0,-2.5) to[bend left=30pt] (0,2.5);
\draw[blue] (0,-2.5) to[bend right=30pt] (0,2.5);
\draw[thick,red,->-] (-2*5/6,-1*5/6) -- (-1*5/6,0);
\draw[thick,red,->-] (-2*5/6,1*5/6) -- (-1*5/6,0);
\draw[thick,red,-<-] (-1*5/6,0) to[out=0,in=135] (1.25,-1.25);
\node[red] at (0,0.25) {$\tau_+^R$};
\end{scope}

\begin{scope}[xshift=6cm,yshift=-19.5cm]
\draw[blue] (2.5,0) -- (0,2.5) -- (-2.5,0) -- (0,-2.5) --cycle;
\draw[blue] (0,-2.5) to (0,2.5);
\quiversquare{0,-2.5}{2.5,0}{0,2.5}{-2.5,0};
\squarefrozen{1/3}{2/3}{0}{0}
{1/3}{2/3}{1/3}{2/3};
\squareunfrozen{1/3}{2/3}{2/3}{1};
\node at (0,-3) {$\sfa_\tri(\tau_+^R)$};
\end{scope}

\begin{scope}[xshift=12cm,yshift=-19.5cm]
\draw[blue] (2.5,0) -- (0,2.5) -- (-2.5,0) -- (0,-2.5) --cycle;
\draw[blue] (0,-2.5) to (0,2.5);
\quiversquare{0,-2.5}{2.5,0}{0,2.5}{-2.5,0};
\squarefrozen{0}{-1}{0}{0}
{0}{-1}{0}{-1};
\squareunfrozen{0}{0}{0}{1};
\node at (0,-3) {$\sfx_\tri(\tau_+^R)$};
\end{scope}

%5th type
\begin{scope}[yshift=-26cm]
\draw[blue] (2.5,0) -- (0,2.5) -- (-2.5,0) -- (0,-2.5) --cycle;
\draw[blue] (0,-2.5) to[bend left=30pt] (0,2.5);
\draw[blue] (0,-2.5) to[bend right=30pt] (0,2.5);
\draw[thick,red,-<-] (-2*5/6,-1*5/6) -- (-1*5/6,0);
\draw[thick,red,-<-] (-2*5/6,1*5/6) -- (-1*5/6,0);
\draw[thick,red,->-] (-1*5/6,0) to[out=0,in=-135] (1.25,1.25);
\node[red] at (0,-0.25) {$\tau_-^L$};
\end{scope}

\begin{scope}[xshift=6cm,yshift=-26cm]
\draw[blue] (2.5,0) -- (0,2.5) -- (-2.5,0) -- (0,-2.5) --cycle;
\draw[blue] (0,-2.5) to (0,2.5);
\quiversquare{0,-2.5}{2.5,0}{0,2.5}{-2.5,0};
\squarefrozen{0}{0}{2/3}{1/3}
{2/3}{1/3}{2/3}{1/3};
\squareunfrozen{2/3}{1/3}{2/3}{1};
\node at (0,-3) {$\sfa_\tri(\tau_-^L)$};
\end{scope}

\begin{scope}[xshift=12cm,yshift=-26cm]
\draw[blue] (2.5,0) -- (0,2.5) -- (-2.5,0) -- (0,-2.5) --cycle;
\draw[blue] (0,-2.5) to (0,2.5);
\quiversquare{0,-2.5}{2.5,0}{0,2.5}{-2.5,0};
\squarefrozen{0}{0}{0}{0}
{0}{0}{0}{0};
\squareunfrozen{0}{0}{0}{-1};
\node at (0,-3) {$\sfx_\tri(\tau_-^L)$};
\end{scope}
\end{tikzpicture}
    \caption{Two types of coordinates of component webs on a quadrilateral $Q$. All the webs shown here have weight $1$. (1/2)}
    \label{fig:quadrilateral_coordinates_comparison}
\end{figure}

\begin{figure}[htbp]
    \centering
\begin{tikzpicture}[scale=0.68, every node/.style={scale=0.85}]
%5th type
\begin{scope}
\draw[blue] (2.5,0) -- (0,2.5) -- (-2.5,0) -- (0,-2.5) --cycle;
\draw[blue] (0,-2.5) to[bend left=30pt] (0,2.5);
\draw[blue] (0,-2.5) to[bend right=30pt] (0,2.5);
\draw[thick,red,-<-] (-2*5/6,-1*5/6) -- (-1*5/6,0);
\draw[thick,red,-<-] (-2*5/6,1*5/6) -- (-1*5/6,0);
\draw[thick,red,->-] (-1*5/6,0) to[out=0,in=-135] (1.25,1.25);
\node[red] at (0,-0.25) {$\tau_-^L$};
\end{scope}

\begin{scope}[xshift=6cm]
\draw[blue] (2.5,0) -- (0,2.5) -- (-2.5,0) -- (0,-2.5) --cycle;
\draw[blue] (0,-2.5) to (0,2.5);
\quiversquare{0,-2.5}{2.5,0}{0,2.5}{-2.5,0};
\squarefrozen{0}{0}{2/3}{1/3}
{2/3}{1/3}{2/3}{1/3};
\squareunfrozen{2/3}{1/3}{2/3}{1};
\node at (0,-3) {$\sfa_\tri(\tau_-^L)$};
\end{scope}

\begin{scope}[xshift=12cm]
\draw[blue] (2.5,0) -- (0,2.5) -- (-2.5,0) -- (0,-2.5) --cycle;
\draw[blue] (0,-2.5) to (0,2.5);
\quiversquare{0,-2.5}{2.5,0}{0,2.5}{-2.5,0};
\squarefrozen{0}{0}{0}{0}
{0}{0}{0}{0};
\squareunfrozen{0}{0}{0}{-1};
\node at (0,-3) {$\sfx_\tri(\tau_-^L)$};
\end{scope}
%6th type
\begin{scope}[yshift=-6.5cm]
\draw[blue] (2.5,0) -- (0,2.5) -- (-2.5,0) -- (0,-2.5) --cycle;
\draw[blue] (0,-2.5) to[bend left=30pt] (0,2.5);
\draw[blue] (0,-2.5) to[bend right=30pt] (0,2.5);
\draw[thick,red,-<-] (-2*5/6,-1*5/6) -- (-1*5/6,0);
\draw[thick,red,-<-] (-2*5/6,1*5/6) -- (-1*5/6,0);
\draw[thick,red,->-] (-1*5/6,0) to[out=0,in=135] (1.25,-1.25);
\node[red] at (0,0.25) {$\tau_-^R$};
\end{scope}

\begin{scope}[xshift=6cm,yshift=-6.5cm]
\draw[blue] (2.5,0) -- (0,2.5) -- (-2.5,0) -- (0,-2.5) --cycle;
\draw[blue] (0,-2.5) to (0,2.5);
\quiversquare{0,-2.5}{2.5,0}{0,2.5}{-2.5,0};
\squarefrozen{2/3}{1/3}{0}{0}
{2/3}{1/3}{2/3}{1/3};
\squareunfrozen{2/3}{1/3}{1/3}{1};
\node at (0,-3) {$\sfa_\tri(\tau_-^R)$};
\end{scope}

\begin{scope}[xshift=12cm,yshift=-6.5cm]
\draw[blue] (2.5,0) -- (0,2.5) -- (-2.5,0) -- (0,-2.5) --cycle;
\draw[blue] (0,-2.5) to (0,2.5);
\quiversquare{0,-2.5}{2.5,0}{0,2.5}{-2.5,0};
\squarefrozen{-1}{0}{0}{0}
{0}{0}{0}{0};
\squareunfrozen{1}{0}{0}{-1};
\node at (0,-3) {$\sfx_\tri(\tau_-^R)$};
\end{scope}
%7th type
\begin{scope}[yshift=-13cm]
\draw[blue] (2.5,0) -- (0,2.5) -- (-2.5,0) -- (0,-2.5) --cycle;
\draw[blue] (0,-2.5) to[bend left=30pt] (0,2.5);
\draw[blue] (0,-2.5) to[bend right=30pt] (0,2.5);
\draw[thick,red,->-] (-2*5/6,-1*5/6) -- (-1*5/6,0);
\draw[thick,red,->-] (-2*5/6,1*5/6) -- (-1*5/6,0);
\draw[thick,red,-<-] (-1*5/6,0) -- (1*5/6,0);
\draw[thick,red,->-] (1*5/6,0) -- (2*5/6,1*5/6);
\draw[thick,red,->-] (1*5/6,0) -- (2*5/6,-1*5/6);
\node[red] at (0,0.5) {$h$};
\end{scope}

\begin{scope}[xshift=6cm,yshift=-13cm]
\draw[blue] (2.5,0) -- (0,2.5) -- (-2.5,0) -- (0,-2.5) --cycle;
\draw[blue] (0,-2.5) to (0,2.5);
\quiversquare{0,-2.5}{2.5,0}{0,2.5}{-2.5,0};
\squarefrozen{2/3}{1/3}{2/3}{1/3}
{1/3}{2/3}{1/3}{2/3};
\squareunfrozen{1/3}{2/3}{1}{1};
\node at (0,-3) {$\sfa_\tri(h)$};
\end{scope}

\begin{scope}[xshift=12cm,yshift=-13cm]
\draw[blue] (2.5,0) -- (0,2.5) -- (-2.5,0) -- (0,-2.5) --cycle;
\draw[blue] (0,-2.5) to (0,2.5);
\quiversquare{0,-2.5}{2.5,0}{0,2.5}{-2.5,0};
\squarefrozen{0}{0}{0}{0}
{0}{-1}{0}{-1};
\squareunfrozen{0}{0}{-1}{1};
\node at (0,-3) {$\sfx_\tri(h)$};
\end{scope}
\end{tikzpicture}
    \caption{Two types of coordinates of component webs on a quadrilateral $Q$. All the webs shown here have weight $1$. (2/2)}
    \label{fig:quadrilateral_coordinates_comparison_2}
\end{figure}
The following states an extension of \cref{thm:cluster_transf_unfrozen} with pinnings/frozen variables, as promised before. 

\begin{thm}\label{thm:cluster_transf}
For any two ideal triangulations $\tri$ and $\tri'$ of $\Sigma$, the coordinate transformation $\sfx_{\tri,\tri'}:=\sfx_{\tri'}\circ \sfx_\tri^{-1}: \bQ^{I(\tri)} \to \bQ^{I(\tri)}$ is a composite of tropical cluster Poisson transformations. In particular, we get an $MC(\Sigma)$-equivariant identification $\sfx_\bullet:\cL^p(\Sigma,\bQ) \xrightarrow{\sim} \X_{\fsl_3,\Sigma}(\bQ^T)$. 
\end{thm}
As a corollary, combining with \cref{lem:cluster_projection}, we get a proof of \cref{thm:cluster_transf_unfrozen}.

\begin{proof}
From \cref{lem:ensemble_bijection} and \cref{prop:ensemble_map}, the statement is true when $\Sigma$ has no puncture (in particular, a quadrilateral). Indeed, the corresponding transformation $\sfa_{\tri,\tri'}:=\sfa_{\tri'}\circ \sfa_\tri^{-1}: \bQ^{I(\tri)} \to \bQ^{I(\tri)}$ is shown to be a composite of tropical cluster $\A$-transformations \cite[Proposition 4.2]{DS20II}. Then $\sfx_{\tri,\tri'}=(\widetilde{p}^{-1})^*\circ \sfa_{\tri,\tri'} \circ \widetilde{p}^*$ is the corresponding composite of tropical cluster $\X$-transformations, since the extended ensemble map commutes with the tropical cluster transformations and is a bijection in this case. 

For the general case, it suffices to consider two triangulations $\tri$, $\tri'$ related by a single flip along an edge $E \in e_{\mathrm{int}}(\tri)$. Let $Q$ be the unique quadrilateral in $\tri$ containing $E$ as a diagonal, and $\Sigma':=\Sigma \setminus \interior Q$ the complement marked surface. It is obvious that the shear coordinates assigned to the vertices outside $Q$ are unchanged. On the other hand, the coordinates assigned to the vertices on $Q$ transform correctly from the argument above under the corresponding coordinate transformation on $\cL^p(Q,\bQ)$. Since $\Sigma$ is obtained by gluing $Q$ with $\Sigma'$ and the shear coordinates are obtained by amalgamating those on $\cL^p(Q,\bQ)$ and $\cL^p(\Sigma',\bQ)$ by  \cref{thm:amalgamation}, the statement follows from the fact that the amalgamations commute with cluster $\X$-transformations \cite[Lemma 2.2]{FG06a}.
\end{proof}

\begin{rem}
For an unpunctured surface $\Sigma$, the fastest way to introduce the coordinate system $\sfx_\tri$ on $\cL^p(\Sigma,\bQ)$ which transforms correctly under the flips would be to define it via the relation \eqref{eq:ensemble_relation} in view of \cref{lem:ensemble_bijection}. 
%On the other hand, our approach to intrinsically introduce the shear coordinates makes the amalgamation formula in \cref{thm:amalgamation} manifest. 
Then, however, it becomes rather difficult to obtain the amalgamation formula in \cref{thm:amalgamation}, since the (tropical) $\A$-coordinates do not behave so simply as the (tropical) $\X$-coordinates under the gluing. Indeed, the following naive diagram does not commute:
\begin{equation*}
    \begin{tikzcd}
    \A_{\fsl_3,\Sigma}(\bQ^T) \ar[r]\ar[d,"\widetilde{p}_\Sigma"'] & \A_{\fsl_3,\Sigma'}(\bQ^T)\ar[d,"\widetilde{p}_{\Sigma'}"] \\
    \X_{\fsl_3,\Sigma}(\bQ^T) \ar[r,"q_{E_L,E_R}"'] & \X_{\fsl_3,\Sigma'}(\bQ^T). 
    \end{tikzcd}
\end{equation*}
Here the top right arrow denotes the quotient map given by the equation $\sfa_i=\sfa_j$ for any pair $\{i,j\}$ of quiver vertices that are identified under the gluing. Actually, we need to ``rescale'' some of the $\A$-coordinates for a correct gluing: see \cite[Section 6.1]{IOS} for a more detail. In particular, the sum $\widetilde{p}_\Sigma^\ast \sfx^\tri_i + \widetilde{p}_\Sigma^\ast \sfx^\tri_j$ does not compute $\widetilde{p}_{\Sigma'}^\ast \sfx^{\tri'}_{\overline{i}}$, where the pair $\{i,j\}$ is amalgamated into $\overline{i}$.
\end{rem}

\subsection{Dynkin involution}\label{subsec:Dynkin}
Let us discuss the equivariance of the shear coordinates under the Dynkin involution \eqref{eq:Dynkin_pinning}. 
The \emph{cluster action} $\ast_\tri$ (see the last paragraph of \cref{sec:appendix}) of the Dynkin involution in the cluster chart associated to $\tri$ is given by the mutation sequence
\begin{align*}
        \mu_\gamma= \sigma_{e(\tri)}\circ \mu_{t(\tri)},
\end{align*}
where $\sigma_{e(\tri)}$ denotes the composite of the transpositions of the labels of the two vertices on each edge of $\tri$, and $\mu_{t(\tri)}$ is the composite of mutations at the vertex on each triangle of $\tri$.  It induces the tropical cluster $\X$-transformation
\begin{align*}
    \ast_\tri^x: 
    \sfx_T &\mapsto -\sfx_T, & \mbox{for $T \in t(\tri)$}, \\
    \sfx_{E,1} &\mapsto \sfx_{E,2}+[\sfx_{T_L}]_+ -[-\sfx_{T_R}]_+, &\\
    \sfx_{E,2} &\mapsto \sfx_{E,1}+[\sfx_{T_R}]_+ -[-\sfx_{T_L}]_+ & \mbox{for $E \in e(\tri)$},
\end{align*}
where we use the local labeling as in \cref{subsec:reconstruction} for each edge $E$.

\begin{prop}\label{prop:Dynkin-cluster}
We have the commutative diagram
\begin{equation*}
    \begin{tikzcd}
    \cL^p(\Sigma,\bQ) \ar[r,"\sfx_\tri"] \ar[d,"\ast"'] & \bQ^{I(\tri)}  \ar[d,"\ast_\tri"] \\
    \cL^p(\Sigma,\bQ) \ar[r,"\sfx_\tri"'] & \bQ^{I(\tri)}. 
    \end{tikzcd}
\end{equation*}
In particular, the orientation-reversing action of the Dynkin involution coincides with the cluster action.
\end{prop}

\begin{proof}
%Since $\ast_\tri$ is a composite of cluster transformations, 
Mutations commute with amalgamations \cite[Lemma 2.2]{FG06a}. Moreover, the permutation term $\sigma_{e(\tri)}$ also commutes with the amalgamation of edge vertices corresponding to the gluing. Hence $\ast_\tri$ commutes with the gluing map. 
It is also clear from the definitions that the Dynkin involution \eqref{eq:Dynkin_pinning} commutes with gluing maps. Therefore it suffices to prove the statement for triangles.

It is easy to verify the equation
\begin{align}\label{eq:Dynkin_check}
    \ast_\tri\circ \sfx_\tri(W) = \sfx_\tri (W^\ast)
\end{align}
for each component web $W$ shown in \cref{fig:triangle_coordinates_comparison} by inspection. 
Consider a disjoint union $W=W_1 \sqcup W_2$ of webs on a triangle $T$, and suppose that the equation \eqref{eq:Dynkin_check} is true for $W=W_1,W_2$. Since sink/source honeycombs cannot co-exist, we have $\{\sgn\sfx_T(W_1),\sgn\sfx_T(W_2)\} \neq \{+,-\}$. 
%Namely, these signs must weakly agree with each other. 
Therefore the coordinate vectors $\sfx_\tri(W_1)$ and $\sfx_\tri(W_2)$ belong to the same cone on which the tropical cluster transformation $\ast_\tri$ is linear. Hence we get
\begin{align*}
    \ast_\tri \circ \sfx_\tri(W)&= \ast_\tri(\sfx_\tri(W_1)+\sfx_\tri(W_1)) \\
    &= \ast_\tri\circ \sfx_\tri(W_1) + \ast_\tri\circ \sfx_\tri(W_2) \\
    &= \sfx_\tri(W_1^\ast) + \sfx_\tri(W_2^\ast) =\sfx_\tri(W^\ast).
\end{align*}
% It suffices to consider the integral points. 
% Recall the gluing rule \eqref{eq:gluing_rule} from a given coordinate tuple $(\sfx_i)_i \in \bZ^{I_\uf(\tri)}$. The cluster Dynkin involution transforms the tuple $(n^+_L,n^-_L,n^+_R,n^-_R)$ into
% \begin{align*}
%     \begin{aligned}
%     \widetilde{n}_{L}^+:&=\sfx_{\topp}+[\sfx_{T_L}]_+-[-\sfx_{T_R}]_+, & &\widetilde{n}_{R}^-:=[-\sfx_{T_R}]_+,
%     \\ 
%     %\widetilde{n}_{L}^-:&=\sfx_{\bott} +[\sfx_{T_R}]_+, & &\widetilde{n}_{R}^+:=0. 
%     \widetilde{n}_{L}^-:&=[-\sfx_{T_L}]_+, & &\widetilde{n}_{R}^+:=\sfx_{\bott}+[\sfx_{T_R}]_+ -[-\sfx_{T_L}]_+. 
%     \end{aligned}
% \end{align*}
% By \cref{rem:gluing_symmetry}, this gluing rule is equivalent to 
% \begin{align*}
%     \begin{aligned}
%     \widetilde{n}_{L}^+:&=[\sfx_{T_L}]_+, & &\widetilde{n}_{R}^-=\sfx_{\topp},
%     \\ 
%     \widetilde{n}_{L}^-:&=\sfx_{\bott}, & &\widetilde{n}_{R}^+=[\sfx_{T_R}]_+. 
%     \end{aligned}
% \end{align*}
% The latter gluing rule is obtained from the original rule by exchanging $\widetilde{n}_Z^\pm \leftrightarrow \widetilde{n}_Z^\mp$ for $Z \in \{L,R\}$, meaning that the resulting lamination has the opposite orientation to the original one. 
\end{proof}

\section{A relation to the graphical basis and quantum duality map}\label{subsec:FG duality}
% We are going to give a recipe for a Fock--Goncharov duality 
% \begin{align*}
%     \X_{\mathfrak{sl}_3,\Sigma}(\bZ^\trop) = \cL^p(\Sigma,\bZ) \to \mathscr{S}_{\mathfrak{sl}_3,\Sigma}^q = \cO_q(\A_{\mathfrak{sl}_3,\Sigma}),
% \end{align*}
% which indeed satisfies the required properties at least when $\Sigma$ is a triangle or a quadrilateral. 

Let $\Sigma$ be a marked surface without punctures. Recall from \cite{IY21} the skein algebra $\mathscr{S}_{\fsl_3,\Sigma}^q$, which is a non-commutative algebra over $\bZ_q:=\bZ[q^{\pm 1/2}]$ consisting of tangled trivalent graphs in $\Sigma$ with endpoints in $\bM$, subject to the $\fsl_3$-skein relations
\begin{align}
		\mathord{
			\ \tikz[baseline=-.6ex,scale=.8]{
				\draw[dashed, fill=white] (0,0) circle(1cm);
				\draw[red, very thick, ->-={.8}{red}] (-45:1) -- (135:1);
				\draw[overarc, ->-={.8}{red}] (-135:1) -- (45:1);
				}
		\ }
		&=q^{2}\mathord{
			\ \tikz[baseline=-.6ex,scale=.8]{
				\draw[dashed, fill=white] (0,0) circle(1cm);
				\draw[very thick, red, ->-] (-45:1) to[out=north west, in=south] (0.4,0) to[out=north, in=south west] (45:1);
				\draw[very thick, red, ->-] (-135:1) to[out=north east, in=south] (-0.4,0) to[out=north, in=south east] (135:1);
			}
		\ }
		+q^{-1}\mathord{
			\ \tikz[baseline=-.6ex,scale=.8]{
				\draw[dashed, fill=white] (0,0) circle [radius=1];
				\draw[very thick, red, ->-] (-45:1) -- (0,-0.4);
				\draw[very thick, red, ->-] (-135:1) -- (0,-0.4);
				\draw[very thick, red, -<-] (45:1) -- (0,0.4);
				\draw[very thick, red, -<-] (135:1) -- (0,0.4);
				\draw[very thick, red, -<-] (0,-0.4) -- (0,0.4);
			}
		\ },\\[3pt]
		\mathord{
			\ \tikz[baseline=-.6ex,scale=.8]{
				\draw[dashed, fill=white] (0,0) circle(1cm);
				\draw[very thick, red, ->-={.8}{red}] (-135:1) -- (45:1);
				\draw[overarc, ->-={.8}{red}] (-45:1) -- (135:1);
			}
		\ }
		&=q^{-2}\mathord{
			\ \tikz[baseline=-.6ex,scale=.8]{
				\draw[dashed, fill=white] (0,0) circle(1cm);
				\draw[very thick, red, ->-] (-45:1) to[out=north west, in=south] (0.4,0) to[out=north, in=south west] (45:1);
				\draw[very thick, red, ->-] (-135:1) to[out=north east, in=south] (-0.4,0) to[out=north, in=south east] (135:1);
			}
		\ }
		+q\mathord{
			\ \tikz[baseline=-.6ex,scale=.8]{
				\draw[dashed, fill=white] (0,0) circle [radius=1];
				\draw[very thick, red, ->-] (-45:1) -- (0,-0.4);
				\draw[very thick, red, ->-] (-135:1) -- (0,-0.4);
				\draw[very thick, red, -<-] (45:1) -- (0,0.4);
				\draw[very thick, red, -<-] (135:1) -- (0,0.4);
				\draw[very thick, red, -<-] (0,-0.4) -- (0,0.4);
			}
		\ },\\[3pt]
		\mathord{
			\ \tikz[baseline=-.6ex, scale=.8]{
				\draw[dashed, fill=white] (0,0) circle [radius=1];
				\draw[very thick, red, -<-={.6}{}] (-45:1) -- (-45:0.5);
				\draw[very thick, red, ->-={.6}{}] (-135:1) -- (-135:0.5);
				\draw[very thick, red, ->-={.6}{}] (45:1) -- (45:0.5);
				\draw[very thick, red, -<-={.6}{}] (135:1) -- (135:0.5);
				\draw[very thick, red, -<-] (45:0.5) -- (135:0.5);
				\draw[very thick, red, ->-] (-45:0.5) -- (-135:0.5);
				\draw[very thick, red, -<-] (45:0.5) -- (-45:0.5);
				\draw[very thick, red, ->-] (135:0.5) -- (-135:0.5);
			}
		\ }
		&=\mathord{
			\ \tikz[baseline=-.6ex, scale=.8]{
				\draw[dashed, fill=white] (0,0) circle [radius=1];
				\draw[very thick, red, -<-] (-45:1) to[out=north west, in=south] (0.4,0) to[out=north, in=south west] (45:1);
				\draw[very thick, red, ->-] (-135:1) to[out=north east, in=south] (-0.4,0) to[out=north, in=south east] (135:1);
			}
		\ }
		+\mathord{
			\ \tikz[baseline=-.6ex, scale=.8,rotate=90]{
				\draw[dashed, fill=white] (0,0) circle [radius=1];
				\draw[very thick, red, ->-] (-45:1) to[out=north west, in=south] (0.4,0) to[out=north, in=south west] (45:1);
				\draw[very thick, red, -<-] (-135:1) to[out=north east, in=south] (-0.4,0) to[out=north, in=south east] (135:1);
			}
		\ },\label{eq:skein:4-gon_rel}\\[3pt]
		\mathord{
			\ \tikz[baseline=-.6ex, scale=.8]{
				\draw[dashed, fill=white] (0,0) circle [radius=1];
				\draw[very thick, red, ->-] (0,-1) -- (0,-0.4);
				\draw[very thick, red, ->-] (0,0.4) -- (0,1);
				\draw[very thick, red, -<-] (0,-0.4) to[out=east, in=south] (0.4,0) to[out=north, in=east] (0,0.4);
				\draw[very thick, red, -<-] (0,-0.4) to[out=west, in=south] (-0.4,0) to[out=north, in=west] (0,0.4);
			}
		\ }
		&=-(q^3+q^{-3})\mathord{
			\ \tikz[baseline=-.6ex, scale=.8]{
				\draw[dashed, fill=white] (0,0) circle [radius=1];
				\draw[very thick, red, ->-] (0,-1) -- (0,1);
			}
		\ },\label{eq:skein_bigon_rel}\\[3pt]
		\mathord{
			\ \tikz[baseline=-.6ex, scale=.8]{
				\draw[dashed, fill=white] (0,0) circle [radius=1];
				\draw[very thick, red, ->-] (0,0) circle [radius=0.4];
			}
		\ }
		&=(q^{6}+1+q^{-6})
		\mathord{
			\ \tikz[baseline=-.6ex, scale=.8]{
				\draw[dashed, fill=white] (0,0) circle [radius=1];
			}
		\ }
		=\mathord{
			\ \tikz[baseline=-.6ex, scale=.8]{
				\draw[dashed, fill=white] (0,0) circle [radius=1];
				\draw[very thick, red, -<-] (0,0) circle [radius=0.4];
			}
		\ },
	\end{align}
and the boundary skein relations 
	\begin{align*}
		\mathord{
			\ \tikz[baseline=-.6ex, yshift=-.4cm]{
				\coordinate (P) at (0,0);
				\draw[very thick, red, ->-] (P) -- (135:1);
				\fill[white] (P) circle [radius=0.2cm];
				\draw[very thick, red, ->-] (P) -- (45:1);
				\draw[dashed] (1,0) arc (0:180:1cm);
				\bline{-1,0}{1,0}{.2}
				%\draw[gray,line width=2pt] (-10,0) -- (10,0);
				\draw[fill=black] (P) circle(2pt);
			\ }
		}
		&=q^{2}
		\mathord{
			\ \tikz[baseline=-.6ex, yshift=-.4cm]{
				\coordinate (P) at (0,0);
				\draw[very thick, red, ->-] (P) -- (45:1);
				\fill[white] (P) circle [radius=0.2cm];
				\draw[very thick, red, ->-] (P) -- (135:1);
				\draw[dashed] (1,0) arc (0:180:1cm);
				\bline{-1,0}{1,0}{.2}
				\draw[fill=black] (P) circle(2pt);
			\ }
		}
% 		&\mathord{
% 			\ \tikz[baseline=-.6ex, scale=.1, yshift=-4cm]{
% 				\coordinate (P) at (0,0);
% 				\draw[very thick, red, -<-] (P) -- (135:10);
% 				\fill[white] (P) circle [radius=2cm];
% 				\draw[very thick, red, -<-] (P) -- (45:10);
% 				\draw[dashed] (10,0) arc (0:180:10cm);
% 				\bline{-10,0}{10,0}{2}
% 				\draw[fill=black] (P) circle [radius=20pt];
% 			\ }
% 		}
% 		&=q^{2}
% 		\mathord{
% 			\ \tikz[baseline=-.6ex, scale=.1, yshift=-4cm]{
% 				\coordinate (P) at (0,0);
% 				\draw[very thick, red, -<-] (P) -- (45:10);
% 				\fill[white] (P) circle [radius=2cm];
% 				\draw[very thick, red, -<-] (P) -- (135:10);
% 				\draw[dashed] (10,0) arc (0:180:10cm);
% 				\bline{-10,0}{10,0}{2}
% 				\draw[fill=black] (P) circle [radius=20pt];
% 			\ }
% 		}\label{rel:parallel}\\
		&\mathord{
			\ \tikz[baseline=-.6ex, yshift=-.4cm]{
				\coordinate (P) at (0,0);
				\draw[very thick, red, -<-] (P) -- (135:1);
				\fill[white] (P) circle [radius=.2cm];
				\draw[very thick, red, ->-] (P) -- (45:1);
				\draw[dashed] (1,0) arc (0:180:1cm);
				\bline{-1,0}{1,0}{.2}
				\draw[fill=black] (P) circle [radius=2pt];
			\ }
		}
		&=q
		\mathord{
			\ \tikz[baseline=-.6ex, yshift=-.4cm]{
				\coordinate (P) at (0,0);
				\draw[very thick, red, ->-] (P) -- (45:1);
				\fill[white] (P) circle [radius=.2cm];
				\draw[very thick, red, -<-] (P) -- (135:1);
				\draw[dashed] (1,0) arc (0:180:1cm);
				\bline{-1,0}{1,0}{.2}
				\draw[fill=black] (P) circle [radius=2pt];
			\ }
		}\\[3pt]
% 		&\mathord{
% 			\ \tikz[baseline=-.6ex, scale=.1, yshift=-4cm]{
% 				\coordinate (P) at (0,0);
% 				\draw[very thick, red, ->-] (P) -- (135:10);
% 				\fill[white] (P) circle [radius=2cm];
% 				\draw[very thick, red, -<-] (P) -- (45:10);
% 				\draw[dashed] (10,0) arc (0:180:10cm);
% 				\bline{-10,0}{10,0}{2}
% 				\draw[fill=black] (P) circle [radius=20pt];
% 			\ }
% 		}
% 		&=q
% 		\mathord{
% 			\ \tikz[baseline=-.6ex, scale=.1, yshift=-4cm]{
% 				\coordinate (P) at (0,0);
% 				\draw[very thick, red, -<-] (P) -- (45:10);
% 				\fill[white] (P) circle [radius=2cm];
% 				\draw[very thick, red, ->-] (P) -- (135:10);
% 				\draw[dashed] (10,0) arc (0:180:10cm);
% 				\bline{-10,0}{10,0}{2}
% 				\draw[fill=black] (P) circle [radius=20pt];
% 			\ }
% 		}\label{rel:antiparallel}\\
		\mathord{
			\ \tikz[baseline=-.6ex, yshift=-.4cm]{
				\coordinate (P) at (0,0);
				\node[inner sep=0, outer sep=0, circle, fill=black] (R) at (45:.7) {};
				\node[inner sep=0, outer sep=0, circle, fill=black] (L) at (135:.7) {};
				\draw[very thick, red, ->-] (L) -- (R);
				\draw[very thick, red, -<-={.7}{}] (P) -- (L);
				\draw[very thick, red, -<-={.7}{}] (R) -- (45:1);
				\draw[very thick, red, ->-={.7}{}] (L) -- (135:1);
				\fill[white] (P) circle [radius=.2cm];
				\draw[very thick, red, ->-={.7}{}] (P) -- (R);
				\draw[dashed] (1,0) arc (0:180:1cm);
				\bline{-1,0}{1,0}{.2}
				\draw[fill=black] (P) circle [radius=2pt];
			\ }
		}
		&=
		\mathord{
			\ \tikz[baseline=-.6ex, yshift=-.4cm]{
				\coordinate (P) at (0,0) {};
				\draw[very thick, red, ->-] (P) -- (135:1);
				\fill[white] (P) circle [radius=.2cm];
				\draw[very thick, red, -<-] (P) -- (45:1);
				\draw[dashed] (1,0) arc (0:180:1cm);
				\bline{-1,0}{1,0}{.2}
				\draw[fill=black] (P) circle [radius=2pt];
			\ }
		}
% 		&\mathord{
% 			\ \tikz[baseline=-.6ex, scale=.1, yshift=-4cm]{
% 				\coordinate (P) at (0,0);
% 				\node[inner sep=0, outer sep=0, circle, fill=black] (R) at (45:7) {};
% 				\node[inner sep=0, outer sep=0, circle, fill=black] (L) at (135:7) {};
% 				\draw[very thick, red, -<-] (L) -- (R);
% 				\draw[very thick, red, ->-={.7}{}] (P) -- (L);
% 				\draw[very thick, red, ->-={.7}{}] (R) -- (45:10);
% 				\draw[very thick, red, -<-={.7}{}] (L) -- (135:10);
% 				\fill[white] (P) circle [radius=2cm];
% 				\draw[very thick, red, -<-={.7}{}] (P) -- (R);
% 				\draw[dashed] (10,0) arc (0:180:10cm);
% 				\bline{-10,0}{10,0}{2}
% 				\draw[fill=black] (P) circle [radius=20pt];
% 			\ }
% 		}
% 		&=
% 		\mathord{
% 			\ \tikz[baseline=-.6ex, scale=.1, yshift=-4cm]{
% 				\coordinate (P) at (0,0);
% 				\draw[very thick, red, -<-] (P) -- (135:10);
% 				\fill[white] (P) circle [radius=2cm];
% 				\draw[very thick, red, ->-] (P) -- (45:10);
% 				\draw[dashed] (10,0) arc (0:180:10cm);
% 				\bline{-10,0}{10,0}{2}
% 				\draw[fill=black] (P) circle [radius=20pt];
% 			\ }
% 		}\label{rel:pbigon}\\
		&\mathord{
			\ \tikz[baseline=-.6ex, yshift=-.4cm]{
				\coordinate (P) at (0,0);
				\node[inner sep=0, outer sep=0, circle, fill=black] (C) at (90:.7) {};
				\draw[very thick, red, ->-] (P) to[out=north west, in=west] (C);
				\fill[white] (P) circle [radius=.2cm];
				\draw[very thick, red, ->-] (P) to[out=north east, in=east] (C);
				\draw[very thick, red, -<-={.7}{}] (C) -- (90:1);
				\draw[dashed] (1,0) arc (0:180:1cm);
				\bline{-1,0}{1,0}{.2}
				\draw[fill=black] (P) circle [radius=2pt];
			\ }
		}
		&=
% 		&\mathord{
% 			\ \tikz[baseline=-.6ex, scale=.1, yshift=-4cm]{
% 				\coordinate (P) at (0,0);
% 				\node[inner sep=0, outer sep=0, circle, fill=black] (C) at (90:7) {};
% 				\draw[very thick, red, -<-] (P) to[out=north west, in=west] (C);
% 				\fill[white] (P) circle [radius=2cm];
% 				\draw[very thick, red, -<-] (P) to[out=north east, in=east] (C);
% 				\draw[very thick, red, ->-={.7}{}] (C) -- (90:10);
% 				\draw[dashed] (10,0) arc (0:180:10cm);
% 				\bline{-10,0}{10,0}{2}
% 				\draw[fill=black] (P) circle [radius=20pt];
% 			\ }
% 		}
% 		&=0\label{rel:degbigon}\\
		\mathord{
			\ \tikz[baseline=-.6ex, yshift=-.4cm]{
				\coordinate (P) at (0,0);
				\node[inner sep=0, outer sep=0, circle, fill=black] (C) at (90:.7) {};
				\draw[very thick, red, -<-] (P) to[out=north west, in=west] (C);
				\fill[white] (P) circle [radius=.2cm];
				\draw[red, very thick] (P) to[out=north east, in=east] (C);
				\draw[dashed] (1,0) arc (0:180:1cm);
				\bline{-1,0}{1,0}{.2}
				\draw[fill=black] (P) circle [radius=2pt];
			\ }
		}
		=0,
% 		&\mathord{
% 			\ \tikz[baseline=-.6ex, scale=.1, yshift=-4cm]{
% 				\coordinate (P) at (0,0) {};
% 				\node[inner sep=0, outer sep=0, circle, fill=black] (C) at (90:7) {};
% 				\draw[very thick, red, ->-] (P) to[out=north west, in=west] (C);
% 				\fill[white] (P) circle [radius=2cm];
% 				\draw[red, very thick] (P) to[out=north east, in=east] (C);
% 				\draw[dashed] (10,0) arc (0:180:10cm);
% 				\bline{-10,0}{10,0}{2}
% 				\draw[fill=black] (P) circle [radius=20pt];
% 			\ }
% 		}
% 		&=0\label{rel:pcircle}
\end{align*}
together with their Dynkin involutions. We included the square-root parameter $q^{1/2}$ so that we can consider the \emph{simultaneous crossing} (or the \emph{Weyl normalization}) as
\begin{align*}
    \mathord{
			\ \tikz[baseline=-.6ex, yshift=-.4cm]{
				\coordinate (P) at (0,0);
				\draw[very thick, red, ->-] (P) -- (135:1);
				\fill[white] (P) circle [radius=0.2cm];
				\draw[very thick, red, ->-] (P) -- (45:1);
				\draw[dashed] (1,0) arc (0:180:1cm);
				\bline{-1,0}{1,0}{.2}
				%\draw[gray,line width=2pt] (-10,0) -- (10,0);
				\draw[fill=black] (P) circle(2pt);
			\ }
		}
		&=q
		\mathord{
			\ \tikz[baseline=-.6ex, yshift=-.4cm]{
				\coordinate (P) at (0,0);
				\draw[very thick, red, ->-] (P) -- (45:1);
				%\fill[white] (P) circle [radius=0.2cm];
				\draw[very thick, red, ->-] (P) -- (135:1);
				\draw[dashed] (1,0) arc (0:180:1cm);
				\bline{-1,0}{1,0}{.2}
				\draw[fill=black] (P) circle(2pt);
			\ }
		}
% 		&\mathord{
% 			\ \tikz[baseline=-.6ex, scale=.1, yshift=-4cm]{
% 				\coordinate (P) at (0,0);
% 				\draw[very thick, red, -<-] (P) -- (135:10);
% 				\fill[white] (P) circle [radius=2cm];
% 				\draw[very thick, red, -<-] (P) -- (45:10);
% 				\draw[dashed] (10,0) arc (0:180:10cm);
% 				\bline{-10,0}{10,0}{2}
% 				\draw[fill=black] (P) circle [radius=20pt];
% 			\ }
% 		}
% 		&=q^{2}
% 		\mathord{
% 			\ \tikz[baseline=-.6ex, scale=.1, yshift=-4cm]{
% 				\coordinate (P) at (0,0);
% 				\draw[very thick, red, -<-] (P) -- (45:10);
% 				\fill[white] (P) circle [radius=2cm];
% 				\draw[very thick, red, -<-] (P) -- (135:10);
% 				\draw[dashed] (10,0) arc (0:180:10cm);
% 				\bline{-10,0}{10,0}{2}
% 				\draw[fill=black] (P) circle [radius=20pt];
% 			\ }
% 		}\label{rel:parallel}\\
		&\mathord{
			\ \tikz[baseline=-.6ex, yshift=-.4cm]{
				\coordinate (P) at (0,0);
				\draw[very thick, red, -<-] (P) -- (135:1);
				\fill[white] (P) circle [radius=.2cm];
				\draw[very thick, red, ->-] (P) -- (45:1);
				\draw[dashed] (1,0) arc (0:180:1cm);
				\bline{-1,0}{1,0}{.2}
				\draw[fill=black] (P) circle [radius=2pt];
			\ }
		}
		&=q^{1/2}
		\mathord{
			\ \tikz[baseline=-.6ex, yshift=-.4cm]{
				\coordinate (P) at (0,0);
				\draw[very thick, red, ->-] (P) -- (45:1);
				%\fill[white] (P) circle [radius=.2cm];
				\draw[very thick, red, -<-] (P) -- (135:1);
				\draw[dashed] (1,0) arc (0:180:1cm);
				\bline{-1,0}{1,0}{.2}
				\draw[fill=black] (P) circle [radius=2pt];
			\ }
		}\ .
\end{align*}
It is proved in \cite{IY21} that the localized skein algebra $\mathscr{S}_{\fsl_3,\Sigma}^q[\partial^{-1}]$ along the oriented arcs parallel to boundary intervals 
is contained in the quantum cluster algebra \cite{BZ} $\mathscr{A}^q_{\fsl_3,\Sigma}$ associated with a certain choice of compatibility pairs over the mutation class $\sfs(\fsl_3,\Sigma)$
% \footnote{Here we need to modify the realization $i_{\mathrm{0}}: \mathscr{S}_{\fsl_3,\Sigma}^q[\partial^{-1}] \hookrightarrow \mathscr{A}^q_{\fsl_3,\Sigma}$ given in \cite{IY21} into the one $i:=\ast \circ i_{0}$ by composing the Dynkin involution, in order to make it fit in with the Fock--Goncharov duality. See \cref{rem:GS_GHKK}. Thanks to the equivariance under the Dynkin involution \cite[Theorem 2]{IY21}, this modification does not change the cluster structure. In the modified realization, the web cluster associated with an ideal triangulation (\emph{i.e.}, $\boldsymbol{s}(T)=+$ for all $T \in t(\tri)$) consists of \emph{source} triads and oriented arcs, where the latter correspond to the cluster variables on their \emph{tails}.
% }.
At least in the classical limit $q=1$, we have the equalities \cite{IOS}
\begin{align}\label{eq:S=A=U_classical}
    \mathscr{S}_{\fsl_3,\Sigma}^1[\partial^{-1}] = \mathscr{A}_{\fsl_3,\Sigma} = \cO(\A_{\fsl_3,\Sigma}).
\end{align}
The skein algebra $\mathscr{S}_{\fsl_3,\Sigma}^q$ has a natural $\bZ_q$-basis $\mathsf{BWeb}_{\fsl_3,\Sigma}$ consisting of \emph{non-elliptic flat trivalent graphs}. Here a flat trivalent graph is an immersed oriented uni-trivalent graph on $\Sigma$ such that each univalent vertex lie in $\bM$, and the other part is embedded into $\interior \Sigma$. In particular, it is required to have simultaneous crossings at each special point. 
It is said to be non-elliptic if it has none of the following \emph{elliptic faces}:
\begin{align}\label{eq:elliptic_face_IY}
    &\begin{tikzpicture}[scale=.8]
		\draw[dashed, fill=white] (0,0) circle [radius=1];
		\draw[very thick, red, fill=pink!60] (0,0) circle [radius=0.4];
	\end{tikzpicture}
	\hspace{2em}
	\begin{tikzpicture}[scale=.8]
		\draw[dashed, fill=white] (0,0) circle [radius=1];
		\draw[very thick, red] (0,-1) -- (0,-0.4);
		\draw[very thick, red] (0,0.4) -- (0,1);
		\draw[very thick, red, fill=pink!60] (0,0) circle [radius=0.4];
	\end{tikzpicture}
	\hspace{2em}
	\begin{tikzpicture}[scale=.8]
		\draw[dashed, fill=white] (0,0) circle [radius=1];
		\draw[very thick, red] (-45:1) -- (-45:0.5);
		\draw[very thick, red] (-135:1) -- (-135:0.5);
		\draw[very thick, red] (45:1) -- (45:0.5);
		\draw[very thick, red] (135:1) -- (135:0.5);
		\draw[very thick, red, fill=pink!60] (45:0.5) -- (135:0.5) -- (225:0.5) -- (315:0.5) -- cycle;
	\end{tikzpicture} 
	\hspace{2em}
	\begin{tikzpicture}[scale=.8,baseline=-0.7cm]
		\coordinate (P) at (0,0) {};
		\coordinate (C) at (90:.6) {};
		\draw[very thick, red, fill=pink!60] (P) to[out=170, in=180] (C) to[out=0, in=10] (P);
		\draw[very thick, red] (C) -- (90:1);
		\draw[dashed] (1,0) arc (0:180:1cm);
		\bline{-1,0}{1,0}{.2}
		\draw[fill=black] (P) circle [radius=2pt];
	\end{tikzpicture}
	\hspace{2em}
	\begin{tikzpicture}[scale=.8,baseline=-0.7cm]
		\coordinate (P) at (0,0) {};
		\coordinate (C) at (90:.6) {};
		\draw[very thick, red, fill=pink!60] (P) to[out=170, in=180] (C) to[out=0, in=10] (P);
		\draw[dashed] (1,0) arc (0:180:1cm);
		\bline{-1,0}{1,0}{.2}
		\draw[fill=black] (P) circle [radius=2pt];
	\end{tikzpicture}
\end{align}
Elements of $\mathsf{BWeb}_{\fsl_3,\Sigma}$ are also called the \emph{basis webs}. 
We are going to relate the integral $\fsl_3$-laminations with pinnings to the basis webs.

\begin{dfn}[negative $\bM$-shifting of webs (cf.~``moving left'' in {\cite[Figure 2]{LY22}})]
Given a web $W$ on $\Sigma$ in the sense of \cref{subsec:webs}, let $W^{\bM} \in \mathscr{S}^q_{\fsl_3,\Sigma}$ be the flat trivalent graph obtained by shifting the endpoints of $W$ to the nearest special point in the negative direction along the boundary (with respect to the orientation induced from $\Sigma$), and taking the simultaneous crossing. See \cref{fig:elementary laminations}. 
\end{dfn}
% Given an integral coweight $\nu_E=\nu_E^+\varpi_1^\vee + \nu_E^-\varpi_2^\vee \in \mathsf{P}^\vee$ assigned to a boundary interval $E \in \mathbb{B}$, define
% \begin{align*}
%     \mathbb{I}_\X^q(\nu_E):= \left[ (e_E^+)^{\nu_E^+}(e_E^-)^{\nu_E^-}\right] \in \mathscr{S}_{\fsl_3,\Sigma}^q[\partial^{-1}],
% \end{align*}
% where $e_E^+$ (resp. $e_E^-$) denotes the oriented arc parallel to $E$ oriented towards the initial (resp. terminal) point, and $[-]$ denotes the simultaneous crossing of their product. Note that $\mathbb{I}_\X^q(\nu_E) \in \mathscr{S}_{\fsl_3,\Sigma}^q$ if and only if $\nu_E \in \mathsf{P}_+^\vee:= \bZ_+\varpi_1^\vee + \bZ_+ \varpi_2^\vee$. 

For an integral $\fsl_3$-lamination with pinnings $(\hL,\nu) \in \cL^p(\Sigma,\bZ)$, represent $\hL$ by a non-elliptic $\fsl_3$-web $W$ only with components with weight one, and define 
\begin{align*}
    \mathbb{I}_\X^q(\hL) := \left[ W^{\bM} \cdot \prod_{E \in \mathbb{B}} (e_E^+)^{\nu_E^+}(e_E^-)^{\nu_E^-} \right] \in \mathscr{S}_{\fsl_3,\Sigma}^q[\partial^{-1}].
\end{align*}
Here $\nu_E=\nu_E^+\varpi_1^\vee + \nu_E^-\varpi_2^\vee \in \mathsf{P}^\vee$ for each $E \in \bB$, and the symbol $[-]$ stands for the Weyl normalization. 
Then $\mathbb{I}_\X^q(\hL)$ does not depend on the choice of the representative $W$, since the loop parallel-move is also realized in the skein algebra (by using the Reidemeister II move twice), and the boundary H-move exactly corresponds to the third boundary skein relation. Moreover, it is a basis web since the two notions of elliptic faces correspond to each other via the shift of endpoints. 
%Thus we get:

Note that $\mathbb{I}_\X^q(\hL) \in \mathscr{S}_{\fsl_3,\Sigma}^q$ if and only if $\nu_E \in \mathsf{P}_+^\vee:= \bZ_+\varpi_1^\vee + \bZ_+ \varpi_2^\vee$ for all $E \in \bB$. In this case, 
we say that $(\hL,(\nu_E)) \in \cL^p(\Sigma,\bZ)$ is \emph{dominant}. Let $\cL^p(\Sigma,\bZ)_+ \subset \cL^p(\Sigma,\bZ)$ denote the subspace of dominant integral $\fsl_3$-laminations. From the above discussion, we get:

\begin{thm}\label{thm:basis_correspondence}
Assume that $\Sigma$ has no punctures. Then we have an $MC(\Sigma)\times \mathrm{Out}(SL_3)$-equivariant bijection
\begin{align*}
    \mathbb{I}_\X^q: \cL^p(\Sigma,\bZ)_+ \xrightarrow{\sim} \mathsf{BWeb}_{\fsl_3,\Sigma} \subset \mathscr{S}_{\fsl_3,\Sigma}^q.
\end{align*}
%where $\cL^p(\Sigma,\bZ)_+ \hookrightarrow \cL^p(\Sigma,\bZ)$ denotes the subspace of dominant integral $\fsl_3$-laminations. 
Moreover, it is extended to a map $\mathbb{I}_\X^q: \cL^p(\Sigma,\bZ) \hookrightarrow \mathscr{S}_{\fsl_3,\Sigma}^q[\partial^{-1}]$, whose image again gives a $\bZ_q$-basis.
\end{thm}
The latter correspondence should be a basic ingredient for a construction of the Fock--Goncharov's \emph{quantum duality map} \cite{FG09} (see \cite[Conjecture 4.14]{Qin21} for a finer formulation as well as \cite{DM}), which requires a basis of the quantum upper cluster algebra parametrized by the tropical set $\X_{\fsl_3,\Sigma}(\bZ^T) = \cL^p(\Sigma,\bZ)$ with certain positivity properties. Let us interpret \cref{thm:basis_correspondence} in this context.

\smallskip
\paragraph{\textbf{Langlands dual coordinates.}}
It turns out that it is more convenient to use a slight modification\footnote{In the language of Goncharov--Shen \cite{GS19}, it amounts to take the decoration at the \emph{terminal} endpoint of a boundary interval rather than its \emph{initial} endpoint along the boundary orientation to make a pinning.} 
of frozen shear coordinates to make the correspondence suited to the Fock--Goncharov conjecture. For an ideal triangulation $\tri$ of $\Sigma$, we define the \emph{Langlands dual coordinates}
\begin{align*}
    \check{\sfx}_\tri=(\check{\sfx}_i^\tri)_{i \in I(\tri)}: \cL^p(\Sigma,\bQ) \xrightarrow{\sim} \bQ^{I(\tri)},
\end{align*}
as follows. For $i \in I_\uf(\tri)$, let $\check{\sfx}_i^\tri:=\sfx_i^\tri$. For $E \in \bB$, we define the frozen coordinates on $E$ by
\begin{align*}
    &\check{\sfx}_{E,1}^\tri(\hL,\nu):=\nu_E^+ + \check{\alpha}_E^+(\hL) + [-\sfx_T(\hL)]_+, \\
    &\check{\sfx}_{E,2}^\tri(\hL,\nu):=\nu_E^- + \check{\alpha}_E^-(\hL).
\end{align*}
Here $T$ is the unique triangle having $E$ as an edge; $\check{\alpha}_E^+(\hL)$ (resp. $\check{\alpha}_E^-(\hL)$) is the total weight of the oriented corner arcs in $\cW \cap T$ bounding the \underline{terminal} endpoint of $E$ in the counter-clockwise (resp. clockwise) direction. Compare with \eqref{eq:frozen_coordinates}. The map $\check{\sfx}_\tri$ gives a bijection, which can be verified similarly to the proof of \cref{prop:shear_P}.

\begin{figure}[h]
\centering
\begin{tikzpicture}
\draw[blue] (0,0) coordinate(A) -- (0,3) coordinate(B) --++(-150:3) coordinate(C) --cycle;
\fill[gray!30] (0,-0.5) -- (0,3.5) -- (0.2,3.5) -- (0.2,-0.5) --cycle;
\draw[thick] (0,-0.5) -- (0,3.5);
\fill(0,0) circle(2pt);
\fill(0,3) circle(2pt);
\draw[red,very thick,-<-] ($(B)!0.55!(C)$) arc(-150:-90:3*0.55);
\draw[red,very thick,->-] ($(B)!0.45!(C)$) arc(-150:-90:3*0.45);
\draw[red,very thick,dotted] ($(B)!0.55!(C)$) arc(30:60:1);
\draw[red,very thick,dotted] ($(B)!0.45!(C)$) arc(30:60:1);
\node at (0.5,1.5) {$E$};
\node at (0.5,3) {$m$};
\node[below] at ($(A)!0.5!(C)$) {$T$};
\node[red] at (-0.5,2.3) {$\alpha_E^+$};
\node[red] at (-0.8,1.1) {$\alpha_E^-$};
\end{tikzpicture}
    \caption{The corner arcs relevant to the Langlands dual coordinate.}
    \label{fig:boundary_coordinate_dual}
\end{figure}
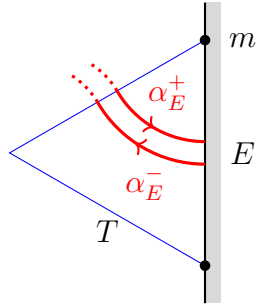
We define the \emph{Langlands dual ensemble map} 
\begin{align}\label{eq:ensemble_Langlands}
    \check{p}: \cL^a(\Sigma,\bQ) \to \cL^p(\Sigma,\bQ)
\end{align}
by forgetting the peripheral components, and defining the pinning $\nu_{E}^+ \in \bQ$ (resp. $\nu_{E}^- \in \bQ$) to be the weight of the peripheral component around the \underline{terminal} endpoint of $E$ in the counter-clockwise (resp. clockwise) direction. 
The name ``Langlands dual'' is inspired by the following property:
\begin{prop}\label{prop:ensemble_Langlands}
The Langlands dual ensemble map \eqref{eq:ensemble_Langlands} satisfies
\begin{align*}
    \check{p}^\ast \sfx_i^\tri = \sum_{j \in I(\tri)} (\ve_{ij}^\tri - m_{ij})\sfa_j^\tri   
\end{align*}
for any ideal triangulation. 
\end{prop}
Compare with \eqref{eq:ensemble_map}, and observe that the presentation matrix is changed to the Langlands dual $-(\ve^\tri+M)^\top = \ve^\tri-M$. The verification of \cref{prop:ensemble_Langlands} is similar to \cref{prop:ensemble_map}, which is left to the reader.

% As a small matter of convention, we are going to use the min-plus semifield $\bZ^\trop=(\bZ,\min,+)$ in order to make it fit in with the convention used in the Goncharov--Shen's construction of canonical basis \cite{GS15} 
% (see \cref{rem:GS_GHKK}). The semifield isomorphism $(-1):\bZ^T \xrightarrow{\sim} \bZ^\trop$ induces a PL isomorphism $\iota: \X_{\fsl_3,\Sigma}(\bZ^T) \xrightarrow{\sim} \X_{\fsl_3,\Sigma}(\bZ^\trop)$, and we get a PL isomorphism
% \begin{align*}
%     x_\tri:=\iota \circ \sfx_\tri: \cL^p(\Sigma,\bZ) \xrightarrow{\sim} \X_{\fsl_3,\Sigma}(\bZ^\trop)
% \end{align*}
% given by minus the shear coordinates\footnote{Compare this to the tropical $\A$-coordinates used in \cite{GS15}, which are $-1/2$ times the intersection number in the $A_1$ case (see \cite[(82)]{GS15}). The minus sign should be inherited to the tropical $\X$-side via the ensemble map.}. 

\smallskip 
For each $v \in \bExch_{\fsl_3,\Sigma}$ and $k \in I$, the \emph{elementary lamination} is the tropical point $\ell_k^{(v)} \in \X_{\fsl_3,\Sigma}(\bZ^T)$ characterized by $\check{\sfx}_i^{(v)}(\ell_k^{(v)})=\delta_{i,k}$. We have the cone
\begin{align*}
    \cC^+_{(v)}:= \mathrm{span}_{\bR_+}\{ \ell_k^{(v)}\mid k \in I\} = \{ \ell \in \X_{\fsl_3,\Sigma}(\bR^T) \mid \check{\sfx}_k^{(v)}(\ell) \geq 0 \mbox{ for all $k \in I$}\}
\end{align*}
and its integral points $\cC^+_{(v)}(\bZ):=\cC^+_{(v)} \cap \X_{\fsl_3,\Sigma}(\bZ^T)$. 
The following gives a partial verification of a condition for the quantum duality map:

\begin{lem}\label{lem:quantum_duality}
For any elementary lamination $\ell_k^{(v)}$ associated with a labeled $\fsl_3$-triangulation $v=(\tri,\ell) \in \bExch_{\fsl_3,\Sigma}$, the element $\mathbb{I}_\X^q(\ell_k^{(v)})$ coincides with the quantum cluster variable $A_k^{(v)} \in \mathscr{A}^q_{\fsl_3,\Sigma}$. In particular, any point $\ell=\sum_k x_k\ell_k^{(v)} \in \cC^+_{(v)}(\bZ)$ gives a quantum cluster monomial $\left[ \prod_k (A_k^{(v)})^{x_k} \right]$.
\end{lem}

\begin{proof}
Via the isomorphism 
\begin{align*}
    \check{\sfx}_\tri^{-1}:\X_{\fsl_3,\Sigma}(\bZ^T) \cong \cL^p(\Sigma,\bZ),
\end{align*}
the elementary laminations $\ell_k^{(v)}$ for unfrozen $k \in I(\tri)_\uf$ corresponds to the integral $\fsl_3$-laminations as shown in the left of \cref{fig:elementary laminations}. The elementary laminations $\ell_k^{(v)}$ for frozen $k=i^s(E) \in I(\tri)_\f$ with $E \in \mathbb{B}$ and $s \in \{1,2\}$ corresponds to the pinning data $\nu_E=\varpi_s^\vee$. Then via the quantum duality map
\begin{align*}
    \mathbb{I}_\X^q: \cL^p(\Sigma,\bZ)_+ \xrightarrow{\sim} \mathsf{BWeb}_{\fsl_3,\Sigma} \subset \mathscr{S}_{\fsl_3,\Sigma}^q %\subset \mathscr{A}_{\fsl_3,\Sigma},
\end{align*}
these laminations are sent to the \emph{elementary webs} associated with $\tri$ in the sense of \cite{IY21}. They corresponds to the quantum cluster variables (\cite[Section 5]{IY21}). 
\end{proof}

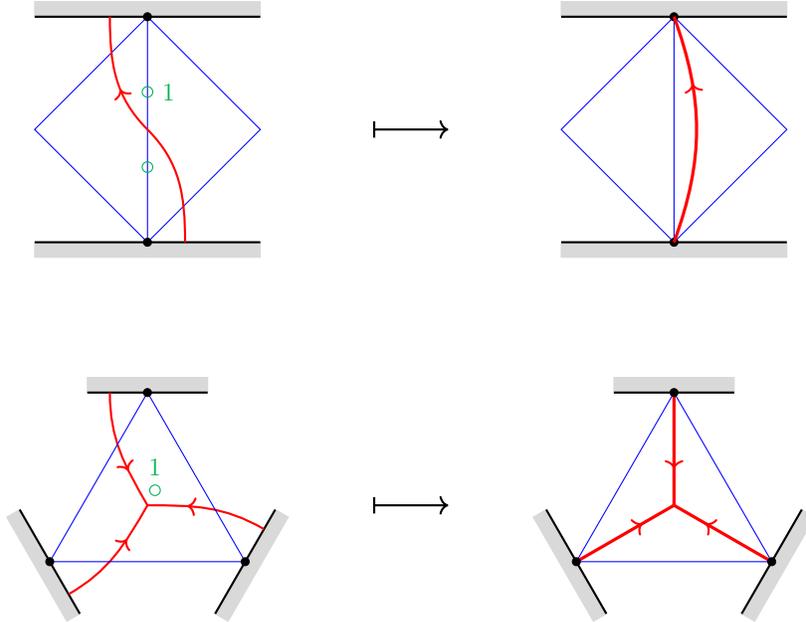
\begin{figure}[htbp]
\centering
\begin{tikzpicture}
\bline{-1.5,-1.5}{1.5,-1.5}{0.2};
\tline{-1.5,1.5}{1.5,1.5}{0.2};
\draw[blue] (-1.5,0) -- (0,-1.5) -- (1.5,0) -- (0,1.5) -- cycle;
\draw[blue] (0,1.5) -- (0,-1.5);
{\color{mygreen}
\draw[mygreen]($(0,1.5)!1/3!(0,-1.5)$) circle(2pt)node[right=0.2em,scale=0.8]{$1$};
\draw[mygreen]($(0,1.5)!2/3!(0,-1.5)$) circle(2pt);
}
\filldraw (0,1.5) circle(1.5pt);
\filldraw (0,-1.5) circle(1.5pt);
\draw[red,thick,->-={0.7}{}] (0.5,-1.5) to[out=90,in=-45] (0,0) to[out=135,in=-90] (-0.5,1.5);

{\begin{scope}[xshift=7cm]
\bline{-1.5,-1.5}{1.5,-1.5}{0.2};
\tline{-1.5,1.5}{1.5,1.5}{0.2};
\draw[blue] (-1.5,0) -- (0,-1.5) -- (1.5,0) -- (0,1.5) -- cycle;
\draw[blue] (0,1.5) -- (0,-1.5);
\filldraw (0,1.5) circle(1.5pt);
\filldraw (0,-1.5) circle(1.5pt);
\draw[red,very thick,->-={0.7}{}] (0,-1.5) to[bend right=20] (0,1.5);
\end{scope}}
\draw[thick,|->] (3,0) -- (4,0);

{\begin{scope}[yshift=-5cm]
\foreach \i in {0,120,240}
{\begin{scope}[rotate=\i]
\tline{-0.8,1.5}{0.8,1.5}{0.2};
\draw[red,thick,-<-={0.4}{}] (0,0) to[out=120,in=-90] (-0.5,1.5);
\end{scope}}
\draw[mygreen] (0.1,0.2) circle(2pt) node[above=0.2em,scale=0.8]{$1$};
\draw[blue] (-30:1.5) -- (90:1.5) -- (210:1.5) --cycle;
\foreach \i in {-30,90,210}
\filldraw (\i:1.5) circle(1.5pt);
\draw[thick,|->] (3,0) -- (4,0);
\end{scope}}

{\begin{scope}[xshift=7cm,yshift=-5cm]
\foreach \i in {0,120,240}
{\begin{scope}[rotate=\i]
\tline{-0.8,1.5}{0.8,1.5}{0.2};
\draw[red,very thick,-<-={0.4}{}] (0,0) -- (0,1.5);
\end{scope}}
\draw[blue] (-30:1.5) -- (90:1.5) -- (210:1.5) --cycle;
\foreach \i in {-30,90,210}
\filldraw (\i:1.5) circle(1.5pt);
\end{scope}}
\end{tikzpicture}
    \caption{Negative $\bM$-shifting of elementary laminations associated with a triangulation. Here exactly one of the Langlands dual coordinates $\check{\sfx}_i^\tri$ is $+1$, while the others are zero (including the frozen ones).}
    \label{fig:elementary laminations}
\end{figure}

\begin{rem}\label{rem:decorated_triangulation}
By the equivariance of the map $\mathbb{I}_\X^q$ under the Dynkin involution, the above lemma can be immediately generalized for \emph{decorated triangulations} (see \cite[Section 1]{IY21}). 
\end{rem}

\begin{rem}\label{rem:GS_GHKK}
% \begin{itemize}
%     \item The $\fsl_3$-laminations having exactly one $+1$ coordinates in $\sfx_\tri$ are shown in \cref{fig:positive laminations}. Here we need to equip the laminations with non-trivial pinning data in order to cancel the frozen coordinates. 
%     \item For a mutation class $\sfs$ of seeds, let $-\sfs$ be the mutation class having minus the exchange matrices of $\sfs$ (in other words, opposite quivers). Then we have a PL isomorphism $\X_{\sfs}(\bZ^\trop) \cong \X_{-\sfs}(\bZ^T)$ which gives the identity on each cluster chart. The image  $\cC_{(v)}^+(\bZ) \subset \X_{-\sfs}(\bZ^T)$ of our positive cone under this isomorphism is the one used in \cite{GHKK}. 
When $\Sigma$ is not a $k$-gon with $k=3,4,5$, the mutation class $\sfs(\fsl_3,\Sigma)$ is of infinite-mutation type. In this case, the union $\bigcup_{v\in \bExch_{\fsl_3,\Sigma}} \cC^+_{(v)}$ is not dense in $\X_{\fsl_3,\Sigma}(\bR^\trop)$ \cite[Theorem 2.27]{Yur21}. Therefore \cref{lem:quantum_duality} is far from characterizing the map $\mathbb{I}_\X^q$.
%\end{itemize}
\end{rem}

For the simplest cases that $\Sigma$ is a triangle or a quadrilateral (where the mutation class $\sfs(\fsl_3,\Sigma)$ is finite types $A_1$ and $D_4$, respectively), we actually get a quantum duality map:

\begin{prop}
When $\Sigma$ is a triangle or a quadrilateral, the image  $\mathbb{I}_\X^q(\cL^p(\Sigma,\bZ)) \subset \cO_q(\A_{\fsl_3,\Sigma})$ gives a $\bZ_q$-basis consisting of quantum cluster monomials. In particular, it has positive structure constants. 
\end{prop}

\begin{proof}
For these cases, it is easy to see that $\mathscr{S}_{\fsl_3,\Sigma}^q[\partial^{-1}] = \mathscr{A}^q_{\fsl_3,\Sigma} = \cO_q(\A_{\fsl_3,\Sigma})$ \cite[Corollary 6.1]{IY21}. Moreover, the tropical set $\X_{\fsl_3,\Sigma}(\bZ^T)$ is covered by finitely many cones $\cC^+_{(v)}(\bZ^T)$ for $v \in \bExch_{\fsl_3,\Sigma}$. 

For the triangle case, we have only two clusters (up to permutations), and hence \cref{lem:quantum_duality} with \cref{rem:decorated_triangulation} already gives the desired statement. For the quadrilateral case (type $D_4$), we have $16$ unfrozen variables and $8$ frozen variables. 
For instance, see \cite[Appendix A and Corollary 6.1]{IY21}.
Up to symmetry, we have already seen in the proof of \cref{lem:quantum_duality} (cf.~\cref{fig:elementary laminations}) that all of them are the images of some elementary laminations under the map $\mathbb{I}_\X^q$, except for the one represented by the elementary web $\tikz[scale=0.9,baseline=2mm]{\draw[blue](0,0)--(1,0)--(1,1)--(0,1)--cycle; \Hweb{0,1}{0,0}{1,0}{1,1}}$. This one also comes from an elementary lamination, as seen from \cref{fig:H-web_elementary}. Thus the assertion is proved. 
\end{proof}

\begin{figure}[htbp]
\centering
\begin{tikzpicture}[scale=.85]
\draw[blue] (2,0) -- (0,2) -- (-2,0) -- (0,-2) --cycle;
\draw[blue] (0,-2) -- (0,2);
\CoG{0,2}{0,-2}{-2,0};
\draw(G) coordinate(G1);
\CoG{0,2}{0,-2}{2,0};
\draw(G) coordinate(G2);
\draw[red,thick,-<-={0.4}{}] (G1) -- (G2);
\draw[red,thick,-<-] (G1) --++(135:1.2);
\draw[red,thick,-<-] (G1) --++(-135:1.2);
\draw[red,thick,->-] (G2) --++(45:1.2);
\draw[red,thick,->-] (G2) --++(-45:1.2);
\draw[-implies, double distance=2pt] (-3.5,-5) -- (-2.5,-5);
\draw[-implies, double distance=2pt] (2.5,-5) -- (3.5,-5);

{\begin{scope}[xshift=-6cm,yshift=-5cm,>=latex]
\draw[blue] (2,0) -- (0,2) -- (-2,0) -- (0,-2) --cycle;
\draw[blue] (0,-2) -- (0,2);
{\color{mygreen}
\quiverplus{2,0}{0,2}{0,-2}
\quiverplus{-2,0}{0,-2}{0,2}
\quiversquare{0,-2}{2,0}{0,2}{-2,0}
\qdlarrow{x122}{x121}
\qdlarrow{x232}{x231}
\qdlarrow{x342}{x341}
\qdlarrow{x412}{x411}
\squarefrozen{0}{0}{0}{0}{0}{-1}{0}{-1}
\node[left=0.2em,scale=0.8] at (x131) {$0$};
\node[right=0.2em,scale=0.8] at (x132) {$0$};
\node[right=0.2em,scale=0.8] at (x241) {$-1$};
\node[left=0.2em,scale=0.8] at (x242) {$1$};
\draw[red](x242) circle(5pt);
}
\end{scope}}

{\begin{scope}[yshift=-5cm,>=latex]
\draw[blue] (2,0) -- (0,2) -- (-2,0) -- (0,-2) --cycle;
\draw[blue] (0,-2) -- (0,2);
{\color{mygreen}
\quiverplus{2,0}{0,2}{0,-2}
\quiverminus{0,2}{-2,0}{0,-2}
\quiversquare{0,-2}{2,0}{0,2}{-2,0}
\qarrowbr{x131}{x132}
\qdlarrow{x122}{x121}
\qdlarrow{x232}{x231}
\qdrarrow{x341}{x342}
\qdrarrow{x411}{x412}
\squarefrozen{0}{0}{0}{0}{0}{0}{0}{0}
\node[left=0.2em,scale=0.8] at (x131) {$0$};
\node[right=0.3em,scale=0.8] at (x132) {$1$};
\node[right=0.2em,scale=0.8] at (x241) {$-1$};
\node[left=0.4em,scale=0.8] at (x242) {$-1$};
\draw[red] (x132) circle(5pt);
}
\end{scope}}

{\begin{scope}[xshift=6cm,yshift=-5cm,>=latex]
\draw[blue] (2,0) -- (0,2) -- (-2,0) -- (0,-2) --cycle;
\draw[blue] (0,-2) -- (0,2);
\quiversquare{0,-2}{2,0}{0,2}{-2,0}
{\color{mygreen}
\qarrow{x122}{x231}
\qarrow{x231}{x241}
\qarrow{x241}{x122}
\qarrow{x121}{x241}
\qarrow{x131}{x121}
\qarrow{x132}{x232}
\qarrow{x241}{x132}

\qarrow{x242}{x132}
\qarrow{x342}{x242}
\qarrow{x412}{x242}
\qarrow{x242}{x341}
\qarrowbl{x132}{x131}
\draw[thick,->,shorten >=2pt,shorten <=2pt] (x132) to[bend left=30] (x411);
\draw[thick,->,shorten >=2pt,shorten <=2pt] (x232) to[bend right=30] (x242);
\draw[thick,->,shorten >=2pt,shorten <=2pt] (x411) to[bend right=10] (x241);
\qdlarrow{x122}{x121}
\qdlarrow{x232}{x231}
\qdrarrow{x341}{x342}
\qdrarrow{x411}{x412}
\squarefrozen{0}{0}{0}{0}{0}{0}{0}{0}
\node[left=0.2em,scale=0.8] at (x131) {$0$};
\node[right=0.2em,scale=0.8] at (x132) {$0$};
\node[right=0.2em,scale=0.8] at (x241) {$-1$};
\node[left=0.2em,scale=0.8] at (x242) {$0$};
}
\end{scope}}
\end{tikzpicture}
    \caption{An elementary lamination of H-shape. Its shear coordinates associated with a triangulation is shown in the bottom left, and their transformations under the mutation sequence shown in red circles continue to the right.}
    \label{fig:H-web_elementary}
\end{figure}
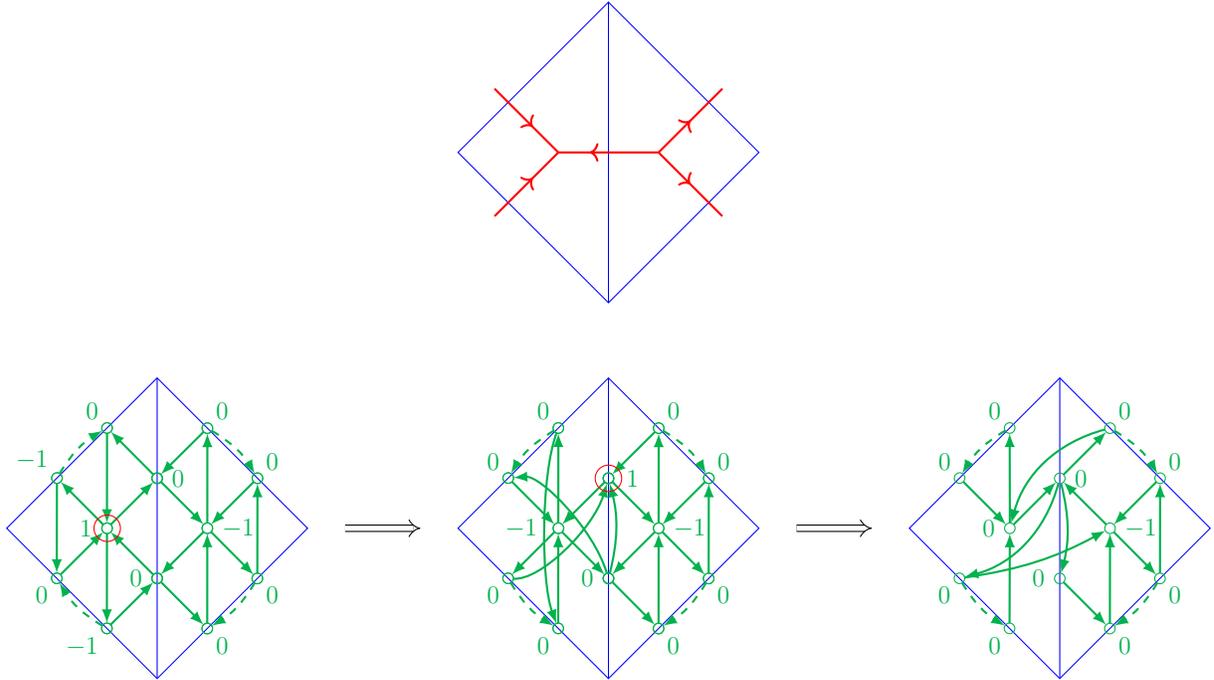

\begin{conj}
The basis $\mathbb{I}_\X^q(\cL^p(\Sigma,\bZ))$ is \emph{parametrized by tropical points} in the sense of \cite[Definition 4.13]{Qin21}. Namely, for any integral $\fsl_3$-lamination $\hL \in \cL^p(\Sigma,\bZ)$, the quantum Laurent expression of $\mathbb{I}_\X^q(\hL) \in \mathscr{A}^q_{\fsl_3,\Sigma}$ in the quantum cluster $\{A_i\}_{i \in I}$ associated with a vertex $\omega \in \bExch_{\fsl_3,\Sigma}$ has the leading term $\big[\prod_{i \in I}A_i^{\check{\sfx}_i(\hL)}\big]$ with respect to the dominance order (\cite[Definition 4.6]{Qin21}), where $\check{\sfx}^{(\omega)}=(\check{\sfx}_i)_{i \in I}$ is the Langlands dual shear coordinate system associated with $\omega$. 
\end{conj}

% \begin{figure}[htbp]
% \begin{tikzpicture}
% \bline{-1.5,-1.5}{1.5,-1.5}{0.2};
% \tline{-1.5,1.5}{1.5,1.5}{0.2};
% \draw[blue] (-1.5,0) -- (0,-1.5) -- (1.5,0) -- (0,1.5) -- cycle;
% \draw[blue] (0,1.5) -- (0,-1.5);
% {\color{mygreen}
% \draw[mygreen]($(0,1.5)!1/3!(0,-1.5)$) circle(2pt);
% \draw[mygreen]($(0,1.5)!2/3!(0,-1.5)$) circle(2pt) node[left=0.2em,scale=0.8]{$+1$};
% }
% \filldraw (0,1.5) circle(1.5pt);
% \filldraw (0,-1.5) circle(1.5pt);
% \draw[red,thick,-<-={0.7}{}] (0.5,-1.5) to[out=90,in=-45] (0,0) to[out=135,in=-90] (-0.5,1.5);
% \node[red] at (0.5,-2) {$\varpi_1^\vee$};
% \node[red] at (-0.5,2) {$\varpi_2^\vee$};

% {\begin{scope}[xshift=6cm]
% \foreach \i in {0,120,240}
% {\begin{scope}[rotate=\i]
% \tline{-0.8,1.5}{0.8,1.5}{0.2};
% \draw[red,thick,-<-={0.4}{}] (0,0) to[out=120,in=-90] (-0.5,1.5);
% \node[red] at (-0.5,2) {$\varpi_2^\vee$};
% \end{scope}}
% \draw[mygreen] (0.1,0.2) circle(2pt) node[above=0.2em,scale=0.8]{$+1$};
% \draw[blue] (-30:1.5) -- (90:1.5) -- (210:1.5) --cycle;
% \foreach \i in {-30,90,210}
% \filldraw (\i:1.5) circle(1.5pt);
% \end{scope}}
% \end{tikzpicture}
%     \caption{Laminations with positive unit shear coordinates $\sfx_i^\tri$. Here the coweights are the pinning data assigned to boundary intervals.}
%     \label{fig:positive laminations}
% \end{figure}

\paragraph{\textbf{Classical limit}}
Recall that the set $\mathsf{BWeb}_{\fsl_3,\Sigma}$ also gives a $\bZ$-basis of the classical (commutative) skein algebra $\mathscr{S}_{\fsl_3,\Sigma}^1$. Then \cref{thm:basis_correspondence} tells us that the map $\mathbb{I}_\X^q$ induces a bijection 
\begin{align*}
    \mathbb{I}_\X: \cL^p(\Sigma,\bZ)_+ \xrightarrow{\sim} \mathsf{BWeb}_{\fsl_3,\Sigma} \subset \mathscr{S}_{\fsl_3,\Sigma}^1,
\end{align*}
which is also extended to a map $\mathbb{I}_\X: \cL^p(\Sigma,\bZ) \hookrightarrow \mathscr{S}_{\fsl_3,\Sigma}^1[\partial^{-1}]$. Then by \eqref{eq:S=A=U_classical}, we get the following:
\begin{cor}
The image $\mathbb{I}_\X(\cL^p(\Sigma,\bZ))$ gives a $\bZ$-basis of the cluster algebra $\mathscr{A}_{\fsl_3,\Sigma}$.
\end{cor}

\section{Proofs of \texorpdfstring{\cref{prop:spiralling_good_position}}{Theorem 3.10} and \texorpdfstring{\cref{prop:injectivity}}{Theorem 3.19}}\label{sec:technicalities}

\subsection{Proof of \texorpdfstring{\cref{prop:spiralling_good_position}}{Theorem 3.10}}\label{sec:good_position}
%Let $B$ be a biangle with special points $m$, $m'$. We say that an asymptotically periodic symmetric strand set $S=(S_1,S_2)$ on $B$ is \emph{periodic towards $m$} if there exists an order-preserving bijection $s: \bZ_{\geq 0} \to S_1$ such that $s(n)$ approaches to $m$ as $n \to \infty$, and the orientations of strands are periodic: $\mathfrak{o}(s(n+\omega))=\mathfrak{o}(s(n))$ for all $n \in \bZ_{\geq 0}$ and a period $\omega \in \bZ_{>0}$. Here $\mathfrak{o}(s)$ stands for the orientation (incoming or outgoing) of $s \in S_1$. 
\paragraph{\textbf{General position.}} 
Recall that an \emph{ideal arc} in $(\Sigma,\bM)$ is an immersed arc $\gamma$ in $\Sigma$ with endpoints in $\bM$ which has no self-intersection except possibly at its endpoints, and not isotopic to one point. 
In particular $\gamma$ is one-sided differentiable at each endpoint $p$, hence there exists a small coordinate neighborhood $D_p$ of $p$ such that $D_p \cap \gamma$ consists of (at most two) rays incident to $p$. 
%Clearly, it suffices to consider such ideal arcs for our purpose. 

We say that two immersed arcs or webs in $\Sigma$ are in \emph{general position} with each other if their intersections are finite, transverse and avoiding the trivalent vertices. 
Moreover, we say that the spiralling diagram $\cW$ (\cref{def:spiralling}) associated with a non-elliptic signed web is in \emph{general position} with an ideal arc if their intersection points do not accumulate in $\interior \Sigma$, transverse and avoiding the trivalent vertices. We may always assume the general position by the concrete construction of a spiralling diagram as logarithmic spirals near punctures.

\paragraph{\textbf{Relative intersection number.}}
Let $\gamma,\gamma'$ be two ideal arcs isotopic to each other with common endpoints $p_1,p_2 \in \bM$, and $\cW$ a spiralling diagram. Assume that these three are in a general position with each other. Then the ideal arcs $\gamma,\gamma'$ bounds a region $B(\gamma,\gamma')$, which is a union of finitely many biangles (or such a region minus small biangles: see $\gamma$ and $\gamma'_2$ in \cref{fig:relative_intersection_diverge}). 

By the construction of the spiralling diagram, there exists a small disk neighborhood $p_i \in D_i$ for $i=1,2$ such that $\rho_i:=\gamma \cap D_i$, $\rho'_i:=\gamma' \cap D_i$ are rays incident to $p_i$, and $\cW \cap D_i$ is a logarithmic spiral. The rays $\rho_i,\rho'_i$ separates $D_i$ into two sectors, and exactly one of them corresponds to the region bounded by $\gamma,\gamma'$. Then we can find a circular segment in this sector which does not intersect with $\cW$, and the restriction of $\cW$ to the circular sector separated by this segment is a periodic ladder-web. We call this circular sector $S(p_i)$ a \emph{cut-off sector} at $p_i$. See \cref{fig:cut-off}. 
Then $\cW_\mathrm{reg}:=\cW\cap (B\setminus S(p_1)\cup S(p_2))$ is a finite web. 
%that bound a biangle $B=B(\alpha,\alpha')$ and in general position with a spiralling diagram $\cW$. Here we allow $\alpha,\alpha'$ to be ideal arcs, so that $\cW_B:=\cW \cap B$ is a possibly infinite web. 

\begin{figure}[htbp]
\centering
\begin{tikzpicture}[scale=.95]
\draw[blue] (0,0) -- (5,0);
\node[blue] at (0.7,-0.3) {$\gamma$};
\draw[blue] (0,0) to[bend left=30] node[midway,above]{$\gamma'$} (1.5,0) to[bend right=30] (3.5,0) to[bend left=30] (5,0);
\draw[dashed] (5,0) circle(0.75cm);
\draw[fill=white] (0,0) circle(2pt) node[below]{$p_1$};
\draw[fill=white] (5,0) circle(2pt) node[below=0.5em]{$p_2$};
\begin{scope}[xshift=5cm]
\draw[red,thick,->-] (135:1.5) node[above]{$\cW$} ..controls ++(-60:0.3) and (-0.4,-0.3).. (0,-0.3)
..controls (0.3,-0.3) and (0.3,0.3).. (0,0.3)
..controls (-0.3,0.3) and (-0.2,-0.1).. (0,-0.1);
\end{scope}
\draw[thick,->] (7.5,0) --(6.5,0);
\begin{scope}[xshift=10.5cm]
\draw[dashed] (0,0) circle(2cm);
\fill[gray!20] (0,0) -- (150:0.5) arc(150:180:0.5) --++(0.5,0);
\draw[blue] (0,0) -- (-2,0);
\draw[blue] (0,0) -- (150:2);
\node[blue] at (180:2.3) {$\rho$};
\node[blue] at (150:2.3) {$\rho'$};
\draw[red,thick,->-] (135:2) ..controls ++(-60:0.3) and (-0.7,-0.5).. (0,-0.5)
..controls (0.5,-0.5) and (0.5,0.5).. (0,0.5)
..controls (-0.5,0.5) and (-0.3,-0.3).. (0,-0.3)
..controls (0.3,-0.3) and (0.3,0.3).. (0,0.3)
..controls (-0.3,0.3) and (-0.2,-0.2).. (0,-0.2);
\draw[thick] (140:0.5) arc(140:190:0.5);
\draw[fill=white] (0,0) circle(2pt);
\end{scope}
\end{tikzpicture}
    \caption{Two isotopic ideal arcs and a spiralling diagram. A cut-off sector is shown in gray in the right.}
    \label{fig:cut-off}
\end{figure}
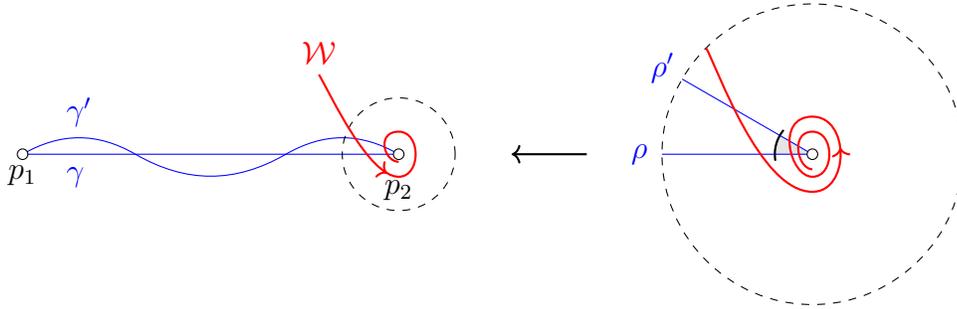

\begin{dfn}[Relative intersection number]
Let $\gamma,\gamma',\cW$ be as above, and choose cut-off sectors $S(p_1),S(p_2)$ at the common endpoints $p_1,p_2 \in \bM$. Then we define the \emph{relative intersection number} of $\cW$ with $(\gamma,\gamma')$ to be
\begin{align*}
    i(\cW;\gamma,\gamma'):=i(\cW_\mathrm{reg},\gamma) - i(\cW_\mathrm{reg},\gamma').
\end{align*}
Here $i(-,-)$ denotes the usual geometric intersection number of two webs. 
\end{dfn}
Notice that it is independent of the choice of the cut-off sectors since a periodic ladder-web has an equal number of intersections with $\gamma$ and $\gamma'$ in each of its period.
Clearly, we have $i(\cW;\gamma',\gamma)= -i(\cW;\gamma,\gamma')$. 

\begin{lem}\label{lem:relative_int_cocycle}
Let $\gamma_1,\gamma_2,\gamma_3$ be three ideal arcs isotopic to each other with common endpoints, and $\cW$ a spiralling diagram. Assume that they are in general position with each other. Then we have
\begin{align*}
    i(\cW;\gamma_1,\gamma_3) = i(\cW;\gamma_1,\gamma_2) + i(\cW;\gamma_2,\gamma_3).
\end{align*}
\end{lem}

\begin{proof}
Immediately verified by choosing a common cut-off sector. 
\end{proof}

% Let $\gamma,\gamma'$ be two ideal arcs in general position and isotopic to each other, and $\cW$ a spiralling diagram in general position with these arcs. We define the relative intersection number of $\cW$ with $(\gamma,\gamma')$ to be the sum
% \begin{align*}
%     i(\cW;\gamma,\gamma'):= \sum_{\alpha \subset \gamma,\alpha'\subset \gamma'} i(\cW_{B(\alpha,\alpha')};\alpha,\alpha'),
% \end{align*}
% where $(\alpha,\alpha')$ runs over the pairs of subarcs that bound biangles $B(\alpha,\alpha')$. Notice also that the above sum is finite by the general position condition. We have a cocycle condition
% \begin{align}\label{eq:cocycle}
%     i(\cW;\gamma,\gamma'') = i(\cW;\gamma,\gamma') + i(\cW;\gamma',\gamma'')
% \end{align}
% for three ideal arcs $\gamma,\gamma',\gamma''$ isotopic to each other and in general position. 

\begin{dfn}
We say that an ideal arc $\gamma$ is in \emph{minimal position} with a spiralling diagram $\cW$ if it satisfies $i(\cW;\gamma',\gamma) \geq 0$ for any ideal arc $\gamma'$ isotopic to $\gamma$ with common endpoints, and in general position with $\cW$. 
\end{dfn}
See \cref{fig:non-minimal position} for an example of an ideal arc not in a minimal position.

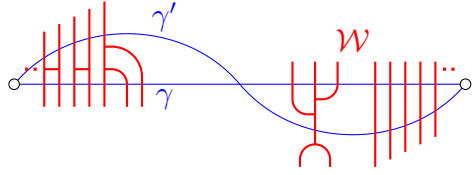
\begin{figure}[ht]
    \centering
\begin{tikzpicture}
\draw[blue](-3,0) -- (3,0);
\draw[blue](-3,0) to[bend left=50pt] (0,0) to[bend right=50pt] (3,0);
\node[blue] at (-1,-0.2) {$\gamma$};
\node[blue] at (-1,0.9) {$\gamma'$};
\draw[red] (1.5,0.3) node[above]{$\cW$};
\foreach \i in {-0.2,0,0.2,0.4,0.6}
\draw[red,thick] (2+\i,-1+\i*0.5) -- (2+\i,0.3);
\draw[red,very thick,dotted] (2.7,0.2) -- (2.9,0.2);
\draw[red,thick] (1,-0.8) -- (1,0.3);
\draw[red,thick,rounded corners=0.2cm] (1,-0.2) --++(0.3,0) -- (1.3,0.3);
\draw[red,thick,rounded corners=0.2cm] (1,-0.4) --++(-0.3,0) -- (0.7,0.3);
\draw[red,thick,rounded corners=0.2cm] (1,-0.8) --++(0.2,0) -- (1.2,-1.1);
\draw[red,thick,rounded corners=0.2cm] (1,-0.8) --++(-0.2,0) -- (0.8,-1.1);
\foreach \i in {-0.2,0,0.2,0.4,0.6}
\draw[red,thick] (-2-\i,-0.3) -- (-2-\i,1-\i*0.5);
\draw[red,thick] (-2,0.2) -- (-2.2,0.2);
\draw[red,thick] (-2.4,0.2) -- (-2.6,0.2);
\draw[red,very thick,dotted] (-2.7,0.2) -- (-2.9,0.2);
\draw[red,thick,rounded corners=0.2cm] (-1.8,0.2) --++(0.3,0) -- (-1.5,-0.3);
\draw[red,thick,rounded corners=0.4cm] (-1.8,0.5) --++(0.5,0) -- (-1.3,-0.3);
\draw[fill=white] (-3,0) circle(2pt);
\draw[fill=white] (3,0) circle(2pt);
\end{tikzpicture}
    \caption{A spiralling diagram $\cW$ that is not in minimal position with an ideal arc $\gamma$. Indeed, $i(\cW;\gamma',\gamma)=-4$.}
    \label{fig:non-minimal position}
\end{figure}

\paragraph{\textbf{Realization of a minimal position}}
We are going to prove:
\begin{prop}[unbounded version of {\cite[Corollary 12]{FS20}}]\label{prop:minimal_position_unbounded}
Let $\cW$ be the spiralling diagram associated with a non-elliptic signed web, and $\gamma$ an ideal arc in a general position with $\cW$. 
Then we can isotope $\cW$ into a spiralling diagram $\cW'$ in minimal position with $\gamma$ via a finite sequence of intersection reduction moves, H-moves, and an isotopy relative to $\gamma$.
\end{prop}
To prove this, the following lemma is useful:

\begin{lem}[unbounded version of {\cite[Lemma 15]{FS20}}]\label{eq:bigon_lemma_unbounded}
Let $B$ be a biangle in $\Sigma$ bounded by two immersed arcs $\alpha,\alpha'$, and $\cW$ a spiralling diagram in a general position. If some of the endpoints of $\alpha,\alpha'$ are punctures, then choose any cut-off sectors and consider $\cW_\mathrm{reg}$ as above. Otherwise, set $\cW_\mathrm{reg}:=\cW$. 
Then $\cW_{\mathrm{reg}}$ can be isotoped through a finite number of intersection reduction moves and H-moves so that $\cW_{\mathrm{reg}}\cap B$ consists of disjoint parallel arcs connecting $\alpha$ and $\alpha'$. This can be done by preserving the cut-off sectors, and the resulting web does not depend on the choice of cut-offs.
\end{lem}

\begin{proof}
%Suppose that there exists $\cW_\mathrm{reg}$ for which the statement does not hold. 
%Let $W'_\mathrm{reg}$ be a web obtained from $\cW_\mathrm{reg}$ by removing its cut-off part.
%and the parallel arcs connecting $\alpha$ and $\alpha'$.
Since $\cW_\mathrm{reg}$ is finite, the statement follows from \cite[Lemma 15]{FS20}.
\end{proof}
%If one of common endpoints of $\alpha,\alpha'$ lies in $\interior \Sigma$, then the asymptotic part of $\cW_B$ near that point is empty, since no spiralling ends of $\cW$ accumulate in the interior of $\Sigma$.
Notice that each of the H-move and the intersection reduction moves is accompanied with a small biangle (shown by dashed lines in \cref{fig:H-move,fig:tightening}) that cuts out a part of the web which we push out. Therefore the finite sequence of these moves in \cref{eq:bigon_lemma_unbounded} is accompanied with a finite collection $\{B^{(j)}\}_{j \in J}$ of biangles that is partially ordered for the inclusion according to the order of moves, which we call the \emph{tightening biangles}. 
Let us denote by $W^{(j)}$ the part of $\cW$ cut out by the tightening biangle $B^{(j)}$, which we call the \emph{loose part} of $\cW$.
%The intersection of the loose part and the two sides of $B$ give rise to two strands sets, where one of them contains strictly smaller number of strands than the other. We call the former strand set $S_\mathrm{tight}(W_B)$ the \emph{tight strand set}, and the latter $S_{\mathrm{loose}}(W_B)$ the \emph{loose strand set}. 
See \cref{fig:loose}. 

\begin{figure}[htbp]
\centering
\begin{tikzpicture}[scale=.95]
\draw[blue] (0,0) to[bend left=30] (5,0);
\draw[blue] (0,0) to[bend right=30] (5,0);
\foreach \i in {0.4,0.8}
{
\draw[red,thick] (40:\i) arc(40:-40:\i);
\draw[red,thick] (40:\i+0.2) arc(40:-40:\i+0.2);
\draw[red,thick] (\i,0) -- (\i+0.2,0);
}
\draw[red,thick] (3.8,-0.8) -- (3.8,0.8);
\draw[red,thick] (3.2,-0.9) -- (3.2,0.9);
\draw[orange,thick] (2,-0.2) node{$?$} ellipse(0.5cm and 0.3cm);
\draw[orange,thick] (2,0.1) --++(0,0.6);
\draw[orange,thick] (2,-0.5) --++(0,-0.4);
\draw[orange,thick] (2,-0.2)++(-0.5*0.866,-0.3*0.5) to[out=-120,in=90] (1.4,-0.8);
\draw[orange] (2,-0.2)++(0.5*0.866,-0.3*0.5) to[out=-60,in=90] (2.6,-0.8);

%\draw[myorange,decorate,decoration={brace,amplitude=3pt,raise=2pt}] (2.7,-1) --node[midway,below=0.4em]{$S_{\mathrm{loose}}(W_B)$} (1.3,-1);
%\draw[myorange,decorate,decoration={brace,amplitude=3pt,raise=2pt}] (1.7,0.8) --node[midway,above=0.4em]{$S_{\mathrm{tight}}(W_B)$} (2.3,0.8);
\begin{scope}
\clip (0,0) to[bend right=30] (5,0) to[bend right=30] (0,0);
\draw[blue,dashed] (1,-1.3) ..controls ++(90:1.5) and (1.7,0.3).. (2,0.3)
..controls (2.3,0.3) and ($(3,-1.3)+(90:1.5)$).. (3,-1.3);
\end{scope}
\draw[fill=white] (0,0) circle(2pt);
\draw[fill=white] (5,0) circle(2pt);
\draw [thick,-{Classical TikZ Rightarrow[length=4pt]},decorate,decoration={snake,amplitude=2pt,pre length=2pt,post length=3pt}](5.5,0) --(7.5,0);
\begin{scope}[xshift=8cm]
\draw[blue] (0,0) to[bend left=30] (5,0);
\draw[blue] (0,0) to[bend right=30] (5,0);
\foreach \i in {0.4,0.8}
{
\draw[red,thick] (40:\i) arc(40:-40:\i);
\draw[red,thick] (40:\i+0.2) arc(40:-40:\i+0.2);
\draw[red,thick] (\i,0) -- (\i+0.2,0);
}
\draw[red,thick] (3.8,-0.8) -- (3.8,0.8);
\draw[red,thick] (3.2,-0.9) -- (3.2,0.9);
{\begin{scope}[yshift=-1cm]
\draw[orange,thick] (2,-0.2) node{$?$} ellipse(0.5cm and 0.3cm);
\draw[orange,thick] (2,0.1) --++(0,1.8);
\draw[orange,thick] (2,-0.5) --++(0,-0.4);
\draw[orange,thick] (2,-0.2)++(-0.5*0.866,-0.3*0.5) to[out=-120,in=90] (1.4,-0.8);
\draw[orange] (2,-0.2)++(0.5*0.866,-0.3*0.5) to[out=-60,in=90] (2.6,-0.8);
%\node[orange] (2,-1.1) {$W_\mathrm{loose}(\alpha,\alpha')$};
\end{scope}}
\draw[fill=white] (0,0) circle(2pt);
\draw[fill=white] (5,0) circle(2pt);
\end{scope}
\end{tikzpicture}
    \caption{The restriction of a spiralling diagram $\cW$ to a biangle bounded by two ideal arcs. Its loose part is shown in orange, which can be pushed out through a sequence of intersection reduction moves and H-moves.}
    \label{fig:loose}
\end{figure}
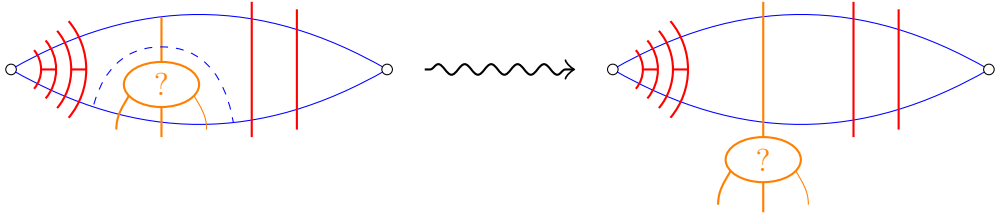

The following lemma ensures that the intersection reduction procedures of a spiralling diagram associated with a non-elliptic signed web always terminate in finite steps. 

\begin{lem}\label{lem:relative_intersection_maximal}
For any spiralling diagram $\cW$ associated with a non-elliptic signed web $W$ and an ideal arc $\gamma$ in general position, the relative intersection number $i(\cW;\gamma,\gamma')$ is bounded from above when $\gamma'$ runs over the ideal arcs homotopic to $\gamma$ and in general position with $\gamma$ and $\cW$. 
\end{lem}

% \begin{rem}
% $i(\cW;\gamma,\gamma')$ is not bounded from below, since we can intersect $\gamma'$ with $\cW$ arbitrarily many times.
% \end{rem}

\begin{proof}
If $W$ has puncture H-faces, then applying appropriate puncture H-moves, we obtain another signed web $W'$ which is puncture-reduced. The corresponding spiralling diagrams $\cW$ and $\cW'$ differ only by some finitely many H-shaped parts in the spiralling part, and hence $i(\cW;\gamma,\gamma')=i(\cW';\gamma,\gamma')$. 
Therefore it suffices to consider the case where the signed web $W$ giving rise to $\cW$ is puncture-reduced.  

We prove the assertion by contradiction: 
suppose that there exists a sequence $\gamma'_n \simeq \gamma$ of ideal arcs satisfying the condition and $i(\cW;\gamma,\gamma'_n) \geq n$ for all $n \in \bZ_{\geq 0}$. Let $\{B_n^{(j)}\}_{j \in J_n}$ be the collection of tightening biangles for the pair $(\gamma,\gamma'_n)$, and $W_n^{(j)} \subset \cW$ the corresponding loose part. 
\begin{itemize}
    \item[(a)] Since we are interested in a sequence $\gamma'_n$ such that $i(\cW;\gamma,\gamma'_n)$ diverges, we may assume that all of the tightening biangles $B_n^{(j)}$ are sticked to $\gamma$ rather than $\gamma'_n$. Otherwise, a biangle sticked to $\gamma'_n$ contributes negatively to $i(\cW;\gamma,\gamma'_n)$. Then we may isotope $\gamma'_n$ to avoid this biangle without decreasing $i(\cW;\gamma,\gamma'_n)$. 
    \item[(b)] Shrinking each tightening biangle (without changing the intersection number of its boundary with $\cW$) if necessary, we may assume that either $B_n^{(j)} \cap B_m^{(\ell)}=\emptyset$, $B_n^{(j)} \subset B_m^{(\ell)}$ or $B_m^{(\ell)} \subset B_n^{(j)}$ holds for any pair in this collection. Also we can ensure that each tightening biangle do not intersect with the cut-off sectors at punctures. 
\end{itemize}
Let us consider the compact interval $K=\gamma \setminus (\text{cut-off sectors})$. From the assumption of general position, the intersection of $\cW$ with $K$ is finite.
The intersections $I_n^{(j)}:=\gamma \cap \interior B_n^{(j)}$ give open intervals in $K$. Observe that the union $\bigcup_{n \geq 0,~j \in J_n} I_n^{(j)}$ has finitely many path-connected components, since each such component contains a distinct point in $\cW \cap K$, which is finite. Therefore we see that there exist subsequences $n_k$ and $j_k \in J_{n_k}$ such that $B_{n_k}^{(j_k)} \subset B_{n_{k+1}}^{(j_{k+1})}$.

Such a nested situation is illustrated in  \cref{fig:relative_intersection_diverge}. Indeed, the situation says that distinct reduction moves are applied infinitely many times, while the original signed web $W$ is finite. It means that there is a portion $P$ of the signed web that is referred infinitely many times. Therefore the nested biangles $B_{n_k}^{(j_k)}$ (or the arcs $\gamma'_{n_k}$) must be winding around one of the punctures $p_1$ or $p_2$, while the portion $\P$ in $\cW$ corresponding to $P$ is spiralling around the same puncture as in the bottom left of the figure. 
%Here the arcs $\gamma'_n$ (or the biangles $B_n$) are nested. A single portion (shown in orange) of the spiralling diagram $\cW$ appears $n$ many times in the biangles $B_n$, which are in a covering $\widetilde{\Sigma}$ of $\Sigma$ seen as in the above left. Such a portion must be in fact spiralling around a puncture (below left). Such a spiralling diagram necessarily arises from the signed web shown in the below right.
Notice that such a spiralling diagram $\cW$ arises from the signed web $W$ shown in the bottom right. 

Moreover, observe the correspondence shown in \cref{fig:side-puncture_correspondence} between the faces stuck to $\gamma$ and the puncture-faces. 
Therefore, the sequence of loose parts $W_{n_k}^{(j_k)}$ must come from these puncture-faces in the signed web $W$, which contradicts to either the puncture-reduced assumption, non-elliptic condition, or the no bad ends condition.
%Here we must have no puncture-faces of the first type by the puncture-reduced assumption, and none of the second type which is elliptic, and none of the third type having bad ends. 
%Since the portion of the spiralling diagram $\cW$ bounded by the nested tightening biangles can be isotoped through the arc $\gamma$ by applying a sequence of the H- and the intersection-reduction moves, this correspondence tells us that the corresponding signed web must have an elliptic face. 
%Here notice that one of the intersection-reduction move is applied first since we have assumed that $W$ is puncture-reduced.
%It contradicts to our assumption that $\cW$ is non-elliptic. 
Thus the assertion is proved.
%Suppose first that the intersection number $i(\gamma,\gamma'_n)$ is bounded. Then there exist subarcs $\alpha_n \subset \gamma$ and $\alpha'_n \subset \gamma'_n$ that bound a bigon $B_n$ such that $i(\cW_{B_n};\alpha_n,\alpha'_n) \to \infty$ as $n \to \infty$. Since $\cW_{B_n}$ is asymptotically periodic, it imples that its compact part contains arbitrarily large number of 2-gon or 3-gon faces attached to subarcs of $\gamma$. Since the spiralling diagram has no accumulation points in $\interior\Sigma$, there exists a subsequence of such faces which converges to an endpoint $p \in \bP$ of $\gamma$. By applying an isotopy to $\alpha'_n$ to push out irrelevant vertices away if necessary, we would arrive at the situation shown in \cref{fig:elliptic_spiral}. 
% Such a situation could only occur when the original signed web had an elliptic face incident to $p$. This contradicts to the non-elliptic assumption.
% %The former contradicts to the finiteness of the trivalent vertices in the original signed web, and the latter produces a sequence of intersection points of $\cW$ and $\gamma$ that accumulates in $\interior\Sigma$, which is also absurd. 
% If the intersection number $i(\gamma,\gamma'_n)$ diverges, then we also get a sequence of 2-gon or 3-gon faces attached to $\gamma$ which must come from elliptic faces incident to a puncture, hence a contradiction. 
\end{proof}

\begin{figure}[ht]
\centering
\begin{tikzpicture}[scale=0.8]
\draw[blue] (0,0) ..controls ++(45:2.5) and ($(7,0)+(135:2.5)$).. (7,0);
\draw[blue] (0,0) ..controls ++(67.5:3.5) and ($(7,0)+(112.5:3.5)$).. (7,0);
%\draw[blue] (0,0) ..controls ++(90:5) and ($(7,0)+(90:5)$).. (7,0);
\draw[blue] (0,0) -- (7,0);
\draw[orange,thick] (4,0.7) node{$?$} ellipse(0.5cm and 0.3cm);
\draw[orange,thick] (4,0.4) --(4,0);
\draw[orange,thick] (4,0.7)++(-0.5*0.866,-0.3*0.5) to[out=-120,in=90] (3.5,0);
\draw[orange,thick] (4,0.7)++(0.5*0.866,-0.3*0.5) to[out=-60,in=90] (4.5,0);
\draw[orange,thick] (4,1) --(4,2.5);
\draw[orange,thick] (2.5,1.8) node{$?$} ellipse(0.5cm and 0.3cm);
\draw[orange,thick] (2.5,1.5) --(2.5,0);
\draw[orange,thick] (2.5,1.8)++(-0.5*0.866,-0.3*0.5) to[out=-120,in=90] (2,1.2) -- (2,0);
\draw[orange,thick] (2.5,1.8)++(0.5*0.866,-0.3*0.5) to[out=-60,in=90] (3,1.2) -- (3,0);
\draw[orange,thick] (2.5,2.1) --(2.5,2.5);
%\draw[orange,thick] (1.5,2.5) node{$?$} ellipse(0.5cm and 0.3cm);
%\draw[orange,thick] (1.5,2.8) --(1.5,4);
%\draw[orange,thick] (1.5,2.5)++(-0.5*0.866,-0.3*0.5) to[out=-120,in=90] (1,1.8) -- (1,0);
%\draw[orange,thick] (1.5,2.5)++(0.5*0.866,-0.3*0.5) to[out=-60,in=90] (2,1.8) -- (2,0);
%\draw[orange,thick] (1.5,2.2) --(1.5,0);
\draw[orange,thick,dotted] (1.8,1.5) --++(-0.5,0);
\node[blue] at (6,-0.3) {$\gamma$};
\node[blue] at (6,0.5) {$\gamma'_{n_1}$};
\node[blue] at (6,1.2) {$\gamma'_{n_2}$};
%\node[blue] at (6,2.4) {$\gamma'_3$};
\draw[blue,thick,dotted] (6,1.6) --++(0,0.5);
\draw[fill=white] (0,0) circle(2pt);
\draw[fill=white] (7,0) circle(2pt);
\node at (-1,1) {$\widetilde{\Sigma}$};
\draw[thick,->] (3.5,-1) -- (3.5,-2);

\begin{scope}[xshift=2cm,yshift=-5cm]
\draw[blue] (0,0) ..controls ++(45:2.5) and ($(5,0)+(135:2.5)$).. (5,0);
\draw[blue] (0,0) ..controls ++(67.5:1) and (-0.5,0.5).. (-0.5,0)
..controls (-0.5,-1) and (1.5,-1).. (1.5,0)
..controls (1.5,1.5) and (3.5,1.5).. (3.5,0)
..controls (3.5,-2) and (-1.5,-2).. (-1.5,0)
..controls (-1.5,3) and ($(5,0)+(112.5:3)$).. (5,0);
\draw[blue] (0,0) -- (5,0);
\draw[orange,thick] (2.5,0.5) node{$?$} ellipse(0.5cm and 0.3cm);
\draw[orange,thick] (2.5,0.2) --(2.5,0)
..controls (2.5,-1.5) and (-1,-1.5).. (-1,0)
..controls (-1,1.2) and (0.7,1.2).. (0.7,0)
..controls (0.7,-0.5) and (-0.3,-0.5).. (-0.3,0);
\draw[orange,thick] (2.5,0.5)++(-0.5*0.866,-0.3*0.5) to[out=-120,in=90] (2,0.2) -- (2,0)
..controls (2,-1.2) and (-0.8,-1.2).. (-0.8,0)
..controls (-0.8,1) and (0.6,1).. (0.6,0)
..controls (0.6,-0.4) and (-0.2,-0.4).. (-0.2,0);
\draw[orange,thick] (2.5,0.5)++(0.5*0.866,-0.3*0.5) to[out=-60,in=90] (3,0.2) -- (3,0)
..controls (3,-1.8) and (-1.2,-1.8).. (-1.2,0)
..controls (-1.2,1.4) and (0.8,1.4).. (0.8,0)
..controls (0.8,-0.6) and (-0.4,-0.6).. (-0.4,0);
\draw[orange,thick] (2.5,0.8) --(2.5,2.5);
\node[blue] at (4,-0.3) {$\gamma$};
\node[blue] at (4,0.5) {$\gamma'_{n_1}$};
\node[blue] at (4,1.6) {$\gamma'_{n_2}$};
\draw[fill=white] (0,0) circle(2pt);
\draw[fill=white] (5,0) circle(2pt);
\node[orange] at (3,-1.5) {$\cW$};
\node at (-3,1) {$\Sigma$};
\end{scope}

\begin{scope}[xshift=10cm,yshift=-5cm]
\draw[blue] (0,0) -- (5,0);
\draw[orange,thick] (0,1) node{$?$} ellipse(0.5cm and 0.3cm);
\draw[orange,thick] (0,0.7) -- (0,0);
\draw[orange,thick] (0,1) ++ (0.5*0.866,-0.3*0.5) to[out=-60,in=45] (0,0);
\draw[orange,thick] (0,1) ++ (-0.5*0.866,-0.3*0.5) to[out=-120,in=135] (0,0);
\node[orange,scale=0.8] at (0.3,0.1) {$+$};
\node[orange,scale=0.8] at (-0.3,0.1) {$+$};
\node[orange,scale=0.8] at (0.15,0.35) {$+$};
\draw[orange,thick] (0,1.3) --++(0,1);
\node[orange] at (1,1) {$W$};
\draw[fill=white] (0,0) circle(2pt);
\draw[fill=white] (5,0) circle(2pt);
\end{scope}
\draw [thick,-{Classical TikZ Rightarrow[length=4pt]},decorate,decoration={snake,amplitude=1.8pt,pre length=2pt,post length=3pt}](9,-4) -- (7.5,-4);
\end{tikzpicture}
    \caption{The situation that the relative intersection numbers $i(\cW;\gamma,\gamma'_n)$ diverge. The top left shows a covering of $\Sigma$ around the puncture. The infinite sequence of portions are projected to the same portion of $\cW$.}
    \label{fig:relative_intersection_diverge}
\end{figure}
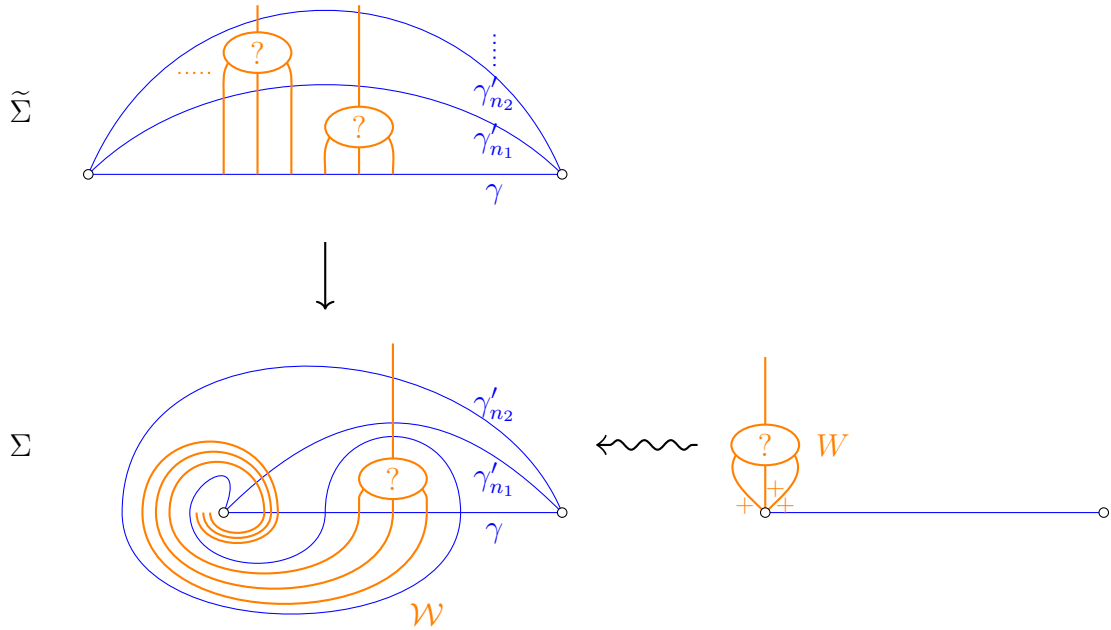

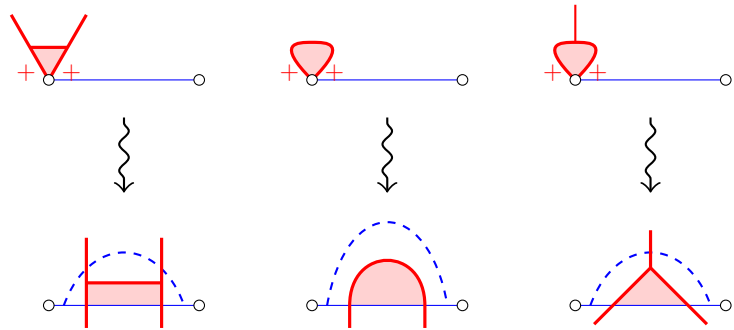
\begin{figure}[ht]
\centering
\begin{tikzpicture}
\begin{scope}[rotate=90]
\fill[pink!70] (0.3,0.5) -- (0.3,-0.5) -- (0,-0.5) -- (0,0.5) --cycle;
%\draw[blue] (2,0) -- (0,2) -- (-2,0) -- (0,-2) --cycle;
\draw[blue] (0,1) -- (0,-1);
\draw[blue,thick,dashed] (0,0.8) ..controls ++(-20:1) and ($(0,-0.8)+(20:1)$).. (0,-0.8);
\draw[red,very thick] (-0.3,0.5) -- (0.9,0.5);
\draw[red,very thick] (-0.3,-0.5) -- (0.9,-0.5);
\draw[red,very thick] (0.3,0.5) -- (0.3,-0.5);
\draw[fill=white] (0,1) circle(2pt);
\draw[fill=white] (0,-1) circle(2pt);
\end{scope}
\begin{scope}[xshift=3.5cm,rotate=90]
\fill[pink!70] (0,0.5) ..controls (0.8,0.5) and (0.8,-0.5).. (0,-0.5) -- (0,0.5);
%\draw[blue] (2,0) -- (0,2) -- (-2,0) -- (0,-2) --cycle;
\draw[blue] (0,1) -- (0,-1);
\draw[red,very thick] (-0.3,0.5) -- (0,0.5) ..controls (0.8,0.5) and (0.8,-0.5).. (0,-0.5) -- (-0.3,-0.5);
\draw[blue,thick,dashed] (0,0.8) ..controls ++(-10:1.5) and ($(0,-0.8)+(10:1.5)$).. (0,-0.8);
\draw[fill=white] (0,1) circle(2pt);
\draw[fill=white] (0,-1) circle(2pt);
\end{scope}
\begin{scope}[xshift=7cm,rotate=90]
\fill[pink!70] (0,0.5) -- (0.5,0) -- (0,-0.5) --cycle;
%\draw[blue] (2,0) -- (0,2) -- (-2,0) -- (0,-2) --cycle;
\draw[blue] (0,1) -- (0,-1);
\draw[blue,thick,dashed] (0,0.8) ..controls ++(-20:1) and ($(0,-0.8)+(20:1)$).. (0,-0.8);
\draw[red,very thick] (-0.25,0.75) -- (0.5,0) -- (-0.25,-0.75);
\draw[red,very thick] (0.5,0) -- (1,0);
\draw[fill=white] (0,1) circle(2pt);
\draw[fill=white] (0,-1) circle(2pt);
\end{scope}
\begin{scope}[yshift=3cm]
\fill[pink!70] (-1,0) --++ (60:0.5) --++(-0.5,0) --cycle;
%\draw[blue] (2,0) -- (0,2) -- (-2,0) -- (0,-2) --cycle;
\draw[blue] (-1,0) -- (1,0);
\draw[red,very thick] (-1,0) --++ (60:1);
\draw[red,very thick] (-1,0) --++ (120:1);
\draw[red,very thick] (-1,0) ++ (60:0.5) --++(-0.5,0);
\node[red,scale=0.8] at (0.3-1,0.1) {$+$};
\node[red,scale=0.8] at (-0.3-1,0.1) {$+$};
\draw[fill=white] (1,0) circle(2pt);
\draw[fill=white] (-1,0) circle(2pt);
\end{scope}
\begin{scope}[xshift=3.5cm,yshift=3cm]
\draw[blue] (-1,0) -- (1,0);
\filldraw[draw=red,fill=pink!70,very thick] (-1,0) ..controls (-1.5,0.5) and (-1.2,0.5).. (-1,0.5) ..controls (-0.8,0.5) and (-0.5,0.5).. (-1,0);
\node[red,scale=0.8] at (0.3-1,0.1) {$+$};
\node[red,scale=0.8] at (-0.3-1,0.1) {$+$};
\draw[fill=white] (1,0) circle(2pt);
\draw[fill=white] (-1,0) circle(2pt);
\end{scope}
\begin{scope}[xshift=7cm,yshift=3cm]
\draw[blue] (-1,0) -- (1,0);
\filldraw[draw=red,fill=pink!70,very thick] (-1,0) ..controls (-1.5,0.5) and (-1.2,0.5).. (-1,0.5) ..controls (-0.8,0.5) and (-0.5,0.5).. (-1,0);
\draw[red,thick] (-1,0.5) --++(0,0.5);
\node[red,scale=0.8] at (0.3-1,0.1) {$+$};
\node[red,scale=0.8] at (-0.3-1,0.1) {$+$};
\draw[fill=white] (1,0) circle(2pt);
\draw[fill=white] (-1,0) circle(2pt);
\end{scope}
\foreach \i in {0,3.5,7}
\draw [thick,-{Classical TikZ Rightarrow[length=4pt]},decorate,decoration={snake,amplitude=1.8pt,pre length=2pt,post length=3pt}](0+\i,2.5) --++ (0,-1);
\end{tikzpicture}
    \caption{The correspondence between the puncture-faces (top) and the faces sticked to $\gamma$ (bottom).}
    \label{fig:side-puncture_correspondence}
\end{figure}

\begin{proof}[Proof of \cref{prop:minimal_position_unbounded}]
Suppose that $\cW$ is not in minimal position with $\gamma$. Then there exists an ideal arc $\gamma_0\simeq \gamma$ such that $i(\cW;\gamma,\gamma_0)>0$ and in general position with $\gamma$ and $\cW$. Choose $\gamma_0$ so that $i(\cW;\gamma,\gamma_0)$ is maximal, whose existence is ensured by \cref{lem:relative_intersection_maximal}. Then for any other ideal arc $\gamma'$ isotopic to $\gamma$, we have
\begin{align*}
    i(\cW;\gamma',\gamma_0) = i(\cW;\gamma',\gamma) + i(\cW;\gamma,\gamma_0) = -i(\cW;\gamma,\gamma') + i(\cW;\gamma,\gamma_0) \geq 0
\end{align*}
by \cref{lem:relative_int_cocycle} and the maximality of $\gamma_0$. It implies that $\gamma_0$ is in minimal position with $\cW$, as desired. 
\end{proof}

\begin{cor}[cf. {\cite[Corollary 12 and Proposition 13]{FS20}}]\label{cor:minimal_position_unbounded}
Any spiralling diagram $\cW$ associated with a signed web on $\Sigma$ can be isotoped through a finite number of intersection reduction moves and H-moves so that it is in minimal position simultaneously with any disjoint finite collection $\{\gamma_i\}_{i=1}^N$ of ideal arcs. Such a minimal position with $\{\gamma_i\}_{i=1}^N$ is unique up to isotopy relative to these arcs, H-moves, periodic H-moves and parallel moves.
\end{cor}

\begin{proof}
%In order to prove the first statement, we need to verify that the procedure to obtain a minimal position with an ideal arc is independent of the others. 
As in the discussion above, we isotope the arcs instead of the spiralling diagram.
Let $\{\gamma_i\}_i$ be the original collection of ideal arcs, and $\{\gamma'_i\}_i$ the collection of modified arcs such that $i(\cW;\gamma_i,\gamma'_i)$ is maximal. Let $B_i$ be the biangle bounded by $\gamma_i$ and $\gamma'_i$. We claim that we can slightly modify $B_i$ as in (b) above so that it does not cross $\gamma'_j$ for any $i \neq j$. 
Indeed, suppose $B_i$ crosses $\gamma'_j$. If we can shrink $B_i$ without changing the intersection with $\cW$, do so. Otherwise, it implies that $\gamma_i$ and $\gamma'_j$ bound together at least one biangle $B' \subset B_j$, for which we can apply a reduction move (see \cref{fig:biangle_crossing}). It contradicts to the maximality of $i(\cW;\gamma_j,\gamma'_j)$. 

Hence, the biangle $B_i$ is either disjoint from $B_j$ or intersect with $B_j$ only through $\gamma_j$. In the former case, the reduction moves are independently applied. In the latter case, some of the reduction moves are common for $\gamma_i$ and $\gamma_j$ but still the minimal positions can be simultaneously realized. Thus we get the first statement. 

The second one is proved by induction on the number $N$ of arcs, just in the same way as the proof of \cite[Proposition 13]{FS20}.
\end{proof}

\begin{figure}[ht]
    \centering
    \begin{tikzpicture}[scale=.7]
\draw[blue] (-4,-1) .. controls (-2.5,-1) and (-1.5,1) .. (0,1) .. controls (1.5,1) and (1,-1) .. (2.5,-1);
\draw [cyan](-4,-1) .. controls (-2.85,0.5) and (1,-1) .. (2,0) .. controls (3,1) and (-2.1,1.5) .. (-1,3);
\draw [red, very thick](-3,-1.5) -- (0,0.5) coordinate (v1) {} -- (0.5,3);
\draw [red, very thick](v1) -- (1,-1.5);
\draw (-1,3) node [draw, fill=white, circle, inner sep=1.3pt] {} -- (-4,-1) node [draw, fill=white, circle, inner sep=1.3pt] {} -- (2.5,-1) node [draw, fill=white, circle, inner sep=1.3pt] {};
\node at (-1,-1.3) {$\gamma_i$};
\node at (-2.75,1.25) {$\gamma_j$};
\node[blue] at (-1.55,0.95) {$\gamma_i'$};
\node[cyan] at (2,1) {$\gamma_j'$};
\draw[cyan, dashed] (-1.5,-0.3) .. controls (-0.8,1.1) and (0.7,1.1) .. (1,-0.4);
\node[cyan,scale=0.9] at (0.54,0) {$B'$};
\end{tikzpicture}
\qquad \qquad
\begin{tikzpicture}[scale=.7]
\draw (-1,3) node [draw, fill=white, circle, inner sep=1.3pt] {} -- (-4,-1) node [draw, fill=white, circle, inner sep=1.3pt] {} -- (2.5,-1) node [draw, fill=white, circle, inner sep=1.3pt] {};
\node at (-1.5,-1.35) {$\gamma_i$};
\node at (-1.25,2.05) {$\gamma_j$};
\node[blue] at (0.5,0.5) {$\gamma_i'$};
\node[cyan] at (-3,3) {$\gamma_j'$};
\draw [red, very thick](-3,1.4) -- (-4.5,2.5);
\draw [red, very thick](-3,1.4) .. controls (-3,0.5) and (-1,0) .. (-1,-1.5);
\draw [red, very thick](-3,1.4) .. controls (-2,1.5) and (0.5,0) .. (0.5,-1.5);
\draw [blue](-4,-1) .. controls (-2.5,-0.75) and (-4.35,0.9) .. (-3.35,1.9) .. controls (-2.35,2.9) and (0.5,-0.5) .. (2.5,-1);
\draw [cyan](-4,-1) .. controls (-3.5,0) and (-4.85,1) .. (-4,2) .. controls (-3.15,3) and (-2,2) .. (-1,3);
\draw [cyan, dashed](-3,0.35) .. controls (-3.65,1.35) and (-3.2,2.2) .. (-2.1,1.5);
\end{tikzpicture}
    \caption{Situations where $B_i$ essentially crosses $\gamma_j'$ (left) and $B_i$ crosses $\gamma_j$ essentially (right). Both pictures show the case where $B_i$ intersects $B_j$ only once.}
    \label{fig:biangle_crossing}
\end{figure}

\paragraph{\textbf{Proof of \cref{prop:spiralling_good_position}: realization of a good position.}}

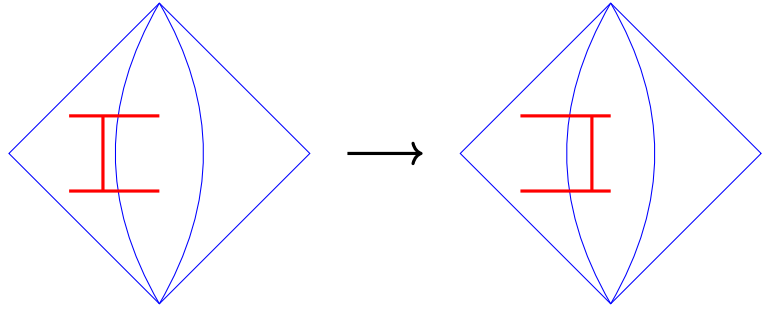
\begin{figure}[ht]
    \centering
\begin{tikzpicture}
\draw[blue] (2,0) -- (0,2) -- (-2,0) -- (0,-2) --cycle;
\draw[blue] (0,-2) to[bend left=30pt] (0,2);
\draw[blue] (0,-2) to[bend right=30pt] (0,2);
\draw[red,very thick] (-1.2,0.5) -- (0,0.5);
\draw[red,very thick] (-1.2,-0.5) -- (0,-0.5);
\draw[red,very thick] (-0.75,0.5) -- (-0.75,-0.5);

\draw[very thick,->] (2.5,0) -- (3.5,0);

\begin{scope}[xshift=6cm]
\draw[blue] (2,0) -- (0,2) -- (-2,0) -- (0,-2) --cycle;
\draw[blue] (0,-2) to[bend left=30pt] (0,2);
\draw[blue] (0,-2) to[bend right=30pt] (0,2);
\draw[red,very thick] (-1.2,0.5) -- (0,0.5);
\draw[red,very thick] (-1.2,-0.5) -- (0,-0.5);
\draw[red,very thick] (-0.25,0.5) -- (-0.25,-0.5);
\end{scope}
\end{tikzpicture}
    \caption{Pushing a ladder into a biangle.}
    \label{fig:tidying_up}
\end{figure}

By \cref{cor:minimal_position_unbounded}, we can place any spiralling diagram $\cW$ in a minimal position with the ideal arcs in the split triangulation $\widehat{\tri}$. 
%Moreover, by applying boundary H-moves if necessary, we may assume that $\cW$ is weakly reduced (\emph{i.e.}, no 4-gon faces on boundary). 
Then by applying a finite number of H-moves and periodic H-moves, we can push all the \emph{ladders} as in \cref{fig:tidying_up} into biangles (the \lq\lq tidying up'' operation in \cite{FS20}). Assume that these moves can be no longer applied to $\cW$. We are going to prove that this position (the \lq\lq joy-sparking'' position in \cite{FS20}) is a good position with respect to $\widehat{\tri}$. 

For each $E \in e(\tri)$, the intersection $\cW\cap B_E$ is an unbounded essential web by \cref{eq:bigon_lemma_unbounded}, since it is in minimal position with the ideal arcs bounding $B_E$. For each $T \in t(\tri)$, we see that the only components of $\cW \cap T$ which do not touch all sides of $T$ are corner arcs by \cref{eq:bigon_lemma_unbounded}. Indeed, such a component can be viewed as a web in a biangle obtained from $T$ by collapsing one edge that is not touched, and the ladders in the periodic part have been pushed into the biangles neighboring to $T$. Let $W'$ be the web obtained from $\cW \cap T$ by removing these corner arcs, which must be finite. Then we see that $W'$ must be a honeycomb in the same as in the last part in the proof of \cite[Theorem 19]{FS20}. Hence $\cW \cap T$ is an unbounded rung-less essential web. The uniqueness statement follows from that of \cref{cor:minimal_position_unbounded}. Thus \cref{prop:spiralling_good_position} is proved.

\subsection{Proof of \texorpdfstring{\cref{prop:injectivity}}{Theorem 3.19}}\label{sec:traveler}
We are going to prove \cref{prop:injectivity} by following the strategy for the proof of \cite[Theorem 47]{DS20I}. We remark here that another proof of the latter statement is given in \cite[Section 14]{FS20} based on the graded skein algebras.

The main issue here is that we have fixed the periodic pattern of corner arcs in the reconstruction procedure. Hence the resulting spiralling diagram may differ from the original one by a periodic permutation of corner arcs (\lq\lq periodic local parallel-moves'') on each triangle. Our claim is that these local adjustments glue together to give a global parallel-move, thus we get equivalent $\fsl_3$-laminations. See \cref{fig:traveler_labelings} for a typical example.

By the $\bQ_{>0}$-equivariance, it suffices to consider integral unbounded $\fsl_3$-laminations, which are represented by signed non-elliptic webs. Therefore it suffices to prove the following statement:

\begin{prop}\label{prop:same_shear}
If two signed non-elliptic webs $W_1,W_2$ have the same shear coordinates $(\sfx_i)_{i \in I_\uf(\tri)}$ with respect to an ideal triangulation $\tri$, then $W_1$ and $W_2$ are equivalent as unbounded $\fsl_3$-laminations.
\end{prop}
In what follows, the index $\nu \in \{1,2\}$ will always given to the objects associated to the web $W_\nu$. For a discrete subset $A \subset \bR$ (e.g., $A=\bZ$), we call a subset $I \subset A$ of the form $I=[a,b] \cap A$ for a (possibly unbounded) interval $[a,b] \subset \bR$ an \emph{interval} in $A$.

\bigskip
\paragraph{\textbf{Global pictures.}}
Let $W_1,W_2$ be as in \cref{prop:same_shear}. 
For $\nu=1,2$, we may assume that the associated spiralling diagram $\cW_\nu$ is placed in a good position with respect to the split triangulation $\widehat{\tri}$ by \cref{prop:spiralling_good_position}. Then its braid representative $\cW_{\nu,\mathrm{br}}^\tri$ has at most one honeycomb component on each triangle. 
Let $\Sigma^\circ$ be the holed surface, which is a compact surface obtained by removing a small open disk $D_T$ in each $T \in t(\Delta)$ from $\Sigma$. We may isotope the unique honeycomb component of $\cW_{\nu,\mathrm{br}}^\tri$ into the disk $D_T$, so that $\picW{\nu}:=\cW_{\nu,\mathrm{br}}^\tri \cap \Sigma^\circ$ is a collection of oriented curves, whose ends either lie on $\partial\Sigma^\circ$ or spiral around punctures. Following \cite{DS20I}, we call $\picW{\nu}$ the \emph{global picture} associated with $\cW_{\nu,\mathrm{br}}^\tri$. It is obvious to reconstruct the braid representative from its global picture. 
We call each oriented curve in $\langle \cW_\nu \rangle$ a \emph{traveler}. 

Recall from Step 1 and Step 2 in the reconstruction procedure (\cref{subsec:reconstruction}) that we can construct a braid representative $W_{\nu,\mathrm{br}}^\tri$ of signed web by replacing the spiralling ends with signed ends. We similarly define its global picture by $\pictW{\nu}:=W_{\nu,\mathrm{br}}^\tri \cap \Sigma^\circ$. For the scheme of our proof, see \cref{fig:proof_scheme_1}. Our strategy is as follows:
\begin{enumerate}
    \item Starting from the assumption in \cref{prop:same_shear}, we are going to make a correspondence between the topological data of global pictures $\pictW{1}$ and $\pictW{2}$ (namely, their travelers and intersection points among them) by an unbounded version of the \lq\lq Fellow-Traveler Lemma`` (\cite[Lemma 57]{DS20I}).
    \item From such a correspondence, we can describe a sequence of elementary moves relation $W_1$ and $W_2$ by just following the argument of Douglas--Sun \cite[Section 7.4--]{DS20I} for the bounded case. 
\end{enumerate}

\begin{figure}[ht]
\centering
\begin{tikzpicture}[scale=.8]
\node[draw,rectangle,rounded corners=0.7em,inner sep=0.6em,anchor=west](A) at (0,0) {Signed web $W_\nu$};
\node[draw,rectangle,rounded corners=0.7em,inner sep=0.6em,anchor=west](B) at (7,0) {Spiralling diagram $\cW_\nu$};
\node[draw,rectangle,rounded corners=0.7em,inner sep=0.6em,anchor=west](C) at (7,-3) {Global picture $\picW{\nu}$};
\node[draw,rectangle,rounded corners=0.7em,inner sep=0.6em,anchor=west](D) at (0,-3) {Global picture $\pictW{\nu}$};
\node(E) at (15,-3) {$\bZ^{I_\uf(\tri)}$};
\draw[->,thick,shorten >=0.5cm,shorten <=0.5cm] (A) -- (B);
\draw[<->,thick,shorten >=0.5cm,shorten <=0.5cm] (D) --node[midway,above]{$1:1$} node[midway,below=0.3em]{Steps 1 \& 2} (C);
\draw[<->,thick,shorten >=0.5cm,shorten <=0.5cm] (B) --node[midway,right]{$1:1$}++(0,-2.7);
\draw[<-,thick,shorten >=0.5cm,shorten <=0.5cm] (A) --node[midway,right]{Step 3}++(0,-2.7);
\draw[<->,thick,shorten >=0.5cm,shorten <=0.5cm] (C) --node[midway,above]{$1:1$} node[midway,below=0.3em]{\cref{lem:braid-shear}} (E);
\draw[->,dashed,->-,thick,rounded corners=1em] (A) --++(0,1) --node[midway,above]{$\sfx_\tri$}++(10.5,0) --++(0,-3) --($(E)+(0,1)$) --(E);
\draw[>-,dashed,-<-,thick,rounded corners=1em] (A) --++(-2.2,0) --++(0,-4) --node[midway,below]{$\xi_\tri$} ($(E)+(0,-1)$) --(E);
\end{tikzpicture}
    \caption{The scheme for a proof of \cref{prop:same_shear}.
    % The scheme for a proof of \cref{prop:injectivity}.
    It is obvious that the three objects $\cW_\nu$,
    $\picW{\nu}$ and $\pictW{\nu}$
    % $\cW_{\mathrm{br}}^\tri$, $W_{\mathrm{br}}^\tri$
    are in one-to-one correspondences (up to strict isotopies), when one fixes a triangulation $\tri$. It will be proved that we get the identity (up to equivalence of signed webs) after going through the square.}
    \label{fig:proof_scheme_1}
\end{figure}
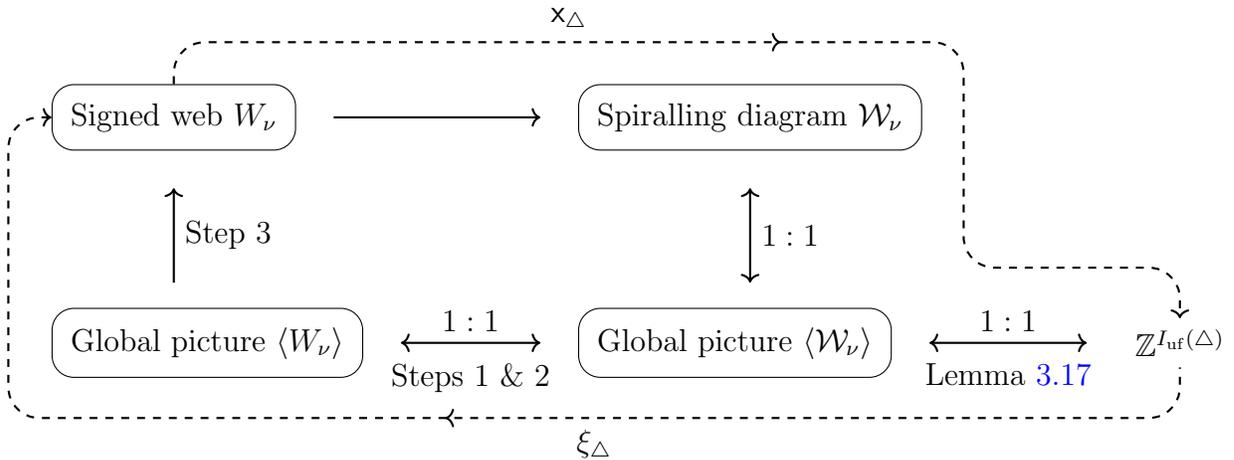

% \begin{figure}[ht]
% \begin{tikzpicture}
% \node[draw,rectangle,rounded corners=0.7em,inner sep=0.6em,anchor=west](A) at (0,0) {Signed web $W_\nu$};
% \node[draw,rectangle,rounded corners=0.7em,inner sep=0.6em,anchor=west](B) at (7,0) {Spiralling diagram $\cW_\nu$};
% \node[draw,rectangle,rounded corners=0.7em,inner sep=0.6em,anchor=west](C) at (7,-3) {Global picture $\picW{\nu}$};
% \node[draw,rectangle,rounded corners=0.7em,inner sep=0.6em,anchor=west](D) at (0,-3) {Global picture $\pictW{\nu}$};
% \node(E) at (15,-3) {$\bZ^{I_\uf(\tri)}$};
% \draw[->,thick,shorten >=0.5cm,shorten <=0.5cm] (A) -- (B);
% \draw[<->,thick,shorten >=0.5cm,shorten <=0.5cm] (D) --node[midway,above]{1:1} node[midway,below=0.3em]{Step 1, Step 2} (C);
% \draw[<->,thick,shorten >=0.5cm,shorten <=0.5cm] (B) --node[midway,right]{1:1}++(0,-2.7);
% \draw[<-,thick,shorten >=0.5cm,shorten <=0.5cm] (A) --node[midway,right]{Step 3, Step 4}++(0,-2.7);
% \draw[<->,thick,shorten >=0.5cm,shorten <=0.5cm] (C) --node[midway,above]{1:1} node[midway,below=0.3em]{\cref{lem:braid-shear}} (E);
% \draw[->,dashed,->-] (A) --++(0,1) --node[midway,above]{$\sfx_\tri$}++(10.5,0) --++(0,-3) --($(E)+(0,1)$) --(E);
% \draw[<-,dashed,-<-] (A) --++(-2,0) --++(0,-4) --node[midway,below]{$\xi_\tri$} ($(E)+(0,-1)$) --(E);
% \end{tikzpicture}
%     \caption{The scheme for a proof of \cref{prop:same_shear}. Comparing to \cref{fig:proof_scheme_1}, the object $\cW_{\nu,\mathrm{br}}^\tri$ (resp. $W_{\nu,\mathrm{br}}^\tri$) is replaced with the equivalent object $\picW{\nu}$ (resp. $\pictW{\nu}$).}
%     \label{fig:proof_scheme_2}
% \end{figure}

\bigskip
\paragraph{\textbf{Unbounded Fellow-Traveler Lemma.}}
For each traveler $\gamma$ in $\picW{\nu}$, fix a basepoint $x_0 \in \gamma$ so that it does not lie on any edge of $\tri$. Associated to such a based traveler $(\gamma,x_0)$ is the \emph{route} $(E_i)_{i \in I}$, where $I \subset \bZ$ is an interval and $E_i$ is the $i$-th edge of $\tri$ crossed by $\gamma$ listed in order according to the orientation of $\gamma$: the $0$-th edge is the first one encountered by $\gamma$ after passing $x_0$. 
We also define the \emph{turning pattern} $(\tau_i)_{i \in I} \subset \{L,S,R\}^I$ of the based traveler $(\gamma,x_0)$ as follows:
\begin{align*}
    \tau_i:=\begin{cases}
        L & \mbox{if $E_{i+1}$ follows $E_{i}$ in the counter-clockwise direction at their common endpoints},\\
        S & \mbox{if $\gamma$ ends at the boundary of $D_T$ right after passing $E_i$},\\
        R & \mbox{if $E_{i+1}$ follows $E_{i}$ in the clockwise direction at their common endpoints}.
    \end{cases}
\end{align*}
The following is immediately verified:

\begin{lem}\label{lem:traveler_types}
The topological types of the travelers $\gamma$ are distinguished by the periodicity of the data $(E_i,\tau_i)_{i \in I}$, as follows:
\begin{itemize}
    \item $\gamma$ is a bounded arc both of whose ends lie on $\partial\Sigma^\circ$ if $I \subset \bZ$ is bounded;
    \item $\gamma$ is a loop if $I=\bZ$ and the route is totally periodic $($namely, $E_{i+k}=E_i$ for some $k \in \bZ)$. Moreover, it is peripheral if the turning pattern $(\tau_i)_{i \in I}$ is constant;
    %all the edges $E_i$ have common endpoints at a puncture;
    \item $\gamma$ has an end spiralling to a puncture $p$, say in the forward direction, if $I \subset \bZ$ is unbounded from above, the route $(E_i)_i$ is not totally periodic but 
    eventually periodic, and the turning pattern $(\tau_i)_i$ is eventually constant in the forward direction. 
    %, and the edges $E_i$ have common endpoints at $p$ for sufficiently large $i$.
\end{itemize}
\end{lem}
We say that two travelers $\gamma^{(1)}$ in $\picW{1}$ and $\gamma^{(2)}$ in $\picW{2}$ are \emph{fellow-travelers} if their data $(E_i^{(1)},\tau_i^{(1)})_{i \in I_1}$ and $(E_i^{(2)},\tau_i^{(2)})_{i \in I_2}$ are the same, in the sense that there exists an order-preserving bijection $f: I_1 \to I_2$ such that $E_{f(i)}^{(2)}=E_i^{(1)}$, $\tau_{f(i)}^{(2)}=\tau_i^{(1)}$ for all $i \in I_1$. Notice that the notion of fellow-traveler does not depend on the choice of basepoints, and that two fellow-travellers have the same topological type by \cref{lem:traveler_types}.

\begin{lem}[Unbounded Fellow-Traveler Lemma, cf.~{\cite[Lemma 57]{DS20I}}]\label{lem:traveler_lemma}
Under the assumption of \cref{prop:same_shear}, there exists a bijection 
\begin{align*}
    \varphi: \{\text{non-peripheral travelers in $\picW{1}$}\} \xrightarrow{\sim} \{\text{non-peripheral travelers in $\picW{2}$}\}
\end{align*}
such that $\gamma$ and $\varphi(\gamma)$ are fellow-travelers. 
\end{lem}

\paragraph{\textbf{Traveler identifier.}} In order to prove \cref{lem:traveler_lemma}, let us introduce another data that identifies the traveler and can be characterized by the shear coordinates. 
%\begin{proof}
Let us consider two triangles $T_L,T_R \in t(\tri)$ that shares a biangle $B_E$. For $Z \in \{L,R\}$, let $E_Z$ denote the edge of $\widehat{\tri}$ shared by $T_Z$ and $B_E$. Let $S_{E_Z}^{+,(\nu)}$ (resp. $S_{E_Z}^{-,(\nu)}$) denote the set of strands on $E_Z$ incoming to (resp. outgoing from) the biangle $B_E$, which are given by the intersections of travelers in $\picW{\nu}$ and $E_Z$ for $\nu=1,2$. We endow $E_Z$ with the orientation induced from the triangle $T_Z$. 

% Choose the points $p_Z^{\pm,(\nu)} \in E_Z \setminus S_{E_Z}^{\pm,(\nu)}$ as follows:
% \begin{itemize}
%     \item The point $p_Z^{+,(\nu)}$ is put in a position right before the strands incoming to $B_E$ which come from the source-honeycomb component of $\picW{\nu}$ in the triangle $T_Z$, if the latter exists. Otherwise, put $p_Z^{+,(\nu)}$ in a position right after the corner arcs of $\picW{\nu}$ in $T_Z$ which surrounds the initial marked point of the oriented edge $E_Z$.  
%     \item Similarly, $p_Z^{-,(\nu)}$ is put right before the strands outgoing from $B_E$ which come from the sink-honeycomb component of $\picW{\nu}$ in $T_Z$, if the latter exists. Otherwise, put $p_Z^{-,(\nu)}$ right after the corner arcs of $\picW{\nu}$ in $T_Z$ which surrounds the initial marked point of the oriented edge $E_Z$.  
% \end{itemize}
% Here the before/after refers to the orientation of $E_Z$. See \cref{fig:traveler_labelings}. 
% Fix two orientation-preserving parametrizations $\psi_{E_Z}^{\pm,(\nu)}: E_Z \to \bR$ so that $\psi^{\pm,(\nu)}_{E_Z}(p_Z^{\pm,(\nu)})=0$. Notice that $\psi^{\pm,(\nu)}_{E_Z}(S_{E_Z}^{\pm,(\nu)}) \subset \frac{1}{2}+\bZ$ is an interval. 

Choose two orientation-preserving parametrizations of $E_Z$ in the same way as in \cref{subsec:reconstruction}. Namely, choose $\phi_{E_Z}^{\pm,(\nu)}: \bR \to E_Z$ so that the inverse image of $S_{E_Z}^{\pm,(\nu)}$ is an interval $I_{E_Z}^{\pm,(\nu)} \subset \frac{1}{2} + \bZ$, and $\phi_{E_Z}^{\pm,(\nu)}(\bR_{<0}) \cap S_{E_Z}^{\pm,(\nu)}$ consists of all the strands coming from the corner arcs around the initial marked point of $E_Z$. Let $f_{E_Z}^{\pm,(\nu)}: E_Z \to \bR$ be the inverse map of $\phi_{E_Z}^{\pm,(\nu)}$. 
For a traveler $\gamma^{(\nu)}$ in $\picW{\nu}$ that intersects with the edge $E_Z$ at a point $x$, its \emph{traveler identifier} at $E_Z$ is the pair $(k,\epsilon) \in (\frac{1}{2}+\bZ) \times \{\pm 1\}$ given by
\begin{align*}
    (k,\epsilon):=\begin{cases}
    (f_{E_Z}^{+,(\nu)}(x),+) & \mbox{if $\gamma$ enters $B_E$ from $E_Z$}, \\
    (f_{E_Z}^{-,(\nu)}(x),-) & \mbox{if $\gamma$ exits $B_E$ from $E_Z$}.
    \end{cases}
\end{align*}
Then we write $\gamma^{(\nu)}=\gamma^{(\nu)}_{E_Z}(k,\epsilon)$. 

\begin{ex}
In the example shown in \cref{fig:traveler_labelings}, we have
\begin{align*}
    &\gamma_1 = \gamma^{(\nu)}_{E_1}(5/2,-) = \gamma^{(\nu)}_{E_2}(-1/2,+) = \gamma^{(\nu)}_{E_3}(3/2,-),\\
    &\gamma_2 = \gamma^{(\nu)}_{E_1}(3/2,-) = \gamma^{(\nu)}_{E_2}(1/2,+),\\
    &\gamma_3 = \gamma^{(\nu)}_{E_1}(1/2,-) = \gamma^{(\nu)}_{E_2}(3/2,+),\\
    &\gamma_4 = \gamma^{(\nu)}_{E_1}(3/2,+) = \gamma^{(\nu)}_{E_2}(-5/2,-) = \gamma^{(\nu)}_{E_3}(5/2,+),\\
    &\gamma_5 = \gamma^{(\nu)}_{E_1}(1/2,+) = \gamma^{(\nu)}_{E_2}(-3/2,-) = \gamma^{(\nu)}_{E_3}(3/2,+),\\
    &\gamma_6 = \gamma^{(\nu)}_{E_1}(-1/2,+) = \gamma^{(\nu)}_{E_2}(-1/2,-) = \gamma^{(\nu)}_{E_3}(1/2,+),\\
    &\gamma_7 = \gamma^{(\nu)}_{E_1}(-3/2,+) = \gamma^{(\nu)}_{E_2}(1/2,+).
\end{align*}
\end{ex}

\begin{figure}[h]
    \centering
    \begin{tikzpicture}[scale=.57]
    \draw [blue](-1,2.5) .. controls (-2,1) and (-2,-1) .. (-1,-2.5);
    \draw [blue](-1,2.5) coordinate (v1) .. controls (0,1) and (0,-1) .. (-1,-2.5);
    \draw [blue](v1) -- (-5,0) -- (-1,-2.5) -- (3,0) -- (v1);
    \draw [blue] (-3,0) ellipse (0.5 and 0.5);
    \draw [blue] (1,0) ellipse (0.5 and 0.5);
    \draw [blue](-1,2.5) .. controls (0.8,3) and (2.75,1.85) .. (3,0);
    \draw [-<-, thick, red](-2.55,0.2) -- (0.5,0);
    \draw [->-={.43}{}, thick, red](-2.95,1.3) .. controls (-2,0.2) and (-0.1,0.2) .. (0.6,1.5) .. controls (0.8,1.9) and (0.5,2.2) .. (0.65,2.5);
    \draw [-<-={.42}{}, thick, red](-2.45,1.55) .. controls (-1.5,0.7) and (-0.55,0.75) .. (0.2,1.7) .. controls (0.45,2.05) and (0.9,2.05) .. (1.1,2.35);
    \draw [->-={.43}{}, thick, red](-1.95,1.85) .. controls (-1.45,1.2) and (-0.2,1.3) .. (0.15,2.6);
    \draw [-<-={.48}{}, thick, red](-2.55,-0.2) .. controls (-1.05,-0.3) and (-0.35,-0.55) .. (0.55,-1.55);
    \draw [->-={.3}{}, thick, red](1.35,0.35) -- (2.3,1.5);
    \draw [->-, thick, red](1.3,-0.4) -- (1.65,-0.85);
    \draw [-<-, thick, red](-3.15,0.5) -- (-3.45,0.95);
    \draw [-<-, thick, red](-3.45,0.2) -- (-3.85,0.7);
    \draw [-<-, thick, red](-3.45,-0.25) -- (-3.85,-0.7);
    \draw [-<-, thick, red](-3.15,-0.5) -- (-3.5,-0.95);
    \draw [->-={.43}{}, thick, red](-2.85,-1.35) .. controls (-1.5,-0.5) and (0.9,0.75) .. (1.7,2.05);
    \draw [->-, thick, orange](-2,-1.85) .. controls (-1.5,-1.35) and (-0.5,-1.35) .. (0,-1.85);
    
    \draw [blue](14,2.5) .. controls (13,1) and (13,-1) .. (14,-2.5);
    \draw [blue](14,2.5) coordinate (v1) .. controls (15,1) and (15,-1) .. (14,-2.5);
    \draw [blue](v1) -- (10,0) -- (14,-2.5) -- (18,0) -- (v1);
    \draw [blue] (12,0) ellipse (0.5 and 0.5);
    \draw [blue] (16,0) ellipse (0.5 and 0.5);
    \draw [blue](14,2.5) .. controls (15.8,3) and (17.75,1.85) .. (18,0);
    \draw [-<-, thick, red](12.45,0.2) -- (15.5,0);
    \draw [->-, thick, red](13.05,1.9) .. controls (13.55,1.2) and (14.35,1.35) .. (15.15,2.6);
    \draw [-<-={.48}{}, thick, red](12.45,-0.2) .. controls (13.95,-0.3) and (14.65,-0.55) .. (15.55,-1.55);
    \draw [->-={.3}{}, thick, red](16.35,0.35) -- (17.3,1.5);
    \draw [->-, thick, red](16.3,-0.4) -- (16.65,-0.85);
    \draw [-<-, thick, red](11.85,0.5) -- (11.55,0.95);
    \draw [-<-, thick, red](11.55,0.2) -- (11.15,0.7);
    \draw [-<-, thick, red](11.55,-0.25) -- (11.15,-0.7);
    \draw [-<-, thick, red](11.85,-0.5) -- (11.5,-0.95);
    \draw [->-={.43}{}, thick, red](12.15,-1.35) .. controls (13.5,-0.5) and (15.9,0.65) .. (16.7,2.05);
    \draw [->-, thick, red](12,1.25) .. controls (13,0) and (14,0.15) .. (15.7,2.5);
    \draw [-<-={.45}{}, thick, red](12.55,1.6) .. controls (13.75,0) and (14.6,-0.05) .. (16.2,2.3);
    
    \draw [blue](6.5,-1) .. controls (5.5,-2.5) and (5.5,-4.5) .. (6.5,-6);
    \draw [blue](6.5,-1) coordinate (v1) .. controls (7.5,-2.5) and (7.5,-4.5) .. (6.5,-6);
    \draw [blue](v1) -- (2.5,-3.5) -- (6.5,-6) -- (10.5,-3.5) -- (v1);
    \draw [blue] (4.5,-3.5) ellipse (0.5 and 0.5);
    \draw [blue] (8.5,-3.5) ellipse (0.5 and 0.5);
    \draw [blue](6.5,-1) .. controls (8.3,-0.5) and (10.25,-1.65) .. (10.5,-3.5);
    \draw [->-, thick, red](5.55,-1.6) .. controls (6.05,-2.3) and (6.85,-2.15) .. (7.7,-0.9);
    \draw [->-, thick, red](4.65,-4.85) .. controls (6,-4) and (8.4,-2.85) .. (9.2,-1.45);
    \draw [->-, thick, red](4.5,-2.25) .. controls (5.5,-3.5) and (6.5,-3.35) .. (8.5,-1.1);
    \draw [->-, thick, orange](5.5,-5.35) .. controls (6,-4.85) and (7,-4.85) .. (7.5,-5.35);
    
    \draw [blue](6.5,6) .. controls (5.5,4.5) and (5.5,2.5) .. (6.5,1);
    \draw [blue](6.5,6) coordinate (v1) .. controls (7.5,4.5) and (7.5,2.5) .. (6.5,1);
    \draw [blue](v1) -- (2.5,3.5) -- (6.5,1) -- (10.5,3.5) -- (v1);
    \draw [blue] (4.5,3.5) ellipse (0.5 and 0.5);
    \draw [blue] (8.5,3.5) ellipse (0.5 and 0.5);
    \draw [blue](6.5,6) .. controls (8.3,6.5) and (10.25,5.35) .. (10.5,3.5);
    \draw [-<-, thick, red](4.95,3.7) -- (8,3.5);
    \draw [-<-, thick, red](4.95,3.3) .. controls (6.45,3.2) and (7.15,2.95) .. (8.05,1.95);
    \draw [->-={.3}{}, thick, red](8.85,3.85) -- (9.8,5);
    \draw [->-, thick, red](8.8,3.1) -- (9.15,2.65);
    \draw [-<-, thick, red](4.35,4) -- (4.05,4.45);
    \draw [-<-, thick, red](4.05,3.7) -- (3.65,4.2);
    \draw [-<-, thick, red](4.05,3.25) -- (3.65,2.8);
    \draw [-<-, thick, red](4.35,3) -- (4,2.55);
    \draw [-<-={.4}{}, thick, red](5.05,5.1) .. controls (6.25,3.75) and (7.1,3.7) .. (8.7,5.8);
    
    \node[scale=.65] at (5.65,2.75) {0 $-$}[scale=.65];
    \node[scale=.65] at (5.6,3.45) {1 $-$};
    \node[scale=.65] at (5.65,4) {2 $-$};
    \node[scale=.65] at (5.85,4.85) {3 $-$};
    \node[scale=.65] at (7.35,4.8) {$-$ $-1$};
    \node[scale=.65] at (7.4,3.9) {$-$ 0};
    \node[scale=.65] at (7.4,3.1) {$-$ 1};
    \node[scale=.65] at (7.2,2.15) {$-$ 2};
    \node[scale=.65,rotate=60] at (9.6,4.05) {$-$};
    \node[scale=.65,rotate=60] at (8.65,4.65) {$-$};
    \node[scale=.65,rotate=60] at (7.5,5.35) {$-$};
    \node[scale=.65] at (9.75,4.25) {0};
    \node[scale=.65] at (8.8,4.85) {1};
    \node[scale=.65] at (7.65,5.55) {2};
    
    \node[scale=.65] at (5.85,-5.35) {$-2$ $-$};
    \node[scale=.65] at (5.6,-4.65) {$-1$ $-$};
    \node[scale=.65] at (5.6,-3.55) {0 $-$};
    \node[scale=.65] at (5.7,-2.4) {1 $-$};
    \node[scale=.65] at (7.5,-2.85) {$-$ $-1$};
    \node[scale=.65] at (7.35,-4.15) {$-$ 0};
    \node[scale=.65] at (7,-5.35) {$-$ 1};
    
    \node[rotate=60,scale=.65] at (8.95,-2.55) {$-$};
    \node[scale=.65,rotate=60] at (8.2,-2.05) {$-$};
    \node[scale=.65,rotate=60] at (7.5,-1.65) {$-$};
    \node[scale=.65] at (9.1,-2.35) {0};
    \node[scale=.65] at (8.35,-1.85) {1};
    \node[scale=.65] at (7.65,-1.45) {2};
    \node[scale=.65] at (7.3,-2.1) {$-$ $-2$};
    \node[scale=.65] at (6.8,-1.55) {$-$};
    \node[scale=.65] at (6,-1.6) {2 $-$};
    \node[scale=.65,rotate=60] at (7.05,-1.35) {$-$};
    
    \node [scale=.9,red] at (8.85,6.05) {$\gamma_1$};
    \node [scale=.9,red] at (8.3,3.5) {$\gamma_2$};
    \node [scale=.9,red] at (8.25,1.7) {$\gamma_3$};
    \node [scale=.9,red] at (7.75,-0.65) {$\gamma_4$};
    \node [scale=.9,red] at (8.6,-0.85) {$\gamma_5$};
    \node [scale=.9,red] at (9.3,-1.2) {$\gamma_6$};
    \node [scale=.9,orange] at (7.7,-5.6) {$\gamma_7$};
    \node [scale=.8,blue] at (5.75,2) {$E_1$};
    \node [scale=.8,blue] at (6.7,5.15) {$E_2$};
    \node [scale=.8,blue] at (9.6,-2.65) {$E_3$};
    \node [scale=.9] at (-4,-2) {$\picW{1}$};
    \node [scale=.9] at (17,-2) {$\picW{2}$};
    \end{tikzpicture}
    \caption{Example of local picture of a pair $(\picW{1},\picW{2})$ having the same shear coordinates. Here the top (resp. bottom) picture shows the collection of oriented curves going through the central biangle from the right to the left (resp. from the left to the right), which is common for $\picW{1}$ and $\picW{2}$ except for $\gamma_7$.}
    \label{fig:traveler_labelings}
\end{figure}
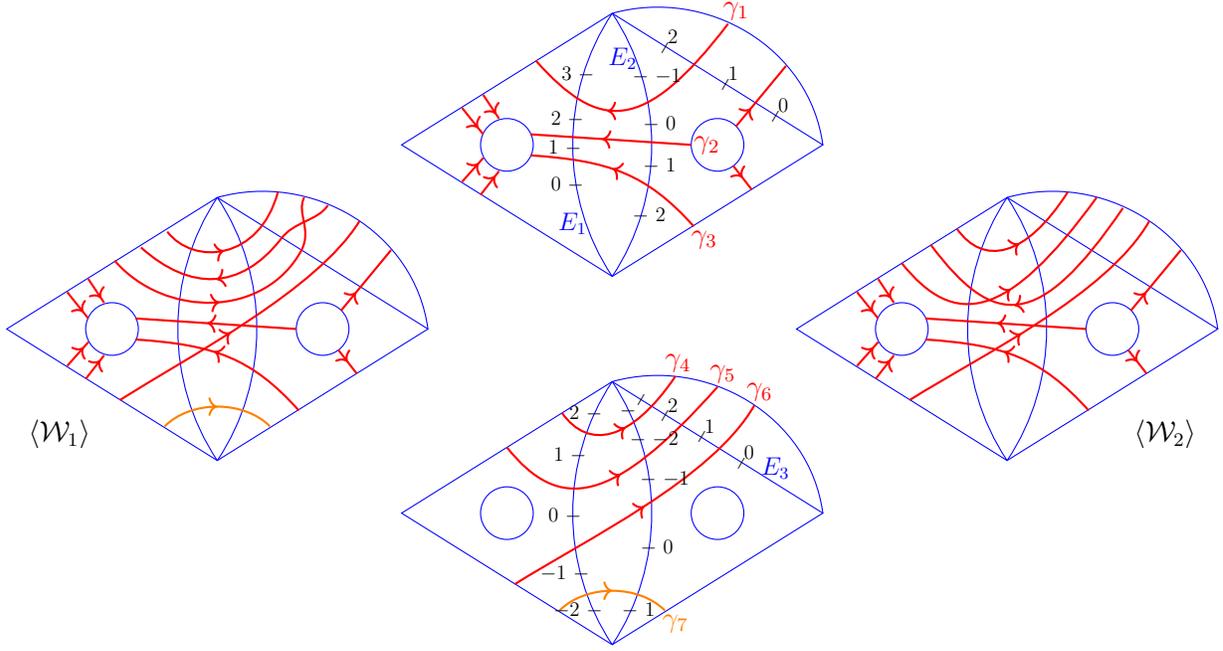

\begin{lem}\label{lem:shear_identifier}
Let $\gamma^{(\nu)}$ be a traveler in $\picW{\nu}$ that passes through $B_E$ from $E_L$ to $E_R$. Let $(k_L,+)$ and $(k_R,-)$ be its traveler identifier at $E_L$ and $E_R$, respectively. Then we have
%Then we can write $\gamma^{(\nu)}=\gamma^{(\nu)}_{E_L}(k_L,+)= \gamma^{(\nu)}_{E_R}(k_R,-)$ for some $k_L,k_R \in \frac 1 2 +\bZ$, which satisfies 
\begin{align*}
    k_L+k_R=\sfx_{\bott}(W_\nu)+[\sfx_{T_R}(W_\nu)]_+.
\end{align*}
If $\gamma^{(\nu)}$ passes through $B_E$ from $E_R$ to $E_L$, then its traveler identifiers $(k_R,+)$ and $(k_L,+)$ satisfy $k_L+k_R=\sfx_{\topp}(W_\nu)+[\sfx_{T_L}(W_\nu)]_+$.
\end{lem}

\begin{proof}
Just observe that our choice of parametrizations $\phi_{E_Z}^{\pm,(\nu)}$ is the same as in the reconstruction procedure (\cref{subsec:reconstruction}), except for the difference that we do not necessarily have an infinite number of corner arcs here. 
Then the assertion is obtained from the gluing rule \eqref{eq:gluing_rule}. 
\end{proof}

\begin{lem}\label{lem:identifier_traveler}
The traveler identifiers characterizes the traveler and its topological type. Namely,
\begin{enumerate}
    \item The traveler identifier determines the data $(E_i,\tau_i)_{i \in I}$ for each traveler. 
    \item If $\gamma^{(1)}_{E}(k,\epsilon)=\gamma^{(1)}_{E'}(k',\epsilon')$ for two edges $E,E'$ of the split triangulation $\widehat{\tri}$, then $\gamma^{(2)}_{E}(k,\epsilon)=\gamma^{(2)}_{E'}(k',\epsilon')$ holds. 
\end{enumerate}
\end{lem}

\begin{proof}
(1): The initial edge $E_0$ is determined from the basepoint $x_0$. Assume that we have determined the data $E_i$ for $0 \leq i\leq k$ and $\tau_j$ for $0\leq j\leq k-1$. Let $E:=E_k$. Then $E_{k-1}$ and $\tau_{k-1}$ tell us from which direction our traveler passes through the biangle $B_E$. Assume it is from $T_L$ to $T_R$, without loss of generality. Then by \cref{lem:shear_identifier}, we have $k_R=\sfx_{\bott}(W_\nu)+[\sfx_{T_R}(W_\nu)]_+ - k_L$. Then by the choice of the parameterization $\phi_{E_R}^{-,(\nu)}$, we see that 
\begin{align*}
    \tau_k=\begin{cases}
        L & \mbox{if $k_R <0$}, \\
        S & \mbox{if $0 < k_R < [\sfx_{T_R}(W_\nu)]_+$},\\
        R & \mbox{if $k_R > [\sfx_{T_R}(W_\nu)]_+$}.
    \end{cases}
\end{align*}
See \cref{fig:enter-label}. 
%For example, the first pattern follows from the condition that $\phi_{E_R}^{\pm,(\nu)}(\bR_{<0}) \cap S_{E_R}^{\pm,(\nu)}$ consists of all the strands coming from the corner arcs around the initial marked point of $E_R$. 
Moreover, the pattern $\tau_k$ tells us the next edge $E_{k+1}$ or its absence. 

(2): Recall that the shear coordinates of $W_1$ and $W_2$ are the same. Since the reconstruction given in (1) is characterized by the shear coordinates, the assertion follows.
\end{proof}

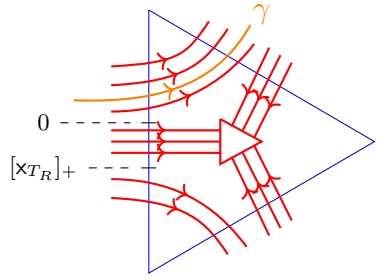
\begin{figure}[ht]
    \centering
    \begin{tikzpicture}
\draw [blue](-1.5,1.75) coordinate (v1) {} -- (-1.5,-1.75) -- (1.5,0) -- (v1);
\draw [red, thick](-0.55,0.3) coordinate (v2) {} -- (-0.55,-0.3) -- (0,0) -- (v2);
\draw [red, thick, ->-](-2,0) -- (-0.55,0);
\draw [red, thick, ->-](-2,-0.15) -- (-0.55,-0.15);
\draw [red, thick, ->-](-2,0.15) -- (-0.55,0.15);
\draw [red, thick, -<-](-0.1,0.05) -- (0.4,1.05);
\draw [red, thick, -<-](-0.25,0.15) -- (0.25,1.15);
\draw [red, thick, -<-](-0.4,0.2) -- (0.1,1.25);
\draw [red, thick, ->-](-2,0.4) .. controls (-1,0.4) and (-0.4,0.9) .. (-0.1,1.3);
\draw [red, thick, ->-](-2,0.75) .. controls (-1,0.8) and (-0.75,1.05) .. (-0.45,1.5);
\begin{scope}[yscale=-1]
\draw [red, thick, -<-](-0.1,0.05) -- (0.4,1.05);
\draw [red, thick, -<-](-0.25,0.15) -- (0.25,1.15);
\draw [red, thick, -<-](-0.4,0.2) -- (0.1,1.25);
\draw [red, thick, ->-](-2,0.5) .. controls (-1,0.5) and (-0.5,1) .. (-0.2,1.4);
\draw [red, thick, ->-](-2,0.75) .. controls (-1,0.8) and (-0.75,1.05) .. (-0.45,1.5);
\end{scope}
\draw [red, thick, ->-](-2,1) .. controls (-1.1,1.05) and (-1,1.15) .. (-0.7,1.6);
\draw [orange, thick, ->-](-2.5,0.55) .. controls (-1.2,0.5) and (-0.55,0.85) .. (-0.15,1.55);
\draw (-1.6,0.25) node (v6) {} -- (-1.4,0.25);
\draw (-1.6,-0.35) node (v4) {} -- (-1.4,-0.35);
\node (v5) at (-2.9,0.25) {\scriptsize $0$};
\node (v3) at (-2.9,-0.35) {\scriptsize $[\sfx_{T_R}]_+$};
\draw [dashed] (v3) edge (v4);
\draw [dashed] (v5) edge (v6);
\node [orange] at (0,1.7) {$\gamma$};
\end{tikzpicture}
    \caption{The turning pattern determined by the value of $k_R$.}
    \label{fig:enter-label}
\end{figure}

\begin{proof}[Proof of \cref{lem:traveler_lemma}]
Define the bijection $\varphi$ by 
\begin{align}\label{eq:traveler_correspondence}
    \varphi: \gamma^{(1)}_{E}(k,\epsilon) \mapsto \gamma^{(2)}_{E}(k,\epsilon).
\end{align}
It is well-defined by \cref{lem:identifier_traveler} (2), and preserves the topological types of travelers by \cref{lem:identifier_traveler} (1). 
\end{proof}

\begin{rem}
From the proof of \cref{lem:shear_identifier}, a traveler $\gamma=\gamma_E^{(1)}(k,\epsilon)$ with 
$k \in I_E^{\epsilon,(1)} \setminus I_E^{\epsilon,(2)}$
% $k \in \psi_{E}^{\epsilon,(1)}(S_E^{\epsilon,(1)}) \setminus \psi_{E}^{\epsilon,(2)}(S_E^{\epsilon,(2)})$
must be peripheral. For example, if $\gamma=\gamma_{E_L}^{(1)}(k_L,+)$ and $k_L < \min I_{E_L}^{\epsilon,(2)}$ is a lower excess, then it must have $\tau_k=R$, since otherwise it has a non-trivial contribution to the edge coordinates. It follows that such a traveler also has an identifier of lower excess in the next biangle, concluding $\tau_k=R$ for all $k\in \bZ$ inductively for both directions. 
%Thus one sees inductively that it must always turn right. 
See $\gamma_7$ in \cref{fig:traveler_labelings} for an example.
\end{rem}

\paragraph{\textbf{Correspondence between the global pictures $\pictW{1}$ and $\pictW{2}$.}}
Let $W_1$ and $W_2$ be as in \cref{prop:same_shear}. 
Then by the unbounded fellow-traveler Lemma, we have a bijective correspondence $\varphi$ between the travelers in $\picW{1}$ and $\picW{2}$. 
%For $\nu=1,2$, replace the spiralling ends of travelers in $\picW{\nu}$ with ends incident at punctures, encoding the spiralling directions in signs by reversing the rule in \cref{fig:spiral}. Then we get a collection $\pictW{\nu}$ of oriented curves with signed ends at punctures. 
Let us consider the global pictures $\pictW{1}$ and $\pictW{2}$, and call each oriented curve in $\pictW{\nu}$ a traveler again. Since the bijection $\varphi$ preserves the spiralling types of travelers in $\picW{\nu}$, it induces a bijection
\begin{align}\label{eq:traveler_correspondence_finite}
    \varphi: \{\text{travelers in $\pictW{1}$}\} \xrightarrow{\sim} \{\text{travelers in $\pictW{2}$}\}.
\end{align}
Here we make the intersection of each traveler in $\pictW{\nu}$ with each edge of $\widehat{\tri}$ minimal, by applying the same isotopy for each pair $(\gamma,\varphi(\gamma))$ of travelers. 
Notice that each traveler in $\pictW{\nu}$ is either a closed loop or a compact arc, and their intersections are finite. 
Therefore we can proceed by applying Douglas--Sun's argument \cite[Section 7.4--]{DS20I} for the rest of discussion. 

Recall the notion of a \emph{shared route} of two ordered travelers $(\gamma,\gamma')$ from \cite[Definition 59]{DS20I}. Roughly speaking, it is a maximal interval shared by the routes of two travelers with opposite orientations. The definition is extended for the travelers in $\pictW{\nu}$ in a straightforward way. 
A shared route is either \emph{crossing} or \emph{non-crossing}. A non-crossing shared route is said to be \emph{left-oriented} if one traveler is always seen on the left from the other traveler. A crossing shared route is said to be \emph{left-oriented} if the same situation occurs near its source-end (\cite[Definition 61]{DS20I}). 

By applying the boundary and puncture H-moves if necessary, we may assume that these webs are reduced
%have no boundary H-faces (the left one in \eqref{eq:boundary_H-move}) and puncture H-faces (the left/middle ones in \eqref{eq:puncture_H-move_1} and \eqref{eq:puncture_H-move_2}). 
Then we see that each shared route has at most one intersection point (cf.~\cite[Lemma 60]{DS20I}). Indeed, two intersecting travelers cannot have a common endpoint at a puncture, since such a situation would come from a puncture H-face. Hence the situation regarding the crossing shared routes is exactly the same as in the bounded case.  
From these observations, together with the bijection \eqref{eq:traveler_correspondence_finite}, we get:

\begin{lem}[cf.~{\cite[Corollary 64]{DS20I}}]
For $\nu=1,2$, let $P_{\pictW{\nu}}$ denote the set of intersections of travelers in $\pictW{\nu}$. Then we have a bijection
\begin{align*}
    \varphi_{\mathrm{int}}: P_{\pictW{1}} \xrightarrow{\sim} P_{\pictW{2}}
\end{align*}
such that the unique intersection point $p$ of a left-oriented shared route of two travelers $(\gamma,\gamma')$ in $\pictW{1}$ is sent to the unique intersection point $\varphi_{\mathrm{int}}(p)$ of the corresponding shared route of $(\varphi(\gamma),\varphi(\gamma'))$ in $\pictW{2}$.
\end{lem}

\paragraph{\textbf{Proof of \cref{prop:same_shear}: a sequence of elementary moves relating $W_1$ and $W_2$.}}
As in the previous paragraph, we may assume that $W_1$ and $W_2$ reduced
%have no boundary/puncture H-faces 
by applying the boundary/puncture H-moves. 
Moreover by applying the loop parallel-moves and the arc parallel-moves (\cref{lem:arc-parallel-moves}), we may assume that both $W_1$ and $W_2$ are \emph{left-oriented} in the sense that for each pair of parallel loop or arc components with opposite orientations, one is always seen on the left from the other. It includes the \emph{closed-left-oriented} condition (\cite[Definition 62]{DS20I}). 
Now we are going to see that the intersection points $p \in P_{\pictW{1}}$ and $\varphi_{\mathrm{int}}(p) \in P_{\pictW{2}}$ can be adjusted to a common position by a sequence of modified H-moves: see \cref{fig:adjustment_intersection}.
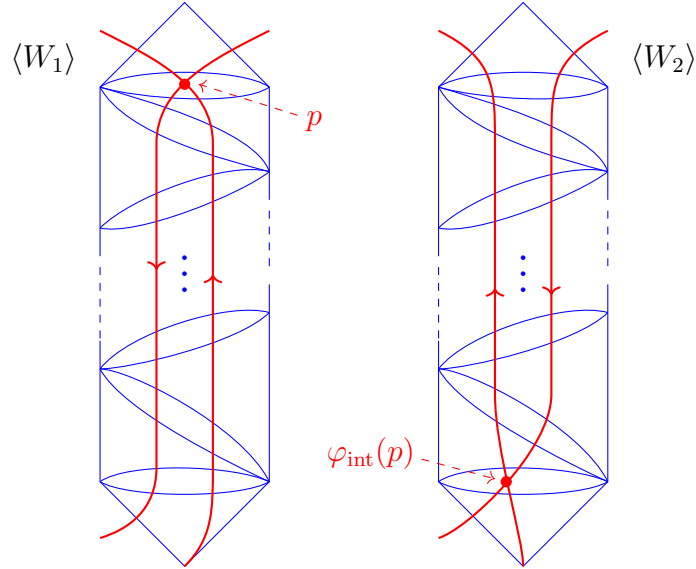
\begin{figure}[h]
    \centering
    \begin{tikzpicture}[scale=.75]
    \draw [blue](-2.5,-1.5) node (v1) {} -- (-2.5,1.5) -- (-1,3) -- (0.5,1.5) -- (0.5,-0.5) node (v2) {};
    \draw [blue](-2.5,-3) -- (-2.5,-5.5) -- (-1,-7) -- (0.5,-5.5) -- (0.5,-2) node (v3) {};
    \draw [blue](-2.5,1.5) .. controls (-2,1.85) and (0,1.85) .. (0.5,1.5);
    \draw [blue](-2.5,1.5) .. controls (-1.85,1.3) and (0,1.15) .. (0.5,1.5);
    \draw [blue](-2.5,1.5) .. controls (-1.85,1.3) and (0.25,0.75) .. (0.5,0);
    \draw [blue](0.5,0) .. controls (-0.25,0.2) and (-2.05,0.65) .. (-2.5,1.5);
    \draw [blue](-2.5,-1) .. controls (-1.95,-0.45) and (-0.25,0.2) .. (0.5,0);
    \draw [blue](-2.5,-1) .. controls (-1.85,-1.15) and (0,-0.4) .. (0.5,0);
    \draw [blue](-2.5,-5.5) .. controls (-2,-5.8) and (0,-5.8) .. (0.5,-5.5);
    \draw [blue](-2.5,-5.5) .. controls (-2,-5.15) and (-0.15,-5.2) .. (0.5,-5.5);
    \draw [blue](0.5,-5.5) .. controls (-0.15,-5.2) and (-2.1,-4.25) .. (-2.5,-3.5);
    \draw [blue](-2.5,-3.5) .. controls (-2,-3.5) and (0,-4.5) .. (0.5,-5.5);
    \draw [blue](-2.5,-3.5) .. controls (-2,-3.5) and (-0.05,-3.05) .. (0.5,-2.5);
    \draw [blue](-2.5,-3.5) .. controls (-2,-2.85) and (-0.2,-2.3) .. (0.5,-2.5);
    \draw [blue, dashed](v1) -- (-2.5,-3);
    \draw [blue, dashed](v2) -- (v3);
    \node[blue, scale=1.5] at (-1,-1.6) {$\vdots$};
    \draw [red, thick, ->-](0.5,2.5) .. controls (-0.5,2) and (-1.5,1.5) .. (-1.5,0.5) .. controls (-1.5,-0.5) and (-1.5,-4.5) .. (-1.5,-5.25) .. controls (-1.5,-6) and (-2,-6.35) .. (-2.5,-6.5);
    \draw [red, thick, ->-](-1,-7) .. controls (-0.5,-6.5) and (-0.5,-6) .. (-0.5,-5.25) .. controls (-0.5,-4.5) and (-0.5,-0.5) .. (-0.5,0.5) .. controls (-0.5,1.5) and (-1.5,2) .. (-2.5,2.5);
    
    \draw [blue](3.5,-1.5) node (v1) {} -- (3.5,1.5) -- (5,3) -- (6.5,1.5) -- (6.5,-0.5) node (v2) {};
    \draw [blue](3.5,-3) -- (3.5,-5.5) -- (5,-7) -- (6.5,-5.5) -- (6.5,-2) node (v3) {};
    \draw [blue](3.5,1.5) .. controls (4,1.85) and (6,1.85) .. (6.5,1.5);
    \draw [blue](3.5,1.5) .. controls (4.15,1.3) and (6,1.15) .. (6.5,1.5);
    \draw [blue](3.5,1.5) .. controls (4.15,1.3) and (6.25,0.75) .. (6.5,0);
    \draw [blue](6.5,0) .. controls (5.75,0.2) and (3.95,0.65) .. (3.5,1.5);
    \draw [blue](3.5,-1) .. controls (4.05,-0.45) and (5.75,0.2) .. (6.5,0);
    \draw [blue](3.5,-1) .. controls (4.15,-1.15) and (6,-0.4) .. (6.5,0);
    \draw [blue](3.5,-5.5) .. controls (4,-5.8) and (6,-5.8) .. (6.5,-5.5);
    \draw [blue](3.5,-5.5) .. controls (4,-5.15) and (5.85,-5.2) .. (6.5,-5.5);
    \draw [blue](6.5,-5.5) .. controls (5.85,-5.2) and (3.9,-4.25) .. (3.5,-3.5);
    \draw [blue](3.5,-3.5) .. controls (4,-3.5) and (6,-4.5) .. (6.5,-5.5);
    \draw [blue](3.5,-3.5) .. controls (4,-3.5) and (5.95,-3.05) .. (6.5,-2.5);
    \draw [blue](3.5,-3.5) .. controls (4,-2.85) and (5.8,-2.3) .. (6.5,-2.5);
    \draw [blue, dashed](v1) -- (3.5,-3);
    \draw [blue, dashed](v2) -- (v3);
    \node[blue, scale=1.5] at (5,-1.6) {$\vdots$};
    \draw [red, thick, ->-](6.5,2.5) .. controls (5.5,2) and (5.5,1.5) .. (5.5,0.5) .. controls (5.5,-0.5) and (5.5,-2.95) .. (5.5,-3.95) .. controls (5.5,-4.95) and (4,-6.35) .. (3.5,-6.5);
    \draw [red, thick, ->-](5,-7) .. controls (5,-6.5) and (4.5,-4.95) .. (4.5,-3.95) .. controls (4.5,-2.95) and (4.5,-0.5) .. (4.5,0.5) .. controls (4.5,1.5) and (4.5,2) .. (3.5,2.5);
    
    \node[red] at (1.3,0.9) {$p$};
    \node[red] at (2.3,-5) {$\varphi_{\mathrm{int}}(p)$};
    \node at (-3.5,2) {$\langle W_1 \rangle$};
    \node at (7.5,2) {$\langle W_2 \rangle$};
    \node [red, fill, circle, inner sep = 1.5] at (-1,1.55) {};
    \node [red, fill, circle, inner sep = 1.5] at (4.7,-5.5) {};
    \draw [red, dashed, ->](3.2,-5.1) -- (4.5,-5.45);
    \draw [red, dashed, ->](1,1) -- (-0.75,1.55);
    \end{tikzpicture}
    \caption{Adjustment of intersection points. Here only the difference from the situation in \cite{DS20I} is that some of the travelers can end at a puncture.}
    \label{fig:adjustment_intersection}
\end{figure}
The techniques developed in \cite[Section 7.8]{DS20I} can be directly applied to our situation without any essential modification, since the situation around a crossing shared route is exactly the same as in the bounded case, and the sets $P_{\pictW{\nu}}$ are finite. Then we get:

\begin{lem}[cf.~{\cite[Lemma 66]{DS20I}}]
There are sequences of modified H-moves applicable to the webs $W_1$ and $W_2$ respectively, after which the bijection $\varphi_{\mathrm{int}}$ satisfies the property that for each intersection point $p$ in $\pictW{1}$, the two points $p$ and $\varphi_{\mathrm{int}}(p)$ lie in the same shared-route-biangle $($\cite[Definition 65]{DS20I}$)$.  
\end{lem}

Apply the sequence of modified H-moves to $W_1$ and $W_2$ prescribed above. We claim that the two signed webs $W_1$ and $W_2$ are now isotopic. 

In the same way as in the proof of \cite[Lemma 67]{DS20I}, we see that the finite sequences of oriented strands on each edge of the split triangulation $\widehat{\tri}$ are the same for $\pictW{1}$ and $\pictW{2}$. We have a correspondence \eqref{eq:traveler_correspondence_finite} that relates the travelers in $\pictW{1}$ and $\pictW{2}$, in particular the ends incident to punctures and their signs. The travelers can intersect with each other inside biangles, whose pattern is uniquely determined by the sequence of oriented strands on the side edges. Thus $\pictW{1}$ and $\pictW{2}$ restricts to the same collection of oriented curves (with signed ends at punctures) in each triangle and biangle in $\widehat{\tri}$. Since we can uniquely recover the honeycombs from these diagrams, we get $W_1=W_2$ up to isotopy. Thus \cref{prop:same_shear} is proved.

\begin{proof}[Proof of \cref{prop:injectivity}]
Let us consider an integral unbounded $\fsl_3$-lamination, which is represented by a signed non-elliptic web $W_1$. Let $W_2:=\xi_\tri \circ \sfx^\uf_\tri(W_1)$ be the signed non-elliptic web obtained from the reconstruction. By \cref{prop:surjectivity}, we have $\sfx^\uf_\tri(W_1)=\sfx^\uf_\tri(W_2)$. Then \cref{prop:same_shear} tells us that $W_1$ and $W_2$ determine an equivalent $\fsl_3$-lamination. Combining with the $\bQ_{>0}$-equivariance, we get the desired assertion. 
\end{proof}

\appendix

\section{Cluster varieties associated with the pair \texorpdfstring{$(\fsl_3,\Sigma)$}{(sl3,Sigma)}}\label{sec:appendix}
Here we briefly recall the general theory of cluster varieties \cite{FG09}, and the construction of the seed pattern $\bs(\fsl_3,\Sigma)$ that encodes the cluster structures of the spaces of $\fsl_3$-laminations in consideration. 

\subsection{Seeds, mutations and the labeled exchange graph}\label{subsec:seeds}
% Let us begin with the Fomin--Zelevinsky's algebraic formulation of \emph{cluster algebras}. 

% \paragraph{\textbf{Seeds and their mutations.}}
Fix a finite set $I=\{1,\dots,N\}$ of indices, and let $\cF_A$ and $\cF_X$ be fields both isomorphic to the field $\bQ(z_1,\dots,z_N)$ of rational functions on $N$ variables. 
We also fix a subset $I_\uf \subset I$ (\lq\lq unfrozen'') and let $I_\f:=I \setminus I_\uf$ (\lq\lq frozen''). 
A \emph{(labeled, skew-symmetric) seed} in $(\cF_A,\cF_X)$ is a triple $(\ve,\mathbf{A},\mathbf{X})$, where
\begin{itemize}
    \item $\ve=(\ve_{ij})_{i,j \in I}$ is a skew-symmetric matrix (\emph{exchange matrix}) with values in $\frac{1}{2}\bZ$ such that $\ve_{ij} \in \bZ$ unless $(i,j) \in I_\f \times I_\f$. 
    \item $\mathbf{A}=(A_i)_{i \in I}$ and $\mathbf{X}=(X_i)_{i \in I}$ are tuples of algebraically independent elements (\emph{cluster $\A$-} and \emph{$\X$-variables}) in $\cF_A$ and $\cF_X$, respectively.
\end{itemize}
The exchange matrix $\ve$ can be encoded in a quiver with vertices parametrized by the set $I$ and $|\varepsilon_{ij}|$ arrows from $i$ to $j$ (resp. $j$ to $i$) if $\varepsilon_{ij} >0$ (resp. $\varepsilon_{ji} >0$). In figures, we draw $n$ dashed arrows from $i$ to $j$ if $\varepsilon_{ij}=n/2$ for $n \in \bZ$, where a pair of dashed arrows is replaced with a solid arrow. 

For an unfrozen index $k \in I_\uf$, the \emph{mutation} directed to $k$ produces a new seed $(\ve',\mathbf{A}',\mathbf{X}')=\mu_k(\ve,\mathbf{A},\mathbf{X})$ according to an explicit formula \cite{FZ-CA4}. See, for instance, \cite[(2.1),(2.3),(2.4)]{IIO21} for a formula which fits in with our convention. 
% \begin{align}
%     \ve'_{ij} &:= 
%     \begin{cases}
%     -\ve_{ij} & \mbox{if $i=k$ or $j=k$}, \\
%     \ve_{ij} + [\ve_{ik}]_+ [\ve_{kj}]_+ - [-\ve_{ik}]_+ [-\ve_{kj}]_+ & \mbox{otherwise},
%     \end{cases} \label{eq:matrix mutation}\\
%     %d'_i&:=d_i,\\
%     A'_i&:=\begin{cases}
%     \displaystyle{ A_k^{-1} \left(\prod_{j \in I}A_j^{[\ve_{kj}]_+} + \prod_{j \in I}A_j^{[-\ve_{kj}]_+}\right)} & \mbox{if $i=k$}, \\
%     A_i & \mbox{if $i \neq k$}, 
%     \end{cases} 
%     \label{eq:A-transf} \\
%     X'_i&:= 
%     \begin{cases}
%     X_k^{-1} & i=k,\\
%     X_i\,(1 + X_k^{-\mathrm{sgn}(\ve_{ik})})^{-\ve_{ik}} & i \neq k.
%   \end{cases} \label{eq:X-transf}
% \end{align}
% Here $[a]_+:=\max\{a,0\}$, and $\sgn (a) \in \{0,\pm 1\}$ denotes the sign for $a \in \bR$. 
A permutation $\sigma \in \mathfrak{S}_{I_\uf} \times \mathfrak{S}_{I_\f}$ induces a transformation $\sigma: (\ve,\mathbf{A},\mathbf{X}) \to (\ve',\mathbf{A}',\mathbf{X}')$ by the rule
\begin{align}\label{eq:seed_permutation}
    \ve'_{ij}:=\ve_{\sigma^{-1}(i),\sigma^{-1}(j)}, \quad 
    A'_i:=A_{\sigma^{-1}(i)},\quad 
    X'_i:=X_{\sigma^{-1}(i)}.
\end{align}
We say that two seeds in $(\cF_A,\cF_X)$ are \emph{mutation-equivalent} if they are transformed to each other by a finite sequence of mutations and permutations. The equivalence class is usually called a \emph{mutation class}. 

The relations among the seeds in a given mutation class $\sfs$ can be encoded in the \emph{(labeled) exchange graph} $\bExch_\sfs$. It is a graph with vertices $v$ corresponding to the seeds $\sfs^{(v)}$ in $\sfs$, together with labeled edges of the following two types:
\begin{itemize}
    \item edges of the form $v \overbar{k} v'$ whenever the seeds $\sfs^{(v)}$ and $\sfs^{(v')}$ are related by the mutation $\mu_k$ for $k \in I_\uf$;
    \item edges of the form $v \overbarnear{\sigma} v'$ whenever the seeds $\sfs^{(v)}$ and $\sfs^{(v')}$ are related by the transposition $\sigma=(j\ k)$ for $(j,k) \in I_\uf \times I_\uf$ or $I_\f \times I_\f$.
\end{itemize}
When no confusion can occur, we simply denote a vertex of the labeled exchange graph by $v \in \bExch_\sfs$ instead of $v \in V(\bExch_\sfs)$. 
When we write $\sfs^{(v)}=(\ve^{(v)},\mathbf{A}^{(v)},\mathbf{X}^{(v)})$, it is known that we have $(\ve^{(v)},\mathbf{A}^{(v)})=(\ve^{(v')},\mathbf{A}^{(v')})$ if and only if $(\ve^{(v)},\mathbf{X}^{(v)})=(\ve^{(v')},\mathbf{X}^{(v')})$ for two vertices $v,v'$ (the \emph{synchronicity phenomenon} \cite{Nak21}). We call $(\ve^{(v)},\mathbf{A}^{(v)})$ and $(\ve^{(v)},\mathbf{X}^{(v)})$ an \emph{$\A$-seed} and an \emph{$\X$-seed}, respectively. We also remark that the labeled exchange graph depends only on the mutation class of the underlying exchange matrices. Indeed, it is unchanged if we transform the cluster variables simultaneously by an automorphism of the ambient field. 
% The quiver corresponding to $\ve^{(v)}$ is denoted by $Q^{(v)}$. 

\begin{rem}\label{rem:cluster_atlas}
In geometric applications, $\A$- and $\X$-seeds are constructed in the field of rational function on a space of interest. For $\cZ \in \{\A,\X\}$, a \emph{cluster $\cZ$-atlas} on a variety (scheme, stack) $V$ is a collection of $\cZ$-seeds in the field $\mathcal{K}(V)$ of rational functions which are mutation-equivalent to each other. A cluster atlas can be uniquely extended to a \emph{cluster $\cZ$-structure}, which is a maximal collection of $\cZ$-seeds in $\mathcal{K}(V)$, thus forming a mutation class $\sfs$. See \cref{rem:moduli} below.
\end{rem}
%In \cref{subsec:cluster_sl3}, we will recall the cluster structures on the moduli spaces $\A_{SL_3,\Sigma}$, $\X_{PGL_3,\Sigma}$ and $\P_{PGL_3,\Sigma}$ related to local systems on $\Sigma$ constructed in \cite{FG06,GS19}.  

\subsection{Cluster varieties}\label{subsec:cluster_variety}
The cluster varieties associated with a mutation class $\sfs$ are constructed by patching algebraic tori parametrized by the vertices of the labeled exchange graph. 

\begin{conv}
A multiplicative algebraic group is denoted by $\bG_m=\Spec \bZ[u,u^{-1}]$. For a lattice $\Lambda$ (\emph{i.e.}, a free abelian group of finite rank), let $\bT_\Lambda:=\Hom (\Lambda,\bG_m)$ denote the associated algebraic torus. 
For a (split) algebraic torus $T \cong (\bG_m)^N$, let 
\begin{align*}
    X^\ast(T):=\Hom (T,\bG_m)\quad \mbox{and} \quad X_\ast(T):=\Hom (\bG_m, T)
\end{align*}
denote the lattices of characters and cocharacters, respectively. These lattices are dual to each other by via the canonical pairing $X_\ast(T) \otimes X^\ast(T) \to \Hom(\bG_m,\bG_m)\cong\bZ$. The contravariant functors $\bT_\bullet: \Lambda \mapsto \bT_\Lambda$ and $X^\ast: T \mapsto X^*(T)$ are inverses to each other: $\Lambda = X^\ast(\bT_\Lambda)$, $T=\bT_{X^\ast(T)}$. A vector $\lambda \in \Lambda$ gives rise to a character $\chi_\lambda:\bT_\Lambda \to \bG_m$. 
\end{conv}

For $v \in \bExch_\sfs$, consider a lattice $N^{(v)}=\bigoplus_{i \in I}\bZ e_i^{(v)}$ with a fixed basis and its dual $M^{(v)}=\bigoplus_{i \in I}\bZ f_i^{(v)}$. Let $\X_{(v)}:=\bT_{N^{(v)}}$ and $\A_{(v)}:=\bT_{M^{(v)}}$ denote the associated algebraic tori of dimension $|I|$. The characters $X_i^{(v)}:=\chi_{e_i^{(v)}}:\X_{(v)} \to \bG_m$ and $A_i^{(v)}:=\chi_{f_i^{(v)}}:\A_{(v)} \to \bG_m$ are called the \emph{cluster coordinates}. 
The exchange matrix $\ve^{(v)}$ defines a $\frac{1}{2}\bZ$-valued bilinear form on $N^{(v)}$ by $(e_i^{(v)},e_j^{(v)}):=\ve^{(v)}_{ij}$, which induces a Poisson and $K_2$-structures on $\X_{(v)}$ and $\A_{(v)}$, respectively. 
The mutation rule turns into  birational maps $\mu_k^x: \X_{(v)} \to \X_{(v')}$ and $\mu_k^a: \A_{(v)} \to \A_{(v')}$, called the \emph{cluster transformations} \cite[(13),(14)]{FG09}. Then the \emph{cluster $\X$-} and \emph{$\A$-varieties} are the schemes defined as
\begin{align*}
    \X_\sfs:= \bigcup_{v \in \bExch_\sfs} \X_{(v)}, \quad  \A_\sfs:= \bigcup_{v \in \bExch_\sfs} \A_{(v)}.
\end{align*}
Here for $(z,\cZ) \in \{(a,\A),(x,\X)\}$, (open subsets of) tori $\cZ_{(v)},\cZ_{(v')}$ are identified via the cluster transformation $\mu_k^z$ if there is an edge of the form $v \overbar{k} v'$, or via the coordinate permutation \eqref{eq:seed_permutation} if there is an edge of the form $v \overbarnear{\sigma} v'$. As a slight variant, let $\X_{(v)}^\uf:=\bT_{N^{(v)}_\uf}$, and $\X_\sfs^\uf:=\bigcup_{v \in \bExch_\sfs} \X^\uf_{(v)}$
the cluster $\X$-variety without frozen coordinates. Since the cluster transformation of unfrozen $\X$-coordinates does not refer the frozen ones, we have a natural projection $\X_\sfs \to \X_\sfs^\uf$. We remark that the cluster varieties are constructed only from the mutation class of the underlying exchange matrices. 

For $(Z,\cZ) \in \{(A,\A),(X,\X)\}$, each pair $(\ve^{(v)},(Z_i^{(v)})_{i \in I})$ of the exchange matrix and the cluster $\cZ$-coordinates defines a $\cZ$-seed in the field $\cF_Z:=\mathcal{K}(\cZ_\sfs)$ of rational functions in the sense of the previous section. The rings $\cO(\A_\sfs) \subset \cF_A$ and $\cO(\X_\sfs) \subset \cF_X$ of regular functions are called the \emph{upper cluster algebra} and the \emph{cluster Poisson algebra}, respectively. The \emph{cluster algebra} \cite{FZ-CA1} is the subring $\mathscr{A}_\sfs \subset \cO(\A_\sfs)$ generated by all the cluster coordinates $A_i^{(v)}$, $i\in I$, $v \in \bExch_\sfs$.

\bigskip
\paragraph{\textbf{Ensemble maps and their extensions.}}\label{subsec:cluster_ensemble}
The cluster varieties $\X_\sfs$ and $\A_\sfs$ are coupled as a \emph{cluster ensemble}. For $v \in \bExch_\sfs$, let $N^{(v)}_\uf \subset N^{(v)}$ denote the sub-lattice spanned by $e_i^{(v)}$ for $i \in I_\uf$. Then by the assumption on the exchange matrix, we have the linear map 
\begin{align}\label{eq:ensemble_map}
    p_{(v)}^{*}: N^{(v)}_\uf \to M^{(v)}, \quad e_i^{(v)} \mapsto \sum_{j \in I}\ve_{ij}^{(v)}f_j^{(v)}.
\end{align}
Moreover, it can be verified that the maps between tori induced by \eqref{eq:ensemble_map} commute with cluster transformations, and combine to give a morphism $p: \A_\sfs \to \X_\sfs^\uf$. We call this map the \emph{ensemble map}, and the triple $(\A_\sfs,\X_\sfs,p)$ the \emph{cluster ensemble} associated with $\sfs$. If we pick up a suitable extension $\widetilde{p}_{(v)}^*:N^{(v)} \to M^{(v)}$ of the map \eqref{eq:ensemble_map} (see \cite[(A.2)]{GHKK} for the required condition), then it still commutes with cluster transformations and hence we get an \emph{extended ensemble map} $\widetilde{p}: \A_\sfs \to \X_\sfs$. It is shown in \cite[Section 13.3]{GS19} that such a choice exactly corresponds to a choice of compatibility pairs \cite{BZ} defining a quantum cluster algebra. 

\bigskip
\paragraph{\textbf{Tropicalizations.}}
The positive structures on the cluster varieties allow us to consider their semifield-valued points. For $\bA=\bZ,\bQ$ or $\bR$, let $\bA^T:=(\bA,\max,+)$ denote the corresponding \emph{tropical semifield} (or the \emph{max-plus semifield}). For an algebraic torus $H$, let $H(\bA^T):=X_\ast(H)\otimes_\bZ (\bA,+)$. A positive rational map $f:H \to H'$ between algebraic tori naturally induces a piecewise-linear (PL for short) map $f^T: H(\bA^T) \to H'(\bA^T)$. We call $f^T$ the \emph{tropicalized map}. 
In particular we have the tropicalized cluster transformations $\mu_k^T:\cZ_{(v)}(\bA^T) \to \cZ_{(v')}(\bA^T)$ for $(\mathsf{z},\cZ) \in \{(\sfa,\A),(\sfx,\X)\}$, explicitly given as:
\begin{align}
    (\mu_k^T)^* \sfx_i^{(v')}&= \begin{cases}
    -\sfx_k^{(v)} & \mbox{if $i=k$},\\
    \sfx_i^{(v)} - \ve_{ik}^{(v)}[ -\sgn(\ve_{ik}^{(v)})\sfx_k^{(v)}]_+ & \mbox{if $i\neq k$}, 
    \end{cases} \label{eq:tropical x-transf}\\
    (\mu_k^T)^* \sfa_i^{(v')}&= \begin{cases}
    -\sfa_k^{(v)} + \max\left\{\sum_{j \in I} [\ve_{kj}^{(v)}]_+\sfa_j^{(v)},\sum_{j \in I} [-\ve_{kj}^{(v)}]_+\sfa_j^{(v)}\right\}  & \mbox{if $i=k$},\\
    \sfa_i^{(v)} & \mbox{if $i\neq k$}.
    \end{cases} \label{eq:tropical a-transf}
\end{align}
Here $\sfx_i^{(v)}$ and $\sfa_i^{(v)}$ are the coordinate functions induced by the basis vectors $e_i^{(v)}$ and $f_i^{(v)}$ respectively, and $[u]_+:=\max\{0,u\}$ for $u \in \bA$. 
We can use them to define the \emph{tropical cluster varieties}
\begin{align*}
    \X_\sfs(\bA^T) := \bigcup_{v \in \bExch_\sfs} \X_{(v)}(\bA^T), \quad \A_\sfs(\bA^T) := \bigcup_{v \in \bExch_\sfs} \A_{(v)}(\bA^T),
\end{align*}
which are naturally equipped with canonical PL structures. Since the PL maps are equivariant for the scaling action of $\bA_{>0}$, the tropical cluster varieties inherit this $\bA_{>0}$-action. We also consider the tropical $\X$-varieties $\X^\uf_\sfs(\bA^T)$ without frozen coordinates. In the body of this paper, the main objects of study are the spaces $\X^\uf_\sfs(\bQ^T)$ and $\X_\sfs(\bQ^T)$ associated with a particular mutation class $\sfs$.

\bigskip
\paragraph{\textbf{Cluster modular group.}}
The cluster ensemble is naturally equipped with a discrete symmetry group. Let $\mathrm{Mat}_\sfs$ denote the mutation class of exchange matrices underlying the mutation class $\sfs$. Then we have a map 
\begin{align*}
    \ve^\bullet: V(\bExch_\sfs) \to \mathrm{Mat}_\sfs, \quad v \mapsto \ve^{(v)}.
\end{align*}
Then the \emph{cluster modular group} $\Gamma_\sfs \subset \mathrm{Aut}(\bExch_\sfs)$ consists of graph automorphism $\phi$ which preserves the fibers of the map $\ve^\bullet$ and the labels on the edges. An element of the cluster modular group is called a \emph{mutation loop}. 
The cluster modular group acts on the cluster varieties $\A_\sfs$ and $\X_\sfs$ so that $\phi^*Z_i^{(v)}=Z_i^{(\phi^{-1}(v))}$ for all $\phi \in \Gamma_\bs$, $v \in \bExch_\bs$ and $i \in I$, where $(Z,\cZ) \in \{(A,\A),(X,\X)\}$. These actions commute with the ensemble map. 
%It also acts on the tori $H_\X$, $H_\A$ by monomial automorphisms, making the cluster exact sequence \eqref{eq:cluster_exact_sequence} $\Gamma_\sfs$-equivariant. 

% Here is an alternative description of these actions. 
% Given $\phi \in \Gamma_\sfs$ and a vertex $v_0 \in \bExch_\sfs$, there exists an edge path $\gamma$ (a \lq\lq mutation sequence'') from $v_0$ to $v:=\phi^{-1}(v_0)$.  Associated to $\gamma$ is a sequence $\mu_\gamma$ of mutations and permutations, satisfying the condition $\ve^{(v)}=\ve^{(v_0)}$.  
% For $(Z,\cZ) \in \{(A,\A),(X,\X)\}$, the restriction of the action of $\phi$ to each cluster chart is given by the composite
% \begin{align*}
%     \phi_{(v_0)}: \cZ_{(v_0)} \xrightarrow{\mu_\gamma} \cZ_{(v)} \xrightarrow{\sim} \cZ_{(v_0)},
% \end{align*}
% where the latter isomorphism is given by $Z_i^{(v_0)} \mapsto Z_i^{(v)}$ for all $i \in I$. The coordinate expression $\phi_{(v_0)}$ is a birational automorphism on $\cZ_{(v_0)}$. 
% A mutation loop can be alternatively defined as a certain equivalence class of such edge paths $\gamma$ in $\bExch_\sfs$ as in \cite{IK20a}. 

Since the actions are by positive rational maps, they induce actions of $\Gamma_\sfs$ on $\A_\sfs(\bA^T)$ and $\X_\sfs(\bA^T)$ by PL automorphisms, which commute with the (extended) ensemble map. Moreover, these actions commute with the rescaling action of $\bA_{>0}$.

%\paragraph{\textbf{Fock--Goncharov duality.}}

\subsection{The cluster ensemble associated with the pair \texorpdfstring{$(\fsl_3,\Sigma)$}{(sl(3),S)}}\label{subsec:cluster_sl3}
Here we quickly recall the cluster structures on the moduli spaces $\A_{SL_3,\Sigma}$, $\X_{PGL_3,\Sigma}$ and $P_{PGL_3,\Sigma}$ constructed in \cite{FG03,GS19}. We are going to recall the \emph{Fock--Goncharov atlas} associated with ideal triangulations of $\Sigma$ and their mutation-equivalences, since it is typical difficult to describe the entire cluster structure.

Let $\tri$ be an ideal triangulation of $\Sigma$. Then we construct a quiver $Q_\tri$ with the vertex set $I(\tri)$ by drawing the quiver
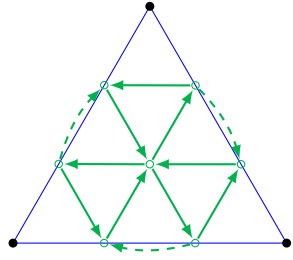
\begin{align*}
\begin{tikzpicture}[scale=0.7]
\draw[blue] (210:3) -- (-30:3) -- (90:3) --cycle;
\foreach \x in {-30,90,210}
\path(\x:3) node [fill, circle, inner sep=1.2pt]{};
\begin{scope}[color=mygreen,>=latex]
\quiverplus{210:3}{-30:3}{90:3};
%dashed arrows
\qdlarrow{x122}{x121}
\qdlarrow{x232}{x231}
\qdlarrow{x312}{x311}
\end{scope}
\end{tikzpicture}
\end{align*}
on each triangle, and glue them by the \emph{amalgamation} construction \cite{FG06a}. In our case, this just means that we glue the quivers on adjacent triangles by identifying the two vertices on the shared edge and eliminate the pair of opposite dashed arrows. The vertices on the boundary intervals of $\Sigma$ are declared to be frozen, forming the subset $I_\f(\tri) \subset I(\tri)$ as in \cref{subsec:notation_marked_surface}. Let $\ve^\tri=(\ve_{ij}^\tri)_{i,j \in I(\tri)}$ be the corresponding exchange matrix. 

These quivers $Q^\tri$ (or the exchange matrices $\ve^\tri$) associated with ideal triangulations of $\Sigma$ are mutation-equivalent to each other. Indeed, the quivers $Q^\tri$, $Q^{\tri'}$ associated with two triangulations $\tri$, $\tri'$ connected by a single flip $f_E:\tri \to \tri'$ are transformed to each other via one of the mutation sequences shown in \cref{fig:flip sequence}. Then the assertion follows from the classical fact that any two ideal triangulations of the same marked surface can be transformed to each other by a finite sequence of flips. 

\begin{figure}
\centering
% \hspace{-5mm}
\begin{tikzpicture}[scale=0.53]
{\color{blue}
\draw (3,0) -- (0,3) -- (-3,0) -- (0,-3) --cycle;
\draw[blue] (0,-3) --node[midway,left=-0.2em]{\scalebox{0.9}{$E$}} (0,3);
}
\draw (-1.5,-2.5) node[anchor=east]{$(\tri,\ell)$};
\foreach \x in {0,90,180,270}
\path(\x:3) node [fill, circle, inner sep=1.2pt]{};
\quiverplus{3,0}{0,3}{0,-3}
\draw[mygreen](x231) node[right=0.2em]{\scalebox{0.9}{$1$}};
\draw[mygreen](G) node[right=0.2em]{\scalebox{0.9}{$4$}};
\draw[mygreen](x232) node[left=0.2em]{\scalebox{0.9}{$3$}};
\draw[mygreen](x311) node[below]{\scalebox{0.9}{$9$}};
\draw[mygreen](x312) node[below]{\scalebox{0.9}{$10$}};
\draw[mygreen](x121) node[above]{\scalebox{0.9}{$11$}};
\draw[mygreen](x122) node[above]{\scalebox{0.9}{$12$}};
\quiverplus{-3,0}{0,-3}{0,3}
\draw[mygreen](G) node[left=0.2em]{\scalebox{0.9}{$2$}};
\draw[mygreen](x311) node[above]{\scalebox{0.9}{$5$}};
\draw[mygreen](x312) node[above]{\scalebox{0.9}{$6$}};
\draw[mygreen](x121) node[below]{\scalebox{0.9}{$7$}};
\draw[mygreen](x122) node[below]{\scalebox{0.9}{$8$}};

\begin{scope}[xshift=4.5cm,yshift=-5cm,scale=0.8,>=latex,yscale=-1]
\draw[blue] (3,0) -- (0,3) -- (-3,0) -- (0,-3) --cycle;
\foreach \x in {0,90,180,270}
\path(\x:3) node [fill, circle, inner sep=1.2pt]{};
{\color{mygreen}
\quiversquare{-3,0}{0,-3}{3,0}{0,3}
    %Arrows
    \qarrow{x241}{x122}
				\qarrow{x122}{x131}
				\qarrow{x131}{x121}
				\qarrow{x121}{x412}
				\qarrow{x412}{x131}
				\qarrow{x131}{x241}
				\qarrow{x241}{x132}
				\qarrow{x132}{x341}
				\qarrow{x341}{x232}
				\qarrow{x232}{x132}
				\qarrow{x132}{x231}
				\qarrow{x231}{x241}
				
				\qarrow{x342}{x242}
				\qarrow{x242}{x411}
				\qarrow{x411}{x342}
				\qarrow{x131}{x242}
				\qarrow{x242}{x132}
				\qarrow{x132}{x131}
}

\end{scope}

\begin{scope}[xshift=4.5cm,yshift=5cm,scale=0.8,>=latex,yscale=-1]
\draw[blue] (3,0) -- (0,3) -- (-3,0) -- (0,-3) --cycle;
\foreach \x in {0,90,180,270}
\path(\x:3) node [fill, circle, inner sep=1.2pt]{};
{\color{mygreen}
\quiversquare{-3,0}{0,-3}{3,0}{0,3}
    %Arrows
    \qarrow{x242}{x342}
				\qarrow{x342}{x132}
				\qarrow{x132}{x341}
				\qarrow{x341}{x232}
				\qarrow{x232}{x132}
				\qarrow{x132}{x242}
				\qarrow{x242}{x131}
				\qarrow{x131}{x121}
				\qarrow{x121}{x412}
				\qarrow{x412}{x131}
				\qarrow{x131}{x411}
				\qarrow{x411}{x242}
				
				\qarrow{x132}{x241}
				\qarrow{x241}{x131}
				\qarrow{x131}{x132}
				\qarrow{x122}{x241}
				\qarrow{x241}{x231}
				\qarrow{x231}{x122}
}
\end{scope}

\begin{scope}[xshift=9cm,>=latex,xscale=-1]
\draw[blue] (3,0) -- (0,3) -- (-3,0) -- (0,-3) --cycle;
\foreach \x in {0,90,180,270}
\path(\x:3) node [fill, circle, inner sep=1.2pt]{};
{\color{mygreen}
\quiversquare{-3,0}{0,-3}{3,0}{0,3}
    %Arrows
    \qarrow{x122}{x241}
				%\qarrow{x122}{x131}
				\qarrow{x131}{x121}
				\qarrow{x121}{x412}
				\qarrow{x412}{x131}
				\qarrow{x241}{x131}
				\qarrow{x132}{x241}
				\qarrow{x132}{x341}
				\qarrow{x341}{x232}
				\qarrow{x232}{x132}
				%\qarrow{x132}{x231}
				\qarrow{x241}{x231}
				\qarrow{x231}{x122}
				
				\qarrow{x342}{x242}
				\qarrow{x242}{x411}
				\qarrow{x411}{x342}
				\qarrow{x131}{x242}
				\qarrow{x242}{x132}
}
\draw[mygreen](x242) node[right=0.2em]{\scalebox{0.9}{$1$
}};
\draw[mygreen](x241) node[left=0.2em]{\scalebox{0.9}{$3$}};
\draw[mygreen](x131) node[below=0.2em]{\scalebox{0.9}{$4$
}};
\draw[mygreen](x132) node[above=0.2em]{\scalebox{0.9}{$2$}};
\draw[mygreen](x342) node[above]{\scalebox{0.9}{$5$}};
\draw[mygreen](x341) node[above]{\scalebox{0.9}{$6$}};
\draw[mygreen](x232) node[below]{\scalebox{0.9}{$7$}};
\draw[mygreen](x231) node[below]{\scalebox{0.9}{$8$}};
\draw[mygreen](x122) node[below]{\scalebox{0.9}{$9$}};
\draw[mygreen](x121) node[below]{\scalebox{0.9}{$10$}};
\draw[mygreen](x412) node[above]{\scalebox{0.9}{$11$}};
\draw[mygreen](x411) node[above]{\scalebox{0.9}{$12$}};

\end{scope}

\begin{scope}[xshift=13.5cm,yshift=5cm,scale=0.8,xscale=-1,rotate=90,>=latex]
\draw[blue] (3,0) -- (0,3) -- (-3,0) -- (0,-3) --cycle;
\foreach \x in {0,90,180,270}
\path(\x:3) node [fill, circle, inner sep=1.2pt]{};
{\color{mygreen}
\quiversquare{-3,0}{0,-3}{3,0}{0,3}
    %Arrows
    \qarrow{x242}{x342}
				\qarrow{x342}{x132}
				\qarrow{x132}{x341}
				\qarrow{x341}{x232}
				\qarrow{x232}{x132}
				\qarrow{x132}{x242}
				\qarrow{x242}{x131}
				\qarrow{x131}{x121}
				\qarrow{x121}{x412}
				\qarrow{x412}{x131}
				\qarrow{x131}{x411}
				\qarrow{x411}{x242}
				
				\qarrow{x132}{x241}
				\qarrow{x241}{x131}
				\qarrow{x131}{x132}
				\qarrow{x122}{x241}
				\qarrow{x241}{x231}
				\qarrow{x231}{x122}
}
\end{scope}

\begin{scope}[xshift=13.5cm,yshift=-5cm,scale=0.8,yscale=-1,rotate=-90,>=latex]
\draw[blue] (3,0) -- (0,3) -- (-3,0) -- (0,-3) --cycle;
\foreach \x in {0,90,180,270}
\path(\x:3) node [fill, circle, inner sep=1.2pt]{};
{\color{mygreen}
\quiversquare{-3,0}{0,-3}{3,0}{0,3}
    %Arrows
    \qarrow{x241}{x122}
				\qarrow{x122}{x131}
				\qarrow{x131}{x121}
				\qarrow{x121}{x412}
				\qarrow{x412}{x131}
				\qarrow{x131}{x241}
				\qarrow{x241}{x132}
				\qarrow{x132}{x341}
				\qarrow{x341}{x232}
				\qarrow{x232}{x132}
				\qarrow{x132}{x231}
				\qarrow{x231}{x241}
				
				\qarrow{x342}{x242}
				\qarrow{x242}{x411}
				\qarrow{x411}{x342}
				\qarrow{x131}{x242}
				\qarrow{x242}{x132}
				\qarrow{x132}{x131}
}
\end{scope}

\begin{scope}[xshift=18cm]
{\color{blue}
\draw (3,0) -- (0,3) -- (-3,0) -- (0,-3) --cycle;
\draw (3,0) --node[midway,above=-0.2em]{\scalebox{0.9}{$E'$}} (-3,0);
}
\foreach \x in {0,90,180,270}
\path(\x:3) node [fill, circle, inner sep=1.2pt]{};
\quiverplus{-3,0}{3,0}{0,3}
\draw[mygreen](G) node[above=0.2em]{\scalebox{0.9}{$1$}};
\draw[mygreen](x121) node[above=0.2em]{\scalebox{0.9}{$2$}};
\draw[mygreen](x122) node[below=0.2em]{\scalebox{0.9}{$4$}};
\draw[mygreen](x311) node[above]{\scalebox{0.9}{$5$}};
\draw[mygreen](x312) node[above]{\scalebox{0.9}{$6$}};
\draw[mygreen](x231) node[above]{\scalebox{0.9}{$11$}};
\draw[mygreen](x232) node[above]{\scalebox{0.9}{$12$}};
\quiverplus{-3,0}{0,-3}{3,0}
\draw[mygreen](G) node[below=0.2em]{\scalebox{0.9}{$3$}};
\draw[mygreen](x231) node[below]{\scalebox{0.9}{$9$}};
\draw[mygreen](x232) node[below]{\scalebox{0.9}{$10$}};
\draw[mygreen](x121) node[below]{\scalebox{0.9}{$7$}};
\draw[mygreen](x122) node[below]{\scalebox{0.9}{$8$}};
\draw (1.5,-2.5) node[anchor=west]{$(\tri',\ell')$};
\end{scope}

\draw[thick,<->] (2,2) --node[midway,above left]{$\mu_1$} (3,3);
\draw[thick,<->] (2,-2) --node[midway,below left]{$\mu_3$} (3,-3);
\draw[thick,<->] (6,3) --node[midway,above right]{$\mu_3$} (7,2);
\draw[thick,<->] (6,-3) --node[midway,below right]{$\mu_1$} (7,-2);
\draw[thick,<->] (16,2) --node[midway,above right]{$\mu_2$} (15,3);
\draw[thick,<->] (16,-2) --node[midway,below right]{$\mu_4$} (15,-3);
\draw[thick,<->] (12,3) --node[midway,above left]{$\mu_4$} (11,2);
\draw[thick,<->] (12,-3) --node[midway,below left]{$\mu_2$} (11,-2);
\end{tikzpicture}
    \caption{Some of the sequences of mutations that realize the flip $f_E: \tri \to \tri'$. Here we partially fix labelings $\ell,\ell'$ of vertices in $I(\tri)$, $I(\tri')$, respectively.}
    \label{fig:flip sequence}
\end{figure}
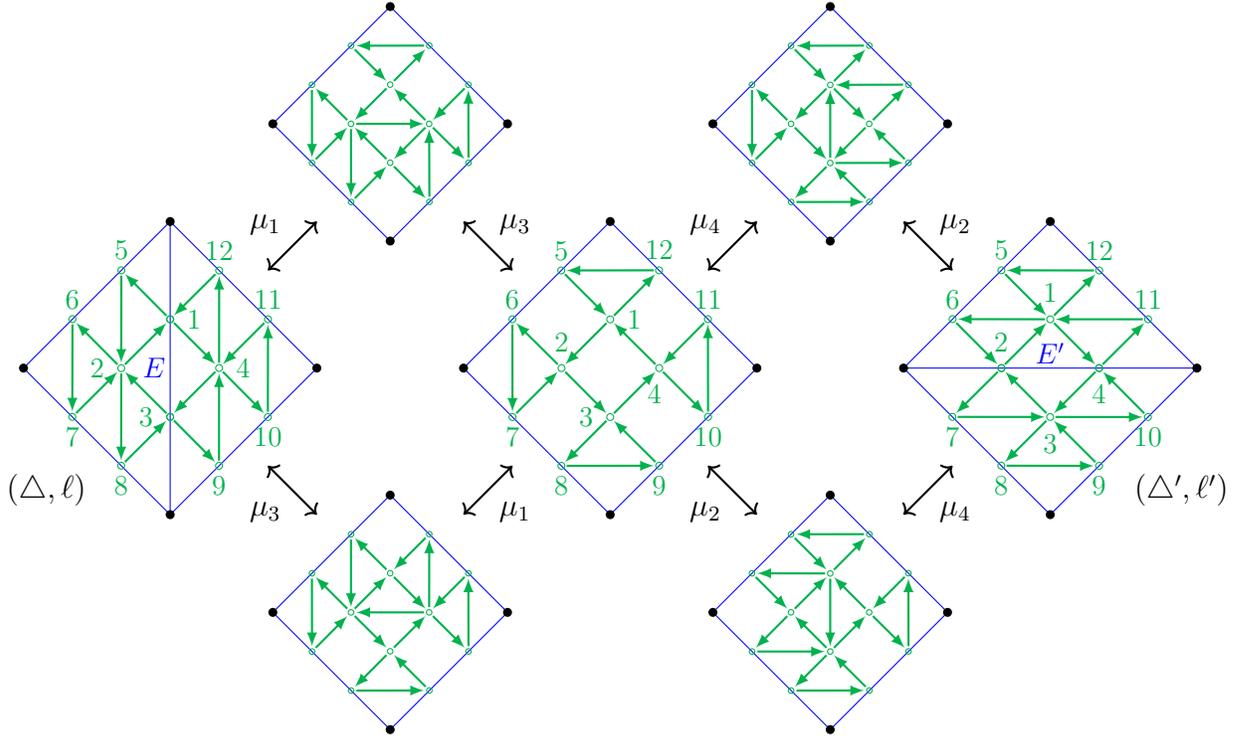

\begin{rem}\label{rem:moduli}
For each ideal triangulation $\tri$, we can associate an $\A$-seed $(\ve^\tri,\mathbf{A}^\tri)$ (resp. $\X$-seed $(\ve^\tri,\mathbf{X}^\tri)$) in the field of rational functions on the moduli space $\A_{SL_3,\Sigma}$ (resp. $\P_{PGL_3,\Sigma}$). Forgetting the frozen part in the latter, we get an $\X$-seed for the moduli space $\X_{PGL_3,\Sigma}$. See \cite[Section 9]{FG03} or \cite[Section 3]{GS19} for construction. These birational coordinate systems define cluster atlases on these moduli spaces in the sense of \cref{rem:cluster_atlas}. 
%A description of the coordinate systems on the positive real part $\X_{PGL_3,\Sigma}(\bR_{>0})$ in terms of convex $\mathbb{RP}^2$-structures on $\Sigma$ is found in \cite{FG06c}.  
\end{rem}
Then there exists a unique mutation class $\sfs(\fsl_3,\Sigma)$ containing the seeds associated with any ideal triangulations $\tri$. 
More precisely, a \emph{labeled $\fsl_3$-triangulation} $(\tri,\ell)$, namely an ideal triangulation $\tri$ together with a bijection $\ell:I(\tri) \to \{1,\dots,N\}$, give rise to vertices of the labeled exchange graph $\bExch_{\sfs(\fsl_3,\Sigma)}$. \cref{fig:flip sequence} describes a subgraph containing $(\tri,\ell)$ and $(\tri',\ell')$, where the labels $\ell$, $\ell'$ are consistently chosen. Let us simply denote the objects related to $\sfs(\fsl_3,\Sigma)$ by
\begin{align*}
    \A_{\fsl_3,\Sigma}:=\A_{\sfs(\fsl_3,\Sigma)},\quad \X_{\fsl_3,\Sigma}:=\X_{\sfs(\fsl_3,\Sigma)}, \quad \bExch_{\fsl_3,\Sigma}:=\bExch_{\sfs(\fsl_3,\Sigma)}, \quad 
    \Gamma_{\fsl_3,\Sigma}:=\Gamma_{\sfs(\fsl_3,\Sigma)},
\end{align*}
and so on. 

The following can be verified from \eqref{eq:tropical x-transf} by a direct computation:
\begin{lem}
For two labeled $\fsl_3$-triangulations $v=(\tri,\ell), v'=(\tri',\ell') \in \bExch_{\fsl_3,\Sigma}$ as in \cref{fig:flip sequence}, the (max-plus) tropical coordinates $\sfx_i:=\sfx_i^{(v)}$ and $\sfx'_i:=\sfx_i^{(v')}$ for $i \in \{1,\dots,12\}$ are related as follows:
\begin{align*}
    \sfx'_1&=\sfx_2 + [\sfx_3,\sfx_4,\sfx_1]_+ - [\sfx_1,\sfx_2,\sfx_3]_+, &
    \sfx'_2&=-\sfx_1-\sfx_2+[\sfx_1]_+ - [\sfx_3]_+, \\
    \sfx'_3&=\sfx_4+[\sfx_1,\sfx_2,\sfx_3]_+-[\sfx_3,\sfx_4,\sfx_1]_+, &
    \sfx'_4&=-\sfx_3-\sfx_4+ [\sfx_3]_+ -[\sfx_1]_+, \\
    \sfx'_5&=\sfx_5+[\sfx_1]_+, & \sfx'_6&=\sfx_6+[\sfx_1,\sfx_2,\sfx_3]_+-[\sfx_1]_+,\\
    \sfx'_7&=\sfx_7+\sfx_1+\sfx_2+[\sfx_3]_+-[\sfx_1,\sfx_2,\sfx_3]_+, & 
    \sfx'_8&=\sfx_8-[-\sfx_3]_+,\\
    \sfx'_9&=\sfx_9+[\sfx_3]_+, & \sfx'_{10}&=\sfx_{10}+[\sfx_3,\sfx_4,\sfx_1]_+-[\sfx_3]_+,\\
    \sfx'_{11}&=\sfx_{11}+\sfx_3+\sfx_4+[\sfx_1]_+-[\sfx_3,\sfx_4,\sfx_1]_+, & \sfx'_{12}&=\sfx_{12}-[-\sfx_1]_+.
\end{align*}
Here $[x]_+:=\max\{0,x\}$ and  $[x,y,z]_+:=\max\{0,x,x+y,x+y+z\}$.
\end{lem}

%Keller's applet (convention corrected)
% y[1]:=1/y1;
% y[2]:=y1*y2/(1 + y1);
% y[4]:=(1 + y1)*y4;
% y[5]:=(1 + y1)*y5;
% y[12]:=y1*y12/(1 + y1);
% y[2]:=(y1*y2 + y1*y2*y3)/(1 + y1);
% y[3]:=1/y3;
% y[4]:=(((1 + y1)*y3)*y4)/(1 + y3);
% y[8]:=y3*y8/(1 + y3);
% y[9]:=(1 + y3)*y9;
% y[1]:=(y2 + y2*y3)/(1 + y1 + y1*y2 + y1*y2*y3);
% y[2]:=(1 + y1)/(y1*y2 + y1*y2*y3);
% y[3]:=(1 + y1 + y1*y2 + y1*y2*y3)/((1 + y1)*y3);
% y[6]:=((1 + y1 + y1*y2 + y1*y2*y3)*y6)/(1 + y1);
% y[7]:=((y1*y2 + y1*y2*y3)*y7)/(1 + y1 + y1*y2 + y1*y2*y3);
% y[1]:=(y2 + y2*y3 + (((1 + y1)*y2)*y3)*y4)/(1 + y1 + y1*y2 + y1*y2*y3);
% y[3]:=((1 + y1 + y1*y2 + y1*y2*y3)*y4)/(1 + y3 + ((1 + y1)*y3)*y4);
% y[4]:=(1 + y3)/(((1 + y1)*y3)*y4);
% y[10]:=((1 + y3 + ((1 + y1)*y3)*y4)*y10)/(1 + y3);
% y[11]:=((((1 + y1)*y3)*y4)*y11)/(1 + y3 + ((1 + y1)*y3)*y4);

% \begin{rem}[Cluster exact sequence in terms of the geometry of moduli spaces]
% We have a canonical projection $p:\A_{SL_3,\Sigma} \to \X_{PGL_3,\Sigma}$ given by forgetting decorations and reducing the gauge group $SL_3$ to $PGL_3$ (\cite[(2.8)]{FG06}). By \cite[Proposition 9.1]{FG06}, this map is identified with the ensemble map $p:\A_{\fsl_3,\Sigma} \to \X^\uf_{\fsl_3,\Sigma}$. There is an action of $H_\A=H_{SL_3}^\bM$ on $\A_{SL_3,\Sigma}$ that rescales the decorations, where $H_{SL_3} \subset SL_3$ denotes the Cartan subgroup. The projection $p$ is a principal $H_\A$-bundle over its image. On the other hand, there is a map $\theta: \X_{PGL_3,\Sigma} \to H_{PGL_3}^{\bP}$ that takes the Cartan part of the monodromy around each puncture, which must be triangular. 
% \end{rem}

\bigskip
\paragraph{\textbf{Goncharov--Shen extension of the ensemble map.}}
Following \cite{GS19}, we choose the following extension of the ensemble map. Let 
\begin{align*}
    C(\fsl_3)=(C_{st})_{s,t \in \{1,2\}}=\begin{pmatrix}
    2 & -1 \\
    -1 & 2
    \end{pmatrix}
\end{align*}
denote the Cartan matrix of the Lie algebra $\fsl_3$. 
For an ideal triangulation $\tri$, let $\widetilde{\ve}^\tri=(\widetilde{\ve}_{ij}^\tri)_{i,j \in I(\tri)}$ be the matrix given by $\widetilde{\ve}_{ij}^\tri:=\ve_{ij}^\tri+m_{ij}$, where 
\begin{align}\label{eq:m-matrix}
    m_{ij}:=\begin{cases}
    -\frac{1}{2}C_{st} & \mbox{if $i=i^s(E),j=i^t(E)$ lie on a common boundary interval $E \in \mathbb{B}$}, \\
    0 & \mbox{otherwise}.
    \end{cases}
\end{align}
Then we define $\widetilde{p}^\ast_\tri: N^\tri \to M^\tri$ by $e_i^\tri \mapsto \sum_{i,j \in I(\tri)} \widetilde{\ve}_{ij}^\tri f_j^\tri$ inducing a morphism 
\begin{align}\label{eq:GS_extension}
    \widetilde{p}_{\mathrm{GS}}: \A_{\fsl_3,\Sigma} \to \X_{\fsl_3,\Sigma},
\end{align}
which we call the \emph{Goncharov--Shen extension of the ensemble map}. This choice naturally comes from the geometry of the moduli spaces of local systems on $\Sigma$, so that $\widetilde{p}_{\mathrm{GS}}$ agrees with the map $p:\A_{SL_3,\Sigma}^\times \to \mathcal{P}_{PGL_3,\Sigma}$ (\cite[Proposition 9.4]{GS19}).

\bigskip
\paragraph{\textbf{Cluster modular group.}}
Although the entire structure of the cluster modular group $\Gamma_{\fsl_3,\Sigma}$ is yet unknown, it is known to include the subgroup $(MC(\Sigma) \times \mathrm{Out}(SL_3)) \ltimes W(\fsl_3)^{\bP} \subset \Gamma_{\fsl_3,\Sigma}$ \cite{GS18}. Here $MC(\Sigma)$ denotes the mapping class group of the marked surface $\Sigma$, $\mathrm{Out}(SL_3))=\mathrm{Aut}(SL_3)/\mathrm{Inn}(SL_3)$ is the outer automorphism group of $SL_3$, and $W(\fsl_3)$ is the Weyl group of the Lie algebra $\fsl_3$. The group $\mathrm{Out}(SL_3)$ has order $2$, and generated by the \emph{Dynkin involution} $\ast: G \to G$, $g \mapsto (g^{-1})^\mathsf{T}$. 
For each element $\phi$ in this subgroup, let us call the induced PL action $\phi:\cZ_{\fsl_3,\Sigma}(\bQ^T) \to \cZ_{\fsl_3,\Sigma}(\bQ^T)$ the \emph{cluster action}, in comparison to the geometric action we introduce in the body of this paper in terms of signed $\fsl_3$-webs.

%%%%%%%%%%%%%%%%%%%%   End of main body of article
%
%                             References
%
%   BiBTeX users uncomment the following line:
%
%\bibliographystyle{gtart}
%


\begin{thebibliography}{GHKK18}

\bibitem[All22]{All}
D. G. L. Allegretti,
\emph{Quantization of canonical bases and the quantum symplectic double},
Manuscripta Math. \textbf{167} (2022), no. 3-4, 613--651.

\bibitem[AK17]{AK}
D. G. L. Allegretti and H. K. Kim,
\emph{A duality map for quantum cluster varieties from surfaces},
Adv. Math. \textbf{306} (2017), 1164--1208.

\bibitem[BZ05]{BZ}
A. Berenstein and A. Zelevinsky,
\emph{Quantum cluster algebras},
Adv. Math. \textbf{195} (2005), no. 2, 405--455.

\bibitem[CF99]{CF}
V.~V.~Fock and L. O. Chekhov, 
\emph{Quantum Teichmüller spaces},
translated from Teoret. Mat. Fiz. \textbf{120} (1999), no. 3, 511--528; Theoret. and Math. Phys. \textbf{120} (1999), no. 3, 1245--1259.

\bibitem[CKKO20]{CKKO20}
S. Y. Cho, H. Kim, H. K. Kim and D. Oh,
\emph{Laurent positivity of quantized canonical bases for quantum cluster varieties from surfaces},
Comm. Math. Phys. \textbf{373} (2020), no. 2, 655--705. 

\bibitem[DM21]{DM}
B. Davison and T. Mandel,
\emph{Strong positivity for quantum theta bases of quantum cluster algebras},
Invent. math. \textbf{226}, 725–843 (2021).

\bibitem[DS20I]{DS20I}
D. C. Douglas and Z. Sun,
{\em Tropical Fock-Goncharov coordinates for $SL_3$-webs on surfaces {I}: construction},
arXiv:2011.01768v1; to appear in Forum Math. Sigma.

\bibitem[DS20II]{DS20II}
D. C. Douglas and Z. Sun,
{\em Tropical Fock-Goncharov coordinates for $SL_3$-webs on surfaces {II}: naturality},
arXiv:2012.14202. 

\bibitem[FG06a]{FG03} 
V. V. Fock and A. B. Goncharov, 
{\em Moduli spaces of local systems and higher Teichm\"uller theory},
Publ. Math. Inst. Hautes \'Etudes Sci., \textbf{103} (2006), 1--211.

\bibitem[FG06b]{FG06a} 
V. V. Fock and A. B. Goncharov, 
{\em Cluster $\mathcal{X}$-varieties, amalgamation and Poisson-Lie groups},
Algebraic geometry and number theory27--68, Progr. Math., \textbf{253}, Birkh\"auser Boston, Boston, MA, 2006.

\bibitem[FG07a]{FG07}
V. V. Fock and A. B. Goncharov, 
{\em Dual Teichm\"uller and lamination spaces},
Handbook of Teichm\"uller theory,  Vol. I, 647--684; IRMA Lect. Math. Theor. Phys., \textbf{11}, Eur. Math. Soc., Z\"urich, 2007. 

\bibitem[FG07b]{FG07c}
V. V. Fock and A. B. Goncharov, 
{\em Moduli spaces of convex projective structures on surfaces},
Adv. Math. \textbf{208} (2007), 249--273. 

\bibitem[FG09]{FG09}
V. V. Fock and A. B. Goncharov, 
\emph{Cluster ensembles, quantization and the dilogarithm},
Ann. Sci. \'Ec. Norm. Sup\'er., \textbf{42} (2009), 865--930.

\bibitem[FG16]{FG16}
V. V. Fock and A. B. Goncharov, 
\emph{Cluster {P}oisson varieties at infinity},
Selecta Math. (N.S.) \textbf{22} (2016), 2569--2589

\bibitem[FKK13]{FKK}
B. Fontaine, J. Kamnitzer and G. Kuperberg, 
\emph{Buildings, spiders, and geometric Satake},
Compos. Math. \textbf{149} (2013), 1871--1912. 

% \bibitem[FP21]{FP21}
% C. Fraser and P. Pykyavskyy,
% {\em Tensor diagrams and cluster combinatorics at punctures},
% arXiv:2107.13069.

\bibitem[FST08]{FST}
S. Fomin, M. Shapiro and D. Thurston,
{\em Cluster algebras and triangulated surfaces. {I}. Cluster complexes},
Acta Math. \textbf{201} (2008), 83--146.

\bibitem[FS22]{FS20}
C. Frohman and A. S. Sikora,
\emph{$SU(3)$-skein algebras and webs on surfaces},
Math. Z. \textbf{300} (2022), 33--56.

\bibitem[FZ02]{FZ-CA1}
S. Fomin and A. Zelevinsky,  
{\em Cluster algebras. I. Foundations}, 
J. Amer. Math. Soc. \textbf{15} (2002), 497--529. 

\bibitem[FZ07]{FZ-CA4}
S. Fomin and A. Zelevinsky,  
{\em Cluster algebras. IV. Coefficients}, 
Compos. Math. \textbf{143} (2007), 112--164. 

\bibitem[GHKK18]{GHKK}
M. Gross, P. Hacking, S. Keel, and M. Kontsevich,
\emph{Canonical bases for cluster algebras},
J. Amer. Math. Soc. \textbf{31} (2018), 497--608. 

% \bibitem[GS15]{GS15}
% A. B. Goncharov and L. Shen,
% {\em Geometry of canonical bases and mirror symmetry},
% Invent. Math. \textbf{202} (2015), 487--633. 

\bibitem[GS18]{GS18} 
A. B. Goncharov and L. Shen,
 {\em Donaldson-Thomas transformations of moduli spaces of $G$-local systems},
Adv. Math. \textbf{327} (2018), 225--348. 

\bibitem[GS19]{GS19}
A. B. Goncharov and L. Shen,
\emph{Quantum geometry of moduli spaces of local systems and representation theory},
arXiv:1904.10491v2.

\bibitem[Ish19]{Ish19}
T. Ishibashi,
{\em On a Nielsen--Thurston classification theory for cluster modular groups},
Annales de l'Institut Fourier, \textbf{69} (2019), 515--560. 

\bibitem[IK21]{IK19}
T. Ishibashi and S. Kano,
{\em Algebraic entropy of sign-stable mutation loops},
Geometriae Dedicata \textbf{214} (2021), 79--118.

\bibitem[IK20a]{IK20a}
T. Ishibashi and S. Kano,
\emph{Sign stability of mapping classes on marked surfaces {I}: empty boundary case},
arXiv:2010.05214.

\bibitem[IK20b]{IK20b}
T. Ishibashi and S. Kano,
\emph{Sign stability of mapping classes on marked surfaces {II}: general case via reductions},
arXiv:2011.14320.

\bibitem[IK]{IKp}
T. Ishibashi and S. Kano,
\emph{Unbounded $\fsl_3$-laminations around punctures},
in preparation.

\bibitem[IIO21]{IIO21}
R. Inoue, T. Ishibashi and H. Oya,
{\em Cluster realizations of Weyl groups and higher Teichmüller theory}, 
Sel. Math. New Ser. \textbf{27}, 37 (2021).

\bibitem[IOS23]{IOS}
T. Ishibashi, H. Oya and L. Shen,
\emph{$\mathscr{A}=\mathscr{U}$ for cluster algebras from moduli spaces of $G$-local systems},
Adv. Math. \textbf{431} (2023).

\bibitem[IY23]{IY21}
T. Ishibashi and W. Yuasa,
{\em Skein and cluster algebras of marked surfaces without punctures for $\mathfrak{sl}_3$},
Math. Z. \textbf{303}, 72 (2023).

\bibitem[Kan23]{Kano_track}
S. Kano,
\emph{Train track combinatorics and cluster algebras},
arXiv:2303.03190.

\bibitem[Kim21]{Kim21}
H. K. Kim,
{\em $SL_3$-laminations as bases for $PGL_3$ cluster varieties for surfaces},
arXiv:2011.14765; to appear in Mem. Am. Math. Soc.	 

\bibitem[Kup96]{Kuperberg}
G. Kuperberg,
\emph{Spiders for rank $2$ {L}ie groups},
Comm. Math. Phys. \textbf{180} (1996), 109--151.

\bibitem[Le16]{Le16}
I. Le,
{\em Higher laminations and affine buildings},
Geom. Topol. \textbf{20} (2016), 1673--1735. 

\bibitem[Le19]{Le19}
I. Le,
 {\em Cluster structure on higher Teichm\"uller spaces for classical groups},
Forum Math. Sigma \textbf{7} (2019), e13, 165 pp.

\bibitem[LY22]{LY22} T. T. Q. L\^{e} and T. Yu, 
\emph{Quantum traces and embeddings of stated skein algebras into quantum tori}, 
Selecta Math. (N.S.) \textbf{28} (2022), no. 4, Paper No. 66, 48 pp.

\bibitem[MQ23]{MQ} T. Mandel and F. Qin, 
\emph{Bracelets bases are theta bases}, 
arXiv:2301.11101. %[math.QA]

% \bibitem[MSW13]{MSW}
% G. Musiker, R. Schiffler and L. Williams,
% \emph{Bases for cluster algebras from surfaces},
% Compos. Math. \textbf{149} (2013), 217--263.

\bibitem[Nak21]{Nak21}
T. Nakanishi,
\emph{Synchronicity phenomenon in cluster patterns},
J. Lond. Math. Soc. (2) \textbf{103} (2021), 1120--1152.

\bibitem[PP93]{PP93}
A. Papadopoulos and R. C. Penner,
{\em The Weil--Petersson symplectic structure at Thurston's boundary},
Trans. Amer. Math. Soc. \textbf{335} (1993), 891--904. 

\bibitem[Pen]{Penner}
R. C. Penner, 
\emph{Decorated {T}eichm\"uller theory}, 
QGM Master Class Series, European Mathematical Society (EMS), Z\"urich, 2012.

\bibitem[Qin21]{Qin21}
F. Qin,
{\em Cluster algebras and their bases},
arXiv:2108.09279.

\bibitem[Thu88]{Thu88} W. P. Thurston, 
{\em On the geometry and dynamics of diffeomorphisms of surfaces}, 
Bull. Amer. Math. Soc. (N.S.) {\bf 19} (1988), 417--431.

\bibitem[Thu14]{Thu14}
D. P. Thurston,
\emph{Positive basis for surface skein algebras},
Proc. Natl. Acad. Sci. USA, \textbf{111} (2014), no.~27, 9725--9732.

\bibitem[Yur23]{Yur21}
T. Yurikusa,
{\em Acyclic cluster algebras with dense $g$-vector fans},
%arXiv:2107.13482.
Adv. Stud. Pure Math. (2023), 437--459. 

\end{thebibliography}
\end{document}